\documentclass[11pt, a4paper,leqno]{amsart}
\usepackage{amsmath,amsthm,amscd,amssymb,amsfonts, amsbsy}
\usepackage{latexsym}
\usepackage{txfonts}
\usepackage{exscale}

\usepackage[colorlinks,citecolor=blue, linkcolor=blue, pagebackref,hypertexnames=false]{hyperref}

\usepackage[utf8]{inputenc}

\usepackage{color}

\calclayout
\allowdisplaybreaks

\theoremstyle{plain}
\newtheorem{theorem}[equation]{Theorem}
\newtheorem{lemma}[equation]{Lemma}
\newtheorem{corollary}[equation]{Corollary}
\newtheorem{proposition}[equation]{Proposition}

\theoremstyle{definition}
\newtheorem{definition}[equation]{Definition}

\theoremstyle{remark}
\newtheorem{remark}[equation]{Remark}

\newtheorem{notation}[equation]{Notation}
\newtheorem{claim}[equation]{Claim}

\numberwithin{equation}{section}

\newcommand{\RR}{{\mathbb{R}}}
\newcommand{\NN}{{\mathbb{N}}}
\newcommand{\ZZ}{{\mathbb{Z}}}

\newcommand{\eps}{\varepsilon}

\newcommand{\dint}{\int\!\!\!\int}

\newcommand{\dist}{\operatorname{dist}}

\newcommand{\re}{\mathbb{R}}
\newcommand{\rn}{\mathbb{R}^n}

\newcommand{\ree}{\mathbb{R}^{n+1}}
\newcommand{\N}{\mathbb{N}}
\newcommand{\dd}{\mathbb{D}}

\newcommand{\C}{\mathrm{CME}}

\newcommand{\po}{\partial\Omega}

\newcommand{\A}{\mathcal{A}}
\newcommand{\F}{\mathcal{F}}
\newcommand{\Qt}{\widetilde{Q}}

\newcommand{\W}{\mathcal{W}}
\newcommand{\B}{\mathcal{B}}

\newcommand{\T}{\mathcal{T}}

\newcommand{\R}{\mathcal{R}}

\newcommand{\sbf}{{\bf S}}

\newcommand{\G}{\mathcal{G}}

\newcommand{\GS}{\Gamma_{\sbf}}

\newcommand{\pom}{\partial\Omega}

\newcommand{\hm}{\omega}
\newcommand{\vp}{\varphi}

\newcommand{\loc}{{\rm loc}}

\renewcommand{\emptyset}{{\textup{\O}}}

\DeclareMathOperator{\supp}{supp}

\DeclareMathOperator{\diam}{diam}

\DeclareMathOperator{\interior}{int}

\def\div{\mathop{\operatorname{div}}\nolimits}

\def\Xint#1{\mathchoice
	{\XXint\displaystyle\textstyle{#1}}%
	{\XXint\textstyle\scriptstyle{#1}}%
	{\XXint\scriptstyle\scriptscriptstyle{#1}}%
	{\XXint\scriptscriptstyle\scriptscriptstyle{#1}}%
	\!\int}
\def\XXint#1#2#3{{\setbox0=\hbox{$#1{#2#3}{\int}$}
		\vcenter{\hbox{$#2#3$}}\kern-.5\wd0}}

\def\fint{\Xint-}

\begin{document}
\allowdisplaybreaks

\title[Transference of estimates]{Transference of scale-invariant estimates from Lipschitz to Non-tangentially accessible to Uniformly rectifiable domains}

\author{Steve Hofmann}

\address{Steve Hofmann
\\
Department of Mathematics
\\
University of Missouri
\\
Columbia, MO 65211, USA} \email{hofmanns@missouri.edu}

\author{Jos\'e Mar{\'\i}a Martell}

\address{Jos\'e Mar{\'\i}a Martell\\
Instituto de Ciencias Matem\'aticas CSIC-UAM-UC3M-UCM\\
Consejo Superior de Investigaciones Cient{\'\i}ficas\\
C/ Nicol\'as Cabrera, 13-15\\
E-28049 Madrid, Spain} \email{chema.martell@icmat.es}

\author{Svitlana Mayboroda}

\address{Svitlana Mayboroda
\\
Department of Mathematics
\\
University of Minnesota
\\
Minneapolis, MN 55455, USA} \email{svitlana@math.umn.edu}

\thanks{The first author was supported by NSF grant DMS-1664047. The second author 
acknowledges financial support from the Spanish Ministry of Science and Innovation, through the ``Severo Ochoa Programme for Centres of Excellence in R\&D'' (SEV-2015-0554) and from the Spanish National Research Council, through the ``Ayuda extraordinaria a Centros de Excelencia Severo Ochoa'' (20205CEX001). He also acknowledges that 
the research leading to these results has received funding from the European Research
Council under the European Union's Seventh Framework Programme (FP7/2007-2013)/ ERC
agreement no. 615112 HAPDEGMT.	The third author was  supported in part by the NSF INSPIRE Award DMS-1344235, NSF CAREER
Award DMS-1220089,  NSF-DMS 1839077, Simons Fellowship, and the Simons Foundation grant 563916, SM}

\date{\today}
\subjclass[2010]{28A75, 28A78, 31B05, 
42B20, 42B25, 42B37} 

\keywords{Carleson measures, square functions, non-tangential maximal functions, 
$\eps$-approximability, 
uniform rectifiability, harmonic functions.}

\begin{abstract} In relatively nice geometric settings, in particular, on Lipschitz domains,  absolute continuity of elliptic measure with respect to the surface  measure is equivalent to Carleson measure estimates,  to square function estimates, and to $\varepsilon$-approximability,  for solutions to the second order divergence form elliptic partial differential equations $ Lu= -{\rm div\,} (A \nabla u)=0$.   In more general situations, notably, in an open set $\Omega$ with a uniformly rectifiable boundary, absolute continuity of elliptic measure with respect to the  surface  measure may fail, already for the Laplacian.  In the present paper, the authors demonstrate that  nonetheless,  Carleson measure estimates, square function estimates, and 
$\varepsilon$-approximability remain valid in such $\Omega$, for solutions of $Lu=0$, 
provided that such solutions enjoy these properties
in Lipschitz subdomains of $\Omega$.  

Moreover, we establish a general real-variable transference principle, from Lipschitz to chord-arc domains, and from chord-arc to open sets with uniformly rectifiable boundary, that is not restricted to harmonic functions or  even to solutions of elliptic equations. In particular, this allows one to deduce the first Carleson measure estimates and square function  bounds for higher order systems on open sets with uniformly rectifiable boundaries and to treat subsolutions and subharmonic functions.
\end{abstract}

\maketitle

\tableofcontents

\section{Introduction}

In the setting of a Lipschitz domain $\Omega\subset \ree,\, n\geq 1$,
for any divergence form elliptic operator $L=-\div (A\nabla)$ with bounded measurable
coefficients,
the following are equivalent:
\begin{list}{$(\theenumi)$}{\usecounter{enumi}\leftmargin=.8cm
\labelwidth=.8cm\itemsep=0.2cm\topsep=.1cm
\renewcommand{\theenumi}{\roman{enumi}}}

\item Every bounded solution $u$, of the equation $Lu=0$ in $\Omega$, satisfies the
{\it Carleson measure estimate} (see Definition~\ref{defCME} with $F=|\nabla u|/\|u\|_{L^\infty(\Omega)}$).

\item Every bounded solution $u$, of the equation $Lu=0$ in $\Omega$, is {\it $\eps$-approximable}, for every $\eps>0$
(see Definition \ref{def1.3}).

\item The elliptic measure associated to $L$, $\omega_L$, is (quantitatively) absolutely continuous with respect to the Lebesgue measure,  $\hm_L\in A_\infty(\sigma)$ on $\pom$.

\item  Uniform {\it Square function/Non-tangential maximal function} (``$S/N$")
estimates 
hold locally in  ``sawtooth" subdomains of $\Omega$ (see Definition~\ref{defsfntmax:traditional}).

\end{list}

Historically, Dahlberg \cite{D} obtained an
extension of Garnett's $\eps$-approximability result,
observing that
$(iv)$ implies $(ii)$ in the harmonic case\footnote{This implication holds more generally for null solutions of
divergence form elliptic equations, see \cite{KKPT} and \cite{HKMP}.}.
The explicit
connection of $\eps$-approximability
with the $A_\infty$ property of harmonic measure, i.e.,
that $(ii)\implies(iii)$, appears in \cite{KKPT}
(where this implication is established not only for the Laplacian,
but for general divergence form elliptic operators).
That $(iii)$ implies $(iv)$ is proved for harmonic functions in \cite{D2}\footnote{And thus all four properties
hold for harmonic functions in Lipschitz domains, by the result of \cite{D1}.}, and, for null solutions of
general divergence form elliptic operators,
in \cite{DJK}.
Finally, Kenig, Kirchheim, Pipher and Toro \cite{KKiPT} have recently shown  
that $(i)$ implies $(iii)$, whereas, on the other hand,  
$(i)$ may be seen, via good-lambda and John-Nirenberg arguments, to be equivalent to the local version of one 
direction of $(iv)$ (the ``$S<N$" direction)\footnote{We will prove this fact in much greater generality in this paper.}.

The main goal of the present paper is to show that while $(iii)$ may fail on general uniformly rectifiable domains even for harmonic functions \cite{BiJo} or might be not applicable in the absence of a suitable concept of elliptic measure (e.g., for systems), $(i), (ii)$ and $(iv)$ carry over from Lipschitz domains to uniformly rectifiable sets by a purely real variable mechanism.  In particular, this both extends and clarifies our previous work in \cite{HMM2}.
But let us start with more historical context.

In the past several decades,  
uniformly rectifiable sets have been identified as the most general geometric setting in which 
many standard harmonic-analytic properties continue to hold. In particular, it was shown in the early 
90's that uniform rectifiability of a set $E$ is equivalent to boundedness of all sufficiently nice
singular integral operators with odd kernels in $L^2(E)$ \cite{DS1}, and, much more recently, that 
uniform rectifiability is equivalent to boundedness of the Riesz transform in $L^2(E)$ (see \cite{MMV} 
for the case $n=1$, and 
\cite{NToV} in general). 

However, it seemed to be vital for many
standard boundary estimates for solutions of elliptic PDEs in a domain $\Omega$ that, in addition to 
uniform rectifiability of its boundary, $\Omega$ should possess 
some additional topological features, ensuring a reasonably nice approach to the boundary.   In some 
respects, this is indeed true.
In particular, it has been known that $(i)$\,--\,$(iv)$ hold for harmonic functions on chord-arc domains, that is, 
non-tangentially accessible domains with Ahlfors-David regular boundaries (see Definitions \ref{defadr} 
and \ref{def1.nta} below, and \cite{JK, DJK, DJe}).
Such domains satisfy an interior and exterior corkscrew condition (quantitative openness) and 
a Harnack chain 
condition (quantitative connectedness). 
At the same time, the 
 counterexample of Bishop and Jones \cite{BiJo} showed that absolute continuity of harmonic measure 
with respect to the Lebesgue measure $(iii)$ may fail on a general set with a uniformly rectifiable boundary:  
they construct a one dimensional
(uniformly) rectifiable set
$E$ in the complex plane, for which harmonic measure with respect to $\Omega= \mathbb{C}\setminus E$,
is singular with respect to Hausdorff $H^1$ measure on $E$.   
Much more recently, under the natural and rather minimal background 
assumptions that $\Omega$ satisfies an interior corkscrew condition, and has an
Ahlfors-David regular boundary, 
quantitative absolute continuity of harmonic 
measure with respect to surface measure 
(either property $(iii)$ above, or the weak-$A_\infty$ property, i.e., property $(iii)$ in the absence of doubling),
has now been characterized in the harmonic case, thus establishing the necessity of some 
connectivity assumption in this context: property 
$(iii)$ (respectively, its weaker non-doubling version)
is equivalent to uniform rectifiability of $\pom$, along with some version of accessibility to the boundary, 
either the semi-uniformity condition of \cite{AH} in the doubling case \cite{Az}, or respectively, 
the ``weak local John condition", which entails access to an ample
portion of the boundary, locally, from each interior point of $\Omega$  \cite{AHMMT}. 
Thus, while some connectivity
is indeed required to obtain property $(iii)$,
in \cite{HMM2} the authors proved that, nonetheless, Carleson measure estimates $(i)$ and 
$\eps$-approximability $(ii)$ for harmonic functions (and implicitly, for solutions of a certain more 
general class of elliptic equations) remain valid on all domains with a 
uniformly rectifiable boundary, in the absence of any connectivity assumption.  Shortly thereafter, 
it was shown that, at least in the presence of interior corkscrew points, each of the necessary
properties $(i)$ and $(ii)$ is also {\it sufficient} for uniform rectifiability \cite{GMT}. 

The present paper introduces a new transference mechanism, which illustrates that
for certain classes of
scale-invariant estimates (e.g., Carleson measure bounds, 
or square function/non-tangential maximal function estimates) the passage from such estimates on Lipschitz domains 
 to analogous results on chord-arc domains and further to the same bounds on all open sets with 
uniformly rectifiable boundaries is, in fact, a real variable phenomenon. That is, for a given function $F$
defined in the
complement of a co-dimension 1, uniformly rectifiable set $E\subset \ree$, if
one has suitable bounds for $F$ on Lipschitz domains, then these 
automatically carry over to $\ree\setminus E$. 
This immediately entails a series of new results in very general PDE settings (for solutions of second order elliptic PDEs with coefficients satisfying a Carleson measure condition, for solutions of higher order systems, for non-negative subsolutions), but clearly the power of having a general, purely real-variable scheme, goes beyond these applications. Let us now discuss the details.  
 
\begin{definition}[\textbf{ADR}]\label{defadr}
We say that a  set $E \subset \ree$ is $n$-dimensional \textbf{Ahlfors-David regular} (or simply \textbf{ADR})
if it is closed, and if there is some uniform constant $C\ge 1$ such that
\begin{equation} \label{eq1.ADR}
C^{-1}r^n \leq \sigma\big(\Delta(x,r)\big)
\leq C\, r^n,\quad\forall r\in(0,\diam (E)),\ x \in E,
\end{equation}
where $\diam(E)$ may be infinite.
Here, $\Delta(x,r):= E\cap B(x,r)$ is the surface ball of radius $r$,
and $\sigma:= H^n|_E$ 
is the surface measure on $E$, where $H^n$ denotes $n$-dimensional
Hausdorff measure.
\end{definition}

\begin{definition}[\textbf{UR} and \textbf{UR character}]\label{defur} 
An $n$-dimensional ADR (hence closed) set $E\subset \ree$
is $n$-dimensional \textbf{uniformly rectifiable} (or simply \textbf{UR}) if and only if it contains \textbf{big pieces of
Lipschitz images} of $\rn$ (\textbf{BPLI}). This means that there are positive constants $\theta, 
M_0>1$, such that for each
$x\in E$ and each $r\in (0,\diam (E))$, there is a
Lipschitz mapping $\rho= \rho_{x,r}: \rn\to \ree$, with Lipschitz constant
no larger than $M_0$,
such that 
$$
H^n\Big(E\cap B(x,r)\cap  \rho\left(\{z\in\rn:|z|<r\}\right)\Big)\,\geq\,\theta^{-1} r^n\,.
$$
Additionally, the \textbf{UR character} of $E$ is just the triple of constants $(\theta,M_0, C)$ where $C$ is the ADR constant; or equivalently,
the quantitative bounds involved in any particular characterization of uniform rectifiability.
\end{definition}

Note that, in particular, a UR set is closed by definition, so that $\ree\setminus E$ is open,
but need not be connected.

We recall that $n$-dimensional rectifiable sets are characterized by the
property that they can be
covered, up to a set of
$H^n$-measure 0, by a countable union of Lipschitz images of $\rn$;
we observe that BPLI  is a quantitative version
of this fact.

It is worth mentioning that there exist sets that are ADR (and that even form the boundary of an open set satisfying 
interior Corkscrew and Harnack Chain conditions),
but that are totally non-rectifiable (e.g., see the construction of Garnett's ``4-corners Cantor set"
in \cite[Chapter1]{DS2}).  

\begin{definition}[\textbf{Corkscrew condition}]  \label{def1.cork}
Following
\cite{JK}, we say that an open set $\Omega\subset \ree$
satisfies the \textbf{Corkscrew condition} if for some uniform constant $C>1$ and
for every surface ball $\Delta:=\Delta(x,r)=B(x, r)\cap\po,$ with $x\in \partial\Omega$ and
$0<r<\diam(\partial\Omega)$, there is a ball
$B(X_\Delta,C^{-1}r)\subset B(x,r)\cap\Omega$.  The point $X_\Delta\subset \Omega$ is called
a \textbf{Corkscrew point} relative to $\Delta.$  We note that  we may allow
$r<C'\diam(\pom)$ for any fixed $C'$, simply by adjusting the constant $C$.
\end{definition}

\begin{definition}[\textbf{Harnack Chain condition}]\label{def1.hc} 
Again following \cite{JK}, we say that an open set $\Omega$ satisfies the \textbf{Harnack Chain condition} if there is a uniform constant $C\ge 1$ such that
for every pair of points $X, X'\in \Omega$
there is a chain of balls $B_1, B_2, \dots, B_N\subset \Omega$ with
\[
N \leq  C\Big (2+\log_2^+ \frac{|X-X'|}{\min\{\dist(X,\pom), \dist(X',\pom)\}}\Big),
\]
$X\in B_1,\, X'\in B_N,$ $B_k\cap B_{k+1}\neq \emptyset$ for every $1\le k\le N-1$,
and $C^{-1}\diam (B_k) \leq \dist (B_k,\partial\Omega)\leq C\diam (B_k)$ for every $1\le k\le N$.  The chain of balls is called
a \textbf{Harnack Chain}.	 We remark that in general, the estimate for $N$ can be worse than logarithmic,
but as is well known, in the presence of an interior corkscrew condition, it is necessarily logarithmic if it holds at all.
\end{definition}

\begin{definition}[\textbf{NTA}, \textbf{1-sided NTA}, \textbf{CAD}, and  \textbf{1-sided CAD}]\label{def1.nta}
We say that an open set $\Omega\subset \ree$ is \textbf{1-sided non-tangentially accessible} (or simply \textbf{1-sided NTA}) if it satisfies the
Harnack Chain condition, and $\Omega$ satisfies the (interior) Corkscrew condition.	Additionally, the \textbf{1-sided NTA character} of $\Omega$ is just the collection of constants involved in the fact that $\Omega$ is 1-sided NTA, that is, the (interior) corkscrew constant, as well as the constant from the Harnack chain condition. 	

As in \cite{JK}, we say that an $\Omega\subset \ree$ is \textbf{non-tangentially accessible} (or simply \textbf{NTA}) if it satisfies the
Harnack Chain condition, and if both $\Omega$ and $\Omega_{\rm ext}:= \ree\setminus \overline{\Omega}$ satisfy the Corkscrew condition. The \textbf{NTA character} of $\Omega$ is the collection of constants involved in the fact that $\Omega$ is NTA, that is, the interior and exterior corkscrew constants, as well as the constant from the Harnack chain condition. 

We say that an open set $\Omega\subset \ree$ is a \textbf{1-sided chord-arc domain}, or simply \textbf{1-sided CAD}, (resp. \textbf{chord-arc domain}, or simply \textbf{CAD}) if it is 1-sided NTA (resp. NTA) and has ADR boundary. The \textbf{1-sided CAD character} (resp. \textbf{CAD character}) is the 1-sided NTA character (resp. NTA character) together with the ADR constant.
\end{definition}

\begin{definition}[\textbf{Lipschitz graph domain}]\label{dD} 
	We say that $\Omega\subset\RR^{n+1}$ is a \textbf{Lipschitz graph domain} if there is some Lipschitz function 
	$\psi:\RR^{n} \longrightarrow\RR$ 
	and some coordinate system such that
	\begin{equation*}
	\Omega 
	=\{(x',t):\ x'\in\RR^{n},\  t>\psi(x')\}. 
	\end{equation*}
	We refer to $ M =\|\nabla\psi\|_{L^\infty(\RR^n)}$ as the  \textbf{Lipschitz constant} of $\Omega$.
	
\end{definition}

\begin{definition}[\textbf{Bounded Lipschitz domain}]  	We say that and open connected set $\Omega\subset\RR^{n+1}$ is a \textbf{bounded Lipschitz domain} if there exist $r_\Omega>0$, $M, C_0, m\ge 1$, 
	$\{x_j\}_{j=1}^{m}\subset \partial\Omega$, $\{r_j\}_{j=1}^m$ with $C_0^{-1}\,r_\Omega<r_j<C_0\, r_\Omega$ for every $1\le j\le m$ such that 
	the following conditions hold. First, $	\partial \Omega\subset \bigcup_{j=1}^{ m } B(x_j,r_j)$. Second, for each $1\le j\le m$ there is some Lipschitz graph domain~$V_j$, with $x_j\in\partial V_j$ and with Lipschitz constant at most~$M$, such that $	U_j\cap\Omega = U_j\cap V_j$	where $U_j$ is a cylinder of height $8(M+1)r_j$, radius~$2r_j$, and with axis parallel to the $t$-axis (in the coordinates associated with~$V_j$). 	We refer to the triple  $(M,m,C_0)$  
	as the  \textbf{Lipschitz character} of $\Omega$. 
\end{definition}

As we pointed out above and as can be seen from the definitions, non-tangentially accessible domains possess certain quantitative topological features. One can show that a CAD satisfies a property analogous to Definition~\ref{defur}, but using Big Pieces of Lipschitz Subdomains, rather than Big Pieces of Lipschitz Images (see Proposition~\ref{IBPLS}), the crucial difference being that in some sense, a nice access to the boundary of a Lipschitz domain is  retained, contrary to the general UR case.

Finally, let us define the scale-invariant estimates at the center of this paper.
\begin{definition}[\textbf{CME}]\label{defCME} Let $\Omega\subset \ree$ be an open set and let  $F\in
L^2_{\rm loc}(\Omega)$.
We say that $F$ satisfies the \textbf{Carleson measure estimate} (or simply \textbf{CME}) on $\Omega$ if 
\begin{equation}\label{eqdefCME}
 \|F\|_{\C(\Omega)}:=\sup_{x\in \pom,\, 0<r<\infty} \,\frac1{r^n}\dint_{B(x,r)\cap \Omega} |F(Y)|^2 \dist(Y,\pom) \,dY<\infty.
\end{equation}
\end{definition}

\begin{definition}[\textbf{$\eps$-approximable}]\label{def1.3} 
Let $\Omega\subset \ree$ be an open set. Let $u\in L^\infty(\Omega)$,
with $\|u\|_{L^\infty(\Omega)} \leq 1$, and let $\eps \in (0,1)$.   We say that $u$
is \textbf{$\eps$-approximable} on $\Omega$, if there is a constant $C_\eps$, and
a function $\vp =\vp^\eps\in W^{1,1}_{\rm loc}(\Omega)$
satisfying
\begin{equation}\label{eq1.4}
\|u-\vp\|_{L^\infty(\Omega)}<\eps\,,
\end{equation}
and
\begin{equation}\label{eq1.5}
\sup_{x\in \pom,\, 0<r<\infty} \,\frac1{r^n}\dint_{B(x,r)\cap\Omega}|\nabla \vp(Y)| \,dY\leq C_\eps\,.
\end{equation}
\end{definition}

Let $\Omega$ be an open set. The cone with vertex at $x\in\pom$ and aperture $\kappa>0$ is defined as 
\begin{equation}\label{conetrad}
\Gamma_{\Omega} (x):= \Gamma_{\Omega,\kappa} (x):= \{Y\in \Omega\cap B(x,r):\, |Y-x|\leq (1+\kappa) \dist(Y, \po)\}, \quad x\in \po. 
\end{equation}
Given $r>0$, we write $\Gamma_{\Omega}^r (x):=\Gamma_{\Omega} (x)\cap B(x,r)$ for the truncated cone. With a slight abuse of notation if $\Omega$ is unbounded and $\pom$ bounded, our cones will be truncated. More precisely, in that scenario, we will write $\Gamma_{\Omega}(\cdot)$ to denote $\Gamma_{\Omega}^{C \diam(\pom)}(\cdot)$ where $C\ge 2$ is a fixed harmless constant. In this way,  when $\pom$ is bounded, so are the cones, all being contained in a $C'\diam(\pom)$-neighborhood of $\pom$. 
We will sometimes refer to these cones as ``traditional'' to distinguish them from some dyadic cones  which will be introduced later, see \eqref{defcone}.

\begin{definition}[\textbf{Non-tangential maximal function}, \textbf{Area integral}, and \textbf{Square function}]\label{defsfntmax:traditional} 
Let $\Omega$ be an open set. For $H\in C(\Omega)$ (i.e., $H$ is continuous function in $\Omega$) we define the \textbf{non-tangential
	maximal function} as
	\begin{equation}\label{defN*-regular}
	N_{*,\Omega}H(x):=N_{*,\Omega,\kappa}H(x):= \sup_{Y\in \Gamma_{\Omega,\kappa}(x)} |H(Y)|\,,\qquad x\in \pom;
	\end{equation}
	 for 
	$G\in L^2_{\rm loc}(\Omega)$, we define the \textbf{area integral}  as
	\begin{equation}\label{defA-regular}
	\A_{\Omega} G(x):=\A_{\Omega,\kappa} G(x):=\left(\dint_{\Gamma_{\Omega,\kappa}(x)}|G(Y)|^2 \dist(Y,\pom)^{1-n} dY\right)^{\frac12}\,,\qquad x\in \pom;
	\end{equation}
	and, for $u\in W^{1,2}_{\loc}(\Omega)$, we define the \textbf{square function}  as
	\begin{equation}\label{defS-regular}
	S_{\Omega} u(x):=S_{\Omega,\kappa} u(x):=\left(\dint_{\Gamma_{\Omega,\kappa}(x)}|\nabla u(Y)|^2 \dist(Y,\pom)^{1-n} dY\right)^{\frac12}\,,\qquad x\in \pom.
	\end{equation}
	For any $r>0$, we write $N_{*,\Omega}^r$, $\A_{\Omega}^r$, and $S_{\Omega}^r$ to denote the \textbf{truncated} non-tangential maximal function, area integral, and square function respectively, where $\Gamma_\Omega(\cdot)$ is replaced by the truncated cone $\Gamma_\Omega^r(\cdot)$.
		\end{definition}


Let us now list some highlights of main results of this paper (see Corollary~\ref{corol:CME:Lip->UR}, Theorem~\ref{theor:CME:CAD->UR} and Theorem~\ref{theor:CME:Lip->CAD} for the precise statements in the body of the paper and also Notation \ref{notation:constants}). First, the Carleson measure estimates on Lipschitz domains imply the Carleson measure estimates in CAD, which, in turn, imply the Carleson measure estimates on the sets with UR boundaries, via the following formalism.

\begin{theorem}[Transference of Carleson measure estimates]\label{t0.1}\footnote{In the statement we have omitted the dependence in the Carleson estimates on the various geometric parameters. The precise statements (see Theorem~\ref{theor:CME:CAD->UR} and Theorem~\ref{theor:CME:Lip->CAD}) given in the body of the paper impose that the Carleson measure estimates hold for any bounded Lipschitz (resp. chord-arc) subdomain with a bound depending on the Lipschitz (resp.~CAD) character. The latter means that for all subdomains with Lipschitz (resp. CAD) character 
controlled by some uniform quantity, say $M$,  the corresponding Carleson measure estimates hold with an associated 
uniform constant depending on $M$. The conclusions should also include that the resulting  Carleson estimates depend on the CAD character of $D$ (resp. UR character of $E$) as well as on the Carleson estimates of $F$ in the subdomains.}  

\begin{list}{$(\theenumi)$}{\usecounter{enumi}\leftmargin=.8cm
		\labelwidth=.8cm\itemsep=0.2cm\topsep=.01cm
		\renewcommand{\theenumi}{\roman{enumi}}}
 	
	\item Let $D\subset \ree$ be a chord-arc domain and $F\in L^2_{\loc}(D)$. If $F$ satisfies the Carleson measure estimate on all bounded Lipschitz subdomains of $D$ then $F$ satisfies the Carleson measure estimate on $D$ as well.

	\item Let $E\subset \ree$ be an $n$-dimensional uniformly rectifiable set and  let $F\in L^2_{\loc}(\ree\setminus E)$. If $F$ satisfies the Carleson measure estimate on all bounded chord-arc subdomains of $\ree\setminus E$,  then $F$ satisfies the Carleson measure estimate on $\ree\setminus E$ as well.
	
	\item Let $E\subset \ree$ be an $n$-dimensional uniformly rectifiable set and  let $F\in L^2_{\loc}(\ree\setminus E)$. If $F$ satisfies the Carleson measure estimate on all bounded Lipschitz subdomains of $\ree\setminus E$,  then $F$ satisfies the Carleson measure estimate on $\ree\setminus E$ as well.

\end{list}

\end{theorem}

%

Secondly, in the class of open sets with UR or ADR boundary, 
or in the class of chord-arc domains or 1-sided chord-arc domains, the Carleson measure estimates are equivalent to local and global  area integral bounds (aka square function estimates). 

\begin{theorem}\label{t0.2}
Let $\Omega\subset\ree$ be an open set with ADR boundary and suppose that we have a collection $\{\Omega'\}_{\Omega'\in\Sigma}$ such that each $\Omega'\in\Sigma$ is an open subset of $\Omega$, $\pom'$ is ADR boundary, and also that all of its local sawtooth subdomains (see Section \ref{sPrelim}) belong to $\Sigma$. Let $G\in L^2_{\rm loc}(\Omega)$ and $H\in C(\Omega)$ and assume that 
\[
\left(\frac{1}{r^{n}}\dint_{B(X,r)}  |G(Y)|^2\, \delta(Y)\,dY\right)^{1/2} \leq C \|H\|_{L^\infty(B(X,2 r))}, \mbox{\ \ for all } B(X,2r)\subset\Omega.
\]
The following statements are equivalent:
	
	\begin{list}{$(\theenumi)$}{\usecounter{enumi}\leftmargin=.8cm
			\labelwidth=.8cm\itemsep=0.2cm\topsep=.01cm
			\renewcommand{\theenumi}{\roman{enumi}}}
		
		\item  $\|G\|_{\C(\Omega')}\lesssim \|H\|_{L^\infty(\Omega')}^2$ for all $\Omega'\in \Sigma$.
		
		\item $\|\A_{\Omega'} G\|_{L^q(\pom')}\leq C \|N_{*,\Omega'} H\|_{L^q(\pom')}$ for all $\Omega'\in \Sigma$ and for some $0<q<\infty$. 
		
		\item  $\|\A_{\Omega'} G\|_{L^q(\pom')}\leq C \|N_{*,\Omega'} H\|_{L^q(\pom')}$ for all $\Omega'\in \Sigma$  and for all $0<q<\infty$. 
	\end{list}
\end{theorem}

This result is a particular case of Theorem~\ref{theor:good-lambda} (and Remarks \ref{Good-lambda:classes}, \ref{remark:rcones}, and \ref{remark:COA}), which actually contains considerably more detailed statements, as well as equivalence to local area integral bounds. 

Finally, we discuss the transference for the converse bounds on non-tangential maximal function in terms of the square function and their connection with $\eps$-approximability. In this context, one has to tie up explicitly the arguments of $\A$ and $N_*$. Our first result is a reduced version of the combination of Theorems \ref{theor:N<S:Lip->CAD} and \ref{theor:N<S:CAD->CAD:all-p} stated in Corollary \ref{corol: N<S:Lip->CAD}.  
We do not explain in detail conditions \eqref{locbdd} and \eqref{revHol} now, but let us mention that, generally, they are harmless bounds on interior cubes, which, in the context of solutions of elliptic PDE follow from well-known interior estimates.

\begin{theorem}
Let $D\subset \ree$ be a chord-arc domain. Let $u\in W^{1,2}_{\rm loc}(D)\cap C(D)$ so that \eqref{locbdd} and \eqref{revHol} hold for some $p>2$. 
Assume that for every bounded Lipschitz subdomain $\Omega\subset \ree\setminus E$ 
\begin{equation}\label{eqn5.9-bis**}
\left\|N_{*,\Omega} (u-u(X_{\Omega}^+))\right\|_{L^2(\pom)}\leq C \left\|S_\Omega u\right\|_{L^2(\pom)}, 
\end{equation}
holds with a constant depending on $n$ and the Lipschitz character of $\Omega$, 
and where $X_{\Omega}^+$ is any interior corkscrew point of $\Omega$ at the scale of $\diam (\pom)$. Then, for every $\kappa>0$, if $\partial D$ is bounded 
\[
\|N_{*,D,\kappa}(u-u(X_D^+))\|_{L^q(\partial D)}\leq C' \|S_{D,\kappa} u\|_{L^q(\partial D)}, \quad \mbox{for all}\quad 0<q<\infty,
\]
and if $\partial D$ is unbounded and $u(X)\to 0$ as $|X|\to\infty$ then 
\[
\|N_{*,D,\kappa}u\|_{L^q(\partial D)}\leq C' \|S_{D,\kappa} u\|_{L^q(\partial D)}, \quad \mbox{for all}\quad 0<q<\infty,
\]
where $C'$ depends on $q$, $n$, the CAD character of $D$, the implicit constants in \eqref{locbdd} and \eqref{revHol}, the constant $C$ in \eqref{eqn5.9-bis**}, and  $\kappa$; and where $X_{D}^+$ is any interior corkscrew point of $D$ at the scale of $\diam (\partial D)$.	
\end{theorem}

Our last result is stated in Theorem \ref{theor:eps-approx}. The interior bound \eqref{oscbdd} is, again, a fairly harmless prerequisite which follows from known interior estimates in the context of solutions of elliptic PDEs. We remark that the estimate \eqref{eqn5.9-bis} itself (see below) would 
not make much sense for general  uniformly rectifiable sets, because of topological obstructions (there is no preferred component for a corkscrew point in such a general context), and for that reason we pass directly to $\eps$-approximability.

\begin{theorem}\label{t0.4} 
	Let $E\subset \ree$ be an $n$-dimensional uniformly rectifiable, and suppose that $u\in W^{1,2}_{\rm loc}(\ree\setminus E) \cap C(\ree\setminus E)\cap L^\infty(\ree\setminus E)$ satisfies \eqref{oscbdd}. Assume, in addition, that 
	\[\|\nabla u\|_{\C(\ree\setminus E)}\le C_0' \|u\|_{L^\infty(\ree\setminus E)}\] 
	and that for every bounded chord-arc subdomain $\Omega\subset \ree\setminus E$ 
\begin{equation}\label{eqn5.9-bis}
\left\|N_{*,\Omega} (u-u(X_{\Omega}^+))\right\|_{L^2(\pom)}\leq C \left\|S_\Omega u\right\|_{L^2(\pom)}, 
\end{equation}
holds with a constant depending on $n$ and the CAD character of $\Omega$, and where $X_{\Omega}^+$ is any interior corkscrew point of $\Omega$ at the scale of $\diam (\pom)$.
Then $u$ is $\eps$-approximable on $\ree\setminus E$, with the implicit constants depending on $n$, the UR character of $E$, the constant in \eqref{oscbdd} and in $C_0'$. 
\end{theorem}

We remark that Theorem \ref{t0.1} (ii) and Theorem \ref{t0.4} were already implicit in our previous work 
\cite{HMM2}, although in the present paper our approach to the former result is simpler than the corresponding 
arguments in \cite{HMM2}.

To conclude, let us reiterate that the fact that our results provide a ``black box" real-variable transference principles allows one to use them considerably beyond the traditional scope. We can treat, for instance, subsolutions and supersolutions of elliptic equations. Another example is higher-order elliptic systems. The best available results to date in this context are restricted to Lipschitz domains \cite{DKPV}. Here we establish, for instance, the following estimates. 

Let $K, m\in\NN$. Let $E$ be an $n$-dimensional uniformly rectifiable set and let $u$ be a weak solution to the system 
\[
Lu
=
\sum_{k=1}^K \sum_{|\alpha|=|\beta|=m} a_{\alpha \beta}^{jk} \partial^\alpha \partial^\beta u^k
=
0, \qquad j=1,\dots,K,
\]
on $\ree\setminus E$. Here, $a_{\alpha \beta}^{jk}$, $1\le \alpha,\beta\le n+1$, $1\le j,k\le K$,  $\alpha=(\alpha_1, \dots, \alpha_{n+1})\in \NN_0^{n+1}$, are real constant symmetric coefficients satisfying the Legendre-Hadamard ellipticity condition (see \eqref{LH-higher}). Then $u$ satisfies the $S<N$ estimates in $\ree\setminus E$, that is,  
\begin{equation*}
\|S_{\ree\setminus E}(\nabla^{m-1}u)\|_{L^p(E)} \leq C \|N_{*,{\ree\setminus E}} (|\nabla^{m-1}u|)\|_{L^p(E)}, \quad 0<p<\infty.
\end{equation*} 
Furthermore, if $D\subset \ree$ is a chord-arc domain with an unbounded boundary and $\nabla^{m-1} u$ vanishes at infinity, we also have the converse estimate 
\begin{equation*}
		\|N_{*,D}\big(\nabla^{m-1} u\big)\|_{L^q(\partial D)}\leq C\big\|S_{D} \big(\nabla^{m-1}u\big)\big\|_{L^q(\partial D)}, \quad \mbox{for all}\quad 0<q<\infty.
		\end{equation*} 

Similar results are valid locally and on bounded domains. We also obtain a version of $\eps$-approximability and Carleson measure estimates in this general context. The reader can consult Section~\ref{appl} for detailed discussion of these results and other applications.

\section{Preliminaries}\label{sPrelim}

We start with some further notation and definitions. 

\begin{list}{$\bullet$}{\leftmargin=0.4cm  \itemsep=0.2cm}
	
	\item We use the letters $c,C$ to denote harmless positive constants, not necessarily
	the same at each occurrence, which depend only on dimension and the
	constants appearing in the hypotheses of the theorems (which we refer to as the
	``allowable parameters'').  We shall also
	sometimes write $a\lesssim b$ and $a \approx b$ to mean, respectively,
	that $a \leq C b$ and $0< c \leq a/b\leq C$, where the constants $c$ and $C$ are as above, unless
	explicitly noted to the contrary.  At times, we shall designate by $M$ a particular constant whose value will remain unchanged throughout the proof of a given lemma or proposition, but
	which may have a different value during the proof of a different lemma or proposition.
	
	\item Given a closed set $E \subset \ree$, we shall
	use lower case letters $x,y,z$, etc., to denote points on $E$, and capital letters
	$X,Y,Z$, etc., to denote generic points in $\ree$ (especially those in $\ree\setminus E$).
	
	\item The open $(n+1)$-dimensional Euclidean ball of radius $r$ will be denoted
	$B(x,r)$ when the center $x$ lies on $E$, or $B(X,r)$ when the center
	$X \in \ree\setminus E$.  A surface ball is denoted
	$\Delta(x,r):= B(x,r) \cap E$ where unless otherwise specified we implicitly assume that $x\in E$.
	
	\item Given a Euclidean ball $B$ or surface ball $\Delta$, its radius will be denoted
	$r_B$ or $r_\Delta$, respectively.
	
	\item Given a Euclidean or surface ball $B= B(X,r)$ or $\Delta = \Delta(x,r)$, its concentric
	dilate by a factor of $\kappa >0$ will be denoted
	$\kappa B := B(X,\kappa r)$ or $\kappa \Delta := \Delta(x,\kappa r).$
	
	\item Given a (fixed) closed set $E \subset \ree$, for $X \in \ree$, we set $\delta(X):= \dist(X,E)$.

	\item We let $H^n$ denote $n$-dimensional Hausdorff measure, and let
	$\sigma := H^n\big|_{E}$ denote the ``surface measure'' on $E$.
	
	\item We will also work with open sets $\Omega\subset\ree$ in which case the previous notations and definitions easily adapt by letting $E:=\partial\Omega$.
	
	\item For a Borel set $A\subset \ree$, we let $1_A$ denote the usual
	indicator function of $A$, i.e. $1_A(x) = 1$ if $x\in A$, and $1_A(x)= 0$ if $x\notin A$.
	
	\item For a Borel set $A\subset \ree$,  we let $\interior(A)$ denote the interior of $A$.
	

	\item Given a Borel measure $\mu$, and a Borel set $A$, with positive and finite $\mu$ measure, we
	set $\fint_A f d\mu := \mu(A)^{-1} \int_A f d\mu$.

	\item We shall use the letter $I$ (and sometimes $J$)
	to denote a closed $(n+1)$-dimensional Euclidean dyadic cube with sides
	parallel to the co-ordinate axes, and we let $\ell(I)$ denote the side length of $I$.
	If $\ell(I) =2^{-k}$, then we set $k_I:= k$.
	Given an ADR set $E\subset \ree$, we use $Q$ to denote a dyadic ``cube''
	on $E$.  The
	latter exist (cf. \cite{DS1}, \cite{Ch}), and enjoy certain properties
	which we enumerate in Lemma \ref{lemmaCh} below.
	
\end{list}

\begin{lemma}[Existence and properties of the ``dyadic grid'', \cite{DS1, DS2, Ch}]\label{lemmaCh}
Suppose that $E\subset \ree$ is an $n$-dimensional ADR set.  Then there exist
constants $ a_0>0,\, \gamma>0$ and $C_1<\infty$, depending only on dimension and the
ADR constant, such that for each $k \in \mathbb{Z},$
there is a collection of Borel sets (``cubes'')
$$
\mathbb{D}_k:=\{Q_{j}^k\subset E: j\in \mathfrak{I}_k\},$$ where
$\mathfrak{I}_k$ denotes some (possibly finite) index set depending on $k$, satisfying

\begin{list}{$(\theenumi)$}{\usecounter{enumi}\leftmargin=.8cm
\labelwidth=.8cm\itemsep=0.2cm\topsep=.1cm
\renewcommand{\theenumi}{\roman{enumi}}}

\item $E=\cup_{j}Q_{j}^k\,\,$ for each
$k\in{\mathbb Z}$.

\item If $m\geq k$ then either $Q_{i}^{m}\subset Q_{j}^{k}$ or
$Q_{i}^{m}\cap Q_{j}^{k}=\emptyset$.

\item For each $(j,k)$ and each $m<k$, there is a unique
$i$ such that $Q_{j}^k\subset Q_{i}^m$.

\item $\diam\big(Q_{j}^k\big)\leq C_1 2^{-k}$.

\item Each $Q_{j}^k$ contains some ``surface ball'' $\Delta \big(x^k_{j},a_02^{-k}\big):=
B\big(x^k_{j},a_02^{-k}\big)\cap E$.

\item $H^n\big(\big\{x\in Q^k_j:{\rm dist}(x,E\setminus Q^k_j)\leq \varrho \,2^{-k}\big\}\big)\leq
C_1\,\varrho^\gamma\,H^n\big(Q^k_j\big),$ for all $k,j$ and for all $\varrho\in (0,a_0)$.
\end{list}
\end{lemma}

A few remarks are in order concerning this lemma.

\begin{list}{$\bullet$}{\leftmargin=0.4cm  \itemsep=0.2cm}

\item In the setting of a general space of homogeneous type, this lemma has been proved by Christ
\cite{Ch}, with the
dyadic parameter $1/2$ replaced by some constant $\delta \in (0,1)$.
In fact, one may always take $\delta = 1/2$ (cf.  \cite[Proof of Proposition 2.12]{HMMM}).
In the presence of the Ahlfors-David
property (\ref{eq1.ADR}), the result already appears in \cite{DS1,DS2}.

\item  For our purposes, we may ignore those
$k\in \mathbb{Z}$ such that $2^{-k} \gtrsim {\rm diam}(E)$, in the case that the latter is finite.

\item  We shall denote by  $\mathbb{D}=\mathbb{D}(E)$ the collection of all relevant
$Q^k_j$, i.e., $$\mathbb{D} := \cup_{k} \mathbb{D}_k,$$
where, if $\diam (E)$ is finite, the union runs
over those $k$ such that $2^{-k} \lesssim  {\rm diam}(E)$. When $E$ is bounded there exists a cube $Q_0\in \dd(\pom)$ such that $Q_0=\pom$ and  $Q\in \dd_{Q_0}$ for any $Q\in\dd(\pom)$.


\item For a dyadic cube $Q\in \mathbb{D}_k$, we shall
set $\ell(Q) = 2^{-k}$, and we shall refer to this quantity as the ``length''
of $Q$.  Evidently, $\ell(Q)\approx \diam(Q).$

\item For a dyadic cube $Q \in \mathbb{D}$, we let $k(Q)$ denote the ``dyadic generation''
to which $Q$ belongs, i.e., we set  $k = k(Q)$ if
$Q\in \mathbb{D}_k$; thus, $\ell(Q) =2^{-k(Q)}$.

\item Given $Q\in\dd$ we write $\widetilde{Q}$ to denote the dyadic parent of $Q$, that is, the unique dyadic cube $\widetilde{Q}$ with $Q\subset\widetilde{Q}$ and $\ell(\widetilde{Q})=2\ell(Q)$. Also, the children of $Q$ are the dyadic cubes $Q'\subset Q$ with $\ell(Q')=\ell(Q)/2$.

\item Properties $(iv)$ and $(v)$ imply that for each cube $Q\in\mathbb{D}$,
there is a point $x_Q\in E$, a Euclidean ball $B(x_Q,r)$ and a surface ball
$\Delta(x_Q,r):= B(x_Q,r)\cap E$ such that
$c\ell(Q)\le r\le\ell(Q)$ for some uniform constant $0<c<1$
and
\begin{equation}\label{cube-ball}
\Delta(x_Q,2r)\subset Q \subset \Delta(x_Q,Cr),\end{equation}
for some uniform constant $C$.
We shall denote this ball and surface ball by
\begin{equation}\label{cube-ball2}
B_Q:= B(x_Q,r) \,,\qquad\Delta_Q:= \Delta(x_Q,r),\end{equation}
and we shall refer to the point $x_Q$ as the ``center'' of $Q$.

\end{list}

\begin{definition}\label{defHLmax} Let $E\subset \ree$ be an $n$-dimensional ADR set. By $M^{\dd}=M^{\dd(E)}$ we denote the dyadic Hardy-Littlewood maximal function on $E$, that is, for $f\in L^1_{\rm loc}(E)$
\[M^{\dd}f(x)=\sup_{x\in Q\in\dd(E)} \fint_Q |f(y)|\,d\sigma(y), 
\]
and, for $0<p<\infty$, we also write $M^{\dd}_{p} f=M^{\dd} (|f|^p)^\frac1p$.  Analogously, if $Q_0\in\dd(E)$, we write $M_{Q_0}^{\dd}$ for the dyadic Hardy-Littlewood maximal function localized to $Q_0$,
\[M_{Q_0}^{\dd}f(x)=\sup_{x\in Q\in\dd_{Q_0}} \fint_Q |f(y)|\,d\sigma(y), 
\]
where $\dd_{Q_0}(E)=\{Q\in\dd(E): Q\subset Q_0\}$, and, for $0<p<\infty$, we also write $M^{\dd}_{Q_0,p} f=M_{Q_0}^{\dd} (|f|^p)^\frac1p$.
\end{definition}

Let $\Omega\subset\ree$ be an open set so that $\pom$ is ADR. Let $\mathcal{W}=\W(\Omega)$ denote a collection
of (closed) dyadic Whitney cubes of $\Omega$, so that the cubes in $\mathcal{W}$
form a pairwise non-overlapping covering of $\Omega$, which satisfy
\begin{equation}\label{Whintey-4I}
4 \diam(I)\leq
\dist(4I,E)\leq \dist(I,\pom) \leq 40\diam(I)\,,\qquad \forall\, I\in \mathcal{W}\,\end{equation}
(just dyadically divide the standard Whitney cubes, as constructed in  \cite[Chapter VI]{St},
into cubes with side length 1/8 as large)
and also
$$(1/4)\diam(I_1)\leq\diam(I_2)\leq 4\diam(I_1)\,,$$
whenever $I_1$ and $I_2$ touch.

Next, we choose a small parameter $0<\tau_0< 2^{-4}$ (depending only on dimension), so that
for any $I\in \W$, and any $\tau \in (0,\tau_0]$,
the concentric dilate
$I^*(\tau):= (1+\tau) I$ still satisfies the Whitney property
\begin{equation}\label{whitney}
\diam I\approx \diam I^*(\tau) \approx \dist\left(I^*(\tau), \pom\right) \approx \dist(I,\pom)\,, \quad 0<\tau\leq \tau_0\,.
\end{equation}
Moreover,
for $\tau\leq\tau_0$ small enough, and for any $I,J\in \W$,
we have that $I^*(\tau)$ meets $J^*(\tau)$ if and only if
$I$ and $J$ have a boundary point in common, and that, if $I\neq J$,
then $I^*(\tau)$ misses $(3/4)J$.

\begin{definition}[\textbf{Whitney-dyadic structure}]\label{def:WD-struct}
Let $\Omega\subset\ree$ be an open set so that $\pom$ is ADR. Let $\mathcal{W}=\W(\Omega)$ denote a collection
of (closed) dyadic Whitney cubes of $\Omega$ as in \eqref{Whintey-4I}. Let $\dd=\dd(\pom)$ be the collection of dyadic cubes from Lemma \ref{lemmaCh} and given the parameters $\eta<1$ and $K>1$, set 
\begin{equation}\label{eq3.1}
\W^0_Q:= \left\{I\in \W:\,\eta^{1/4} \ell(Q)\leq \ell(I)
\leq K^{1/2}\ell(Q),\ \dist(I,Q)\leq K^{1/2} \ell(Q)\right\},
\end{equation}
A \textbf{Whitney-dyadic structure} for $\Omega$ with parameters $\eta$ and $K$ is a family $\{\W_Q\}_{Q\in\dd}\subset\W$ satisfying the following conditions:  
\begin{list}{$(\theenumi)$}{\usecounter{enumi}\leftmargin=.8cm
		\labelwidth=.8cm\itemsep=0.2cm\topsep=.1cm
		\renewcommand{\theenumi}{\roman{enumi}}}
\item $\W_Q^0\neq\emptyset$ for every $Q\in\dd$.

\item $\W^0_Q\subset \W_Q$ for every $Q\in\dd$.

\item 	There exists $C\ge 1$  such that, for every $Q\in\dd$,

\begin{equation}\label{eq2.whitney2}
\begin{gathered}
C^{-1}\eta^{1/2} \ell(Q)\leq \ell(I) \leq CK^{1/2}\ell(Q)\,, \quad \forall I\in \mathcal{W}_Q,
\\
\dist(I,Q)\leq CK^{1/2} \ell(Q)\,,\quad\forall I\in \mathcal{W}_Q.
\end{gathered}
\end{equation}
\end{list}	
\end{definition}

In principle, for the previous definition, $\eta$ and $K$ are arbitrary, but we will typically need to assume that $\eta$ is sufficiently small and $K$ is sufficiently large. We will do so and as a consequence the constant $C$ will be independent of $\eta$ and $K$ and will depend on dimension, ADR, and some other intrinsic constants depending on the different scenarios on which we work. In particular, it is convenient to assume, and we will do so, that $K\ge 40^2 n$ so that given any $I\in\W$ such that $\ell(I) \lesssim \diam(E)$, if we write $Q_I^*$ for (one) nearest dyadic cube to $I$ with $\ell(I) = \ell(Q_I^*)$ then $I\in \W_{Q_I^*}^0\subset \W_{Q_I^*}$. Note that there can be more than one choice of $Q_I^*$, but at this point we fix one so that in what follows $Q_I^*$ is unambiguously defined.

Below we will discuss a few special cases depending on
whether we have some extra information about $\Omega$ or $\pom$. The main idea consists in constructing some kind of ``Whitney regions'' which will allow us to introduce some ``Carleson boxes'' and ``sawtooth subdomains''. The construction of the Whitney regions depends very much on the background assumptions, having extra information about $\Omega$ or $\pom$ will allow us to augment the collections $\W_Q^0$ to define $\W_Q$ so that we gain some connectivity on the corresponding Whitney regions and hence the resulting subdomains would have better properties. We consider four cases. In the first one, treated in Section \ref{sections:ADR}, we assume only that $\Omega=\ree\setminus E$ where $E$ is ADR (but is not necessarily UR)
and we set $\W_Q=\W_Q^0$ (here we do not gain any connectivity). The second case is considered in Section \ref{sections:UR} and deals with $\Omega=\ree\setminus E$ where $E$ is UR, in which case we can invoke Lemma \ref{lemma2.1} below  and use the Lipschitz graphs associated to the 
good regimes so that the augmented collection $\W_Q$ creates two nice Whitney regions, 
one each lying respectively above and below the Lipschitz graph. Third, when $\Omega$ is a 1-sided CAD we can augment $\W_Q^0$ using that $D$ is Harnack chain connected so that the resulting collections $\W_Q$ give some Whitney regions which produce Carleson boxes and sawtooth subdomains which are 1-sided CAD, see Section \ref{sections:1-sided:CAD}. We repeat the same construction in our last case in Section \ref{sections:CAD},  where $\Omega$ is a CAD. The fact that $\Omega$ satisfies the exterior corkscrew condition allows us to conclude that Carleson boxes and sawtooth subdomains are as well.

To continue with our discussion let $\Omega\subset\ree$ be an open set so that $\pom$ is ADR. Let $\mathcal{W}=\W(\Omega)$ and $\dd=\dd(\pom)$ be as above and let $\{\W_Q\}_{Q\in\dd}$ be a Whitney-dyadic structure for $\Omega$ with some parameters $\eta$ and $K$ (we will assume that $\eta$ is sufficiently small and $K$ is sufficiently large). Fix $0<\tau\le \tau_0/4$ as above. Given an  arbitrary $Q\in\dd$, we may define an associated
\textbf{Whitney region} $U_Q$ (not necessarily connected), as follows:
\begin{equation}\label{eq3.3aa}
U_Q=U_{Q,\tau}:= \bigcup_{I\in \W_Q} I^*(\tau)
\end{equation}
For later use, it is also convenient to introduce some fattened version of $U_Q$
\begin{equation}\label{eq3.3aa-fat}
\widehat{U}_Q=U_{Q,2\,\tau}:= \bigcup_{I\in \W_Q} I^*(2\,\tau). 
\end{equation}
When the particular choice of $\tau\in (0,\tau_0]$ is not important,
for the sake of notational convenience, we may
simply write $I^*$ and $U_Q$  in place of $I^*(\tau)$ and $U_{Q,\tau}$.

We may also define the \textbf{Carleson box} relative to $Q\in\dd$, by
\begin{equation}\label{eq3.3a}
T_Q=T_{Q,\tau}:=\interior\left(\bigcup_{Q'\in\dd_Q} U_{Q,\tau}\right)\,,
\end{equation}
where
\begin{equation}\label{eq3.4a}
\dd_Q:=\left\{Q'\in\dd:Q'\subset Q\right\}\,.
\end{equation}
Let us note that we may choose $K$ large enough so that, for every $Q$,
\begin{equation}\label{eq3.3aab}
T_{Q,\tau}\subset T_{Q,\tau_0} \subset B_Q^*:= B\left(x_Q,K\ell(Q)\right).
\end{equation}
We also observe that for any $N\ge 1$ we have 
\begin{equation}\label{eq3.3aab:inte}
B_Q\cap \Omega \subset T_{Q,\tau/N}.
\end{equation}
To see this, let $Y\in B_Q\cap \Omega  =B(x_Q,r)\cap\Omega$ (cf. \eqref{cube-ball}, \eqref{cube-ball2}) and pick $I\in\W$ with $I\ni Y$. Note that $\ell(I)\le\dist(I,\pom)/4\le |Y-x_Q|/4<r/4\le \ell(Q)/4$.
Take $\widehat{y}\in Q$ so that $\dist(Y,Q)=|Y-\widehat{y}|$ and select $Q_Y\ni\widehat{y}$ with $\ell(Q_Y)=\ell(I)\le \ell(Q)/4$.  Thus, $Q_Y\in\dd_{Q}$ and 
\[
\dist(I,Q_Y)\le |Y-\widehat{y}|=\dist(Y,Q)\le |Y-x_Q|<r\le\ell(Q).
\]
All these show that $I\in \W_Q^0\subset \W_Q$ and consequently $Y\in \interior(I^*(\tau/N))\subset T_{Q,\tau/N}$ as  desired. 

It is convenient to introduce the \textbf{Carleson box} $T_\Delta$ relative to $\Delta=\Delta(x,r)$ with $x\in \pom$ and $0<r<\diam(\pom)$. Let $k(\Delta)$ denote the unique $k\in\ZZ$ such that $2^{-k-1}<200 r\le 2^{-k}$ and set
\[
\dd^\Delta:=\{Q\in\dd_{k(\Delta)}: Q\cap 2\Delta\neq\emptyset\}.
\]
We then define 
\begin{equation}\label{eq3.3a:Delta}
T_\Delta=T_{\Delta,\tau}:=\interior\left(\bigcup_{Q\in\dd^\Delta} \overline{T_Q}\right).
\end{equation}
Much as in \cite[(3.60)]{HM-I} if we write $B_\Delta=B(x,r)$ so that $\Delta=B_\Delta\cap E$, we have by taking $K$ possibly larger
\begin{equation}\label{eq3.3aab:Delta}
\frac54 B_\Delta\cap\Omega \subset T_{\Delta}\subset B(x, Kr)\cap\Omega.
\end{equation}

For future reference, we also introduce dyadic sawtooth regions as follows.
Given a family $\mathcal{F}$ of disjoint cubes $\{Q_j\}\subset \mathbb{D}$, we define
the {\bf global discretized sawtooth} relative to $\F$ by
\begin{equation}\label{eq2.discretesawtooth1}
\dd_{\F}:=\dd\setminus \bigcup_{\F} \dd_{Q_j}\,,
\end{equation}
i.e., $\dd_{\F}$ is the collection of all $Q\in\dd$ that are not contained in any $Q_j\in\F$.
Given some fixed cube $Q$,
the {\bf local discretized sawtooth} relative to $\F$ by
\begin{equation}\label{eq2.discretesawtooth2}
\dd_{\F,Q}:=\dd_Q\setminus \bigcup_{\F} \dd_{Q_j}=\dd_\F\cap\dd_Q.
\end{equation}
Note that we can also allow $\F$ to be empty in which case $\dd_{\emptyset}=\dd$ and $\dd_{\emptyset,Q}=\dd_Q$.

Similarly, we may define geometric sawtooth regions as follows.
Given a family $\mathcal{F}\subset\dd$ of disjoint cubes as before
we define the {\bf global sawtooth} and the {\bf local sawtooth} relative to $\mathcal{F}$ by respectively
\begin{equation}\label{eq2.sawtooth1}
\Omega_{\mathcal{F}}:= {\rm int } \bigg( \bigcup_{Q'\in\dd_\F} U_{Q'}\bigg)\,,
\qquad
\Omega_{\mathcal{F},Q}:=  {\rm int } \bigg( \bigcup_{Q'\in\dd_{\F,Q}} U_{Q'}\bigg)\,.
\end{equation}
Note that $\Omega_{\emptyset,Q}=T_Q $.
For the sake of notational convenience, we set
\begin{equation}\label{Def-WF}
\W_{\F}:=\bigcup_{Q'\in\dd_{\F}}\W_{Q'}\,,\qquad
\W_{\F,Q}:=\bigcup_{Q'\in\dd_{\F,Q}}\W_{Q'}\,,\end{equation}
so that in particular, we may write
\begin{equation}\label{eq3.saw}
\Omega_{\mathcal{F},Q}={\rm int }\,\bigg(\bigcup_{I\in\,\W_{\F,Q}} I^*\bigg)\,.
\end{equation}

Finally, for every $x\in \pom$, we define non-tangential approach regions, \textbf{dyadic cones}, as 
\begin{equation}\label{defcone}
\Gamma(x) = \bigcup_{Q\in\dd:\,Q\ni x} U_Q.
\end{equation}
Their local (or truncated) versions are given by 
\begin{equation}\label{defconetrunc}
\Gamma^Q(x) = \bigcup_{Q'\in\dd_Q:\,Q'\ni x} U_{Q'},
\qquad x\in Q.
\end{equation}
When $\pom$ is bounded, there exists a cube $Q_0\in \dd(\pom)$ such that $Q_0=\pom$ and  $Q\in \dd_{Q_0}$ for any  $Q\in\dd(\pom)$. In particular, $\Gamma_{Q}(\cdot)\subset \Gamma_{Q_0}(\cdot)\subset \{X\in\Omega:\dist(X,\pom)\lesssim \diam(\pom)\}$ and all the cones are bounded.

Note that all the previous objects have been defined using the Whitney regions $U_Q$ (made out of dilated Whitney cubes $I^*(\tau)$). One can analogously use the fattened Whitney regions $\widehat{U}_Q$ (composed
of the union of dilated Whitney cubes $I^*(2\tau)$). In that case 
we will use the notation $\widehat{T}_Q$, $\widehat{T}_\Delta$, $\widehat{\Omega}_\F$, $\widehat{\Omega}_{\F,Q}$, $\widehat{\Gamma}(\cdot)$, $\widehat{\Gamma}^Q(\cdot)$.

We will always assume that $K$ is large enough (say $K\ge 10^4n$) so that $\widehat\Gamma_{\Omega,1}(x)\subset \Gamma(x)$ (cf. \eqref{conetrad}) for every $x\in \pom$. Indeed, let $Y\in \Gamma_{\Omega,1}(x)$ and pick $I\in\W$ with $Y\in I$. Take $Q\in\dd$ with $Q\ni x$ and $\ell(Q)=\ell(I)$. Then,
\[
\dist(I,Q)\le |Y-x|\le 2\dist(Y,\pom)\le 2(\diam(I)+\dist(I,\pom))\le 82\,\diam(I)<100\sqrt{n}\ell(Q).
\]
Hence, $I\in \W^0_Q\subset \W_Q$ provided $100\sqrt{n}\le \sqrt{K}$ and thus $I\subset U_Q\subset\Gamma(x)$ as desired.

\begin{remark}\label{remark:CMO0} 
	It is convenient to introduce 
a condition on interior Whitney balls, that is much weaker than CME itself.
Let $\Omega\subset \ree$ be an open set. For every $F\in L^2_{\rm loc}(\Omega)$ we set
	\begin{equation}\label{eqdefCME-interior:def}
	\|F\|_{\C_0(\Omega)} := 
	\sup_{X\in \Omega} \,\frac1{\delta(X)^{n-1}}\dint_{B(X,\delta(X)/2)} |F(Y)|^2 \,dY,
	\end{equation}
	where $\delta(\cdot)=\dist(\cdot,\pom)$.
	
	Note that for any $X\in \Omega$ we have that $B(X,\delta(X)/2)\subset B(\hat{x},3\delta(X)/2)\cap \Omega$ with $\hat{x}\in \pom$ so that $\delta(X)=|X-\hat{x}|$, and $\delta(Y)\ge \delta(X)/2$ for every $Y\in B(X,\delta(X)/2)$. Hence,
	\begin{equation}\label{eqdefCME-interior}
	\|F\|_{\C_0(\Omega)} 
	\le
	2\,\Big(\frac32\Big)^n\|F\|_{\C(\Omega)},
	\end{equation}
	and $\|F\|_{\C_0(\Omega)}<\infty$ is necessary  for \eqref{eqdefCME} to hold. 
	
	We note that  in all applications to the CME for 
	solutions of elliptic PDEs, $\|F\|_{\C_0(\Omega)}$ will be bounded
	automatically, by Caccioppoli's inequality (since $F$ will be of the form  $\nabla u$ or $\nabla^m u$ with $u$ being a bounded solution). We shall discuss this in more detail together with the corresponding applications. 
\end{remark}

We introduce a dyadic version of Definition \ref{defCME}. Given $\Omega\subset \ree$, an open set with $\pom$ being ADR, let $\{\W_Q\}_{Q\in\dd(\pom)}$ be a Whitney-dyadic structure for $\Omega$ with some parameters $\eta$ and $K$. We define, for every $F\in L^2_\loc(\Omega)$,
\begin{align}\label{def:CME:dyadic}
\|F\|_{\C^{\rm dyad}(\Omega)}:= \sup_{Q\in\dd(\pom)}\frac1{\sigma(Q)}\dint_{T_{Q} } |F(X)|^2 \dist(X,\pom) \,dX
\end{align}
We are going to show that 
\begin{equation}\label{CME:dyadic}
\|F\|_{\C(\Omega)}\lesssim \|F\|_{\C^{\rm dyad}(\Omega)}+ \|F\|_{\C_0(\Omega)},
\end{equation}
To obtain this, fix $x\in \pom$ and $0<r<\infty$. Set $\mathcal{W}_{x,r}=\{I\in\mathcal W(\Omega): I\cap B(x,r)\neq \emptyset\}$
and note that given $I\in\mathcal{W}_{x,r}$, if we pick $Z_I\in I\cap B(x,r)$, then \eqref{Whintey-4I} implies
\begin{equation}\label{25tw3}
\diam(I)\le\dist(I,\pom)\le |Z_I-x|<r.
\end{equation}
Set 
\begin{equation*}\label{CME-dyad-regular:2}
\mathcal{W}_{x,r}^{\rm small}=\{I\in\mathcal{W}_{x,r}: \ell(I)<\diam(\pom)/4\},
\qquad 
\mathcal{W}_{x,r}^{\rm big}=\{I\in\mathcal{W}_{x,r} : \ell(I)\ge \diam(\pom)/4\},
\end{equation*}
with the understanding that $\mathcal{W}_{x,r}^{\rm big}=\emptyset$ if $\diam(\pom)=\infty$.
Using this notation and writing $\delta=\dist(\cdot,\pom)$ we have
\begin{align}\label{CME-dyad-regular:3}
\dint_{B(x,r)\cap\Omega }|F|^2\delta\,dX
\le
\sum_{I\in \mathcal{W}_{x,r}^{\rm small}}\dint_{I}|F|^2\delta\,dX
+
\sum_{I\in \mathcal{W}_{x,r}^{\rm big}}\dint_{I}|F|^2\delta\,dX
=
\mathrm{I}+\mathrm{II},
\end{align}
here we understand that $\mathrm{II}=0$ if $\mathcal{W}_{x,r}^{\rm big}=\emptyset$.

To estimate $\mathrm{I}$ we set $r_0=\min\{r, \diam(\pom)/4\}$ and pick $k_2\in\mathbb{Z}$ so that $2^{k_2-1}\le r_0<2^{k_2}$. Set
\[
\mathcal{D}_1=\{Q\in\mathbb{D}(\pom): \ell(Q)=2^{k_2},\: Q\cap B(x,3r)\neq\emptyset\}.
\]
Given $I\in  \mathcal{W}_{x,r}^{\rm small}$ we pick $y\in \pom$ so that $\dist(I,\pom)=\dist(I,y)$. Hence there exists a unique $Q_I\in\mathbb{D}(\pom)$ so that
$y\in Q_I$ and $\ell(Q_I)=\ell(I)<r_0\le \diam(\pom)/4$ by \eqref{25tw3}. Also, 
\[
\dist(I,Q_I)\le \dist(I,y)=\dist(I,\pom)\le 40\diam(I)=40\sqrt{n}\ell(Q).
\]
This implies that  $I\in \mathcal{W}_{Q_I}^0\subset \W_{Q_I}$, provided $0<\eta\le 1$ and $K\ge 40\sqrt{n}$. On the other hand, by \eqref{25tw3}
$$
|y-x|
\le
\dist(y, I)+\diam(I)+|Z_I-x|
<3r,
$$
hence there exists a unique $Q\in \mathcal{D}_1$ so that $y\in Q$. Since $\ell(Q_I)<r_0<2^{k_2}=\ell(Q)$  we conclude that $Q_I\subset Q$
and consequently $I\subset \interior(U_{Q_I})\subset T_{Q}$. In short we have shown that  if $I\in  \mathcal{W}_{x,r}^{\rm small}$ there exists $Q\in\mathcal{D}_1$ so that  $I\subset T_{Q}$. Thus,
\begin{align*}
\mathrm{I}
\lesssim
\sum_{Q\in\mathcal{D}_1} \dint_{T_{Q}}|F|^2\delta\,dX
\le
\|F\|_{\C^{\rm dyad}(\Omega )} \sum_{Q\in\mathcal{D}_1} \sigma(Q)
\lesssim
\|F\|_{\C^{\rm dyad}(\Omega)}  r^n,
\end{align*}
where we have used the fact that $\mathcal{D}_1$ is a pairwise disjoint family, that $\bigcup_{Q\in\mathcal{D}_1} Q\subset B(x,Cr)\cap\pom$ (with $C$ depending on dimension and ADR), and that $\pom$ is ADR.

We now estimate $\mathrm{II}$ when non-empty, in which case $\diam(\pom)<\infty$. Using the properties of the Whitney cubes and recalling \eqref{eqdefCME-interior:def} we arrive at
\begin{multline*}
\mathrm{II}
\lesssim
\sum_{I\in \mathcal{W}_{x,r}^{\rm big}}\ell(I)\dint_{I}|F|^2\,dX
\lesssim
\|F\|_{\C_0(\Omega)}\sum_{I\in \mathcal{W}_{x,r}^{\rm big}}\ell(I)^n
\\
\leq
\|F\|_{\C_0(\Omega)}\sum_{{\diam(\pom)}/4\le 2^k<r } 2^{kn} \#\{I\in \mathcal{W}_{x,r}^{\rm big}:\ell(I)=2^{k}\}.
\end{multline*}
To estimate the last term we observe that if $Y\in I\in  \mathcal{W}_{x,r}^{\rm big}$ we have by \eqref{Whintey-4I}
\[
|Y-x|
\le
\diam(I)+\dist(I,\pom)+\diam(\pom)\lesssim \ell(I).
\]
This and the fact that Whitney cubes have non-overlapping interiors imply
\begin{multline}\label{pior4fgagv}
\#\{I\in \mathcal{W}_{x,r}^{\rm big}:\ell(I)=2^{k}\}
=
2^{-k(n+1)}\sum_{I\in \mathcal{W}_{x,r}^{\rm big}:\ell(I)=2^{k}} |I|
\\
=
2^{-k(n+1)}\Big|\bigcup_{I\in \mathcal{W}_{x,r}^{\rm big}:\ell(I)=2^{k}} I\Big|
\le
2^{-k(n+1)}
|B(x, C2^k)|
\lesssim
1.
\end{multline}
Therefore,
\[
\mathrm{II}
\lesssim
\|F\|_{\C_0(\Omega)}\sum_{\diam(\pom)/4\le 2^k<r } 2^{kn}
\lesssim
\|F\|_{\C_0(\Omega)}r^n.
\]
Collecting the estimates for $\mathrm{I}$ and $\mathrm{II}$ we obtain \eqref{CME:dyadic}.

\begin{definition}[\textbf{Dyadic} \textbf{Non-tangential maximal function}, \textbf{Area integral}, and \textbf{Square function}]\label{defsfntmax:dyadic} 
	Let $\Omega\subset \ree$ be an open set with $\pom$ being ADR and let $\{\W_Q\}_{Q\in\dd(\pom)}$ be a Whitney-dyadic structure for $\Omega$ with some parameters $\eta$ and $K$. For $H\in C(\Omega)$ (i.e., $H$ is continuous function in $\Omega$), we define the \textbf{dyadic non-tangential
		maximal function} as
	\begin{equation}\label{defN*}
	N_*H(x):= \sup_{Y\in \Gamma(x)} |H(Y)|\,,\qquad x\in \pom;
	\end{equation}
	for $G\in L^2_{\rm loc}(\Omega)$, we define the \textbf{dyadic area integral}  as
	\begin{equation}\label{defA}
	\A G(x):=\left(\dint_{\Gamma(x)}|G(Y)|^2 \dist(Y,E)^{1-n} dY\right)^{\frac12}\,,\qquad x\in \pom;
	\end{equation}
	and, for $u\in W^{1,2}_{\loc}(\Omega)$, we define the \textbf{dyadic square function}  as
	\begin{equation}\label{defS}
	S u(x):=\left(\dint_{\Gamma(x)}|\nabla u(Y)|^2 \dist(Y,\pom)^{1-n} dY\right)^{\frac12}\,,\qquad x\in \pom.
	\end{equation}
	For any $Q\in\dd(\pom)$, we write $N_*^Q$, $\A^Q$, and $S^Q$ to denote the \textbf{local} (or \textbf{truncated}) dyadic non-tangential maximal function, area integral, and square function respectively, where $\Gamma(\cdot)$ is replaced by the local cone $\Gamma^Q(\cdot)$. Finally, $\widehat{N}_*$, $\widehat{\A}$, $\widehat{S}$ or
	$\widehat{N}_*^Q$, $\widehat{\A}^Q$, $\widehat{S}^Q$ stand for the corresponding objects associated to the fattened cones $\widehat{\Gamma}(\cdot)$ or their local versions $\widehat{\Gamma}^Q(\cdot)$.	
\end{definition}

%

\begin{remark}\label{remark:rcones}  It is convenient to compare the two types 
of cones, the ``traditional'' and the dyadic (cf. \eqref{conetrad} and \eqref{defcone}). Fix a Whitney-dyadic structure $\{\W_Q\}_{Q\in\dd(\pom)}$  for $\Omega$ with parameters $\eta$ and $K$. 	It is straightforward to see that there exists $\kappa$ such that the dyadic cones $\Gamma(x)$ are contained in $\Gamma_\Omega(x)$ for all $x\in\po$. 
	Indeed, if $Y\in I^*(2\tau)$ with $I\in \W_Q$ and $Q\ni x$ then by \eqref{eq2.whitney2}
	\begin{multline*}
	|Y-x|\le \diam(I^*(2\tau))+\dist(I,Q)+\diam(Q)
	\lesssim 
	K^{1/2}\ell(Q)
	\lesssim 
	K^{1/2}\,\eta^{-1/2}\ell(I)
	\\
	\lesssim
	K^{1/2}\,\eta^{-1/2}\dist(I,\pom)
	\le
	K^{1/2}\,\eta^{-1/2}\dist(Y,\pom),
	\end{multline*}
	hence $Y\in\Gamma_{\Omega, \,K^{1/2}\,\eta^{-1/2}}(x)$. And we have shown that $\Gamma(x)\subset \widehat{\Gamma}(x)\subset \Gamma_{\Omega, \,K^{1/2}\,\eta^{-1/2}}$. 
	Conversely, given $\kappa>0$, there exist $\eta$ and $K$ (depending on $\kappa$) such that if $\{\W_Q\}_{Q\in\dd(\pom)}$ is a Whitney-dyadic structure for $\Omega$ with parameters $\eta$ and $K$ then $\Gamma_{\Omega,\kappa}(x)\subset \Gamma(x)$ for all $x\in\po$. As a matter of fact, given $Y\in \Gamma_{\Omega,\kappa}(x)$, let $I\in\W$ with $I\ni Y$ and pick $Q\in\dd(\pom)$ with $Q\ni x$ and $\ell(I)=\ell(Q)$ (recall that if $\pom$ is bounded we have assumed that $\delta(Y)\lesssim \diam(\pom)$, hence such a cube $Q$ always exists). Then,
	\[
	\dist(I,Q)
	\le
	|Y-x|
	\le
	(1+\kappa)\dist(Y,\pom)
	\le
	(1+\kappa)(\diam(I)+\dist(I,\pom))
	\lesssim
	(1+\kappa)\ell(I)
	=
	(1+\kappa)\ell(Q).
	\]
	Thus, if $K^{1/2}\gg 1+\kappa$, then $I\in \W_Q^0\subset \W_Q$ and $Y\in I\subset U_Q\subset\Gamma(x)$ as desired.

\end{remark}

\begin{remark}\label{remark:COA} 
	In the previous remark we have been able to compare the dyadic and the traditional cones and this give comparisons 	between the associated non-tangential maximal functions, area integrals, or square functions by adjusting the different parameters. It is also convenient to see how to incorporate  the ``change on the aperture'' on the traditional cones via or on the dyadic cones. In the case of traditional cones this amounts to considering 
different values of the aperture $\kappa$. For the dyadic cones one can ``change the aperture'' using $U_Q=U_{Q,\tau}$ versus
$\widehat{U}_Q=U_{Q,2\tau}$, or even by considering Whitney-dyadic structures with different parameters.

	In the case of the traditional cones, one has for every $0<p<\infty$ and $\kappa, \kappa'$ and for every $F\in C(\Omega)$ and  $G\in W^{1,2}_\loc(\Omega)$,
	\begin{equation}\label{COA-traditional}
	\|N_{*,\Omega,\kappa} F\|_{L^p(\pom)}
	\approx_{\kappa,\kappa'}
	\|N_{*,\Omega,\kappa'} F\|_{L^p(\pom)}
	\qquad
	\|\A_{\Omega,\kappa} G\|_{L^p(\pom)}
	\lesssim_{\kappa,\kappa'}
	\|\A_{\Omega,\kappa'} G\|_{L^p(\pom)}
	\end{equation}
	The first estimate can be found in \cite[Proposition 2.2]{HoMiTa}. For the second estimate we refer to  \cite[Proposition 4.5]{MPT} in the case $\Omega$ being a CAD, a simpler argument (valid also in the former case) can be carried out by adapting \cite[Proposition 3.2, part $(i)$]{MP}. Further details are left to the interested reader.

	For the dyadic cones,  Remark \ref{remark:rcones} says that if  $\{\W_Q\}_{Q\in\dd(\pom)}$ is a Whitney-dyadic structure for $\Omega$ with parameters $\eta\ll 1$ and $K\gg 1$ then  $\Gamma(x)\subset  \widehat{\Gamma}(x)\subset \Gamma_{\Omega, \kappa}(x)$ for some large $\kappa>0$ and for every $x\in\pom$. On the other hand, let  $\{\W_Q'\}_{Q\in\dd(\pom)}$ is a Whitney-dyadic structure for $\Omega$ with parameters $\eta'\ll 1$ and $K'\gg 1$ and we write $\Gamma'(x)$ for the associated dyadic cone. As observed before we have that $\Gamma_{\Omega, 1}(x)\subset \Gamma'(x)$.  Write $N_*$ and $\A$ (resp. $N_*'$ and $\A'$) as in \eqref{defN*} and \eqref{defA} for the cones $\Gamma$ (resp.$ \Gamma'$). These and \eqref{COA-traditional} allow us to obtain that for every $0<p<\infty$ and for every $F\in C(\Omega)$
	\[
	\|N_* F\|_{L^p(\pom)}
	\le
	\|\widehat{N}_* F\|_{L^p(\pom)}
	\le
	\|N_{*,\Omega,\kappa} F\|_{L^p(\pom)}
	\lesssim_\kappa
	\|N_{*,\Omega,1} F\|_{L^p(\pom)}
	\le
	\|N_*' F\|_{L^p(\pom)}
	\]
	and, for every $G\in W^{1,2}_\loc(\Omega)$,
	\[
	\|\A G\|_{L^p(\pom)}
	\le
	\|\widehat{\A} G\|_{L^p(\pom)}
	\le
	\|\A_{\Omega,\kappa} G\|_{L^p(\pom)}
	\lesssim_\kappa
	\|\A_{\Omega,1} G\|_{L^p(\pom)}
	\le
	\|\A' G\|_{L^p(\pom)}.
	\]
\end{remark}

\subsection{Case ADR}\label{sections:ADR} 
Here we assume that $\Omega=\ree\setminus E$ where $E$ is merely ADR, but possibly not UR. Let us set 
$\W_Q=\W^0_Q$ (cf. \eqref{eq3.1}) and we clearly have $(ii)$ and $(iii)$ with $C=1$ in Definition \ref{def:WD-struct}. For $(i)$, we see that $\W^0_Q$ is non-empty,
	provided that we choose $\eta$ small enough, and $K$ large enough, depending only on dimension
	and the ADR constant of $E$. Indeed, given $Q\in\dd(E)$, consider the ball $B_Q=B(x_Q,r)$, as defined in \eqref{cube-ball}-\eqref{cube-ball2},
	with $r\approx\ell(Q)$, so that $\Delta_Q=B_Q\cap E\subset Q$.
	By
	\cite[Lemma 5.3]{HM-I}, we have that for some $C=C(n,ADR)$,
	\[
	\big|\{Y\in\ree\setminus E: \,\dist(Y,E)<\epsilon r\}\cap B_Q\big|\le C\,\epsilon\,r^{n+1}\,,
	\]
	for every $0<\epsilon<1$. Consequently, fixing $0<\epsilon_0<1$ small enough, 
	there exists $X_Q\in \frac12\,B_Q$, with $\dist(X_Q,E)\ge \epsilon_0\,r$. 
	Thus, $B(X_Q,\epsilon_0\,r/2)\subset B_Q\setminus E$.  We shall refer to this point $X_Q$
	as a ``Corkscrew point" relative to $Q$, that is, relative to the surface ball $\Delta_Q$ (cf. \eqref{cube-ball} and \eqref{cube-ball2}).  Now observe that $X_Q$ belongs to some
	Whitney cube $I\in\W$, which will belong to $\W^0_Q$, for $\eta$ small enough and $K$ large enough. Hence, $\{\W_Q\}_{Q\in\dd(E)}$ is a Whitney-dyadic structure for $\ree\setminus E$.

In  \cite{HMM2} it was shown that the ADR property is inherited by all dyadic local sawtooths and all Carleson boxes:

\begin{proposition}[{\cite[Proposition A.2]{HMM2}}]
Let $E\subset\ree$ be a closed $n$-dimensional ADR set and let $\{\W_Q\}_{Q\in\dd(E)}$ be a Whitney-dyadic structure for $\ree\setminus E$ with parameters $\eta\ll 1$ and $K\gg 1$. Then all dyadic local sawtooths $\Omega_{\mathcal{F},Q}$ and all Carleson boxes $T_Q$ have $n$-dimensional ADR boundaries. In all cases, the implicit constants are uniform and depend only on dimension, the ADR constant of $E$ and the parameters $\eta$, $K$, and $\tau$.
\end{proposition}

\subsection{Case UR}\label{sections:UR} 

Here we assume that $\Omega=\ree\setminus E$ where we further assume that $E$ is UR. Much as before, since $E$ is in particular ADR, if we take $\eta\ll 1$ and $K\gg 1$ (depending on $n$ and the ADR constant of $E$) we can guarantee that $\W_Q^0\neq\emptyset$. In this case we will exploit the additional fact that $E$ is UR to construct some  Whitney-dyadic structure with better properties. To do so, we would like to recall some results from \cite{HMM2} but we first give a definition to then continue with the main geometric lemma from \cite{HMM2}.

\begin{definition}\cite{DS2}.\label{d3.11}   
	Let $\sbf\subset \dd(E)$. We say that $\sbf$ is
	``coherent" if the following conditions hold:
	\begin{itemize}\itemsep=0.1cm
		
		\item[$(a)$] $\sbf$ contains a unique maximal element denoted by $Q(\sbf)$ which contains all other elements of $\sbf$ as subsets.
		
		\item[$(b)$] If $Q$  belongs to $\sbf$, and if $Q\subset \widetilde{Q}\subset Q(\sbf)$, then $\widetilde{Q}\in {\bf S}$.
		
		\item[$(c)$] Given a cube $Q\in \sbf$, either all of its children belong to $\sbf$, or none of them do.
		
	\end{itemize}
	We say that $\sbf$ is ``semi-coherent'' if only conditions $(a)$ and $(b)$ hold. 
\end{definition}

\smallskip

\begin{lemma}[{The bilateral corona decomposition, \cite[Lemma 2.2]{HMM2}}]\label{lemma2.1}  Suppose that $E\subset \ree$ is $n$-dimen\-sional UR.  Then given any positive constants
	$\eta\ll 1$
	and $K\gg 1$, there is a disjoint decomposition
	$\dd(E) = \G\cup\B$, satisfying the following properties.
	\begin{list}{$(\theenumi)$}{\usecounter{enumi}\leftmargin=.8cm
			\labelwidth=.8cm\itemsep=0.2cm\topsep=.1cm
			\renewcommand{\theenumi}{\roman{enumi}}}
		
		\item  The ``Good" collection $\G$ is further subdivided into
		disjoint stopping time regimes, such that each such regime {\bf S} is coherent (cf. Definition \ref{d3.11}).

		\item The ``Bad" cubes, as well as the maximal cubes $Q(\sbf)$ satisfy a Carleson
		packing condition:
		$$\sum_{Q'\subset Q, \,Q'\in\B} \sigma(Q')
		\,\,+\,\sum_{\sbf: Q(\sbf)\subset Q}\sigma\big(Q(\sbf)\big)\,\leq\, C_{\eta,K}\, \sigma(Q)\,,
		\quad \forall Q\in \dd(E)\,.$$
		\item For each $\sbf$, there is a Lipschitz graph $\Gamma_{\sbf}$, with Lipschitz constant
		at most $\eta$, such that, for every $Q\in \sbf$,
		\begin{equation}\label{eq2.2a}
		\sup_{x\in \Delta_Q^*} \dist(x,\Gamma_{\sbf} )\,
		+\,\sup_{y\in B_Q^*\cap\Gamma_{\sbf}}\dist(y,E) < \eta\,\ell(Q)\,,
		\end{equation}
		where $B_Q^*:= B(x_Q,K\ell(Q))$ and $\Delta_Q^*:= B_Q^*\cap E$.
	\end{list}
\end{lemma}

As we have assumed that $E$ is UR we make the corresponding bilateral corona decomposition of Lemma \ref{lemma2.1} with $\eta\ll 1$ and $K\gg 1$. Our goal is to 
construct, for each stopping time regime $\sbf$
in Lemma \ref{lemma2.1}, a pair of CAD domains $\Omega_{\sbf}^\pm$, which provide a good approximation to $E$, at the scales within $\sbf$, in some appropriate sense.  To be a bit more precise, $\Omega_{\sbf}:= \Omega_{\sbf}^+\cup \Omega_{\sbf}^-$ will be constructed as a sawtooth region
relative to some family of dyadic cubes, and the nature of this construction will be essential to the dyadic analysis that we will use below.

Given $Q\in \dd(E)$, for this choice of $\eta$ and $K$, we set as above $B_Q^*:= B(x_Q, K\ell(Q))$, where we recall that $x_Q$ is the center of $Q$ (see \eqref{cube-ball}-\eqref{cube-ball2}). For a fixed stopping time regime $\sbf$, we choose a co-ordinate system
so that $\Gamma_{\sbf} =\{(z,\vp_{\sbf}(z)):\, z\in \rn\}$, where $\vp_{\sbf}:\rn\longrightarrow \re$ is a Lipschitz function with
$\|\vp\|_{\rm Lip} \leq\eta$.

\begin{claim}[{\cite[Claim 3.4]{HMM2}}]\label{c3.1} If $Q\in \sbf$, and $I\in \W^0_Q$, then $I$ lies either above
	or below $\Gamma_{\sbf}$.  Moreover,
	$\dist(I,\Gamma_{\sbf})\geq \eta^{1/2} \ell(Q)$ (and therefore, by \eqref{eq2.2a},  
	$\dist(I,\Gamma_{\sbf}) \approx \dist(I,E)$, with  implicit constants that may depend on $\eta$
	and $K$).
\end{claim}

Next, given $Q\in\sbf$,  we augment $\W^0_Q$.  We split $\W^0_Q=\W_Q^{0,+} \cup \W_Q^{0,-}$,
where $I\in\W_Q^{0,+}$ if $I$ lies above
$\GS$, and  $I\in\W_Q^{0,-}$ if $I$ lies below
$\GS$.   Choosing $K$ large and $\eta$ small enough, by \eqref{eq2.2a}, 
we may assume that both $\W_Q^{0,\pm}$ are non-empty.
We focus on $\W_Q^{0,+}$, as the construction for $\W_Q^{0,-}$ is the same.
For each $I\in \W_Q^{0,+}$, let $X_I$ denote the center of $I$.
Fix one particular $I_0\in \W_Q^{0,+}$, with center $X^+_Q:= X_{I_0}$.
Let $\widetilde{Q}$ denote the dyadic parent of $Q$ (that is, the unique dyadic cube $\widetilde{Q}$ with $Q\subset\widetilde{Q}$ and $\ell(\widetilde{Q})=2\ell(Q)$), unless $Q=Q(\sbf)$; in the latter case
we simply set $\Qt=Q$.  Note that
$\widetilde{Q}\in\sbf$,
by the coherency of $\sbf$.
By Claim \ref{c3.1}, for each $I$ in $\W_Q^{0,+}$, or in   $\W_{\widetilde{Q}}^{0,+}$,
we have
$$\dist(I,E)\approx\dist(I,Q)\approx\dist(I,\GS)\,,$$
where the implicit constants may depend on $\eta$ and $K$.  Thus, for each
such $I$, we may fix a Harnack chain, call it
$\mathcal{H}_I$, relative to the Lipschitz domain
$$
\Omega_{\GS}^+:=\left\{(x,t) \in \ree: t>\vp_{\sbf}(x)\right\}\,,
$$
connecting $X_I$ to $X_Q^+$.  By the bilateral approximation condition
\eqref{eq2.2a}, the definition of $\W^0_Q$, and the fact that $K^{1/2}\ll K$,
we may construct this Harnack Chain so that it consists of a bounded
number of balls (depending on $\eta$ and $K$), and stays a distance at least
$c\eta^{1/2}\ell(Q)$ away from $\GS$ and from $E$.
We let $\W^{*,+}_Q$ denote the set of all $J\in\W$ which meet at least one of the Harnack
chains $\mathcal{H}_I$, with $I\in \W_Q^{0,+}\cup\W_{\widetilde{Q}}^{0,+}$
(or simply $I\in \W_Q^{0,+}$, if $Q=Q(\sbf)$), i.e.,
$$\W^{*,+}_Q:=\left\{J\in\W: \,\exists \, I\in \W_Q^{0,+}\cup\W_{\Qt}^{0,+}
\,\,{\rm for\, which}\,\, \mathcal{H}_I\cap J\neq \emptyset \right\}\,,$$
where as above, $\widetilde{Q}$ is the dyadic parent of $Q$, unless
$Q= Q(\sbf)$, in which case we simply set $\Qt=Q$ (so the union is redundant).
We observe that, in particular, each $I\in \W^{0,+}_Q\cup \W^{0,+}_{\Qt}$
meets $\mathcal{H}_I$, by definition, and therefore
\begin{equation}\label{eqW}
\W_Q^{0,+}\cup \W^{0,+}_{\Qt} \subset \W_Q^{*,+}\,.
\end{equation}
Of course, we may construct $\W^{*,-}_Q$ analogously.
We then set
$$\W^*_Q:=\W^{*,+}_Q\cup \W^{*,-}_Q\,.$$
It follows from the construction of the augmented collections $\W_Q^{*,\pm}$ 
that there are uniform constants $c$ and $C$ such that
\begin{equation}\label{eq2.whitney2:*}
\begin{gathered}
c\eta^{1/2} \ell(Q)\leq \ell(I) \leq CK^{1/2}\ell(Q)\,, \quad \forall I\in \mathcal{W}^*_Q,
\\
\dist(I,Q)\leq CK^{1/2} \ell(Q)\,,\quad\forall I\in \mathcal{W}^*_Q.
\end{gathered}
\end{equation}

It is convenient at this point to introduce some additional terminology.
\begin{definition}\label{d3.2}  Given $Q\in \G$, and hence in  some $\sbf$, we shall refer to the point
	$X_Q^+$ specified above, as the ``center" of $U^{+}_Q$ (similarly, the analogous
	point $X_Q^-$, lying below
	$\GS$, is the ``center" of $U^{-}_Q$).  We also set $Y_Q^\pm := X^\pm_{\Qt}$,
	and we call this point the ``modified center" of $U_Q^\pm$, where as above $\Qt$ is the dyadic parent of
	$Q$, unless $Q=Q(\sbf)$, in which case $Q=\Qt$, and $Y_Q^\pm=X_Q^\pm$.
\end{definition}

Observe that $\W_Q^{*,\pm}$ and hence also $\W^*_Q$
have been defined for any $Q$ that belongs to some stopping time regime $\sbf$,
that is, for any $Q$ belonging to the ``good" collection $\G$ of Lemma \ref{lemma2.1}.  We now set
\begin{equation}\label{Wdef}
\W_Q:=\left\{
\begin{array}{l}
\W_Q^*\,,
\quad Q\in\G,
\\[6pt]
\W_Q^0\,,
\quad Q\in\B,
\end{array}
\right.\,
\end{equation}
and for $Q \in\G$ we shall henceforth simply write $\W_Q^\pm$ in place of $\W_Q^{*,\pm}$. Note that by \eqref{eq3.1} when $Q\in\B$ and by \eqref{eq2.whitney2:*} when $Q\in\G$ we clearly obtain \eqref{eq2.whitney2} with $C$ depending on $n$ and the UR character of $E$. By construction $\W_Q^0\subset \W_Q$. All these show that, provided $\eta\ll 1$ and $K\gg 1$ (depending on $n$ and the UR character of $E$), $\{\W_Q\}_{Q\in\dd(E)}$ is a Whitney-dyadic structure for $\ree\setminus E$ with parameter $\eta$ and $K$ and with $C$ depending on $n$ and the UR character of $E$.

Given an  arbitrary $Q\in\dd(E)$ and $0<\tau\le\tau_0/4$, we may define an associated Whitney region $U_Q$ (not necessarily connected) as in \eqref{eq3.3aa}
or the fattened version of $\widehat{U}_Q$ as in \eqref{eq3.3aa-fat}. In the present situation, if $Q\in\G$, then $U_Q$ splits into exactly two connected components
\begin{equation}\label{eq3.3b}
U_Q^\pm= U_{Q,\tau}^{\pm}:= \bigcup_{I\in \W^{\pm}_Q} I^*(\tau)\,.
\end{equation}
We note that for $Q\in\G$, each $U_Q^{\pm}$ is Harnack chain connected, by construction
(with constants depending on the implicit parameters $\tau, \eta$ and $K$);
moreover, for a fixed stopping time regime $\sbf$,
if $Q'$ is a child of $Q$, with both $Q',\,Q\in \sbf$, then
$U_{Q'}^{+}\cup U_Q^{+}$
is Harnack Chain connected, and similarly for
$U_{Q'}^{-}\cup U_Q^{-}$.

We may also define the Carleson boxes $T_Q$, global and local sawtooth regions  $\Omega_\F$, $\Omega_{\F,Q}$, cones $\Gamma$, and local cones $\Gamma^Q$  as in \eqref{eq3.3a} \eqref{eq2.sawtooth1}, \eqref{defcone}, and \eqref{defconetrunc}.

\smallskip

\begin{remark}\label{r3.11a}
	We recall that, by construction
	(cf. \eqref{eqW}, \eqref{Wdef}),  given $Q\in\G$, one has $\W_{\Qt}^{0,\pm}\subset \W_Q$, where $\widetilde{Q}$ is the dyadic parent of $Q$. Therefore,
	$Y_Q^\pm \in U_Q^\pm\cap U^\pm_{\Qt}$.  Moreover, since $Y_Q^\pm$ is the center of some $I\in \W_{\Qt}^{0,\pm}$,
	we have that $\dist(Y_Q^\pm, \partial U_Q^\pm)\approx \dist(Y_Q^\pm, \partial U_{\Qt}^\pm) \approx \ell(Q)$
	(with implicit constants possibly depending on
	$\eta$ and/or $K$)
\end{remark}

\smallskip

\begin{remark}\label{r3.11}
	Given a stopping time regime $\sbf$ as in Lemma \ref{lemma2.1}, for any semi-coherent
	subregime (cf. Definition \ref{d3.11}) $\sbf'\subset \sbf$ (including, of course, $\sbf$ itself), we now set
	\begin{equation}\label{eq3.2}
	\Omega_{\sbf'}^\pm = {\rm int}\left(\bigcup_{Q\in\sbf'} U_Q^{\pm}\right)\,,
	\end{equation}
	and let $\Omega_{\sbf'}:= \Omega_{\sbf'}^+\cup\Omega^-_{\sbf'}$.
	Note that implicitly, $\Omega_{\sbf'}$ depends upon $\tau$ (since $U_Q^\pm$ has such dependence).
	When it is necessary to consider the value of $\tau$
	explicitly, we shall write $\Omega_{\sbf'}(\tau)$.
\end{remark}

The main geometric lemma for the associated sawtooth regions is the following.

\begin{lemma}[{\cite[Lemma 3.24]{HMM2}}]\label{lemma3.15} Let $\sbf$ be a given
	stopping time regime as in Lemma \ref{lemma2.1},
	and let $\sbf'$ be any nonempty, semi-coherent  subregime of $\sbf$.
	Then for $0<\tau\leq\tau_0$, with $\tau_0$ small enough,
	each of $\Omega^\pm_{\sbf'}$ is a CAD with character depending only on $n,\tau,\eta, K$, and the
	UR character of $E$.
\end{lemma}

\subsection{Case 1-sided CAD}\label{sections:1-sided:CAD} 
Here we assume that $\Omega$ is a 1-sided CAD. In this case, we are basically in the situation which is similar to being within one regimen ${\bf S}$, at least as far as the construction of $\W_Q$  is concerned.  

With $\W=\W(\Omega)$ and $\dd=\dd(\pom)$ as above, and for some give parameters $\eta<1$, $K> 1$, we consider $\W_Q^0$ (cf. \eqref{eq3.1}). For any $Q\in\dd$ we let $X_Q$ be a corkscrew point relative to $Q$, more specifically, relative to $\Delta_Q$ (cf. \eqref{cube-ball}-\eqref{cube-ball2}). We note that in this scenario the existence of such point comes from the fact that $\Omega$ satisfies the (interior) Corkscrew condition). For $\eta\ll 1$ and $K\gg 1$ depending on the CAD character of $\Omega$ we can guarantee that for every $Q\in\dd$, if $I\in\W$ is so that $I\ni X_Q$ then $I\in \W_Q^0$. We then augment $\W_Q^0$ to $\W_Q^*$ as done in \cite[Section 3]{HM-I}. More precisely, use the fact that one can construct a Harnack chain to connect $X_Q$ with any of the centers of the Whitney cubes in $\W_Q^0\cup\W_{\widetilde{Q}}^0$ where 
$\widetilde{Q}$ is  the dyadic parent of $Q$. Then $\W_Q^*$ is the  family of all Whitney cubes which meet at least one ball in all those Harnack chains. 
Note that in the case when $E$ is UR and $Q\in {\bf S}$ we have used a similar idea, the main difference is that the Harnack chain in that case comes from the fact that $\Omega_{\GS}^+$ is a Lipschitz domain, whereas here such property comes from the assumption that $\Omega$ is a 1-sided CAD and hence the Harnack chain condition holds.  Set then $\W_Q=\W_Q^*$ and one can see that (with the appropriate choice of a sufficiently small $\eta$ and a sufficiently large $K$ depending on $n$ and the CAD character of $D$), that \eqref{eq2.whitney2} hold. Moreover, the construction guarantees that $\W_Q^{0}\cup \W^{0}_{\Qt}\subset \W_Q$, that we can cover with the Whitney cubes in $\W_Q$ all the Harnack chains connecting $X_Q$ with any center of $I\in \W_Q^{0}\cup \W^{0}_{\Qt}\subset \W_Q$, and also that if $I$, $J$ are such that $I\ni X_Q$ and $J\ni X_{\Qt}$ then $I,J\in \W_Q$. We note that by construction the Harnack chain condition holds in each Whitney region $U_Q$ and so it does in $U_Q\cup U_{\widetilde{Q}}$. In either case the corresponding constant depends on the CAD character of $D$ and the parameters $\eta$, $K$, $\tau$. 

In the present situation we have the following geometric result:

\begin{lemma}[{\cite[Lemma 3.61]{HM-I}}]\label{lemma:1-sidedCAD-geom} 
Let $\Omega\subset\ree$ be a 1-sided CAD and let $\{\W_Q\}_{Q\in\dd(\pom)}$ be a Whitney-dyadic structure for $\Omega$ with parameters $\eta\ll 1$ and $K\gg 1$ as just constructed. Then all of its dyadic sawtooths regions $\Omega_\F$ and $\Omega_{\mathcal{F},Q}$, and all Carleson boxes $T_Q$ and $T_\Delta$ are also 1-sided CAD with character depending only on dimension, the 1-sided CAD character of $\Omega$,  and the parameters $\eta$, $K$, and $\tau$.
\end{lemma}

\subsection{Case CAD}\label{sections:CAD} 
Here we assume that $\Omega$ is a CAD.  This is, strictly speaking, a sub-case of 
the Case 1-sided CAD above, but the extra assumption that $\Omega$  has 
exterior corkscrews can be inferred to the associated sawtooth regions and Carleson boxes. 

With $\W=\W(\Omega)$ and $\dd=\dd(\pom)$ as above, and for some give parameters $\eta<1$, $K> 1$, we consider $\W_Q^0$ (cf. \eqref{eq3.1}) and construct $\W_Q$ exactly as in the 1-sided CAD case since a CAD is in particular 1-side CAD. Hence, we have the very same properties, in particular, Lemma \ref{lemma:1-sidedCAD-geom} applies. But we can additionally obtain the exterior corkscrew condition: 

\begin{lemma}\label{lemma:CAD-geom} 
	Let $\Omega\subset\ree$ be a CAD and let $\{\W_Q\}_{Q\in\dd(\pom)}$ be a Whitney-dyadic structure for $\Omega$ with parameters $\eta\ll 1$ and $K\gg 1$ as just constructed. Then all of its dyadic sawtooths regions $\Omega_\F$ and $\Omega_{\mathcal{F},Q}$, and all Carleson boxes $T_Q$ and $T_\Delta$ are also CAD with character depending only on dimension, the CAD character of $\Omega$,  and the parameters $\eta$, $K$, and $\tau$.
\end{lemma}

\begin{proof}
As mentioned above we can apply Lemma \ref{lemma:1-sidedCAD-geom}, hence all the $\Omega_\F$, $\Omega_{\mathcal{F},Q}$, $T_Q$, and $T_\Delta$ are 1-sided CAD domains. It remains to see that any of them satisfy the exterior corkscrew condition. Let $\Omega_\star$ be one of these subdomains and take $x_\star\in\partial \Omega_\star$ and $0<r<\diam(\partial \Omega_\star)$. By construction $\partial \Omega_\star\subset\overline{\Omega}$ and we consider two cases $0\le \dist(x_\star,\pom)\le r/2$ and $\dist(x_\star,\pom)> r/2$. In the first scenario we pick $x\in\pom$ so that $|x_\star-x|=\dist(x_\star,\pom)\le r/2$ (notice that $x=x_\star$ if $x_\star\in\pom\cap\pom_\star$). Since $\Omega$ is a CAD it satisfies the exterior Corkscrew condition, hence we can find $X\in \Omega_{\rm ext}=\ree\setminus \overline{\Omega}$ so that
$B(X,c_0 r/2)\subset B(x,r/2)\cap \Omega_{\rm ext}$ where $c_0$ is the exterior corkscrew constant. Note that $\Omega_\star\subset \Omega$, hence $B(X,c_0 r/2)\subset (\Omega_\star)_{\rm ext}$. Also,
$B(X,c_0 r/2)\subset B(x,r/2)\subset B(x_\star,r)$. This shows that $X$ is an exterior corkscrew point relative to the surface ball $B(x_\star,r)\cap\partial \Omega_\star$ for the domain $\Omega_\star$ with constant $c_0/2$. Consider next the case on which $\dist(x_\star,\pom)> r/2$. Note that in particular $x_\star\in \Omega$ and therefore we can find two Whitney cubes $I,J\in \W$ so that $x\in \partial I^*\cap J$, $\partial I\cap\partial J\neq\emptyset$, $\interior(I^*) \subset \Omega_\star$ and $J$ is a Whitney cube which does not belong to any of the $\W_Q$ that define $\Omega_\star$. Note that $\ell(J)\ge \dist(x_\star,\pom)/C>r/(2C)$ for some uniform constant $C\ge 1$, that $I^*$ misses $\frac34 J$ as observed before and that the center of $J$ satisfies $X(J)\in (\Omega_\star)_{\rm ext}$. It is then clear that the open segment joining  $x_\star$ with $X(J)$ is contained in $(\Omega_\star)_{\rm ext}$ and we pick $X$ in that segment so that $|X-x_\star|=r/(8C)$ and hence $B(X,r/(16C))\subset B(x_\star,r)\cap \Omega_\star$. This shows that $X$ is an exterior corkscrew point relative to the surface ball $B(x_\star,r)\cap\partial \Omega_\star$ for the domain $\Omega_\star$ with constant $1/(16C)$. Therefore, we have shown that $\Omega$ satisfies  the exterior Corkscrew condition with implicit constant uniformly controlled by the CAD character of $\Omega$. 
\end{proof}

\subsection{Some important notation}

To complete this section we introduce the following notation which will be used in our main statements:

\begin{notation}\label{notation:constants} 
	In the statements of our main results, we will assume that some estimates (e.g., Carleson estimates, ``$\mathcal{A}<N$'', ``$N<S$'', etc.) hold for a given family of subsets with constants depending on the character of those subsets and our goal is to transfer those estimates to the original set. It is crucial to explain how this dependence on the character is understood. To set the stage suppose that we are given some set $\mathbb{X}\subset\ree$ 
and a  family $\mathbb{S}_\mathbb{X}:=\{\mathbb{Y}\}_{\mathbb{Y}\in\mathbb{S}_\mathbb{X}}$, $\mathbb{Y}\subset \mathbb{X}$.
 We assume that associated with $\mathbb{X}$ there is some collection of non-negative parameters $M_{\mathbb{X}}\in [1,\infty)^{N_1}$ called its character and also that each $\mathbb{Y}\in\mathbb{S}_\mathbb{X}$ has some associated character $M_{\mathbb{Y}}\in [1,\infty)^{N_2}$, a collection of non-negative parameters. Using this notation when we say that certain estimate holds for all $\mathbb{Y}\in \mathbb{S}_\mathbb{X}$ with constant $C_\mathbb{Y}$ depending on the character of $\mathbb{Y}$, we mean that $C_\mathbb{Y}=\Theta(M_\mathbb{Y})$ with $\Theta:[1,\infty)^{N_2}\to (0,\infty)$ being a non-decreasing function in each variable. Implicit in the arguments to transfer the desired estimate to $\mathbb{X}$, 
we will use only those sets  $\mathbb{Y}\in \mathbb{S}_\mathbb{X}$ whose parameters in the 
character are all uniformly controlled by some constant $M_0$ (which will  depend on the 
character of $\mathbb{X}$), and 
then all the corresponding constants in the assumed estimates for 
those sets will be controlled by $\Theta(M_0,\dots, M_0)<\infty$, 
and eventually the desired 
estimate on $\mathbb{X}$  will depend on $\Theta(M_0,\dots, M_0)$.

	It is illustrative to present some examples explaining the previous abstract notation in some particular cases. Suppose that the goal is to show that some function $F$ satisfies the Carleson measure estimate \eqref{eqdefCME} in $\mathbb{X}=\ree\setminus E$, with $E$ being UR (see  the second part of 
Theorem \ref{theor:CME:CAD->UR}). In this case $M_{\mathbb{X}}\in [1,\infty)^3$ is the 
UR character of $E$, and we let $\mathbb{S}_\mathbb{X}$ be the collection of 
bounded chord-arc subdomains of $\mathbb{X}$, in which case  $M_{\mathbb{Y}}\in [1,\infty)^4$ is 
the CAD character of $\mathbb{Y}$. With this in hand we show that there is a 
constant $M_0$ (depending 
only on $M_\mathbb{X}$, dimension, and the harmless discretionary
parameters $\tau,\eta$ and $K$, 
and thus independent of $F$; see Lemma \ref{lemma3.15})
so that the resulting estimate can be transferred from the 
collection of CAD with parameters in the character at most $M_0$, and hence the 
Carleson estimate \eqref{eqdefCME} holds with a constant depending only on 
$\Theta(M_0,M_0, M_0, M_0)$, and other harmless parameters.
Similarly, another example is the case that 
$\mathbb{X}=D$ is a CAD, hence $M_{\mathbb{X}}\in [1,\infty)^4$ is its CAD character, and $\mathbb{S}_\mathbb{X}$ is some collection of bounded Lipschitz chord-subdomains of $\mathbb{X}$, then $M_{\mathbb{Y}}\in [1,\infty)^3$ is the Lipschitz CAD character of $\mathbb{Y}$. 
\end{notation}

\section{Transference of Carleson measure estimates}

In this section we show how to transfer CME estimates from Lipschitz to CAD (see Theorem \ref{theor:CME:Lip->CAD}) and then from CAD to the complement of a UR set (see Theorem \ref{theor:CME:CAD->UR}). These two 
independent results, each interesting in its own right,  can be combined to give immediately 
the following:

\begin{corollary}\label{corol:CME:Lip->UR}  	
	Let $E\subset \ree$ be an $n$-dimensional  UR set and let $F\in L^2_{\loc}(\ree\setminus E)$. If $F$ satisfies the Carleson measure estimate \eqref{eqdefCME} for every bounded Lipschitz subdomain of $\ree\setminus E$ with constant depending on the Lipschitz character (see Notation \ref{notation:constants}), then $F$ satisfies the Carleson measure estimate \eqref{eqdefCME} in $\ree\setminus E$ as well. More precisely, there exists a large constant $M_0$ 
(depending only $n$ and the UR character of $E$\footnote{Our estimates depend also on
the discretionary parameters 
$\tau, \eta$ and $K$
introduced above, but in turn each of these may be chosen to depend
at most on $n$ and the $UR$ character of $E$.}) so that using the notation in \eqref{eqdefCME} there holds
	\begin{equation}\label{eqn2.2-general*:generalthm}
	\|F\|_{\C(\ree\setminus E)} 	\le 
	C\,\sup_{\Omega\subset \ree\setminus E}  \|F\|_{\C(\Omega)},
	\end{equation}
	where the sup runs over all bounded Lipschitz subdomains $\Omega\subset  \ree\setminus E$ with parameters in the Lipschitz character at most $M_0$, and $C$ depends as before only on $n$, and the UR character of $E$. 
\end{corollary}

\begin{remark}\label{remark:CME-E-OMmga}
	The previous result (and also Theorem \ref{theor:CME:CAD->UR}) easily yields a version of itself where everything is localized to some open subset with UR boundary. More precisely, let $\Omega\subset\ree$ be an open set with $\pom$ being UR and let $F\in L^2_{\loc}(\Omega)$. Then
	\begin{equation}\label{eqn2.2-general:Omega}
	\|F\|_{\C(\Omega)} 	\le 
	C\,\sup_{D\subset\Omega}  \|F\|_{\C(D)},
	\end{equation}
	where the sup runs over all bounded Lipschitz subdomains $D\subset  \Omega$ with parameters in the Lipschitz character at most $M_0$, and $C$ depends only on $n$ and the UR character of $\pom$. 
	
	To see this, write $F_\Omega:=F$ in $\Omega$ and $F_\Omega=0$ in $\ree\setminus\overline\Omega$ so that $F\in L^2_{\loc}(\ree\setminus \pom)$. Since $\pom$ is UR we can apply Corollary \ref{corol:CME:Lip->UR} to $E=\pom$ and \eqref{eqn2.2-general*:generalthm} easily yields
	\[
	\|F\|_{\C(\Omega)} 	
	=
	\|F_\Omega\|_{\C(\ree\setminus \pom)} 	
	\le 
	C\,
	\sup_{D\subset\ree\setminus \pom}  \|F_\Omega\|_{\C(D)}
	=
	C\,
	\sup_{D\subset\Omega}  \|F_\Omega\|_{\C(D)}
	.
	\]
\end{remark}

\subsection{Transference of Carleson measure estimates: 
from Lipschitz to chord-arc domains}\label{sjn}
In this section we present a method to transfer the CME estimates from Lipschitz domains to CAD. Our main result is as follows:
	
	\begin{theorem}\label{tn3.1******} 
		Let $D\subset \ree$ be a given CAD and assume that $F\in L^2_{\rm loc}(D)$ satisfies \eqref{eqdefCME-interior:def}. If $F$ satisfies the Carleson measure estimate \eqref{eqdefCME} on all bounded Lipschitz subdomains of $D$ with the constant $C=C_0$ depending on the Lipschitz constants of the underlying domains only, then $F$ satisfies the Carleson measure estimate \eqref{eqdefCME} in $D$ as well, with the bound depending on $C_0$, the constant in \eqref{eqdefCME-interior:def}, the NTA constants of $D$ and the ADR constants of $\partial D$ only.
	\end{theorem}

\begin{theorem}\label{theor:CME:Lip->CAD}  	
	Let $D\subset \ree$ be a given CAD and let $F\in L^2_{\rm loc}(D)$. If $F$ satisfies the Carleson measure estimate \eqref{eqdefCME} for every bounded Lipschitz subdomain of $D$ with constant depending on the Lipschitz character (see Notation \ref{notation:constants}), then $F$ satisfies the Carleson measure estimate \eqref{eqdefCME} in $D$ as well. More precisely, there exists a large constant $M_0$ (depending only $n$ and the CAD character of $D$) so that using the notation in \eqref{eqdefCME} there holds
	\begin{equation}\label{eqn2.2-general*}
	\|F\|_{\C(D)} 	\le 
	C\,\sup_{\Omega\subset D}  \|F\|_{\C(\Omega)},
	\end{equation}
	where the sup runs over all bounded Lipschitz subdomains $\Omega\subset  D$ with parameters in the Lipschitz character at most $M_0$, and $C$ depends as before only on $n$, and the CAD character of $D$. 
\end{theorem}

Let us remark that in the course of the proof we ensure a suitable choice of a (sufficiently small) $\eta$ and a (sufficiently large) $K$ is \eqref{eq3.1} which strictly speaking affect the constant in \eqref{eqn2.2-general*}. However, as all choices depend on dimension and the CAD character only, this does not affect the result as stated above. 

In preparation to prove the previous result we start with the following version of the John-Nirenberg inequality. It is a suitable modification of  \cite[Lemma~10.1]{HM08} which, in turn, was inspired by \cite[Lemma~2.14 ]{AHLT}.  Here we present an alternative proof along the lines in \cite[Lemma A.1]{MMM}. Given $\Omega$ an open set with an ADR boundary, let $Q_0$ will be either $\pom$ in which case $\dd_{Q_0}=\dd(\pom)$ or some fixed dyadic cube in $\dd(\pom)$ in which case $\dd_{Q_0}$ is defined in \eqref{eq3.4a}.

\begin{lemma}\label{lemma:J-N} 
	Let $\Omega$ be an open set with an ADR boundary, let $Q_0$ be either $\pom$ or a fixed cube in $\dd(\pom)$, and for some given $\eta\ll 1$ and $K\gg 1$, consider a  Whitney-dyadic structure $\{\W_Q\}_{Q\in\dd(\pom)}$ for $\Omega$ with parameters $\eta$ and $K$ as in Definition \ref{def:WD-struct}. 
	Let $F\in
	L^2_{\loc}(\Omega)$ and suppose that there exist $0<\alpha<1$ and $0<N<\infty$ such that 
	\begin{equation}\label{eqjn10.1}
	\sigma\left\{x\in Q:\,\A^Q F(x)> N\right\}\leq 
	\alpha\, \sigma(Q), \qquad\forall\, Q\in \dd_{Q_0}.
	\end{equation}
	Then, for every $0<p<\infty$ there exists $C_{\alpha,p}$ depending only on $p$ and $\alpha$ such that 
	\begin{equation}\label{eqjn10.2}
	\sup_{Q\in\dd_{Q_0}}\,\fint_Q \A^Q F(x)^p\,d\sigma(x)
	\leq 
	C_{\alpha,p} N^p.
	\end{equation}
\end{lemma}

\begin{proof} 
We first claim that for all $Q\in\dd_{Q_0}$
	\begin{equation}\label{GQ}
	\A^Q F(x) \le \A^{Q'} F(x)+\inf_{y\in \widetilde{Q}'} \A^Q F(y),
	\qquad \forall x\in Q'\in\dd_Q \setminus \{Q\},
	\end{equation}
	where $ \widetilde{Q}'$ is the dyadic parent of $Q'$. 
	This follows easily from the fact that if $x\in Q'\in\dd_Q \setminus \{Q\}$  and $y\in \widetilde{Q}'$ then 
	\[
	\Gamma^Q(x)\setminus\Gamma^{Q'}(x)
	\subset
	\bigcup_{x\in P \in\dd_Q\setminus\dd_{Q'}} U_{P}
	=
	\bigcup_{\widetilde{Q}'\subset P\subset Q}U_{P}
	\subset 
	\Gamma^Q(y).
	\]
		Next, let us set
	\begin{equation}
	\Xi(t):=\sup_{Q\in\mathbb{D}_{Q_0}}\frac{\sigma(E_Q(t))}{\sigma(Q)}
	:=
	\sup_{Q\in\mathbb{D}_{Q_0}} \frac{\sigma\{x\in Q: \A^Q F(x)>t\}}{\sigma(Q)},\qquad 0<t<\infty.
	\end{equation}
	From \eqref{eqjn10.1} it follows that
	\begin{equation}\label{eq:esti-unif-level-sets}
	\sigma(E_Q(N))
	:=
	\sigma\{x\in Q:\,\A^Q F(x)>N\}	\leq\alpha\sigma(Q),
	\qquad\forall\,Q\in\mathbb{D}_{Q_0}.
	\end{equation}
	Fix now $Q\in\mathbb{D}_{Q_0}$, $\beta\in(\alpha,1)$ (we will eventually let $\beta \to 1^{+})$ and, recalling the notation introduced in Definition \ref{defHLmax} with $E=\pom$,
	set
	\begin{equation}\label{643rUHG}
	F_Q(N) :=\{x\in Q:\,M_Q^{\dd}\big(1_{E_Q(N) }\big)(x)>\beta\}.
	\end{equation}
	Note that \eqref{eq:esti-unif-level-sets} ensures that 
	\begin{equation}\label{6f4rdDD}
	\fint_Q 1_{E_Q(N)}(y)\,d\sigma(y)=\frac{\sigma(E_Q(N))}{\sigma(Q)}\leq\alpha<\beta,
	\end{equation}
	hence we can extract a family of pairwise disjoint stopping-time cubes 
	$\{Q_j\}_j\subset\mathbb{D}_{Q}\setminus\{Q\}$ so that $F_Q(N)=\cup_j Q_j$ and for every $j$
	\begin{equation}\label{eq:stopping-time}
	\frac{\sigma(E_Q(N) \cap Q_j)}{\sigma(Q_j)}>\beta;\quad
	\qquad\frac{\sigma(E_Q(N) \cap Q')}{\sigma(Q')}\leq\beta,
	\quad Q_j\subsetneq Q'\in\mathbb{D}_Q.
	\end{equation}
	
	Fix $t>N$. Observe that $E_Q(t)\subset E_Q(N)$ and 
	\begin{equation}\label{tdr-aafDD}
	\beta<1=1_{E_Q(N)}(x)\leq M_Q^{\dd}\big(1_{E_Q(N)}\big)(x),
	\qquad \text{for $\sigma$- a.e.~}x\in E_{Q}(t).
	\end{equation}
	Hence,
	\[
	\sigma(E_Q(t))=\sigma(E_Q(t)\cap F_Q(N))=\sum_j \sigma(E_Q(t)\cap Q_j).
	\]
	For every $j$, by the second estimate in \eqref{eq:stopping-time} 
	applied to $\widetilde{Q}_j$, the dyadic parent of $Q_j$, we have  
	$\sigma(E_Q(N)\cap\widetilde{Q}_j)/\sigma(\widetilde{Q}_j)\leq\beta<1$, therefore 
	$\sigma(\widetilde{Q}_j\setminus E_Q(N))/\sigma(\widetilde{Q}_j)\ge 1-\beta>0$. 
	In particular, we can pick $x_j\in \widetilde{Q}_j\setminus E_Q(N)$. This and \eqref{GQ} imply that for all $x\in Q_j$
	\[
	\A^Q F(x)
	\leq 
	\A^{Q_j} F(x)+ 
	\inf_{y\in \widetilde{Q}_j} \A^Q F(y)
	\le
	\A^{Q_j} F(x)+ \A^{Q_j} F(x_j)
	\le
	\A^{Q_j} F(x)+ N.
	\]
	Consequently, $\A^{Q_j} F(x)>t-N$ for every $x\in E_Q(t)\cap Q_j$ which further implies
	\[
	\sigma(E_Q(t)\cap Q_j)
	\leq
	\sigma\{x\in Q_j:\,\A^{Q_j} F(x)>t-N\}
	\leq
	\Xi(t-N)\,\sigma(Q_j).
	\]
	All these give
	\begin{multline}
	\sigma(E_Q(t))
	=
	\sum_j\sigma(E_Q(t)\cap Q_j)
	\leq
	\Xi(t-N)
	\sum_j \sigma(Q_j)
	\\
	\leq
	\Xi(t-N)
	\frac1{\beta}\sum_j \sigma(E_Q(N)\cap Q_j)
	\leq
	\Xi(t-N)\frac1\beta \sigma(E_Q(N))
	\leq
	\Xi(t-N)\frac\alpha\beta\sigma(Q),
	\end{multline}
	where we have used the first estimate in \eqref{eq:stopping-time}, that the cubes $\{Q_j\}_j$ are pairwise 
	disjoint and, finally, \eqref{eq:esti-unif-level-sets}. Dividing by $\sigma(Q)$ and taking the 
	supremum over all $Q\in\mathbb{D}_{Q_0}$ we obtain
	\begin{equation}\label{6tgEDac}
	\Xi(t)\leq\frac\alpha\beta\Xi(t-N),\qquad t>N. 
	\end{equation}
	Since this estimate is valid for all $\beta \in (\alpha,1)$, we can now let $\beta\to 1^{+}$, iterate the previous expression, and use 
	the fact that $\Xi(t)\leq 1$ to conclude that
	\[
	\Xi(t)\leq \alpha^{-1}e^{-\frac{\log(\alpha^{-1})}{N}t},\qquad t>0. 
	\]
	We finally see how the just obtained estimate implies \eqref{eqjn10.2}: for any $0<p<\infty$,
	\begin{multline*}
	\fint_{Q} \A^{Q} F(x)^p\,d\sigma(x)
	=
	p\int_0^\infty\frac{\sigma\{x\in Q:\,\A^{Q} F(x)>t\}}{\sigma(Q)}\,t^p\,\frac{dt}{t}
	\\
	\leq
	p\int_0^\infty\Xi(t)\, t^p\,\frac{dt}{t}
	\leq
	p\alpha^{-1}\int_0^\infty\, e^{-\frac{\log(\alpha^{-1})}{N}t}\,t^p\,\frac{dt}{t}
	\\
	=
	p \alpha^{-1} \left(\frac{N}{\log(\alpha^{-1})}\right)^p
	\int_0^\infty e^{-t}\,t^p\,\frac{dt}{t}
	=
	C_{\alpha,p} N^p. 
	\end{multline*}
	This completes the proof.
\end{proof}

To address the transference of the Carleson measure condition from Lipschitz to chord-arc domains we shall use the fact that chord-arc domains contain interior big pieces of Lipschitz subdomains due to \cite{DJe}.

\begin{proposition}[\cite{DJe}]\label{IBPLS}
	Given $\Omega\subset \ree$, a CAD, there exist constants $C\ge 2$ and $0<\theta<1$ such that for every surface ball $\Delta(x,r)=B(x,r)\cap \partial \Omega$, $x\in \pom $, $0<r<\diam (\pom)$, 
	there exists a bounded Lipschitz domain $\Omega'$ for which we have the following conditions: 
\begin{list}{$(\theenumi)$}{\usecounter{enumi}\leftmargin=.8cm
		\labelwidth=.8cm\itemsep=0.2cm\topsep=.1cm
		\renewcommand{\theenumi}{\roman{enumi}}}

		\item $H^n(\pom\cap \po'\cap B(x,r))\geq \theta H^n(\Delta(x,r))\approx \theta r^n$.
		\item There exists $X_{\Delta}$ so that $B(X_\Delta,r/C)\subset B(x,r)\cap \Omega\cap \Omega'$. 
		\item $\Omega'\subset \Omega\cap B(x, r)$.
	\end{list}
	The Lipschitz character of $\Omega'$ as well as $0<\theta<1$ and $C\ge 2$ depend on $n$, the CAD character of $D$ only (and are independent of $x,r$). 
\end{proposition}

We remark that in \cite{DJe}, Proposition~\ref{IBPLS} is proved under  
 weaker assumptions, namely, ADR and an interior corkscrew condition, and
 a ``weak exterior corkscrew condition" which entails  exterior disks 
 rather than exterior balls,
and with no hypothesis of Harnack chains  ---but if the Harnack chain condition is assumed, \cite{AHMNT} yields the exterior corkscrew condition, hence exterior disks implies exterior balls.  Later on, in \cite{Ba}, existence of 
big pieces of Lipschitz subdomains was also proved for usual 
NTA domains, with no upper ADR assumption on $\pom$ (the lower ADR  bound holds automatically
in the presence of a two-sided corkscrew condition, by virtue of the relative isoperimetric inequality). For the applications that we have in mind here, neither amelioration is significant, and we will simply work with CAD domains  in the sense of Definition~\ref{def1.nta}.

For future reference we also would like to provide the following corollary.  
\begin{corollary}\label{cIBPLS} 
Let $\Omega\subset \ree$ be a CAD. There exist constants $C\ge 2$ and $0<\theta<1$ such that, for every $Q\in \dd(\pom)$,
there exists a bounded Lipschitz domain $\Omega_Q\subset \Omega$ for which, using the notation $B_Q=B(x_Q,r)$, $\Delta_Q=B_Q\cap\pom$, with $c\ell(Q)\le r\le \ell(Q)$ in \eqref{cube-ball}, \eqref{cube-ball2}, we have the following:
\begin{list}{$(\theenumi)$}{\usecounter{enumi}\leftmargin=.8cm
		\labelwidth=.8cm\itemsep=0.2cm\topsep=.1cm
		\renewcommand{\theenumi}{\roman{enumi}}}
\item $\sigma(\po_Q\cap Q)\geq \theta\, \sigma (Q)\approx \theta \ell(Q)^n$.

\item For every $Q'\in \dd(Q)$ such that there exists a point $y_{Q'}\in Q'\cap \partial \Omega_Q$, there exists $Y_{Q'}$ so that $B(Y_{Q'}, \ell(Q')/C)\subset B(y_{Q'}, \ell(Q'))\cap\Omega\cap\Omega_{Q}$, that is, $Y_{Q'}$ is a corkscrew relative to $B(y_{Q'}, \ell(Q'))\cap\Omega$ and $\pom$, and $B(y_{Q'}, \ell(Q'))\cap\partial\Omega_Q$ and $\Omega_Q$. Furthermore, with the appropriate choice of $\eta$ and $K$ in \eqref{eq3.1}, we have $B(Y_{Q'}, \ell(Q')/C)\subset U_{Q'}$.

\item $\Omega_Q\subset \Omega\cap B_Q$.
\end{list}
The Lipschitz character of $\Omega_Q$ as well as $0<\theta<1$, $C\ge 2$, depend on $n$, and the CAD character of $\Omega$ only (and are uniform in $Q$, $Q'$).
\end{corollary}

\begin{proof} 
The corollary follows directly from Proposition~\ref{IBPLS}. Indeed, for any $Q\in \dd(\pom)$ there exists $\Delta_Q\subset Q$ as in \eqref{cube-ball}, \eqref{cube-ball2}. 
One can then build a Lipschitz domain $\Omega_Q$ from Proposition~\ref{IBPLS} corresponding to $\Delta_Q$, and then the conditions $(i)$, $(iii)$ in Proposition~\ref{IBPLS} entail $(i)$ and $(iii)$ in Corollary~\ref{cIBPLS}, respectively. The condition $(ii)$ in Corollary~\ref{cIBPLS} follows from the fact that a Lipschitz domain $\Omega_Q$ is, in particular, a CAD, and hence, it has a corkscrew point relative to $B(y_{Q'}, r')\cap \po_Q$ since $r'\le \ell(Q')\le \ell(Q)\approx\diam (\po_Q)$ (the $\approx$  follows from $(ii)$ and $(iii)$ in Proposition~\ref{IBPLS}). Using the fact that $\Omega_Q\subset \Omega$, one can easily see that $Y_Q$ is also a corkscrew point in $\Omega$ relative to $B(y_{Q'}, r')\cap \pom$. It remains to observe that a suitable choice of $\eta$ and $K$ (uniform in $Q'$) ensures that such a corkscrew point always belongs to $U_{Q'}$ and moreover, $B(Y_{Q'}, C^{-1}\ell(Q'))\subset U_{Q'}$. 
\end{proof}

We are now ready to prove Theorem \ref{theor:CME:Lip->CAD}:

\begin{proof}[Proof of Theorem \ref{theor:CME:Lip->CAD}]  
By \eqref{CME:dyadic} and Remark \ref{remark:CME0-Lip} we can reduce matters to estimate $\|F\|_{\C^{\rm dyad}(D)}$. Fix some $Q\in \dd$. According to Corollary~\ref{cIBPLS}  (along with the inner regularity property of the measure) 
there exists a bounded Lipschitz domain $\Omega_Q$  such that $\sigma(\partial\Omega_Q \cap Q)\geq \theta\,\sigma (Q)$,  and the Lipschitz character of $\Omega_Q$ as well as $0<\theta<1$ depend only on $n$ and the CAD character of $D$ (and are uniformly in $Q$). The domain $\Omega_Q$ further satisfies properties $(i)$--$(iii)$ in Corollary~\ref{cIBPLS}. Given $x\in Q\setminus \partial\Omega_Q$, since $\partial\Omega_Q$ is closed, there exists $r_x>0$ such that $B(x,r_x)\cap\partial\Omega_Q=\emptyset$. Pick then $Q_x\in\dd$ with $\ell(Q_x)\ll \min\{\ell(Q),r_x\}$ so that $x\in Q_x$. Then, $x\in Q\cap Q_x$ and necessarily $Q_x\subset Q$. Also, $Q_x\subset B(x,r_x)$ since $x\in Q_x$ and $\diam(Q_x)\approx \ell(Q_x)\ll r_x$. Thus, $Q_x\subset Q\setminus \partial\Omega_Q$ and there exists a cube with maximal size $Q_x^{\rm max}\in \dd_{Q}$ so that $Q_x^{\rm max}\subset Q\setminus \partial\Omega_Q$. Note that $Q_x^{\rm max}\subsetneq Q$ since  $\sigma(\partial\Omega_Q \cap Q)>0$. Thus, by maximality, $\partial\Omega_Q\cap Q'\neq\emptyset$ for every $Q'$ with $Q_x^{\rm max}\subsetneq Q'\subset Q$. Consider then $\F=\{Q_j\}_j\subset \dd_{Q}\setminus\{Q\}$ the collection of such maximal cubes. By construction, the cubes in $\F$ are pairwise disjoint and also $Q\setminus \partial\Omega_Q=\cup_j Q_j$. Associated with $\F$ we build the corresponding local sawtooth $\Omega_{\mathcal{F},Q}$ (cf. \eqref{eq2.sawtooth1}).

Note that if $Q'\subset Q_j \in\F$, then $Q'\subset Q_j \subset Q\setminus \partial\Omega_Q$, hence $\partial\Omega_Q\cap Q'=\emptyset$. Conversely, if $Q'\in\dd_Q$ is so that $\partial\Omega_Q\cap Q'=\emptyset$, then $Q'\subset Q\setminus\partial\Omega_Q=\cup_j Q_j$ and there is $Q_j\in \F$ such that $Q'\cap Q_j\neq\emptyset$. If $Q_j\subsetneq Q'$ then by the maximality of $Q_j$ we have $\partial \Omega_Q\cap Q'\neq\emptyset $ which is  a contradiction. As a result, necessarily $Q'\subset Q_j$. All in one, for every $Q'\in\dd_Q$, we have that $Q'\subset Q_j\in\F$ if and only if $\partial\Omega_Q\cap Q'=\emptyset$. Equivalently, given $Q'\in\dd_Q$, one has that $Q'\in\dd_{\F,Q} $ if and only if $Q'\cap \partial\Omega_Q\neq\emptyset$.

Let $N\ge 1$ to be chosen and by Chebyshev's inequality
\begin{align*}
&\sigma\left\{x\in \partial\Omega_Q \cap Q:\,\A^Q F(x)> N\right\} 
\\
&\qquad\quad\leq \frac{1}{N^2} \int_{\partial\Omega_Q \cap Q}\dint_{{\Gamma}^Q(x)}
|F(Y)|^2\,\delta(Y)^{1-n}\,dY
\\
&
\qquad\quad\le 
\frac{1}{N^2} \sum_{Q'\in\dd_Q} \sigma(\partial\Omega_Q \cap Q')\dint_{U_{Q'}} |F(Y)|^2\,\delta(Y)^{1-n}\,dY 
\\
&\qquad\quad\approx
\frac{1}{N^2} \sum_{Q'\in\dd_{\F,Q}} \frac{\sigma(\partial\Omega_Q \cap Q')}{\sigma(Q')}\dint_{U_{Q'}} |F(Y)|^2\,\delta(Y)\,dY 
\\
&
\qquad\quad\lesssim \frac{1}{N^2} \dint_{\Omega_{\F,Q}} |F(Y)|^2\,\delta(Y)\,dY,
\end{align*}
where we have used that that $\delta(Y)\approx\ell(Q')$ for every $Y\in U_Q'$ and also that the family $\{U_Q'\}_{Q'\in\dd}$ has bounded overlap. We claim that 
\begin{equation}\label{eqjn10.1-bis2}
\frac1{\sigma(Q)}\dint_{\Omega_{\mathcal{F},Q}} |F(X)|^2 \delta(X) \,dX\,
\leq 
C\,\big(\sup_{\Omega\subset D}  \|F\|_{\C(\Omega)}+ \|F\|_{\C_0(D)}\big)
\end{equation}
where the sup runs over all bounded Lipschitz subdomains $\Omega\subset  D$ with parameters in the Lipschitz character at most $M_0$, and $C$ depends as before only on $n$, and the CAD character of $D$. Assuming this momentarily, and invoking \eqref{eqdefCME-D-interior:Lip},  we conclude that 
\begin{align*}
&\sigma\left\{x\in Q:\,\A^Q F(x)> N\right\} 
\le
\sigma(Q\setminus \partial\Omega_Q)+ \frac{C}{N^2} \dint_{\Omega_{\F,Q}} |F(Y)|^2\,\delta(Y)\,dY
\\
&
\qquad \le
(1-\theta)\,\sigma(Q) + \frac{C\,}{N^2}\sup_{\Omega\subset D}  \|F\|_{\C(\Omega)}\,\sigma(Q)
\le
(1-\theta/2)\,\sigma(Q),
\end{align*}
provided $N^2=\frac{2C}{\theta}\sup_{\Omega\subset D}  \|F\|_{\C(\Omega)}$. Applying then the John-Nirenberg inequality, Lemma~\ref{lemma:J-N} with $Q_0=E=\partial D$ which is ADR by assumption, extending $F$ as $0$ in $\ree\setminus\overline{D}$, and with 
$p=2$ we then conclude that
\[
\sup_{Q\in\dd_{Q_0}}\,\fint_Q \A^QF(x)^2 \sigma(x)
\lesssim 
\sup_{\Omega\subset D}  \|F\|_{\C(\Omega)}.
\]
In turn, this yields
\begin{multline*}
\dint_{T_{Q}}|F|^2\delta\,dX
\le
\sum_{Q'\in\dd_Q} \dint_{U_{Q'}}|F|^2\delta\,dX
\approx
\sum_{Q'\in\dd_Q} \sigma(Q')\dint_{U_{Q'}}|F|^2\delta^{1-n}\,dX
\\
=
\sum_{Q'\in\dd_Q} \int_{Q'}\Big(\dint_{U_{Q'}}|F|^2\delta^{1-n}\,dX\Big) \,d\sigma
\lesssim
\int_Q\Big(\dint_{\Gamma^Q(x)}
|F|^2\,\delta^{1-n}\,dY\Big)\, d \sigma(x)
\\
=
\int_Q \A^QF(x)^2 \sigma(x)
\lesssim
\sigma(Q)\sup_{\Omega\subset D}  \|F\|_{\C(\Omega)}.
\end{multline*}
Here we have used that $\delta(\cdot)\approx \ell(Q')$ in $U_{Q'}$ and the fact that the family $\{U_Q\}_{Q\in\dd}$ has bounded overlap.

We are then left with showing \eqref{eqjn10.1-bis2}. To this end, let us write 
\begin{equation}\label{eqn3.4}
\dint_{\Omega_{\mathcal{F},Q}} |F(X)|^2 \delta(X) \,dX\,
\le \sum_{Q'\in \dd_{\F, Q}} \dint_{U_{Q'}}
|F(Y)|^2\,\delta(Y)\,dY
=
\sum_{Q'\in \dd_{\F, Q}^1} \dots  
+\sum_{Q'\in \dd_{\F, Q}^2} \dots
\end{equation}
where, for some $\eps>0$ to be chosen,
\[
\dd_{\F, Q}^1:=\Big\{Q'\in \dd_{\F, Q}: \dist (U_{Q'}, \partial D)\leq \frac 1 \eps \dist (U_{Q'}, \po_Q)\Big\}, \qquad \dd_{\F, Q}^2:=\dd_{\F, Q}\setminus\dd_{\F, Q}^1.
\]
Note that, in principle, $U_{Q'}$ can intersect $\po_Q$. For later use it is convenient to record that  $\ell(Q')\approx \dist (U_{Q'}, \partial D)\approx\dist(U_{Q'},Q')$ by  \eqref{eq3.1}, \eqref{whitney}, \eqref{eq2.whitney2}, \eqref{eq3.3aa}.

Let $Q'\in \dd_{\F, Q}^1$, the fact that $Q'\in \dd_{\F, Q}$ implies that there exists $y\in Q'\cap \partial\Omega_Q$, hence
\begin{multline}\label{eqn3.5}
\epsilon\ell(Q')
\approx 
\epsilon\dist(U_{Q'}, \partial D)
\le 
\dist (U_{Q'}, \po_Q)
\\
\le
\dist(U_{Q'}, y) \le \dist(U_{Q'}, Q')+\diam(Q')\lesssim \ell(Q').
\end{multline}
In particular, for every $Y\in U_{Q'}$ with $Q'\in\dd_{\F,Q}^1$  we have 
\begin{equation}\label{nnhtfwafwr}
\delta(Y)=\dist (Y, \partial D)\lesssim \ell(Q')+ \dist(U_{Q'},\partial D)
\lesssim
\epsilon^{-1}
\dist (U_{Q'}, \po_Q)
\lesssim 
\epsilon^{-1}
\dist (Y, \po_Q).
\end{equation}
Note also that since $y'\in Q'\cap \partial\Omega_Q\neq \varnothing$, according to Corollary~\ref{cIBPLS} part $(ii)$, we can find $Y_{Q'}$ so that $B(Y_{Q'}, \ell(Q')/C)\subset B(y_{Q'}, \ell(Q'))\cap\Omega\cap\Omega_{Q}\cap U_{Q'}$. Hence, $\Omega_Q\cap U_{Q'}\neq \varnothing$, and then due to \eqref{eqn3.5} and the fact that $U_{Q'}$ is connected by construction, we conclude that $U_{Q'}\subset \Omega_Q$. As a result, 
\begin{multline}\label{eqn3.11a}
\sum_{Q'\in \dd_{\F, Q}^1}\dint_{U_{Q'}}
|F(Y)|^2\,\delta(Y)\,dY
\lesssim 
\epsilon^{-1}
\sum_{Q'\in \dd_{\F, Q}^1} \dint_{U_{Q'}}
|F(Y)|^2\,\dist(Y, \partial\Omega_Q)\,dY\\[4pt]
\lesssim  \dint_{\Omega_{Q}}
|F(Y)|^2\,\dist(Y, \partial\Omega_Q)\,dY \leq \sigma(Q)\sup_{\Omega\subset D}  \|F\|_{\C(\Omega)},
\end{multline} 
where we used \eqref{nnhtfwafwr}, the finite overlap property of the family $\{U_{Q'}\}_{Q'\in\dd}$, and the fact that $\Omega_Q$ is bounded Lipschitz subdomains of $D$ with character controlled by the CAD parameters
in the last one. Note that $\Omega_Q\subset B(x_Q, C\ell(Q))$ for some uniform constant $C$, which justifies the bound by $\sigma(Q)$.

Consider next the family $\dd_{\F, Q}^2$ and we shall demonstrate that they satisfy a packing condition. Indeed, recall from above that $\ell(Q')\approx \dist (U_{Q'}, \partial D)$, so that in particular, if $Q'\in \dd_{\F, Q}^2$, then
\begin{equation}
\label{eqn3.12}
\dist (U_{Q'}, \po_Q)\lesssim \eps \ell(Q').
\end{equation} 
It follows that for a suitably small $\eps$ depending on the implicit constant in \eqref{eqn3.12} and $\tau$, we can ensure that fattened regions $\widehat U_{Q'}$  corresponding to $U_{Q'}$ (cf. \eqref{eq3.3aa-fat}) necessarily intersect $\po_Q$ and, moreover, $H^n(\widehat U_{Q'}\cap \po_Q) \approx \ell(Q')^n$, while  the family $\{\widehat U_{Q'}\}_{Q'}$ still have finite overlap. Since the Lipschitz character of $\po_Q$ is depends on the CAD character of $D$, we have that $H^n(\po_Q)\approx \diam(\partial\Omega_Q)^n\approx \diam(\Omega_Q)^n\approx \ell(Q)\approx \sigma(Q)$ with implicit constants which are uniform in $Q$. Thus, all in all, 
\begin{equation}\label{eqn3.13}
\sum_{Q'\in \dd_{\F, Q}^2}\sigma (Q') 
\approx
\sum_{Q'\in \dd_{\F, Q}^2} \ell(Q')^n
\approx 
\sum_{Q'\in \dd_{\F, Q}^2} H^n(\widehat U_{Q'}\cap \po_Q)
\lesssim
H^n(\po_Q)
\approx
\sigma(Q).
\end{equation} 
Consequently, using that one can cover $U_{Q'}$ by a uniform number of balls of the form $B(X,\delta(X)/2)$ with $X\in U_{Q'}$ (and hence $\delta(X)\approx\ell(Q'))$ we arrive at
\begin{align}\label{eqn3.14}
\sum_{Q'\in \dd_{\F, Q}^2} \dint_{U_{Q'}}
|F(Y)|^2\,\delta(Y)\,dY
\lesssim 
\|F\|_{\C_0(D)}
\sum_{Q'\in \dd_{\F, Q}^2}\sigma (Q')
\lesssim
\sigma(Q)
\|F\|_{\C_0(D)}
, 
\end{align} 
simply recalling the notation introduced in \eqref{eqdefCME-interior:def}. 

Collecting \eqref{eqn3.4}, \eqref{eqn3.11a}, and \eqref{eqn3.14} we conclude as desired \eqref{eqjn10.1-bis2} completing the proof. 
\end{proof}


\subsection{Transference of Carleson measure estimates:  from chord-arc domains to the complement of a UR set}\label{sCME-NTA-UR}

Let us now discuss the ``transference" mechanism allowing one to pass from the Carleson measure estimates on CAD to those open sets with UR boundaries. The main idea consists in showing that if for some given $F$ one can prove \eqref{eqdefCME} on $D\subset \ree\setminus E$, any bounded CAD,  then \eqref{eqdefCME} holds for $\ree\setminus E$. This was proved in \cite[Theorem 1.1]{HMM2} for $F=|\nabla u|/\|u\|_{L^\infty(\ree\setminus E)}$ with $u$ being a bounded harmonic function in $\ree\setminus E$. On the other hand, it was already observed in \cite[Remark 4.28]{HMM2} that harmonicity is not really needed and that one could take for instance $F=|\nabla u|/\|u\|_{L^\infty(\ree\setminus E)}$ 
with $u$ being a bounded solution of a second order elliptic PDE or, more generally, $F=|\nabla^m u|/\|\nabla^{m-1}u\|_{L^\infty(\ree\setminus E)}$ with $u$ being a bounded solution of a $2m$-th order elliptic PDE, $m\in \NN$. We shall come back to this point with more details in Section~\ref{appl}, 
and for now try to keep the discussion general for as long as possible.

\begin{remark}\label{rglitch} 
	There is a slightly glitchy point of notation point. For reasons of homogeneity, 
one might prefer to normalize so that $F=\dist(\cdot,E) |\nabla u|/\|u\|_{L^\infty(\ree\setminus E)}$. However, making the function $F$ and later on $G$ and $H$ in Section~\ref{section:sgl} depend on the open set (via its distance to the boundary) has its own dangers and kills the beauty of the generality here.  
\end{remark}

The following result is stated in \cite[Theorem 1.1]{HMM2} exclusively for harmonic functions,
but as noted in \cite[Remark 4.28]{HMM2}, the same proof applies verbatim to any bounded function satisfying Caccioppoli's inequality along with CME in chord-arc subdomains. The argument further extends
to the following formulation with a few changes. For the sake of self-containment we present  below a somewhat different and more direct argument.

\begin{theorem}\label{theor:CME:CAD->UR}  
	Let $E\subset \ree$ be an $n$-dimensional  UR set and let $F\in L^2_{\loc}(\ree\setminus E)$.  Given $\eta\ll 1$ and $K\gg 1$, consider the decomposition $\dd(E)=\G\cup\B$ from Lemma \ref{lemma2.1} as well as a Whitney-dyadic structure $\{\W_Q\}_{Q\in\dd(E)}$ for $\ree\setminus E$ with parameters $\eta$ and $K$, see  Section \ref{sections:UR}. Then using the notation in  \eqref{eqdefCME} and \eqref{eqdefCME-interior:def} there holds
	\begin{equation}\label{eqn2.2}
	\|F\|_{\C(\ree\setminus E)} 	\le 
	C\,
	\max\Big\{\|F\|_{\C_0(\ree\setminus E)}, 	\sup_{\sbf\subset \G} \|F\|_{\C(\Omega_{\sbf}^\pm)}\Big\},
	\end{equation}
	where $\Omega_{\sbf}^\pm$ is defined by \eqref{eq3.2} (with $\sbf'=\sbf$) and where $C$ depends only on $n$, the UR character of $E$, and the choice of $\eta, K, \tau$.

	In particular, if $F$ satisfies the Carleson measure estimate \eqref{eqdefCME} for every bounded chord-arc subdomain $D\subset\ree\setminus E$ with constants depending on the CAD character (see Notation \ref{notation:constants}) then $F$ satisfies the Carleson measure estimate  \eqref{eqdefCME} on $\ree\setminus E$. More precisely, there exists a large constant $M_0$ (depending only $n$ and the UR character of $E$) so that using the notation in \eqref{eqdefCME}  there holds
	\begin{equation}\label{eqn2.2-general}
	\|F\|_{\C(\ree\setminus E)} 	\le 
	C\,\sup_{D\subset\ree\setminus E}  \|F\|_{\C(D)},
	\end{equation}
	where the sup runs over all bounded chord-arc subdomains $D\subset  \ree\setminus E$ with parameters in the CAD character at most $M_0$, and $C$ depends as before only on $n$ and the UR character of $E$. 
\end{theorem}

We note that much as in Remark \ref{remark:CME-E-OMmga} one can easily get a version of this result valid where everything is localized to some open subset with UR boundary. The precise statement and the details are left to the interested reader.

\begin{remark}\label{remark:CME0-Lip} 
	As already mentioned in Remark \ref{remark:CMO0} and for PDE applications, the quantities $\|F\|_{\C_0(\ree\setminus E)}$ or $\|F\|_{\C_0(D)}$ are harmless terms since they are typically finite, whether or not $F$ satisfies Carleson measure estimates on some family of nice subdomains. However, one can also see that these terms are under controlled when one imposes Carleson measure estimates on bounded Lipschitz subdomains. Let $E\subset \ree$ be an $n$-dimensional ADR set, write $\delta(\cdot)=\dist(\cdot,E)$, and let  $F\in
	L^2_{\rm loc}(\ree\setminus E)$. 	Note that $\Omega_X=B(X,\delta(X))$ is a bounded Lipschitz subdomain of $\ree\setminus E$ with all the parameters in the Lipschitz character bounded by $M_n\ge 1$ which depends just on $n$. Also if $Y\in B(X,\delta(X)/2)$ then 
	$\dist(Y,\partial\Omega_X)\ge \delta(X)/2$ and $Y\in B(z,2\delta(X))$ for any $z\in\partial\Omega_X$. Thus, for any $z\in\partial\Omega_X$
	\begin{align*}
	\frac1{\delta(X)^{n-1}}\dint_{B(X,\delta(X)/2)} |F(Y)|^2 \,dY
	\le
	\frac2{\delta(X)^{n}}\dint_{B(z,2\delta(X))} |F(Y)|^2 \dist(Y,\partial\Omega_X)\,dY
	\end{align*}
	and, consequently,  
	\begin{equation}\label{eqdefCME-Ec-interior:Lip}
	\|F\|_{\C_0(\ree\setminus E)} 
	\le
	2^{n+1} \sup_{D\subset \ree\setminus E}  \|F\|_{\C(D)},
	\end{equation}
	where the sup runs over all bounded Lipschitz subdomains of  $\ree\setminus E$ with all the parameters in the Lipschitz character at most $M_n\ge 1$.  Analogously, if $F\in L^2_{\rm loc}(\Omega)$ where $\Omega\subset \ree$ is an open set with $\partial \Omega$ being $n$-dimensional ADR, then 
	\begin{equation}\label{eqdefCME-D-interior:Lip}
	\|F\|_{\C_0(\Omega)} 
	\le
	2^{n+1} \sup_{D\subset\Omega}  \|F\|_{\C(D)},
	\end{equation}
	where the sup runs over all bounded Lipschitz subdomains of  $\Omega$ with all the parameters in the Lipschitz character at most $M_n\ge 1$.
	
\end{remark}

%
%
%


\begin{proof}
	
	We write $\delta(\cdot)=\dist(\cdot,E)$ and Define $\beta_Q=\dint_{U_{Q,\tau/2}} |F|^2\delta\,dX$ for every $Q\in\dd=\dd(E)$. Fix $Q_0\in\dd$. Using the decomposition $\dd(E)=\G\cup\B$ from Lemma \ref{lemma2.1}
	\begin{multline*}
	\dint_{T_{Q_0,\tau/2} } |F(X)|^2 \delta(X) \,dX
	\le
	\sum_{Q\in\dd_{Q_0}} \beta_Q
	=
	\sum_{Q\in\dd_{Q_0}\cap\B} \beta_Q
	+
	\sum_{Q\in\dd_{Q_0}\cap\G} \beta_Q
	\\
	=
	\sum_{Q\in\dd_{Q_0}\cap\B} \beta_Q
	+
	\sum_{\sbf:\dd_{Q_0}\cap \sbf\neq\emptyset}  \sum_{Q\in\dd_{Q_0}\cap\sbf} \beta_Q 
	=:\Sigma_1+\Sigma_2,
	\end{multline*}
	and we estimate each term in turn. For $\Sigma_1$ we observe that by construction the $U_{Q,\tau/2}$'s are uniformly bounded unions of Whitney cubes of size of the order of $\ell(Q)$ and with distance to $E$ of the order of $\ell(Q)$ and it follows easily that $\beta_Q\lesssim C_0\sigma(Q)$ where the implicit constants depend only on $n$, the UR character of $E$, and the choice of $\eta, K, \tau$. Hence,
	\begin{equation}\label{Sigma1}
	\Sigma_1
	\lesssim
	\|F\|_{\C_0(\ree\setminus E)}\sum_{Q\in\dd_{Q_0}\cap\B} \sigma(Q)
	\lesssim
	\|F\|_{\C_0(\ree\setminus E)}\sigma(Q_0)
	\end{equation}
	where in the last estimate we have used Lemma \ref{lemma2.1} part $(ii)$.
	
	Let us estimate $\Sigma_2$. Fix $\sbf$ so that $\dd_{Q_0}\cap \sbf\neq\emptyset$ and write $Q_1=Q_1(\sbf)=Q_0\cap Q(\sbf)$. Note that if $Q\in\dd_{Q_0}\cap\sbf$ then $Q\subset Q_1\subset Q(\sbf)$ and by the coherency of $\sbf$ we conclude that $Q_1\in \sbf$. 
	Set $\delta_{\sbf}^{\pm}(\cdot)=\dist(\cdot,\partial\Omega_{\sbf}^{\pm})$ (see \eqref{eq3.2} with $\sbf'=\sbf$). Note that $\Omega_\sbf^{\pm}$ is comprised of Whitney regions of the form $U_Q^{\pm}=U_{Q,\tau}^{\pm}$. thus for $X\in U_{Q,\tau/2}^{\pm}$ with $Q\in\sbf$, we have that $\delta(X)\approx \delta_{\sbf}^{\pm}(X)$ where the implicit constants depend on $\tau$. This, the fact that the family  $\{U_Q^\pm\}_{Q\in\dd}$ has bounded overlap and \eqref{eq3.3aab} easily give 
	\[
	\sum_{Q\in\dd_{Q_0}\cap\sbf} \beta_Q
	=
	\sum_{Q\in\dd_{Q_1}\cap\sbf} \beta_Q 
	\approx
	\sum_{Q\in\dd_{Q_1}\cap\sbf} \dint_{U_Q^{\pm}} |F|^2\delta_{\sbf}^{\pm}\,dX
	\lesssim
	\dint_{B_{Q_1}^*\cap\Omega_\sbf^{\pm}} |F|^2\delta_{\sbf}^{\pm}\,dX
	\]
	where $B^{*}_{Q_1}:= B(x_{Q_1}, K\ell(Q_1))$.  Pick now $X_1^{\pm}\in U_{Q_1,\tau/2}^{\pm}$ and choose $x_1^{\pm}\in\partial \Omega_\sbf^{\pm}$ so that $|X_1^{\pm}-x_1^{\pm}|=\delta_\sbf^{\pm}(X_1^{\pm})\approx \delta(X_1^{\pm})\approx\ell(Q_1)$. Therefore, $B_{Q_1}^*\subset B_{Q_1}^{**}=B_{Q_1}(x_1^{\pm}, C\ell(Q_1))$ where $C$ depends on $n$, the UR character of $E$ and $\eta$, $K$ and $\tau$. Thus,
	\begin{equation}\label{est-sigma:aux}
	\sum_{Q\in\dd_{Q_0}\cap\sbf} \beta_Q 
	\lesssim
	\dint_{B_{Q_1}^{**}\cap\Omega_{\sbf}^{\pm}} |F|^2\delta_{\sbf}^{\pm}\,dX
	\lesssim
	\|F\|_{\C(\Omega_{\sbf}^{\pm})}\,\ell(Q_1)^n
	\approx
	\|F\|_{\C(\Omega_{\sbf}^{\pm})} \sigma(Q_1).
	\end{equation}
	Using this and recalling that  $Q_1=Q_1(\sbf)=Q_0\cap Q(\sbf)$ we can  bound $\Sigma_2$ as follows:
	\begin{multline*}
	\Sigma_2
	=
	\sum_{\sbf:\dd_{Q_0}\cap \sbf\neq\emptyset}  \sum_{Q\in\dd_{Q_0}\cap\sbf} \beta_Q 
	\lesssim
	\sup_{\sbf\subset \G} \|F\|_{\C(\Omega_{\sbf}^{\pm})} 
	\sum_{\sbf:\dd_{Q_0}\cap \sbf\neq\emptyset}  
	\sigma(Q_0\cap Q(\sbf))
	\\
	=
	\sup_{\sbf\subset \G} \|F\|_{\C(\Omega_{\sbf}^{\pm})} 
	\Big(
	\sum_{\sbf:Q(\sbf)\subset Q_0}  
	\sigma(Q(\sbf))
	+
	\sum_{\substack{\sbf:\dd_{Q_0}\cap \sbf\neq\emptyset\\ Q_0\subsetneq Q(\sbf)}} 
	\sigma(Q_0)
	\Big)
	\end{multline*}
	Using Lemma \ref{lemma2.1} part $(ii)$ we easily obtain 
	\[
	\sum_{\sbf:Q(\sbf)\subset Q_0}  
	\sigma(Q(\sbf))
	\lesssim
	\sigma(Q_0)
	\]
	where the implicit constant depends only on $n$, the UR character of $E$, and the choice of $\eta, K, \tau$. 
	For the other term we note that the facts $\dd_{Q_0}\cap \sbf\neq\emptyset$ and $Q_0\subsetneq Q(\sbf)$ imply that $Q_0\in\sbf$ by the coherency of $\sbf$,
	hence $\Sigma_{22}=0$ if $Q_0\in\B$. On the other hand, if $Q_0\in\G$ there is a unique $\sbf_0\subset\G$ so that $Q_0\in \sbf_0$ and $\dd_{Q_0}\cap\sbf=\emptyset$ for every $\sbf\neq \sbf_0$ with $Q_0\subsetneq Q(\sbf)$. This clearly implies that in this case
	\[
	\sum_{\substack{\sbf:\dd_{Q_0}\cap \sbf\neq\emptyset\\ Q_0\subsetneq Q(\sbf)}} 
	\sigma(Q_0)
	=
	\sigma(Q_0)
	\]
	If we finally collect all the obtained estimates we conclude that
	\begin{multline}\label{CME:dyadic-est}
	\|F\|_{\C^{\rm dyad}(\ree\setminus E)}=	\sup_{Q\in\dd(E)}\frac1{\sigma(Q)}\dint_{T_{Q,\tau/2} } |F(X)|^2 \delta(X) \,dX
	\\
	\le
	C
	\max\Big\{\|F\|_{\C_0(\ree\setminus E)}, 	\sup_{\sbf\subset \G} \|F\|_{\C(\Omega_{\sbf}^\pm)}\Big\},
	\end{multline}
	where $C$ depends only on $n$, the UR character of $E$, and the choice of $\eta, K, \tau$. Thus, the desired estimates follows from \eqref{CME:dyadic}. 
	
	To complete the proof we look at the second part of the statement. By \eqref{eqdefCME-Ec-interior:Lip} and the fact that bounded Lipschitz domains are CAD with all the parameters in the CAD character by the Lipschitz character we have $\|F\|_{\C_0(\ree\setminus E)}\lesssim \sup_{D\subset\ree\setminus E}  \|F\|_{\C(D)}$ where the sup runs over all bounded CAD subdomains with character at most $M_n$. On the other hand, Lemma \ref{lemma3.15} establishes that all the $\Omega_{\sbf}^{\pm}$'s are CAD with parameters in the CAD character all controlled by $M_0'\ge 1$ (depending on the allowable parameters). They are also bounded since every $\sbf$ has a maximal cube $Q(\sbf)$ and hence $\Omega_{\sbf}^{\pm}\subset B_{Q(\sbf)}^*$ (cf. \eqref{eq3.3aab}). Consequently,
	\[
	\sup_{\sbf\subset \G} \|F\|_{\C(\Omega_{\sbf}^\pm)}
	\le
	\sup_D  \|F\|_{\C(D)}
	\]
	where the second sup runs over all bounded CAD with character at most $M_0'$. Taking $M_0=\max\{M_n, M_0'\}$ we easily see that \eqref{eqn2.2} along with the above observations readily yield \eqref{eqn2.2-general}. This completes the proof. 
\end{proof}

\section{Carleson estimates, $\A<N$ estimates and good-$\lambda$ arguments}\label{section:sgl}

 Given an open set $\Omega\subset\ree$ with ADR boundary we recall the definitions of the area integral $\A$ and the non-tangential maximal function $N_*$ from Definition~\ref{defsfntmax:dyadic} or the corresponding fattened versions $\widehat{\A}$ and  $\widehat{N}_*$ or the corresponding local versions. These are defined with respect to a $\{\W_Q\}_{Q\in\dd}$, some Whitney-dyadic structure for $\Omega$ with some implicit parameters $\eta$ and $K$. Note that according to these definitions, the cones are unbounded when $\pom$ is unbounded. On the other hand, when $\pom$ is bounded, so are the cones, all being contained in a $C\diam(\pom)$-neighborhood of $\pom$. We note also that when $\pom$ is bounded, there exists a cube $Q_0\in \dd(\pom)$ such that $Q_0=\pom$ and for any $Q\in\dd(\pom)$ we have $Q\in \dd_{Q_0}$. It is, however, particularly useful to work with local versions $\A^Q$ and $\widehat{N}^Q_*$ or  $\widehat{\A}^Q$ and $\widehat{N}^Q_*$.  


\begin{definition}[$\A<N$  {\bf estimates}]\label{defsn}   
Let $\Omega\subset \ree$ be an open set with $\pom$ being ADR and let $\{\W_Q\}_{Q\in\dd(\pom)}$ be a Whitney-dyadic structure for $\Omega$ with some parameters $\eta$ and $K$. Consider also $G\in L^2_{\rm loc}(\Omega)$, $H\in C(\Omega)$, and $0<q<\infty$. 
We say that ``$\A <N$" estimates hold for $G, H$ on $L^q(\pom)$ if 
\begin{equation}\label{eq.sn}
\|\A G\|_{L^q(\pom)}\leq C \|\widehat N_* H\|_{L^q(\pom)}\,,
\end{equation}
where the $L^q$ norms are taken with respect to surface measure $\sigma:=H^n|_{\pom}$. 
Similarly, we will say that ``$\A^\dd <N^\dd$" estimates hold for $G, H$ on $L^q(\pom)$ if 
\begin{equation}\label{eq.sn.loc}
\|\A^Q G\|_{L^q(Q)}\leq C \|\widehat N^Q_* H\|_{L^q(Q)}\,, \mbox{ for all } Q\in \dd(\pom),
\end{equation}
with $C$ independent of $Q$.
\end{definition}

\begin{remark}\label{remark:N<S:cubes->balls}
	We observe that by Remarks \ref{remark:rcones} and \ref{remark:COA}, $\A<N$ estimates imply an analogous estimate for traditional cones, that is, for every $\kappa>0$
	\[
	\|\A_{\Omega, \kappa}  G\|_{L^q(\pom)}\leq C \|N_{*,\Omega, \kappa}  H\|_{L^q(\pom)}\,,
	\]
	and the implicit constant depend on $q$, $n$, the ADR constant of $\pom$, the choice of $\eta, K,\tau$, the constant in  $\A <N$, and $\kappa$.		
	On the other hand $\A^\dd <N^\dd$ estimates imply also some local $\A<N$ estimates with traditional cones. More precisely, 	
	for any $x\in\pom$ and $0<r<2\diam(\pom)$, using the notation in Definition \ref{defsfntmax:traditional}, there exists $K'$ depending on $n$, the ADR constant of $\pom$, the choice of $\eta, K,\tau$, and the constant in Definition \ref{def:WD-struct} part $(iii)$, such that for every $\kappa>0$
\begin{equation}\label{eq.sn.loc-balls}
\|\A_{\Omega,\kappa}^r G\|_{L^q(\Delta(x,r))}\lesssim \|N^{K'r}_{*,\Omega,\kappa} H\|_{L^q(\Delta(x,K' r))},
\end{equation}
where $\Delta(x,r)=B(x,r)\cap\pom$, and the implicit constant depends on $q$, $n$, the ADR constant of $\pom$, the choice of $\eta, K,\tau$, the constant in  $\A^\dd <N^\dd$, and $\kappa$.	

Fix then $\{\W_Q\}_{Q\in\dd(\pom)}$ a Whitney-dyadic structure for $\Omega$ with some parameters $\eta$ and $K$. Given $x\in\pom$ and $0<r<2\diam(\pom)$, write $\Delta=\Delta(x,r)$ and $B=B(x,r)$. We first consider the case $r\ll\diam(\pom)$. 
Note that for every $y\in\Delta$ we have $\Gamma^r(y)\subset 2B$. Also, if $\Gamma_{\Omega,1}(z)\cap 2B\neq\emptyset$ then $z\in 6\,\Delta$. Recall that we have always assumed that $K$ is large enough (say $K\ge 10^4n$) so that $\Gamma_{\Omega,1}(y)\subset\Gamma(y)$ for every $y\in\pom$. All these, together with Remark \ref{remark:COA}, give
\[
\|\A^r_{\Omega, \kappa} G\|_{L^q(\Delta)}
\le
\|\A_{\Omega, \kappa} (G 1_{2B})\|_{L^q(\pom)}
\lesssim
\|\A_{\Omega, 1} (G 1_{2B})\|_{L^q(\pom)}
\le
\|\A (G 1_{2B})\|_{L^q(6\Delta)}.
\]
Let 
\begin{align}\label{def:D-Delta}
\mathcal{D}_\Delta=\{Q\in\dd(\pom): Q\cap 6\Delta\neq\emptyset, C(\eta\,n)^{-1/2}\,r/4\le \ell(Q)<C(\eta\,n)^{-1/2}\,r/2\},
\end{align}
where $C$ is the constant in \eqref{eq2.whitney2} (it is here we use that $r\ll\diam(\pom)$ so that $C(\eta\,n)^{-1/2}\,r/2<\diam(\pom)$, thus $\mathcal{D}_\Delta\neq\emptyset$). Suppose that $Q\subsetneq Q'$ with $Q\in\mathcal{D}_\Delta$ and let $Y\in U_{Q'}$. Then there is $I'\in\W_{Q'}$ with $Y\in\partial I^*(\tau)$ and by \eqref{Whintey-4I}
\begin{multline*}
C\,(\eta\,n)^{-1/2}2^{-1}r
\le 
2\,\ell(Q)
\le 
\ell(Q')
\le
C\,\eta^{-1/2}\ell(I')
\\
\le
C\,(\eta\,n) 4^{-1}\dist(4I',\pom)
\le
C\,(\eta\,n) 4^{-1}\dist(Y,\pom).
\end{multline*}
Hence, $\dist(Y,\pom)\ge 2r$ and $\Gamma(y)\cap2B\subset \Gamma^Q(y)$ for every $y\in Q\in\mathcal{D}_\Delta$. Thus the $\A^\dd <N^\dd$ estimates give
\begin{align*}
\|\A^r_{\Omega, \kappa} G\|_{L^q(\Delta)}^q
\lesssim 
\sum_{Q\in\mathcal{D}_\Delta} \| \A(G 1_{2B})\|_{L^q(Q)}^q
\le
\sum_{Q\in\mathcal{D}_\Delta} \|\A^Q G\|_{L^q(Q)}^q
\lesssim
\sum_{Q\in\mathcal{D}_\Delta} \|N_*^Q H\|_{L^q(Q)}^q.
\end{align*}
Note next that for every $y\in Q\in\mathcal{D}_\Delta$ we have by \eqref{eq3.3aab} that $\Gamma^Q(y)\subset  B(x_Q, K\ell(Q))\cap\Omega\subset K'B\cap  \Omega$. Hence, using again Remark \ref{remark:COA} we have
\[
\|\A^r_{\Omega, \kappa} G\|_{L^q(\Delta)}
\lesssim 
\|N_* (H1_{K'B})\|_{L^q(\pom)}
\lesssim
\|N_{*,\Omega,\min\{1,\kappa\}} (H1_{K'B})\|_{L^q(\pom)}
\le
\|N_{*,\Omega,\kappa}^{3K'r} H\|_{L^q(3K'\Delta)}
,
\]
where we have used that $\Gamma_{\Omega,1}(z)\cap K'B\neq\emptyset$ then $z\in 3K'\,\Delta$.

To conclude we consider the case $r\approx\diam(\pom)$. Hence $\pom$ is bounded and $\pom$ is itself a dyadic cube $Q_0$ and $\dd(\pom)=\dd_{Q_0}$. Then we easily obtain using some of the previous observations
\begin{multline}
\|\A^r_{\Omega, \kappa} G\|_{L^q(\Delta)}
\le
\|\A_{\Omega, \kappa} G)\|_{L^q(\pom)}
\lesssim
\|\A_{\Omega, 1} G\|_{L^q(\pom)}
\\
\le
\|\A^{Q_0} G\|_{L^q(\pom)}
\lesssim
\|N_*^{Q_0} H\|_{L^q(\pom)}
\lesssim
\|N_{*,\Omega,\kappa} H\|_{L^q(\pom)}
=
\|N^{K'r}_{*,\Omega,\kappa} H\|_{L^q(\Delta(x,K' r))},
\end{multline}
where the last estimate uses our convention that in the case $\Omega$ unbounded and $\pom$ bounded  $\Gamma_{\Omega}(\cdot)$ is indeed $\Gamma_{\Omega}^{C \diam(\pom)}(\cdot)$.
\end{remark}

\begin{theorem}\label{theor:good-lambda}	 	
Let $\Omega\subset \ree$ be an open set with $\pom$ being ADR and let $\{\W_Q\}_{Q\in\dd(\pom)}$ 
be a Whitney-dyadic structure for $\Omega$ with some parameters $\eta$ and $K$. Given 
$G\in L^2_{\rm loc}(\Omega)$, $H\in C(\Omega)$, and 
$0<q<\infty$, consider the following statements:

\begin{list}{}{\usecounter{enumi}\leftmargin=1.3cm
		\labelwidth=1.3cm\itemsep=0.3cm\topsep=.1cm
		}

\item[$(A)_{\phantom{q}}$] Carleson measure estimate holds for $F=G/\|H\|_{L^\infty(\Omega)}$ on $\Omega$, that is, $\|G\|_{\C(\Omega)}\lesssim \|H\|_{L^\infty(\Omega)}^2$ (cf. \eqref{eqdefCME}). 

\item[$(A)^\dd$] Dyadic Carleson measure estimate holds for $F=G/\|H\|_{L^\infty(\Omega)}$ on $\Omega$, that is, $\|G\|_{\C^{\rm dyad}(\Omega)}\lesssim \|H\|_{L^\infty(\Omega)}^2$ (cf. \eqref{def:CME:dyadic}).

\item[$(A_{\rm loc})_{\phantom{q}}$]  Carleson measure estimate holds on any (bounded) local sawtooth subdomain of $\Omega$, 
in the sense that for any $Q\in \dd(\pom)$ and any pairwise disjoint family of cubes $\F\subset\dd_Q$, one has that $F=G/\|H\|_{L^\infty(\widehat{\Omega}_{\F,Q})}$ satisfies the Carleson measure estimate  on $\widehat{\Omega}_{\F,Q}$, that is,
\[
\sup_{Q,\F} \|G\|_{\C(\widehat{\Omega}_{\F,Q})}/\|H\|_{L^\infty(\widehat{\Omega}_{\F,Q})}^2<\infty,
\]
where the sup runs over all $Q\in \dd(\pom)$ and all pairwise disjoint family of cubes $\F\subset\dd_Q$. 

\item[$(B)_q$] $\A<N$ on $L^q(\pom)$ holds for $G$ and $H$, in the sense of Definition~\ref{defsn}, i.e., \eqref{eq.sn} is valid.

\item[$(B_\loc)_q$] $\A<N$ on $L^q(\partial \widehat{\Omega}_{\F,Q})$ holds for $G$ and $H$ in the sense of Definition~\ref{defsn} for any $Q\in \dd(\pom)$ and any pairwise disjoint family of cubes $\F\subset\dd_Q$, i.e., \eqref{eq.sn} is valid in $\widehat{\Omega}_{\F,Q}$.

\item[$(B)_q^\dd$] $\A^\dd<N^\dd$ on $L^q(\pom)$  holds for $G$ and $H$, in the sense of Definition~\ref{defsn}, i.e., \eqref{eq.sn.loc} is valid.

\item[$(G\lambda)_{\phantom{q}}$]  There exists $\theta>0$ such that for every $\eps, \gamma>0$ and for all $\alpha>0$
\begin{align}\label{eqgl1}
\sigma\{x\in \pom:\, \A G(x)>(1+\eps)\, \alpha,\,\widehat{N}_* H(x)\leq \gamma \alpha\} 
\leq  
 C\,(\gamma/\eps)^\theta
\,\sigma\{x\in \pom:\, \A G(x)>\alpha\}.
\end{align}
\item[$(G\lambda)^\dd$] There exists $\theta>0$ such that for every $\eps, \gamma>0$ and for all $\alpha>0$
\begin{multline}\label{eqgl1-loc}
\sigma\{x\in Q:\, \A^Q G(x)>(1+\eps)\, \alpha,\,\widehat{N}^Q_* H(x)\leq \gamma \alpha\} \\[4pt]
\leq  
 C\,(\gamma/\eps)^\theta
\,\sigma\{x\in Q:\,\A^Q G(x)>\alpha\}, \quad \mbox{ for any } Q\in\dd(\pom).
\end{multline}
\end{list}
Consider, in addition, the condition
\begin{equation}\label{locCacc}
\left(\frac{1}{\delta(X)^{n}}\dint_{B(X,\delta(X)/2)}  |G(Y)|^2\, \delta(Y)\,dY\right)^{1/2} \leq C \|H\|_{L^\infty(B(X,3\delta(X)/4))}, \mbox{\ \ for all } X\in \Omega.
\end{equation}

Then the following implications hold: 
\begin{align}
\label{eq:Aloc->Glambda}
&\mbox{$(A_{\rm loc}) \Longrightarrow (G\lambda)^\dd \Longrightarrow (G\lambda)$.}
\\[4pt]
\label{eq:Aloc->B-dyadic}
&\mbox{$(A_{\rm loc}) \Longrightarrow (B)_q^\dd$, \quad for all \ $0<q<\infty$.}
\\[4pt]
\label{eq:B-dyadic->B}
&\mbox{$(B)_q^\dd\,\, \mbox{ for some $0<q<\infty$ } \Longrightarrow (B)_q$.}
\\[4pt]
\label{eq:B-dyadic->A}
&(B)_q^\dd \,\, \mbox{ for some $0<q<\infty$ }\Longrightarrow (A)^\dd.
\\[4pt]
\label{eq:A-loc->A}
&(A)^\dd \ \&\ \eqref{locCacc} \Longrightarrow (A).
\\[4pt]
\label{eq:B-loc->A}
&(B_\loc)_q \,\, \mbox{ for some $0<q<\infty$ } \Longrightarrow (A)^\dd.
\end{align}
In the previous implications the implicit constants of each of the conclusions depend on $n$, $q$, the ADR character of $\pom$, the choice of $\eta, K, \tau$, the constant in Definition \ref{def:WD-struct} part $(iii)$, as well as the implicit constants in the corresponding hypotheses.
\end{theorem}

\begin{remark}\label{remark:H-bounded}
 In the previous result it is understood that $(A)$ or $(A)^\dd$ are vacuous, unless $H\in L^\infty(\Omega)$. Regarding $(A_\loc)$, if $H\notin L^\infty(\widehat{\Omega}_{\F,Q})$, for some $Q\in\dd(\pom)$ and for some pairwise disjoint family of cubes $\F\subset\dd_Q$, then it is understood that $F=G/\|H\|_{L^\infty(\widehat{\Omega}_{\F,Q})}=0$ and  $\|G\|_{\C(\widehat{\Omega}_{\F,Q})}/\|H\|_{L^\infty(\widehat{\Omega}_{\F,Q})}=0$. Hence, in the sup the only relevant sawtooths $\widehat{\Omega}_{\F,Q}$ are those on which $H$ is essentially bounded.
\end{remark}

\begin{remark}
We note that the assumption \eqref{locCacc} in \eqref{eq:A-loc->A} is only needed when $\Omega$ is unbounded and $\pom$ is bounded because all dyadic cones are contained in a $C\,\diam(\pom)$-neighborhood of $E$. Hence from $(A)^\dd$ we only get information for $F$ in that region. However,  in all practical applications to solutions of elliptic PDEs \eqref{locCacc}  is easily justified by Caccioppoli's inequality.
\end{remark}

\begin{remark}\label{Good-lambda:classes}
It is possible to show the equivalence of previous conditions upon assuming that they hold in some class of sets. To be more precise, let $\Omega\subset\ree$ be an open set with ADR boundary and suppose that we have a collection $\{\Omega'\}_{\Omega'\in\Sigma}$ such that each $\Omega'\in\Sigma$ is an open subset of $\Omega$, $\pom'$ is ADR boundary, and also that $\widehat{\Omega}_{\F,Q}\in\Sigma$ for every $Q\in \dd(\pom')$ and any pairwise disjoint family of cubes $\F\subset\dd_Q$. Assume further that 
\begin{equation}\label{locCacc-alt}
\left(\frac{1}{r^{n}}\dint_{B(X,r)}  |G(Y)|^2\, \delta(Y)\,dY\right)^{1/2} \leq C \|H\|_{L^\infty(B(X,2 r))}, \mbox{\ \ for all } B(X,2r)\subset\Omega.
\end{equation}
Then, $(A)$ holds on  every $\Omega'\in \Sigma$ iff $(B)_q^\dd$ holds for every $\Omega'\in \Sigma$ and for all (some) $0<q<\infty$ iff $(B)_q$ holds for every $\Omega'\in \Sigma$ and for all (some) $0<q<\infty$; with the understanding that all implicit constants in the statements above are uniform within $\Sigma$. 
We have several examples of classes $\Sigma$. Suppose first that $\Omega=\ree\setminus E$ with $E$ being UR (resp. ADR). In that case $\Sigma$ is the class of open sets $\Omega'\subset \ree\setminus E$ with $\partial\Omega$ being UR (resp. ADR) and the implicit constant in each condition should depend on the UR (resp. ADR) character of each $\Omega'$.  Another interesting example is that when $\Omega$ is some given CAD (resp. 1-sided CAD) and $\Sigma$ is the collection of chord-arc subdomains (resp. 1-sided chord-arc subdomains) $\Omega'\subset\Omega$, in that case the implicit constant in each condition should depend on the CAD (1-sided CAD) character of each $\Omega'$. 
\end{remark}

\begin{lemma}\label{lemma:Aloc-L2} 
Let $\Omega\subset \ree$ be an open set with $\pom$ being ADR and let $\{\W_Q\}_{Q\in\dd(\pom)}$ be a Whitney-dyadic structure for $\Omega$ with some parameters $\eta$ and $K$. If $(A_{\rm loc})$ holds for $G\in L^2_{\rm loc}(\Omega)$ and $H\in C(\Omega)$, then 
\begin{equation}
\|\A^{Q_0}G\|_{L^2(F)} 
\le C
\sigma(Q_0)^{\frac12}\,\Big(
\sup_{\F} \|G\|_{\C(\widehat{\Omega}_{\F,Q_0})}/\|H\|_{L^\infty(\widehat{\Omega}_{\F,Q_0})}^2
\Big)^{\frac12}\|\widehat{N}_*^{Q_0} H\|_{L^\infty(F)},
\end{equation}
for every $Q_0\in\dd(\pom)$ and every Borel set $F\subset Q_0$, and where the sup is taken over all families $\F\in\dd_{Q_0}$ which are pairwise disjoint. The constant $C$ depends on 
$n$, the ADR character of $\pom$, the choice of $\eta, K, \tau$, and the constant in Definition \ref{def:WD-struct} part $(iii)$.
\end{lemma}	

\begin{proof} 
We may assume without lost of generality that $\sigma(F)>0$ and also that $\|\widehat{N}_*^{Q_0} H\|_{L^\infty(F)}<\infty$. Subdivide $Q_0\in\dd(\pom)$ dyadically and stop the first time that $Q\cap F=\emptyset$. This generates a possibly empty maximal (hence pairwise disjoint) family $\F=\{Q_j\}_j\subset\dd_{Q_0}\setminus\{Q_0\}$, so that $Q_j\cap F=\emptyset$ for every $Q_j\in\F$, and $Q\cap F\neq\emptyset$ for every $Q\in\dd_{\F,Q_0}$. 

Let us observe that if $Q\cap F\neq\emptyset$ then necessarily $Q\in\dd_{\F,Q_0}$, otherwise $Q\subset Q_j\in\F$ and hence $Q\cap F=\emptyset$, which is a contradiction. Recall that by construction for every $Y\in U_Q$ we have $\delta(Y)\approx\ell(Q)\approx \dist(Y,\partial \widehat \Omega_{\F,Q_0})$ since, as explained above, $\widehat \Omega_{\F,Q_0}$ is composed of fattened Whitney regions $\widehat{U}_{Q}$, which, in turn, have bounded overlap. Writing $\delta(\cdot)=\dist(\cdot, \pom)$, all these yield
\begin{align*}
\int_F \A^{Q_0}G(x)^2 d\sigma(x)
&\le
\int_F \sum_{x\in Q\in \dd_{Q_0}} \dint_{U_Q} G(Y)^2\,\delta(Y)^{1-n}\,dY\,d\sigma(x)
\\
&
=\sum_{Q\in \dd_{Q_0}} \sigma(F\cap Q) \dint_{U_Q} G(Y)^2\,\delta(Y)^{1-n}\,dY
\\
&
\lesssim
\sum_{Q\in \dd_{\F, Q_0}} \dint_{\widehat{U}_Q} G(Y)^2\,\dist(Y, \partial \widehat \Omega_{\F,Q})\,dY
\\
&
\lesssim
\dint_{\widehat \Omega_{\F,Q_0}} G(Y)^2\,\dist(Y, \partial \widehat \Omega_{\F,Q_0})\,dY.
\end{align*}
Pick then $y\in\partial \widehat \Omega_{\F,Q_0}$ and use $(A_{\rm loc})$ in the sawtooth domain $\widehat \Omega_{\F,Q_0}$ to conclude
\begin{multline*}
\int_F \A^{Q_0}G(x)^2 d\sigma(x)
\lesssim
\dint_{B(y, 2\diam(\widehat \Omega_{\F,Q})) \cap \widehat \Omega_{\F,Q_0}}  G(Y)^2\,\dist(Y, \partial \widehat \Omega_{\F,Q_0})\,dY
\\
\lesssim 
\|G\|_{\C(\widehat \Omega_{\F,Q_0})}\le C_0 \|H\|_{L^\infty(\widehat \Omega_{\F,Q})}^2\diam(\widehat \Omega_{\F,Q_0})^n 
\approx
C_0
\|G\|_{\C(\widehat \Omega_{\F,Q_0})} \|H\|_{L^\infty(\widehat \Omega_{\F,Q_0})}^2\sigma(Q_0),
\end{multline*}
where $C_0=\sup_{\F} \|G\|_{\C(\widehat{\Omega}_{\F,Q_0})}/\|H\|_{L^\infty(\widehat{\Omega}_{\F,Q_0})}^2$.
To conclude we observe that if $Y\in \widehat \Omega_{\F,Q_0}$, then $Y\in\widehat{U}_Q$ for some $Q\in\dd_{\F,Q_0}$. The latter implies that we can find $z\in Q\cap F\neq\emptyset$. Hence $Y\in \Gamma^{Q_0}(z)$ and $|H(Y)|\le \widehat{N}_*^{Q_0} H(z)\le\|\widehat{N}_*^{Q_0} H\|_{L^\infty(F)}$. As a result, 
\[
\int_F \A^{Q_0}G(x)^2 d\sigma(x)\lesssim
C_0  \|\widehat{N}_*^{Q_0} H\|_{L^\infty(F)}^2\sigma(Q_0).
\]
This completes the proof.
\end{proof}

\subsection{Proof of Theorem~\ref{theor:good-lambda}: $(A_{\rm loc}) \Longrightarrow (G\lambda)^\dd$}

Fix $Q_0\in\dd=\dd(\pom)$ and for any $\alpha>0$, set
\[
E_\alpha=\{x\in Q_0: \A^{Q_0} G(x)>\alpha\},
\qquad
F_\alpha=\{x\in Q_0: \widehat{N}_*^{Q_0} H(x)\le \alpha\}.
\]
Note that if $E_\alpha=\emptyset$ then \eqref{eqgl1-loc} (with $Q=Q_0$) is trivial and there is nothing to prove. Assume then that $E_\alpha\neq\emptyset$.

We momentarily suppose that $E_\alpha\subsetneq Q_0$. Given $x\in E_\alpha$, the monotone convergence theorem guarantees that there exists $k_x\ge 0$ such that
\begin{equation}\label{prdgbrtes}
\dint_{\bigcup\limits_{\substack{x\in Q\in \dd_{Q_0}\\ \ell(Q)\ge 2^{-k_x}}} U_Q} |G(Y)|^2\,\delta(Y)^{1-n}>\alpha^2,
\end{equation}
where $\delta(\cdot)=\dist(\cdot,\pom)$.

Let $Q_x\in\dd_{Q_0}$ be the unique cube with $Q_x\ni x$ and $\ell(Q_x)=2^{-k_x}$ and note that for every $y\in Q_x$
\[
\Gamma^{Q_0}(y)
=
\bigcup_{y\in Q\in\dd_{Q_0}} U_Q\supset \bigcup_{Q_x\subset  Q\in\dd_{Q_0}} U_Q
=
\bigcup_{\substack{x\in Q\in \dd_{Q_0}\\ \ell(Q)\ge 2^{-k_x}}} U_Q.
\]
This and \eqref{prdgbrtes} implies that $\A^{Q_0} G(y)>\alpha$. We have then show that for every $x\in E_\alpha$ there exists $Q_x\in\dd_{Q_0}$ such that $Q_x\subset E_\alpha$. We can then take $Q_x^{\rm max}$, with $Q_x\subset Q_x^{\rm max}\subset Q_0$, the maximal cube so that $Q_x^{\rm max}\subset E_\alpha$. Note that $Q_x\subsetneq Q_0$ since $E_\alpha\subsetneq Q_0$. Write then $\F=\{Q_j\}_j\subset\dd_{Q_0}\setminus\{Q_0\}$ for the collection of maximal (hence pairwise disjoint) cubes $Q_x^{\rm max}$ with $x\in E_\alpha$. By construction, $E_\alpha=\bigcup_{Q_j\in\F} Q_j$ and for every $Q_j\in\F$, by maximality, we can find $x_j \in\widetilde{Q}_j\setminus E_\alpha$, where $\widetilde{Q}_j$ is the dyadic parent of $Q_j$. In the latter scenario, if $x\in Q_j$
\[
\Gamma^{Q_0}(x)
=
\bigcup_{x\in Q \in\dd_{Q_0}} U_{Q}
=
\Big(\bigcup_{x\in Q \in\dd_{Q_j}} U_{Q}\Big)
\cup
\Big(\bigcup_{Q_j\subsetneq Q\subset Q_0} U_{Q}\Big)
\subset 
\Gamma^{Q_j}(x)\cup \Gamma^{Q_0}(x_j)
\]
and, consequently,
\[
\A^{Q_0}G(x)
\le
\A^{Q_j}G(x)
+
\A^{Q_0}G(x_j)
\le
\A^{Q_j}G(x)+\alpha, \qquad x\in Q_j.
\]
Using this, for every $\epsilon>0$ we have
\begin{align*}
E_{(1+\epsilon)\alpha}
=
E_{(1+\epsilon)\alpha}\cap E_\alpha 
=
\bigcup_{Q_j\in \F} E_{(1+\epsilon)\alpha}\cap Q_j 
\subset
\bigcup_{Q_j\in \F} \{x\in Q_j: \A^{Q_j}G(x)>\epsilon\alpha\}.
\end{align*}
This holds under the assumption  $E_\alpha\subsetneq Q_0$ but it clearly extends to the case  $E_\alpha\subsetneq Q_0$ by setting $\F=\{Q_0\}$. Hence, invoking  
Chebyshev's and Lemma \ref{lemma:Aloc-L2} in every $Q_j$ inequality we arrive at 
\begin{align*}
\sigma(E_{(1+\epsilon)\alpha}\cap F_{\gamma\alpha})
&\le
\sum_{Q_j\in \F} \sigma\big(\{x\in Q_j: \A^{Q_j}G(x)>\epsilon\alpha\}\cap F_{\gamma\alpha}\big)
\\
&\le
\frac1{(\epsilon\alpha)^2}\sum_{Q_j\in \F}\int_{F_{\gamma\alpha}\cap Q_j} \A^{Q_j}G(x)^2\,d\sigma(x)
\\
&\lesssim
\frac{1}{(\epsilon\alpha)^2} \Big(\sup_{Q_0, \F}\|G\|_{\C(\widehat \Omega_{\F,Q_0})}/\|H\|_{L^\infty(\widehat{\Omega}_{\F,Q_0})}^2\Big) \sum_{Q_j\in \F} \|\widehat{N}_*^{Q_0} H\|_{L^\infty(F_{\gamma\alpha}\cap Q_j)}^2\,\sigma(Q_0)
\\
&\le
\Big(\frac{\gamma}{\epsilon}\Big)^2 \Big(\sup_{Q_0, \F}\|G\|_{\C(\widehat \Omega_{\F,Q_0})}/\|H\|_{L^\infty(\widehat{\Omega}_{\F,Q_0})}^2 \Big)\,
\sum_{Q_j\in \F}\sigma(Q_j)
\\
&\le
\Big(\frac{\gamma}{\epsilon}\Big)^2 \Big(\sup_{Q_0, \F}\|G\|_{\C(\widehat \Omega_{\F,Q_0})}/\|H\|_{L^\infty(\widehat{\Omega}_{\F,Q_0})}^2 \Big)\,
\sigma\Big(\bigcup_{Q_j\in \F}Q_j\Big)
\\
&
=
\Big(\frac{\gamma}{\epsilon}\Big)^2 \Big(\sup_{Q_0, \F}\|G\|_{\C(\widehat \Omega_{\F,Q_0})} /\|H\|_{L^\infty(\widehat{\Omega}_{\F,Q_0})}^2 \Big)\,\sigma(E_\alpha),
\end{align*}
where the sup is taken over all $Q_0\in \dd$ and over all families $\F\in\dd_{Q_0}$ which are pairwise disjoint. This completes the proof. \qed

\subsection{Proof of Theorem~\ref{theor:good-lambda}: $(A_{\rm loc}) \Longrightarrow (B)_q^\dd$, for all $0<q<\infty$}

We start by observing that if $G\in L^2_{\loc}(\Omega)$ then for every $\Omega'\subset \Omega$ one has $\|G\,1_{\Omega'}\|_{\C(\widehat{\Omega}_{\F,Q})}\le \|G\|_{\C(\widehat{\Omega}_{\F,Q})}$ for every $Q\in\dd=\dd(\pom)$ and for every family of pairwise disjoint cubes $\F\in\dd_Q$. This means that if $(A_{\rm loc})$ holds for $G$ and $H$ then so it does for $G\,1_{\Omega'}$ and $H$ uniformly in $\Omega'$. Therefore, from what we have proved so far, $(G\lambda)^\dd$ holds for $G\,1_{\Omega'}$ and $H$ uniformly in $\Omega'$.

Fix $x_0\in \pom$ and given $k\in \N$ set 
\[
\Omega_k=\big\{X\in B(x_0,k)\cap\Omega: |G(X)|\le k, \delta(X)\ge k^{-1}\big\}
\]
and note that for every $0<q<\infty$ and for every $x\in \pom$,
\[
\A(G\,1_{\Omega_k})(x)^2
=
\dint_{\Gamma(x)\cap \Omega_k} |G(Y)|^2\,\delta(Y)^{1-n}\,dY
\le
k^{n+1}|B(x_0,k)|
\approx k^{2(n+1)}.
\]
On the other hand, suppose that $x\in \pom$ is so that $\Gamma(x)\cap\Omega_k\neq\emptyset$. Pick $Z\in \Gamma(x)\cap\Omega_k\neq\emptyset$, then $Z\in I^*$ with $I\in\W_Q$ and $x\in Q\in\dd$. Using \eqref{eq2.whitney2} it follows that 
\begin{multline*}
|x-x_0|
\le
|x-x_Q|+\diam(Q)+\dist(I,Q)+\diam(I^*)+|Z-x_0|
\lesssim
\ell(I)+k
\\
\approx
\delta(Z)+k
\lesssim
|X-z_0|+k
\le 2k.
\end{multline*}
As a consequence, $\supp \A(G\,1_{\Omega_k})\subset B(x_0, C,K)$. These, together with the fact that $\A^Q(G\,1_{\Omega_k})(x)\le \A(G\,1_{\Omega_k})(x)$ for every $x\in \pom$, allow us to conclude that $\A(G\,1_{\Omega_k}), \A^Q(G\,1_{\Omega_k})\in L^\infty_c(\pom)\subset L^q(\pom)$ for every $Q\in \dd$, albeit with bounds that depend on $k$.

Using the previous observations and invoking  $(G\lambda)^\dd$ with $G\,1_{\Omega_k}$ and $H$ (with constant that is independent of $k$) we have for every $Q\in\dd$
\begin{align}\label{eqgl13}
&\|\A^Q(G\,1_{\Omega_k})\|_{L^q(Q)}^q= (1+\eps)^q \int_0^\infty q\alpha^q \sigma \{x\in Q:\, \A^Q (G\,1_{\Omega_k})(x)>(1+\eps)\,\alpha\}\, \frac{d\alpha}{\alpha}
\\[4pt] \nonumber
&\quad\leq  (1+\eps)^q\int_0^\infty q\alpha^q \sigma \{x\in Q:\, \A^Q (G\,1_{\Omega_k})(x)>(1+\eps)\,\alpha, \,\widehat{N}_*^Q H(x)\leq \gamma\alpha \}\, \frac{d\alpha}{\alpha} 
\\[4pt]\nonumber
&\hskip4cm+ 
(1+\eps)^q\int_0^\infty q\alpha^q \sigma \{x\in Q:\, \widehat{N}_*^Q H(x)> \gamma\alpha \}\, \frac{d\alpha}{\alpha} 
\\[4pt]\nonumber
&\quad\leq  C\,\left(\frac{\gamma}{\eps}\right)^\theta (1+\eps)^q\int_0^\infty q\alpha^q \sigma \{x\in Q:\, \A^Q (G\,1_{\Omega_k})(x)>\alpha  \}\, \frac{d\alpha}{\alpha} 
+ \left(\frac{1+\eps}{\gamma}\right)^q \|\widehat{N}_*^Q H\|_{L^q(Q)}^q
\\[4pt]\nonumber
&\quad= C\left(\frac{\gamma}{\eps}\right)^\theta   (1+\eps)^q\,\|\A^Q (G\,1_{\Omega_k})\|_{L^q(Q)}^q + \left(\frac{1+\eps}{\gamma}\right)^q \|\widehat{N}_* H\|_{L^q(Q)}^q.
\end{align}
Pick $\eps=1$ and choose $\gamma$ sufficiently small to ensure that $C\,\gamma^\theta\, 2^q<\frac12$. Using that $\|\A^Q (G\,1_{\Omega_k})\|_{L^q(Q)}^q<\infty$ we can hide this term on the left-hand side of \eqref{eqgl13} and conclude that
\begin{equation}\label{bftghy}
\|\A^Q (G\,1_{\Omega_k})\|_{L^q(Q)}^q
\lesssim
\|\widehat{N}_*^Q H\|_{L^q(Q)}^q,
\end{equation}
with an implicit constant depending on $n$, the ADR character of $\pom$, the choice of $\eta, K, \tau$, the constant in Definition \ref{def:WD-struct} part $(iii)$, $q$, and the implicit constant in $(G\lambda)$, but nonetheless independent of $k$. By the monotone convergence theorem and the fact that $|G(X)|<\infty$ for a.e.~$X\in\Omega$, since $G\in L^2_{\loc}(\Omega)$, it follows that $\A^Q (G\,1_{\Omega_k})(x)\nearrow \A^Q G(x)$. Then we can use the monotone convergence theorem to obtain from \eqref{bftghy} 
\[
\|\A^Q G\|_{L^q(Q)}^q
=
\lim_{k\to\infty}
\|\A^Q (G\,1_{\Omega_k})\|_{L^q(Q)}^q
\lesssim
\|\widehat{N}_*^Q H\|_{L^q(Q)}^q,
\]
completing the proof. \qed

\begin{remark}
The previous arguments easily yield that for any $0<q<\infty$, one has that $(G\lambda)^\dd  \Longrightarrow (B_\loc)_q$ provided $\|\A^Q G\|_{L^q(Q)}<\infty$. A very similar argument gives that $(G\lambda) \Longrightarrow (B)_q$ provided $\|\A G\|_{L^q(\pom)}<\infty$. Details are left to the interested reader.
\end{remark}

\subsection{Proof of Theorem~\ref{theor:good-lambda}: $(G\lambda)^\dd \Longrightarrow (G\lambda)$}

We note that if $\pom$ is bounded, then $\pom$ itself is the largest cube in $\dd=\dd(\pom)$, say $\pom=Q_0$, hence, $ (G\lambda)$ is a particular case of $(G\lambda)^\dd$. Consider next the case $\pom$ unbounded and for every $k\in\N$ write
\[
\Gamma^k(x) =\bigcup_{\substack{x\in Q\in \dd\\ \ell(Q)\le 2^k}} U_Q, \qquad x\in \pom,
\]
and associated with these cones define $\A^k$ and $\widehat{N}^k_*$. Given $Q\in\dd_{-k}$, i.e., $\ell(Q)=2^{-k}$, one easily see that $\Gamma^{Q_0}(x)=\Gamma^k(x)$ for every $x\in Q_0$. Hence, for every $k\in\N$, using  $(G\lambda)^\dd$  we obtain
\begin{align}\label{vbjmfg}
&
\sigma\{x\in \pom:\, \A^k G(x)>(1+\eps)\, \alpha,\,\widehat{N}_* H(x)\leq \gamma \alpha\} 
\\
&\qquad\qquad
\le
\sigma\{x\in \pom:\, \A^k G(x)>(1+\eps)\, \alpha,\,\widehat{N}^k_* H(x)\leq \gamma \alpha\} 
\nonumber\\
&\qquad\qquad
=
\sum_{Q\in\dd_{-k}} \sigma\{x\in Q:\, \A^k G(x)>(1+\eps)\, \alpha,\,\widehat{N}^k_* H(x)\leq \gamma \alpha\} 
\nonumber\\
&\qquad\qquad
=
\sum_{Q\in\dd_{-k}} \sigma\{x\in Q:\, \A^Q G(x)>(1+\eps)\, \alpha,\,\widehat{N}^Q_* H(x)\leq \gamma \alpha\} 
\nonumber\\
&\qquad\qquad
 \lesssim (\gamma/\eps)^\theta \sum_{Q\in\dd_{-k}} \sigma\{x\in Q:\, \A^Q G(x)>\alpha\} 
\nonumber\\
&\qquad\qquad
 =(\gamma/\eps)^\theta \sum_{Q\in\dd_{-k}} \sigma\{x\in Q:\, \A^k G(x)>\alpha\} 
\nonumber\\
&\qquad\qquad 
=(\gamma/\eps)^\theta \sigma\{x\in \pom\, \A^k G(x)>\alpha\} 
\nonumber\\
&\qquad\qquad
 \le (\gamma/\eps)^\theta \sigma\{x\in \pom:\, \A G(x)>\alpha\}. 
\nonumber
\end{align}
On the other hand, the monotone convergence theorem gives that $\A^k G(x)\nearrow \A G(x)$ as $k\to\infty$ and for every $x\in \pom$. Hence, another use of the monotone convergence theorem and \eqref{vbjmfg} yield
\begin{multline*}
\sigma\{x\in \pom:\, \A G(x)>(1+\eps)\, \alpha,\,\widehat{N}_* H(x)\leq \gamma \alpha\} 
\\
=\lim_{k\to\infty}
\sigma\{x\in \pom:\, \A^k G(x)>(1+\eps)\, \alpha,\,\widehat{N}_* H(x)\leq \gamma \alpha\} 
\lesssim
(\gamma/\eps)^\theta \sigma\{x\in \pom:\, \A G(x)>\alpha\}, 
\end{multline*}
and the proof is complete. \qed

\subsection{Proof of Theorem~\ref{theor:good-lambda}: $(B)_q^\dd$ for some $0<q<\infty \Longrightarrow (B)_q$}
We note that if $\pom$ is bounded, then $\pom$ itself is the largest cube in $\dd= \dd(\pom)$, say $\pom=Q_0$, hence, $(B)_q$ is a particular case of $(B_\loc)_q$. 
If $\pom$ is unbounded we use  the same argument  as in the previous proof
\begin{multline*}
\int_{\pom} \A^k G(x)^q\,d\sigma(x)
=
\sum_{Q\in\dd_{-k}} \int_Q \A^k G(x)^q\,d\sigma(x)
=
\sum_{Q\in\dd_{-k}} \int_Q \A^Q G(x)^q\,d\sigma(x)
\\
\lesssim
\sum_{Q\in\dd_{-k}} \int_Q \widehat{N}_*^Q H(x)^q\,d\sigma(x)
=
\sum_{Q\in\dd_{-k}} \int_Q \widehat{N}_*^k H(x)^q\,d\sigma(x)
\\
=
\int_{\pom} \widehat{N}_*^k H(x)^q\,d\sigma(x)
\le
\int_{\pom} \widehat{N}_* H(x)^q\,d\sigma(x).
\end{multline*}
From here the fact that $\A^k G(x)\nearrow \A G(x)$ as $k\to\infty$ and for every $x\in \pom$ and the monotone convergence theorem gives the desired estimate. \qed

\subsection{Proof of Theorem~\ref{theor:good-lambda}: $(B)_q^\dd$  for some $0<q<\infty \Longrightarrow (A)^\dd$.}
Assume  that $(B_{\rm loc})_q$, for some $0<q<\infty$ holds. We may assume that $H\in L^\infty(\Omega)$. Hence, for every $Q\in\dd(\pom)$,
\[
\int_Q \A^Q G(x)^q\,d\sigma(x)
\le
C_q^q\, \int_Q \widetilde{N}_*^Q H(x)^q\,d\sigma(x)
\le
C_q^q\,\|H\|_{L^\infty(\Omega)}^q\,\sigma(Q)
\]
Writing  $F:= G \,\left(2 ^{1/q}\,C_q\, \|H\|_{L^\infty(\Omega)}\right)^{-1}$ we have by Chebyshev's
\begin{align*}
\sigma\{x\in Q:\,\A^Q F(x)> 1\}\leq 
\int_Q \A^Q F(x)^q\,d\sigma(x)
\le\frac12\,\sigma(Q).
\end{align*}
We then invoke Lemma \ref{lemma:J-N} with $p=2$ and obtain
\[
\sup_{Q\in\dd_{Q_0}}\,\fint_Q \A^Q F(x)^2\,d\sigma(x)
\lesssim 1.
\]
On the other hand, writing $\delta(\cdot)=\dist(\cdot,\pom)$ recalling that the family $\{U_{Q'}\}_{Q'\in\dd(\pom)}$ has bounded overlap
\begin{multline}\label{Fubini-ME-SFE}
\dint_{T_Q} 
F^2\delta\,dY
\approx
\sum_{Q'\in\dd_Q} \dint_{U_{Q'}} 
F^2\delta\,dY
\approx
\sum_{Q'\in\dd_Q} \sigma(Q') \dint_{U_{Q'}} 
F^2\delta^{1-n}\,dY
\\
=
\int_Q \sum_{x\in Q'\in\dd_Q}  \dint_{U_{Q'}} F^2\delta^{1-n}\,dY\,d\sigma(x)
\\
\approx 
\int_Q\dint_{\Gamma^Q(x)} F^2\delta^{1-n}\,dY\,d\sigma(x)
=
\int_Q \A^Q F(x)^2\,d\sigma(x),
\end{multline}
Thus,
\[
\|F\|_{\C^{\rm dyad}(\Omega)}=\sup_{Q\in\dd_{Q_0}}\,\frac{1}{\sigma(Q)}\dint_{T_Q} 
F(Y)^2\delta(X)\,dY
\lesssim 1,
\]
and the proof is complete. \qed

\subsection{Proof of Theorem~\ref{theor:good-lambda}: $(A)^\dd \ \&\ \eqref{locCacc} \Longrightarrow (A) $.}

This follows trivially from \eqref{CME:dyadic}:
\[
\|G\|_{\C(\Omega)}\lesssim \|G\|_{\C^{\rm dyad}(\Omega)}+ \|G\|_{\C_0(\Omega)}
\lesssim
\|H\|_{L^\infty(\Omega)}^2,
\]
which is the desired estimate.
\qed

\subsection{Proof of Theorem~\ref{theor:good-lambda}: $(B_\loc)_q \,\, \mbox{ for some $0<q<\infty$ } \Longrightarrow (A)^\dd$.}
Write $\dd=\dd(\pom)$ and $\delta(\cdot)=\dist(\cdot,\pom)$. Assume $(B_\loc)$ and fix $Q_0\in\dd$. We may suppose that $H\in L^\infty(\Omega)$, otherwise there is nothing to prove. Recall that $\widehat{T}_{Q_0}=\widehat{\Omega}_{\emptyset,Q_0}$, hence $(B_\loc)$ implies that $A<N$ on $L^q(\partial T_{Q_0})$. This, Remark \ref{remark:N<S:cubes->balls} yields for  every $\kappa>0$
\begin{multline}\label{qw3basvef}
\|\A_{\widehat{T}_{Q_0},\kappa}G\|_{L^q(\partial \widehat{T}_{Q_0})}^q
\lesssim
\|N_{*,\widehat{T}_{Q_0},\kappa}H\|_{L^q(\partial \widehat{T}_{Q_0})}^q
\\
\le
\|H\|_{L^\infty(\Omega)}^q H^n(\partial T_Q)
\lesssim
\|H\|_{L^\infty(\Omega)}^q \diam(\partial T_Q)^n
\lesssim 
\|H\|_{L^\infty(\Omega)}^q \ell(Q)^n
\approx
\|H\|_{L^\infty(\Omega)}^q \sigma(Q)^n,
\end{multline}
where we have used that $\partial T_Q$ is upper ADR (see Remark \ref{remark:ADR-sawtooth}), \eqref{eq3.3aab}, and that $\pom$ is ADR.

Let $x\in Q_0$ and $Y\in \Gamma^{Q_0}(x)$. Then $Y\in I^*$ with $I\in\W_Q$  with $x\in Q\in\dd_{Q_0}$. Recalling that $I^*=I^*(\tau)$ and that $\widehat{T}_{Q_0}$ is defined using fattened Whitney cubes of the form $J^*(2\tau)$ we clearly see that $Y\in\T_{Q_0}\subset \widehat{T}_{Q_0}$ with $\delta(Y)\approx\dist(Y, \partial \widehat{T}_{Q_0})$. Consequently, 
\[
|Y-x|
\le
\diam(I)+\dist(I,Q)+\diam(Q)
\lesssim
\ell(I)
\approx
\delta(Y)\approx\dist(Y, \partial \widehat{T}_{Q_0}).
\]
Then we can find $\kappa$ depending on $n$, the ADR constants of $\pom$, $\eta$, $K$, and the constant in Definition \ref{def:WD-struct} part $(iii)$ such that
$Y\in\Gamma_{\widehat{T}_{Q_0},\kappa}(x)$. Since $Q_0\subset\partial \widehat{T}_{Q_0}$  (see \cite[Proposition 6.1]{HM-I})  we then obtain 
\begin{multline*}
\A^{Q_0} G(x)=\left(\dint_{\Gamma^Q(x)}|G(Y)|^2 \delta(Y)^{1-n} dY\right)^{\frac12}
\approx
\left(\dint_{\Gamma^Q(x)}|G(Y)|^2 \dist(Y,\partial T_{Q_0})^{1-n} dY\right)^{\frac12}
\\
\le
\left(\dint_{\Gamma_{\widehat{T}_{Q_0},\kappa}(x)}|G(Y)|^2 \dist(Y,\partial T_{Q_0})^{1-n} dY\right)^{\frac12}
=\A_{\widehat{T}_{Q_0},\kappa}G(x).
\end{multline*}
This and \eqref{qw3basvef} imply
\[
\fint_{Q_0} \A^{Q_0} G(x)^q\,d\sigma(x)
\lesssim
\fint_{Q_0} \A_{\widehat{T}_{Q_0},\kappa}G(x)^q\,d\sigma(x)
\le
C \|H\|_{L^\infty(\Omega)}^q.
\]
Writing $F=G (C\|H\|_{L^\infty(\Omega)})^{-1}$ and for $N$ large enough it follows from Chebyshev's inequality 
\[
{\sigma\left\{x\in Q_0:\,\A^{Q_0} F(x)> 1\right\}}
\leq 
\int_{Q_0} \A^{Q_0} F(x)^q\,d\sigma(x)
\le
\frac12\sigma(Q).
\]
Since $Q_0\in\dd$ is arbitrary we can apply Lemma \ref{lemma:J-N} with $p=2$ and obtain
\[
\sup_{Q\in\dd_{Q_0}}\,\fint_Q \A^Q F(x)^2\,d\sigma(x)
\lesssim
1. 
\]
This and \eqref{Fubini-ME-SFE} give 
\[
\|F\|_{\C^{\rm dyad}(\Omega)}=\sup_{Q\in\dd_{Q_0}}\,\frac{1}{\sigma(Q)}\dint_{T_Q} 
F(Y)^2\delta(X)\,dY
\lesssim 1,
\]
which is the desired estimate. \qed

\section{Transference of $N<S$ estimates: from Lipschitz to chord-arc domains} 
Before starting, we introduce some notation. Let $D\subset \ree$ be a bounded CAD. Given $Q\in\dd(\partial D)$ or $\Delta=\Delta(x,r)$, with $x\in\partial D$ and $0<r\lesssim\diam(\partial D)$, we will write $X_Q^+$ and $X_\Delta^+$ to denote respectively some interior corkscrew points relative to $Q$ (that is, relative to $\Delta_Q$, see \eqref{cube-ball}) and $\Delta$. When $\partial D$ is bounded, we write $X_D^+$ to denote a corkscrew point relative to a surface ball $\Delta(x,3\diam(\partial D)/2)=\partial D$ for some $x\in\partial D$.


Also, recall the  dyadic Hardy-Littlewood maximal function from Definition~\ref{defHLmax}. In addition, we will be using its continuous analogue. 
Let $E\subset \ree$ be an $n$-dimensional  ADR set. By $M=M_{E}$ we denote the continuous (non-centered) Hardy-Littlewood maximal function on $E$, that is, for $f\in L^1_{\rm loc}(E)$
$$M f(x)=\sup_{\Delta\ni x} \fint_\Delta |f(y)|\,d\sigma(y), $$
where the sup is taken over all $\Delta$, surface balls on $E$ containing $x$. For $0<p<\infty$, we also write $M_{p} f=M(|f|^p)^\frac1p$. 
It is clear from \eqref{cube-ball} that $M^\dd f(x)\lesssim Mf(x)$ for every $x\in E$. The converse might fail pointwise, but both maximal functions are bounded in $L^p(E)$, $p>1$. 

We are now ready to state the main result of this section:

\begin{theorem}\label{theor:N<S:Lip->CAD}
Let $D\subset \ree$ be a  CAD. Let $u\in W^{1,2}_{\rm loc}(D)\cap C(D)$ and 
assume that there exists $C_0>0$ such that for any $c\in\RR$, and for any cube $I$ with $2I\subset D$, 
\begin{equation}\label{locbdd}
\sup_{X\in I}|u(X)-c|\leq C_0 \left(\ell(I)^{-n-1}\dint_{2 I} |u-c|^2 \, dX\right)^{\frac12}.
\end{equation}
Suppose that the $N<S$ estimates are valid on $L^2$ on all bounded Lipschitz subdomains $\Omega\subset D$, that is, for any  bounded Lipschitz subdomain $\Omega\subset D$ there holds
\begin{equation}\label{eqn5.9}
\left\|N_{*, \Omega} \big(u-u(X_{\Omega}^+)\big)\right\|_{L^2(\pom)}\leq C_{\Omega} \left\|S_{\Omega} u\right\|_{L^2(\po)}. 
\end{equation}
Here $X_{\Omega}^+$ is any interior corkscrew point of $\Omega$ at the scale of $\diam (\Omega)$, and the constant  $C_{\Omega}$ in \eqref{eqn5.9} depends on the Lipschitz character of $\Omega$,  the dimension $n$, the implicit choice of $\kappa$ (the aperture of the cones in $N_{*, \Omega}$ and $S_{\Omega}$),  and the implicit corkscrew constant for the point $X_{\Omega}^+$.  

Given $\eta\ll 1$ and $K\gg 1$, consider $\{\W_Q\}_{Q\in\dd(\partial D)}$ a Whitney-dyadic structure for $D$ with parameters $\eta$ and $K$, see Section \ref{sections:CAD}. Then there exists $0<c_0\ll 1$ and $C>0$, depending on $n$, the CAD character of $D$, the choice of $\eta, K, \tau$, such that for every $\eps>0$, every $0<\gamma<c_0\,\eps/C_0$, for all $\alpha>0$, and for all $Q\in\dd(\partial D)$
\begin{multline}\label{eqn5.1:local}
\sigma\{x\in Q:\, N_*^Q (u-u(X_Q^+))(x)> (1+\eps)\, \alpha, \,M^{\dd}_{Q_0,2}(\widehat{S}^{Q}u)(x)\leq \gamma  \alpha\} \\[4pt]
\leq  
C^*_{\gamma, \eps}
\,\sigma\{x\in Q:\, N_*^{Q}(u-u(X_Q^+))(x)>\alpha\},
\end{multline}
where $C_{\gamma,\eps}^*=\big(1-\theta+ C\,({\gamma}/{\eps})^2 \big)$ and $\theta\in (0,1)$ is from Corollary~\ref{cIBPLS} (hence depends on $n$ and the CAD character of $D$). Therefore
\begin{equation}\label{eqn6.5:local}
\|N_*^Q(u-u(X_Q^+))\|_{L^q(Q)}\leq C' \|\widehat{S}^Q u\|_{L^q(Q)}, \quad \mbox{for all}\quad q>2.
\end{equation} 
where $C'$ depends  on $n$, the CAD character of $D$, $C_0$, the choice of $\eta, K, \tau$, and $q$. 

As a consequence, for any $x\in\partial D$ and $0<r<2\diam(\partial D)$ there exists $K'$ depending on $n$, the CAD character of $D$ such that for every $\kappa>0$
\begin{equation}\label{eqn6.5:local-balls}
\|N_{*,D,\kappa}^r (u-u(X_{\Delta(x,r)}^+)\|_{L^q(\Delta(x,r))}\le C'' \|S^{K'r}_{D,\kappa} u\|_{L^q(\Delta(x,K' r))},  \quad \mbox{for all}\quad q>2.
\end{equation}
where $\Delta(x,r)=B(x,r)\cap\pom$, and where $C''$ depends on $q$, $n$, the CAD character of $D$, $C_0$, and $\kappa$.	
In particular, if $\partial D$ is bounded 
\begin{equation}\label{eqn6.5}
\|N_{*,D,\kappa}(u-u(X_D^+))\|_{L^q(\partial D)}\leq C'' \|S_{D,\kappa} u\|_{L^q(\partial D)}, \quad \mbox{for all}\quad q>2
\end{equation} 
and if $\partial D$ is unbounded and $u(X)\to 0$ as $|X|\to\infty$ then 
\begin{equation}\label{eqn6.5:unbounded}
\|N_{*,D,\kappa}u\|_{L^q(\partial D)}\leq C'' \|S_{D,\kappa} u\|_{L^q(\partial D)}, \quad \mbox{for all}\quad q>2.
\end{equation} 
\end{theorem}

We remark that contrary to the previous sections, we do not consider general $\A G$ and $N_* H$ anymore. This is a necessity, as the argument of the area integral has to be the gradient of the argument of the non-tangential maximal function. The assumption \eqref{locbdd} is a standard interior regularity estimate for solutions of elliptic equations 
(also known as Moser's local boundedness estimate). In principle, we need a somewhat different version.  
Recall that $\widehat{U}_Q$ is a fattened version of the Whitney region $U_Q$.  We have 
\begin{equation}\label{locbdd1}
|u(Y_Q)-c|\leq C_0 \left(\ell(Q)^{-n-1}\dint_{\widehat{U}_Q} |u-c|^2 \, dX\right)^{\frac12},
\end{equation}
where $Y_Q$ is any point lying in $U_Q$, so that there is
a ball centered at $Y_Q$, of radius proportional to $\ell(Q)$, which lies inside $\widehat{U}_Q$. 
We note that if we assumed \eqref{locbdd} or \eqref{locbdd1}  without enlarging the integrals on the respective right-hand sides, we  could obtain a version of \eqref{eqn5.1:local}--\eqref{eqn6.5:local} without enlarging the ``aperture of cones" on the right hand side (that is, with $S^Q$ in place of $\widehat{S}^Q$). But that is minor and \eqref{locbdd} looks a bit more familiar and more in line with \eqref{oscbdd} below.

\begin{proof} 
For starters, write $\dd=\dd(\partial D)$ and $\delta(\cdot)=\dist(\cdot, \partial D)$. Fix $\eta\ll 1$ and $K\gg 1$ and consider $\{\W_Q\}_{Q\in\dd(\partial D)}$ a  Whitney-dyadic structure for $D$ with parameters $\eta$ and $K$ from Section \ref{sections:CAD}. We a claim that for every $Q\in\dd$
\begin{equation}\label{locbdd:SFE}
\sup_{X,Y\in U_Q} |u(X)-u(Y)|
\le
C\, C_0\, \inf_{z\in Q }\widehat{S}^{Q}u (z)
\le
C\, C_0\, \fint_Q \widehat{S}^{Q}u\,d\sigma
\end{equation}
where $C$ depends on $n$, $\eta$, $K$, $\tau$, and the CAD character of $D$, and $C_0$ is the constant in \eqref{locbdd}. 
To see this observe that for every $Q\in\dd$ and $X\in U_Q$ we have that $X\in I^*(\tau)$ for some $I\in \W_Q$. Let $I_X\subset D$ be the cube centered at $X$ with side length $\tau\,\ell(I)$ so that $2 I_X\subset I^*(2\tau)\subset \widehat{U}_Q$. Note that $\ell(I_X)\approx\ell(I)\approx\ell(Q)$. Then, \eqref{locbdd} yields,  for every $c\in\re$,
\begin{equation}\label{locbdd:UQ}
|u(X)-c|
\leq C_0 \left(\ell(I_X)^{-n-1}\dint_{2I_X} |u-c|^2 \, dX\right)^{\frac12}
\lesssim
C_0 \left(\ell(Q)^{-n-1}\dint_{\widehat{U}_Q} |u-c|^2 \, dX\right)^{\frac12}.
\end{equation}
With this at hand, let $Q\in\dd$ and $X,Y\in U_Q$ and $z\in Q$. Setting 
\[
c_Q
:= 
\frac{1}{|\widehat U_Q|}\dint_{\widehat{U}_Q} v\, dZ\] 
obtain 
\begin{multline*}
|u(X)-u(Y)| \leq |u(X)-c_Q|+|u(Y)-c_Q|
\lesssim
C_0 \left(\ell(Q)^{-n-1}\dint_{\widehat U_Q} |u-c_Q|^2\, dZ\right)^{\frac12} 
\\
\lesssim C_0 \left(\ell(Q)^{-n+1}\dint_{\widehat U_Q} |\nabla u|^2\, dX\right)^{\frac12}
\approx 
C_0 \left(\dint_{\widehat U_Q} |\nabla u|^2\,\delta^{1-n}\, dX\right)^{\frac12}
\le
C_0 \widehat{S}^{Q}u (z),
\end{multline*}
where the second inequality follows from \eqref{locbdd:UQ}, the third from  Poincar\'e's inequality in the context of Whitney regions (see the argument in \cite[Proof of Lemma 3.1]{HMT}), and the last from the fact that $\delta(\cdot)\approx\ell(Q)$ in $\widehat U_Q$. This proves our claim.

Let us fix $Q_0\in\dd$ and write $v:=u-u(X_{Q_0}^+)$, with $X_{Q_0}^+$ begin the corkscrew relative to $Q_0$, that is, relative to the surface ball $\Delta_{Q_0}$ (cf. \eqref{cube-ball} and \eqref{cube-ball2}). For every $\alpha>0$ we set
\[
E_\alpha:=\{x\in Q_0:\, N_*^{Q_0} v(x)>\alpha\},
\qquad
F_\alpha:=\{x\in Q_0:\, M^{\dd}_{Q_0,2}(\widehat{S}^{Q_0}v)( x)\le \alpha\},
\]
where $M^{\dd}_{Q,2}$ was defined in Definition \ref{defHLmax}. Our goal is to obtain for every $\alpha,\gamma,\eps>0$ with  $0<\gamma\ll \eps/C_0$ there holds
\begin{align}\label{eqn5.1:local:*}
\sigma(E_{(1+\eps)\alpha}\cap F_{\gamma\,\alpha})
\le 
C^*_{\gamma, \eps}
\,\sigma(E_\alpha),
\end{align}
and we will me more specific about the constant $C^*_{\gamma, \eps}$ momentarily. With this goal in mind we fix  $\alpha,\gamma,\eps>0$. We may assume that the set $E_\alpha\neq\emptyset$ otherwise \eqref{eqn5.1:local} is trivial. 

 Let $x\in E_\alpha$, then there exist $Q_x\in\dd_{Q_0}$ with $x\in Q_x$ and $Y\in U_{Q_x}$ such that $|v(Y)|>\alpha$. Note that $U_{Q_x} \subset\Gamma^{Q_0}(y)$ for every $y\in Q_x$, hence $N_*^{Q_0}v(y)\ge |v(Y)|>\alpha$ and $Q_x\subset E_\alpha$. We can then take $Q_x^{\rm max}$, with $Q_x\subset Q_x^{\rm max}\subset Q_0$, the maximal cube so that $Q_x^{\rm max}\subset E_\alpha$. Write then $\F=\{Q_j\}_j\subset\dd_{Q_0}$ for the collection of maximal (hence pairwise disjoint) cubes $Q_x^{\rm max}$ with $x\in E_\alpha$. By construction, $E_\alpha=\bigcup_{Q_j\in\F} Q_j$.

Given $Q\in\F $, invoke Corollary~\ref{cIBPLS} and take a bounded Lipschitz domain $\Omega_Q\subset D$ satisfying properties $(i)$--$(iii)$ in the statement. In particular, we set $F_Q:=\partial\Omega_Q \cap Q\subset Q$ such that  $\sigma(F_Q)\geq \theta \,\sigma (Q)$. Our goal is to show that 
\begin{align}\label{eqn5.4:local}
\sigma(E_{(1+\eps)\alpha}\cap F_{\gamma\,\alpha}\cap F_Q)
\le 
C\, \Big(\frac{\gamma}{\eps}\Big)^2C_{\Omega_Q}
\,\sigma(Q),
\end{align}
where $C_{\Omega_Q}$ is the constant from \eqref{eqn5.9}, hence it depends on the Lipschitz character of $\Omega_Q$, which in turn depends only on the CAD character of $D$, and $C$ depends as well on the CAD character of $D$
Assuming this momentarily, we obtain \eqref{eqn5.1:local}:
\begin{multline*}
\sigma(E_{(1+\eps)\alpha}\cap F_{\gamma\,\alpha})
=
\sigma(E_{(1+\eps)\alpha}\cap E_\alpha\cap F_{\gamma\,\alpha})
=
\sum_{Q\in\F}\sigma(E_{(1+\eps)\alpha}\cap E_\alpha\cap Q)
\\
\le
\sum_{Q\in\F} \big( \sigma(Q\setminus F_Q) +\sigma(E_{(1+\eps)\alpha}\cap E_\alpha\cap F_Q)\big)
\\
\le
\Big(1-\theta+ C\, \Big(\frac{\gamma}{\eps}\Big)^2 \sup_{Q\in\dd}C_{\Omega_Q} \Big)\sum_{Q\in\F}\sigma(Q)
= 
C_{\gamma,\eps}^*
\sigma(E_\alpha)
\end{multline*}
where $C_{\gamma,\eps}^*=\big(1-\theta+ C\,({\gamma}/{\eps})^2 \sup_{Q\in\dd} C_{\Omega_Q}\big)$. Note that  $\sup_{Q\in\dd} C_{\Omega_Q}<\infty$ and ultimately depends on the CAD character of $D$, since all the Lipschitz characters of the $\Omega_Q$'s are uniformly bounded depending on the CAD character of $D$ (cf.  Corollary~\ref{cIBPLS}) and our assumption states that $C_{\Omega_Q}$ depends on the Lipschitz character of $\Omega_Q$,  the dimension $n$, and the choice of $\kappa$ (the aperture of the cones).

Let us then obtain \eqref{eqn5.4:local}. We may assume that the left-hand side is non-zero, hence we can pick $z_Q\in E_{(1+\eps)\alpha}\cap F_{\gamma\,\alpha}\cap F_Q$.  
Let $Y_Q$ be from Corollary~\ref{cIBPLS} part $(ii)$ whose existence is guaranteed by part $(i)$ and note that $Y_Q\in U_Q$.

We need to consider two separate cases. First assume that $Q\subsetneq Q_0$. By the maximality of $Q\in\F$, we can find $\widetilde{x} \in\widetilde{Q}\setminus E_\alpha$, where $\widetilde{Q}$ is the dyadic parent of $Q$. That is, $N_*^{Q_0} v(\widetilde{x})\le \alpha$ and, in particular, $|v(X)|\le \alpha$ for every $X\in U_{\widetilde{Q}}$ since $U_{\widetilde{Q}}\subset \Gamma^{Q_0}(\widetilde{x})$. Note then that if $x\in Q$, then
\begin{multline}\label{localiz-NQ0}
N_*^{Q_0} v(x)
=
\sup_{Y\in \Gamma^{Q_0}(x)}|v(Y)|
=
\max\Big\{\sup_{Y\in \Gamma^{Q}(x)}|v(Y)|, \max_{Y\in U_Q, \widetilde{Q}\subset Q\subset Q_0 }|v(Y)|\Big\}
\\
\le
\max\Big\{N_*^{Q} v(x),N_*^{Q_0} v(\widetilde{x})\Big\}
\le
\max\Big\{N_*^{Q} v(x),\alpha\Big\}.
\end{multline}
Since  $|v(X)|\le \alpha$ for every $X\in U_{\widetilde{Q}}$, we have that $|v(X_{\widetilde{Q}}^+)|\le\alpha$ where $X_{\widetilde{Q}}^+$ is the interior corkscrew point relative to $\widetilde{Q}$ (with respect to $D$ which is a CAD). Then, recalling that the construction of $\W_Q$ guarantees that $X_{\widetilde{Q}}^+\in U_Q$, and that $Y_Q\in U_Q$, 
we have, by \eqref{locbdd:SFE}, 
\begin{multline}\label{Pdet}
|v(X_{\widetilde Q}^+)-v(Y_{Q})| = |u(X_{\widetilde Q}^+)-u(Y_{Q})|
\le
C\,C_0 \fint_Q \widehat{S}^{Q}u\,d\sigma
\le
C\,C_0 \fint_Q \widehat{S}^{Q_0}u\,d\sigma
\\
=
C\,C_0 \fint_Q \widehat{S}^{Q_0}v\,d\sigma
\le M^{\dd}_{Q_0,2}(\widehat{S}^{Q_0}v)( z_Q)
\le
C\, C_0 \gamma\,\alpha
,
\end{multline}
where we have used that $z_Q\in Q\cap F_{\gamma\,\alpha}$. As a consequence,
\begin{equation}\label{eqn5.6:new}
|v(Y_{Q})|\le |v(Y_{Q})-v(X_{\widetilde Q}^+)|+|v(X_{\widetilde Q}^+)|\le (1+C\,C_0\,\gamma)\,\alpha \le (1+\eps/2)\,\alpha,
\end{equation}
where $C$ depends on the CAD character of $D$, and provided $\gamma<(2\,C\, C_0)^{-1}\,\eps=:2\,c_0\,\eps$. As a result, using \eqref{localiz-NQ0}, for every $x\in E_{(1+\eps)\alpha}$ we arrive at
\[
(1+\eps)\,\alpha
< 
N_*^{Q_0}v(x)
=
N_*^{Q}v(x)
\le 
N_*^{Q}(v-v(Y_Q))(x) +|v(Y_Q)| 
\le 
N_*^{Q}(v-v(Y_Q))(x) + (1+\eps/2)\,\alpha,
\]
and, consequently, 
\begin{equation}\label{eqn5.5}
E_{(1+\eps)\alpha}\cap F_{\gamma\,\alpha}\cap F_Q
\subset
\{x\in F_{\gamma\,\alpha}\cap  F_Q:  N_*^{Q}(v-v(Y_Q))(x)> \eps\,\alpha/2\},
\end{equation}
where we recall that we are currently considering the case $E_\alpha\subsetneq Q$.

In the second case $Q=Q_0$, hence $\F=\{Q\}$ and $E_\alpha=Q$. Since $Y_Q, X_{Q_0}^+\in U_{Q_0}$  we can invoke \eqref{locbdd:SFE} to obtain
\begin{multline}\label{adqeaFWEfw}
|v(Y_Q)|=|u(Y_Q)-u(X_{Q_0})^+|
\le
C\,C_0 \fint_{Q_0} \widehat{S}^{Q_0}u\,d\sigma
=
C\,C_0 \fint_{Q_0} \widehat{S}^{Q_0}v\,d\sigma
\\
\le 
M^{\dd}_{Q_0,2}(\widehat{S}^{Q_0}v)( z_Q)
\le
C\, C_0 \gamma\,\alpha
\le (1+C\,C_0\,\gamma)\,\alpha \le (1+\eps/2)\,\alpha,
\end{multline}
where $C$ depends on the CAD character of $D$, and provided $\gamma<(2\,C\, C_0)^{-1}\,\eps=:2\,c_0\,\eps$. Consequently, 
for every $x\in E_{(1+\eps)\alpha}$ we arrive at
\[
(1+\eps)\,\alpha
< 
N_*^{Q_0}v(x)
\le 
N_*^{Q_0}(v-v(Y_Q))(x) +|v(Y_Q)|
\le 
N_*^{Q}(v-v(Y_Q))(x) + (1+\eps/2)\,\alpha
\]
Thus, $N_*^{Q}(v-v(Y_Q))(x)= N_*^{Q_0}(v-v(Y_Q))(x)> \eps\,\alpha/2$ and \eqref{eqn5.5} holds also in this case.

We can now merge the two cases and with the proof. Pick $x\in F_{\gamma\,\alpha}\cap  F_Q =\partial\Omega_Q\cap Q$ be such that $N_*^{Q}(v-v(Y_Q))(x)> \eps\,\alpha/2$. Then, there exist $Q'\in \dd_Q$ with $Q'\ni x$ and $Y\in U_{Q'}$ such that $|v(Y)-v(Y_Q)|> \eps\alpha/2$. Thus, $y_{Q'}:=x\in Q'\cap\partial\Omega_Q= F_Q\cap Q'$ and applying once again condition $(ii)$ of Corollary~\ref{cIBPLS}
we can find the corresponding ponding $Y_{Q'}\in U_{Q'}$ so that
\[
|Y_{Q'}-y_{Q'}|<\ell(Q')\le C\,\dist(Y_{Q'},\partial\Omega_Q):=(1+\kappa)\,\dist(Y_{Q'},\partial\Omega_Q).
\]
where $C\ge 2$ is the constant from Corollary~\ref{cIBPLS}. This means that $Y_{Q'}\in \Gamma_{\Omega_Q}(y_{Q'})=\Gamma_{\Omega_Q}(x)$ (cf.  \eqref{conetrad}). On the other hand, since $Y, Y_{Q'}\in U_{Q'}$ and $x\in F_{\gamma\,\alpha}\cap Q'$, one can see that \eqref{locbdd:SFE} yields
\begin{multline}\label{Afgaeget}
|v(Y)-v(Y_{Q'})|
=
|u(Y)-u(Y_{Q'})|
\le
C\,C_0 \fint_{Q'}\widehat{S}^{Q'}u\,d\sigma
\le
C\,C_0 \fint_{Q'}\widehat{S}^{Q_0}u\,d\sigma
\\
=
C\,C_0 \fint_{Q'}\widehat{S}^{Q_0}v\,d\sigma
\le
C\,C_0 M^{\dd}_{Q_0,2}(\widehat{S}^{Q_0}v)(x)
\le
C\, C_0 \gamma\,\alpha
\le
\eps\,\alpha/4,
\end{multline}
provided $\gamma<c_0\,\eps=(4\,C\, C_0)^{-1}\,\eps$. Hence, 
\[
 \eps\alpha/2<|v(Y)-v(Y_Q)|
 \le 
|v(Y)-v(Y_{Q'})|+ |v(Y_Q')-v(Y_{Q})|
 \le 
 \eps\,\alpha/4+ |v(Y_Q')-v(Y_{Q})|
\]
and
\[
N_{*,\Omega_Q}(v-v(Y_Q))(x)
=
\sup_{Z\in \Gamma_{Q}(x)} |v(Y)-v(Y_Q)|
\ge
|v(Y_{Q'})-v(Y_Q)|
\ge
 \eps\,\alpha/4.
\]
All these yield
\begin{equation*}
E_{(1+\eps)\alpha}\cap F_{\gamma\,\alpha}\cap F_Q
\subset
\{x\in \partial\Omega_Q:  N_{*,\Omega_Q}(v-v(Y_Q))(x) > \eps\,\alpha/4\}
.
\end{equation*}
Use Chebyshev's inequality and the assumption \eqref{eqn5.9} we write 
\begin{multline}\label{eqn5.10-0}
\sigma(E_{(1+\eps)\alpha}\cap F_{\gamma\,\alpha}\cap F_Q) 
\leq \sigma\{x\in \partial\Omega_Q:\, N_{*, \Omega_Q} (v-v(Y_Q))(x)> \eps\alpha/4\}
\\ 
 \le
\left(\frac{4}{\eps\alpha}\right)^2 \int_{\po_Q} N_{*, \Omega_Q} (v-v(Y_Q))(x)^2\,dH^n(x) 
\lesssim C_{\Omega_Q} \frac{16}{(\eps\alpha)^2} \int_{\po_Q}\left(S_{\Omega_Q} v(x)\right)^2\,dH^n(x), 
\end{multline}
where $C_{\Omega_Q}$ depends on $n$ and the CAD of $D$, and so do all the implicit constants. 
Note that
\begin{align}\label{eqn5.10-1}
&\int_{\po_Q}\left(S_{\Omega_Q} v(x)\right)^2\,dH^n(x) 
=
\int_{\po_Q}\dint_{|Y-x|\le (1+\kappa)\dist(Y,\pom_Q)}|\nabla v(Y)|^2\dist(Y,\pom_Q)^{1-n}\,dY\,dH^n(x)
\\ \nonumber
&\qquad\le
\dint_{\Omega_Q} |\nabla v(Y)|^2\dist(Y,\pom_Q)^{1-n}\,H^n(B(Y, (2+\kappa)\dist(Y,\pom_Q))\cap\pom_Q)\,dY
\\ \nonumber
&
\qquad\lesssim
\dint_{\Omega_Q}|\nabla v(Y)|^2\dist (Y, \po_Q)\,dY
\\ \nonumber
&\qquad\le \dint_{T_Q}|\nabla v(Y)|^2\delta(Y)\,dY,
\end{align}
where we have used that $\pom_Q$ is ADR with constant depending on the CAD character of $\Omega_G$, hence ultimately on the CAD character of $D$, and the last inequality follows from  to the fact that $\Omega_Q\subset D\cap B_Q\subset T_Q$ (see $(iii)$ in Corollary~\ref{cIBPLS} and \eqref{eq3.3aab:inte}) and, in particular,  $\dist (Y, \po_Q)\leq \dist (Y, \partial D)=\delta(Y)$ for every $Y\in \Omega_Q$. Note that, \eqref{Fubini-ME-SFE} with $\widetilde{G}=|\nabla v|$ implies
\begin{multline}\label{eqn5.10-2}
\dint_{T_Q}|\nabla v|^2\delta \,dY
\approx 
\int_Q\dint_{\Gamma^Q(x)} |\nabla v|^2\delta^{1-n}\,dY\,d\sigma(x)
\\
=
\int_Q \widehat{S}^{Q} v^2\,d\sigma
\le  M^{\dd}_{Q_0,2}(\widehat{S}^{Q_0}v)(z_Q)^2\,\sigma(Q)
\le
(\gamma\,\alpha)^2 \,\sigma(Q),
\end{multline}
where we have used that $z_Q \in F_{\gamma\,\alpha}$. Thus, \eqref{eqn5.10-0}, \eqref{eqn5.10-1}, and \eqref{eqn5.10-2}, imply
\[
\sigma(E_{(1+\eps)\alpha}\cap F_{\gamma\,\alpha}\cap F_Q) \lesssim  C_{\Omega_Q} (\gamma/\eps)^2 \,\sigma(Q),
\]
which is \eqref{eqn5.4:local}. 

To continue the proof, having at hand \eqref{eqn5.1:local}, an argument analogous to  \eqref{eqgl13} yields \eqref{eqn6.5:local}.  To be specific, we show that taking $\eps>0$ small enough depending on $n$ and the CAD character of $D$ and then taking $\gamma>0$ small enough depending on the same parameters and $\eps$, the estimate \eqref{eqn5.1:local} yields \eqref{eqn6.5:local}. It is here that we use a possibility to pick $\eps>0$ sufficiently small. Indeed, fix any $q>2$, $Q_0\in\dd$ and write $v:=u-u(X_{Q_0})$. Then, much as in \eqref{eqgl13}, for every $N>1$
\begin{align}\label{awffevev}
&\mathrm{I}_N:=\int_0^N q\alpha^q \sigma \{x\in Q_0:\, N_*^{Q_0} v(x)>\alpha\}\, \frac{d\alpha}{\alpha}
\\ \nonumber
&\ = (1+\eps)^q \int_0^{N/(1+\epsilon)} q\alpha^q \sigma \{x\in Q_0:\, N_*^{Q_0} v(x)>(1+\eps)\,\alpha\}\, \frac{d\alpha}{\alpha}
\\[4pt] \nonumber
&\ \leq  (1+\eps)^q\int_0^{N} q\alpha^q \sigma \{x\in Q_0:\, N_*^{Q_0}v (x)>(1+\eps)\alpha, \,M^{\dd}_{Q_0,2}(\widehat{S}^{Q_0}v)(x)\leq \gamma\alpha \}\, \frac{d\alpha}{\alpha} 
\\[4pt]\nonumber
&\ \hskip2cm+ \left(\frac{1+\eps}{\gamma}\right)^q \|M^{\dd}_{Q_0,2}(\widehat{S}^{Q_0}v)\|_{L^q(Q_0)}^q 
\\[4pt]\nonumber
&\ \leq  C_{\gamma, \eps}^*\, \,(1+\eps)^q\int_0^N q\alpha^q \sigma \{x\in Q_0:\, N_*^{Q_0}v (x)>\alpha  \}\, \frac{d\alpha}{\alpha} 
+ 
(1+\eps)^q/\gamma^q\,\|M^{\dd}_{Q_0,2}(\widehat{S}^{Q_0}v)\|_{L^q(Q_0)}^q
\\[4pt]\nonumber
&
\ =
\big(1-\theta+ C\,({\gamma}/{\eps})^2 \sup_{Q\in\dd} C_{\Omega_Q}\big)\,(1+\eps)^q\,\mathrm{I}_N+ (1+\eps)^q/\gamma^q\,\|M^{\dd}_{Q_0,2}(\widehat{S}^{Q_0}v)\|_{L^q(Q_0)}^q.
\end{align}
At this point we first choose $\eps>0$ small enough so that $(1-\theta)\,(1+\eps)^q<1/4$, and once $\eps$ is fixed we take $0<\gamma<c_0 \,\eps/C_0$ small enough so that $C\,({\gamma}/{\eps})^2 \sup_{Q\in\dd} C_{\Omega_Q}\,(1+\eps)^q<1/4$. With these choices and using that $I_N\le N^q\,\sigma(Q_0)<\infty$, we can hide this term with $I_N$on the left-hand side of \eqref{awffevev} to obtain 
\[
\mathrm{I}_N\le 2\,  (1+\eps)^q/\gamma^q\,\|M^{\dd}_{Q_0,2}(\widehat{S}^{Q_0}v)\|_{L^q(Q_0)}^q.
\]
Noting that $\mathrm{I}_N\nearrow \|N_*^{Q_0} v\|_{L^q(Q_0)}^q$ as $N\to\infty$, and using that $M^{\dd}_{Q_0,2}$ is bounded on $L^q(Q_0)$ since $q>2$ we obtain as desired \eqref{eqn6.5:local}.

We next see how to obtain \eqref{eqn6.5:local-balls} using the ideas in Remark \ref{remark:N<S:cubes->balls}. Proceeding as there, once we have fixed $\{\W_Q\}_{Q\in\dd(\partial D)}$ a Whitney-dyadic structure for $D$ with some parameters $\eta$ and $K$. Given $x\in\partial D$ and $0<r<2\diam(\partial D)$, write $\Delta=\Delta(x,r)$ and $B=B(x,r)$ and  consider the case $r\ll\diam(\partial D)$.  Then $\Gamma^r(y)\subset 2B$ for every $y\in\Delta$, if $z\notin 6\,\Delta$ then $\Gamma_{\Omega,1}(z)\cap 2B=\emptyset$ , and $\Gamma_{\Omega,1}(y)\subset\Gamma(y)$ for every $y\in \partial D$. All these and with Remark \ref{remark:COA} imply
\begin{multline*}
\|N^r_{*,D, \kappa} (u-u(X_\Delta^+))\|_{L^q(\Delta)}
\le
\|N_{*,D, \kappa} ((u-u(X_\Delta^+))1_{2B})\|_{L^q(\partial D)}
\\
\lesssim
\|N_{*,D, 1} ((u-u(X_\Delta^+))1_{2B})\|_{L^q(\partial D)}
\le
\|N_{*} ((u-u(X_\Delta^+))1_{2B})\|_{L^q(6\Delta)}.
\end{multline*}
We introduce $\mathcal{D}_\Delta$ as in \eqref{def:D-Delta}. Let $Q\in\mathcal{D}_\Delta$ and note that $\delta(X_Q^+)\approx\ell(Q)\approx r\approx\delta(X_\Delta^+)$ and also $|X_Q^+-X_\Delta^+|\lesssim r$. Hence we can use the Harnack chain condition to find a collection of cubes $I_1,\dots, I_{N}$ with $N\lesssim 1$ so that $X_Q^+\in I_0$, $X_\Delta^+\in I_N$, $\dist(4I_j,\partial D)\approx\ell(I_j)\approx r\approx\ell(Q)$ for $1\le j\le N$, and there exists $X_{j}\in I_j\cap I_{j+1}\neq\emptyset$ for each $1\le j\le N-1$. Write $X_0=X_{Q^+}$, $X_{N}=X_\Delta^+$, and note that for every $1\le j\le N$
\[
\dist(I_j, Q)
\le
|X_j-x_Q|
\le
|X_j-X_Q^+|+|X_Q^+-x_Q|
\lesssim
\sum_{k=0}^{j-1} |X_k-X_{k+1}|+\ell(Q)
\le
\sum_{k=0}^{j-1} \diam(I_{k+1})+\ell(Q)
\approx\ell(Q).
\]
Thus, there exist $\eta'$, and $K'$ depending on $n$, the CAD character of $D$, and fixed parameters $\eta$ and $K$, such that if $\{\W'_Q\}_{Q\in\dd(\partial D)}$ is Whitney-dyadic structure for $D$ with parameters $\eta'$ and $K'$, and if $I\in\W$ with $I\cap 2\,I_j\neq\emptyset$, then $I\in (\W'_Q)^0\subset \W'_Q$. Consequently, $2I\subset U_Q'$ (the Whitney region corresponding to $Q$ with the Whitney-dyadic structure $\{\W'_Q\}_{Q\in\dd(\partial D)}$). All these and \eqref{locbdd} yield
\begin{align*}
|u(X_Q^+)-u(X_\Delta^+)|
&=
|u(X_0)-u(X_{N})|
\le
\sum_{j=0}^{N-1} |u(X_j)-u(X_{j+1})|
\\
&\le
\sum_{j=0}^{N-1} \Big|u(X_j)-\ell(2I_{j+1})^{-n-1}\dint_{2\,I_{j+1}}u\,dY\Big|+\Big|u(X_{j+1})-\ell(2I_{j+1})^{-n-1}\dint_{2\,I_{j+1}}u\,dY\Big|
\\
&\lesssim
\sup_{1\le j\le N} \sup_{X\in I_j}  \Big|u(X)-\ell(2I_{j})^{-n-1}\dint_{2\,I_{j}}u\,dY\Big|
\\
&\lesssim
C_0\sup_{1\le j\le N} \Big( \ell(2I_{j})^{-n-1}\dint_{2\,I_{j}} \Big|u-\ell(2I_{j})^{-n-1}\dint_{2\,I_{j}}u\Big|^2\,dY\Big)^{\frac12}
\\
&\lesssim
C_0\sup_{1\le j\le N} \ell(I_j)\Big( \ell(2I_{j})^{-n-1}\dint_{2\,I_{j}}|\nabla u|^2\,dY\Big)^{\frac12}
\\
&
\approx C_0\sup_{1\le j\le N} \Big(\dint_{2\,I_{j}}|\nabla u|^2\,\delta^{1-n}\,dY\Big)^{\frac12}
\\
&
\le C_0\sup_{1\le j\le N} \Big(\dint_{U_Q'}|\nabla u|^2\,\delta^{1-n}\,dY\Big)^{\frac12}
\\
&
\le C_0\inf_{y\in Q} \A^{\prime,Q}(\nabla u)(y),
\end{align*}
where $\A^{\prime,Q}$ is the local area integral to the cones $\Gamma'(\cdot)$ made up with the Whitney regions $U'_{Q'}$'s. 

On the other hand for each $Q\in\mathcal{D}_\Delta$ we have much as before that $\Gamma(y)\cap2B\subset \Gamma^Q(y)\subset \widehat{\Gamma}^Q(y)\subset B(x_Q, K\ell(Q))\cap D\subset K'B\cap  D$ for every $y\in Q\in\mathcal{D}_\Delta$. Taking $K'$ even larger we also have that $\Gamma^{\prime, Q}(y)\subset D\subset K'B\cap  D$ for every $y\in Q\in\mathcal{D}_\Delta$. 
  Thus all the previous considerations, \eqref{eqn6.5:local} for $q>2$, and Remark \ref{remark:COA} give
\begin{align*}
\|N^r_{*,D, \kappa} (u-u(X_\Delta^+))\|_{L^q(\Delta)}^q
&\lesssim 
\sum_{Q\in\mathcal{D}_\Delta} \|N_{*} ((u-u(X_\Delta^+))1_{2B})\|_{L^q(Q)}^q
\\
&\le
\sum_{Q\in\mathcal{D}_\Delta} \big(\|N_*^Q ((u-u(X_Q^+))1_{2B})\|_{L^q(Q)}^q + |u(X_Q^+)-u(X_\Delta^+)|^q\sigma(Q)\big)
\\
&\lesssim
\sum_{Q\in\mathcal{D}_\Delta} \big(\|\widehat{S}^Q u\|_{L^q(Q)}^q+C_0\inf_{y\in Q} \A^{\prime,Q}(\nabla u)(y)^q\sigma(Q)\big).
\\
&\le
\|\widehat{\A}(|\nabla u|1_{k'B})\|_{L^q(\partial D )}^q+ C_0 \|\A'(|\nabla u|1_{k'B})\|_{L^q(\partial D )}^q
\\
&\lesssim
(1+C_0)\|{\A}_{D,\min\{1,\kappa\}}(|\nabla u|1_{k'B})\|_{L^q(\partial D )}^q
\\
&\lesssim
(1+C_0)\|{\A}_{D,\kappa}^{3 K' r}(|\nabla u)|\|_{L^q(3K'\Delta)}^q
\\
&=
(1+C_0)\|S_{D,\kappa}^{3 K' r} u\|_{L^q(3K'\Delta)}^q
\end{align*}
where we have used that $\Gamma_{\Omega,1}(z)\cap K'B\neq\emptyset$ then $z\in 3K'\Delta$. This proves \eqref{eqn6.5:local-balls}.

To complete the proof we observe that if $\partial D$ is bounded then for any $x\in\partial D$ we have that $\partial D=\Delta(x,3\diam(\partial D)/2)$. Thus  \eqref{eqn6.5} readily follows from \eqref{eqn6.5:local-balls}. On the other hand, to obtain for \eqref{eqn6.5:unbounded} fix $x_0=\in\partial D$ and write $\Delta_R=\Delta(x_0, R)$. Given $\epsilon>0$, there exist $R_\epsilon$ such that $|u(X)|<\eps$ for every $|X-x_0|\ge R_\epsilon$ with $X\in D$. By the Corkscrew condition $B(X_{\Delta_R}^+, R/C)\subset B(x_0, R)$  for some $C>1$ and then $|X_{\Delta_R}^+-x_0|\ge R/C$

Fix $y\in \partial D$ and let $R> 2\max\{C\,R_\epsilon,|y-x_0|\}$ so that $B(x_0, R_\epsilon)\subset B(y, R)$ and $|X_{\Delta_R}^+-x_0|\ge R/C>R_\epsilon$. Hence, $|u(X_{\Delta_R}^+)|<\epsilon$, $|u(Z)|<\epsilon$ for every $D\setminus B(y,R)$, and 
\begin{multline*}
|N_{*,\kappa } u(y)- N_{*,k}^R (u- u(X_{\Delta_R}^+)(y)1_{\Delta_R}(y)|
=
\big|N_{*,\kappa } u(y)- N_{*,k}  \big((u- u(X_{\Delta_R}^+)1_{B(y,R)})\big)(y)\big|
\\
\le
\big|N_{*,\kappa } \big( u- (u- u(X_{\Delta_R}^+)1_{B(y,R)})\big)(y)|
\le
\big|N_{*,\kappa } (u 1_{D\setminus B(y,R)})\big)(y)| + |u(X_{\Delta_R}^+)|<2 \epsilon.
\end{multline*}
This shows that for every $y\in \partial D$
\[
\lim_{R\to\infty}  N_{*,k}^R (u- u(X_{\Delta_R}^+)(y)1_{\Delta_R}(y)=N_{*,\kappa } u(y).
\]
Thus Fatou's Lemma and \eqref{eqn6.5:local-balls} imply for every $q>2$
\begin{multline*}
\int_{\partial D} N_{*,\kappa } u(y)^q\,d\sigma(y)
\le
\liminf_{R\to\infty} \int_{\Delta_R} N_{*,k}^R (u- u(X_{\Delta_R}^+)(y)^q\,d\sigma(y)
\\
\le
C''\liminf_{R\to\infty} \int_{K'\Delta_R} S^{K'R} u(y)^q\,d\sigma(y)
\le
C''\int_{\partial D} S u(y)^q\,d\sigma(y).
\end{multline*}
This completes the proof.
%
%
\end{proof}

%

Our next goal is to extend the previous result so that we have the $N<S$ estimates in all $L^q$. 
We need to introduce some notation. Recall that if $D$ is a CAD, we have constructed a Whitney-dyadic structures $\{\W_Q\}_{Q\in\dd(\partial D)}$ for $D$ with parameters $\eta$ and $K$, see Section \ref{sections:CAD}. In the following result we will need to work with two different Whitney-dyadic structures associated with different parameters  and we need to introduce some notation to distinguish between the associated objects. More specifically, let $\{\W_Q\}_{Q\in\dd(\partial D)}$ (resp.  $\{\W'_Q\}_{Q\in\dd(\partial D)}$) be a Whitney-dyadic structure for $D$  with parameters $\eta\ll 1$ and $K\gg 1$ (resp. ${\eta}'\ll 1$ and ${K}'\gg 1$). Associated with $\{\W_Q\}_{Q\in\dd(\partial D}$ (resp. $\{\W_Q'\}_{Q\in\dd(\partial D)}$) we define the Whitney regions $U_{Q}$, the dyadic cones $\Gamma$ and the local dyadic cones $\Gamma^Q$ (resp. $U_Q'$, $\Gamma'$, $\Gamma^{\prime,Q}$) as in \eqref{eq3.3aa}, \eqref{defcone}, or \eqref{defconetrunc}. With this we define $N_*$, $N_*^Q$, $S$, $S^Q$ (resp. $N_*'$, $N_*^{\prime, Q}$, $S'$, $S^{\prime,Q}$) as in Definition \ref{defsfntmax:dyadic}.

\begin{theorem}\label{theor:N<S:CAD->CAD:all-p} 
Let $D\subset \ree$ be a CAD. Let $u\in W^{1,2}_{\rm loc}(D)\cap C(D)$ be so that \eqref{locbdd} holds and assume that there exists $C_0'>0$ and $p>2$ such that for any cube $I$ with $2I\subset D$, 
\begin{equation}\label{revHol}
 \left(\ell(I)^{-n-1}\dint_{I} |\nabla u|^p \, dX\right)^{\frac1{p}}\le C_0'\left(\ell(I)^{-n-1}\dint_{2I} |\nabla u|^2 \, dX\right)^{\frac12}. 
\end{equation}
Suppose that the $N<S$ estimates are valid on $L^p$ on all bounded chord-arc subdomains  $\Omega\subset D$, that is, for any bounded chord-arc subdomain  $\Omega\subset D$, there holds
\begin{equation}\label{eqn6.5-bis}
\|N_{*,\Omega}(u-u(X_\Omega^+))\|_{L^p(\pom)}\leq C_\Omega \|S_{\Omega} u\|_{L^p(\pom)}. 
\end{equation} 
Here $X_{\Omega}^+$ is any interior corkscrew point of $\Omega$ at the scale of $\diam (\Omega)$, and the constant  $C_{\Omega}$ depends on the CAD character of $\Omega$, the dimension $n$, $p$, the implicit choice of $\kappa$ (the aperture of the cones in $N_{*,\Omega}$ and $S_{\Omega}$), and the implicit corkscrew constant for the point $X_{\Omega}^+$. 

Given $\eta\ll 1$ and $K\gg 1$, consider $\{\W_Q\}_{Q\in\dd(\partial D)}$ a Whitney-dyadic structure for $D$ with parameters $\eta$ and $K$, see Section \ref{sections:CAD}.
Then, there exist $\eta'\ll \eta$ and $K'\gg K$ (depending on $n$, the CAD character of $D$, and  the choice of $\eta, K, \tau$) so that if $\{\W_Q'\}_{Q\in\dd(\partial D)}$ is a  Whitney-dyadic structure for $D$  with parameters $\eta$ and $K$, 
for every $Q\in\dd(\partial D)$,
\begin{equation}\label{eqn6.5:new:local}
\|N_*^Q(u-u(X_Q^+))\|_{L^q(Q)}\leq C' \|S^{\prime,Q} u\|_{L^q(Q)}, \quad \mbox{for all}\quad 0<q<\infty. 
\end{equation} 
where $C'$ depends  on $n$, the CAD character of $D$, $C_0$, $C_0'$, $q$, and  the choice of $\eta, K, \tau$. Here $N_*^Q$ is the non-tangential maximal function associated with the Whitney-dyadic structure $\{\W_Q\}_{\dd(\partial D)}$ while ${S}^{\prime,Q}$ is the square function with the associated with the Whitney-dyadic structure $\{\W_Q'\}_{\dd(\partial D)}$.

As a consequence, for any $x\in\partial D$ and $0<r<2\diam(\partial D)$ there exists $K'$ depending on $n$, the CAD character of $D$ such that for every $\kappa>0$
\begin{equation}\label{eqn6.5:new:local:balls}
\|N_{*,D,\kappa}^r (u-u(X_{\Delta(x,r)}^+)\|_{L^q(\Delta(x,r))}\le C'' \|S^{K'r}_{D,\kappa} u\|_{L^q(\Delta(x,K' r))},  \quad \mbox{for all}\quad 0<q<\infty.
\end{equation}
where $\Delta(x,r)=B(x,r)\cap\pom$, and where $C''$ depends on $q$, $n$, the CAD character of $D$, $C_0$, $C_0'$, and $\kappa$.	
In particular, if $\partial D$ is bounded 
\begin{equation}\label{eqn6.5:new}
\|N_{*,D,\kappa}(u-u(X_D^+))\|_{L^q(\partial D)}\leq C'' \|S_{D,\kappa} u\|_{L^q(\partial D)}, \quad \mbox{for all}\quad 0<q<\infty,
\end{equation} 
and if $\partial D$ is unbounded and $u(X)\to 0$ as $|X|\to\infty$ then 
\begin{equation}\label{eqn6.5:new:unbounded}
\|N_{*,D,\kappa}u\|_{L^q(\partial D)}\leq C'' \|S_{D,\kappa} u\|_{L^q(\partial D)}, \quad \mbox{for all}\quad 0<q<\infty.
\end{equation} 
\end{theorem}

We note that \eqref{revHol} can be relaxed so that it suffices to assume that it holds for $I=2J$ with $J\in\W(D)$. We also note that the same proof allows us to work with 1-sided CAD. That is, if $D$ is a 1-sided CAD and \eqref{eqn6.5-bis} holds for all bounded 1-sided chord-arc subdomains then \eqref{eqn6.5:new:local} and \eqref{eqn6.5:new} hold for $D$. Further details are left to the interested reader.

\begin{proof} 
For starters we fix $\eta\ll 1$ and $K\gg 1$ and let $\{\W_Q\}_{Q\in D}$ be Whitney-dyadic structure for $D$ with parameters $\eta$ and $K$.  Let $\eta'\ll\eta$ be small enough and $K'\gg K$ large enough to be chosen and let $\{\W_Q'\}_{Q\in D}$ be Whitney-dyadic structure for $D$ with parameters $\eta'$ and $K'$. Taking into account \eqref{eq2.whitney2} if $(\eta')^{1/4}\le C^{-1}\eta^{1/2}$ and $(K')^{1/2}\ge C K^{1/2}$, then $\W_Q\subset (\W'_Q)^0\subset {\W}_Q'$ for every $Q\in\dd=\dd(\partial D)$.  Consequently, 
$\widehat{\Gamma}^Q(x)\subset \Gamma_Q'(x)$
and
$\widehat{S}^Q v(x)\le {S}^{\prime,Q} v(x)$  for every $x\in\partial D$, $Q\in\dd$, and $v\in W^{1,2}_{\rm loc}(D)$.

Much as in the proof of Theorem~\ref{theor:N<S:Lip->CAD}, matters can be reduced to showing that for every $\alpha,\gamma,\eps>0$ with $0<\gamma\ll\eps/C_0$
and for any given $Q_0\in\dd$, 
\begin{multline}\label{eqn5.1-bis}
\sigma\{x\in Q_0:\, N_*^{Q_0} (u-u(X_{Q_0}^+))(x)> (1+\eps)\, \alpha, \, S^{\prime, Q_0}u(x)\leq \gamma  \alpha\} \\[4pt]
\leq   C^*_{\gamma, \eps}
\,\sigma\{x\in Q_0:\, N_*^{Q_0}(u-u(X_{Q_0}^+))(x)>\alpha\}, 
\end{multline}
and we will me more specific about the constant $C^*_{\gamma, \eps}$ momentarily.

Let us fix $Q_0\in\dd$ and write $v:=u-u(X_{Q_0}^+)$, with $X_{Q_0}^+$ begin the corkscrew relative to $Q_0$, that is, relative to the surface ball $\Delta_{Q_0}$ (cf. \eqref{cube-ball} and \eqref{cube-ball2}). For every $\alpha>0$ we set
\[
E_\alpha:=\{x\in Q_0:\, N_*^{Q_0} v(x)>\alpha\},
\qquad
\widetilde{F}_\alpha:=\{x\in Q_0:\, S^{\prime, Q_0}v( x)\le \alpha\},
\]
Our goal is to obtain 
\begin{align}\label{eqn5.1:local:new}
\sigma(E_{(1+\eps)\alpha}\cap \widetilde{F}_{\gamma\,\alpha})
\le 
C^*_{\gamma, \eps}
\,\sigma(E_\alpha),
\end{align}
where $C_{\gamma,\eps}^*=C\,({\gamma}/{\eps})^p (1+C_0')\sup_{Q\in\dd,\widetilde{\F}} C_{\Omega_{\widetilde{\F},Q}}\big)$, where the sup runs over all $Q\in\dd$ and all pairwise disjoint families $\widetilde{F}\subset\dd_Q\setminus\{Q\}$. Note that  $\sup_{Q\in\dd,\widetilde{\F}} C_{\Omega_{\widetilde{\F},Q}}\big)<\infty$ and ultimately depends on the CAD character of $D$, since all the sawtooth subdomains $\Omega_{\widetilde{\F},Q}$
are CAD with uniform constants (see Lemma \ref{lemma:CAD-geom}) and our assumption states that $C_{\Omega_{\widetilde{\F},Q}}$ depends on the CAD character of $\Omega_{\widetilde{\F},Q}$.

With this goal in mind we fix  $\alpha,\gamma,\eps>0$. We may assume that the set $E_\alpha\neq\emptyset$, otherwise \eqref{eqn5.1-bis} is trivial. As in the proof of Theorem~\ref{theor:N<S:Lip->CAD}we can find $\F=\{Q_j\}_j\subset\dd_{Q_0}$, a family  of maximal (hence pairwise disjoint) cubes with respect to the property $Q\subset E_\alpha$, so that $E_\alpha=\bigcup_{Q_j\in\F} Q_j$. We then fix $Q\in\F$ and we just need to see that
\begin{equation}\label{eqn5.5-bis}
\sigma(E_{(1+\eps)\alpha}\cap \widetilde{F}_{\gamma\,\alpha}\cap Q)
\le 
C^*_{\gamma, \eps}
\,\sigma(Q),
\end{equation}
assuming that $\gamma <c_0\, \eps$ with a suitably small $c_0$ depending on $n$, the CAD character of $D$. We may assume that $\sigma(E_{(1+\eps)\alpha}\cap \widetilde{F}_{\gamma\,\alpha}\cap Q)>0$ and pick $z_Q\in E_{(1+\eps)\alpha}\cap \widetilde{F}_{\gamma\,\alpha}\cap Q$. We follow the same argument of the proof of Theorem~\ref{theor:N<S:Lip->CAD}taking into account that the set $F_{\gamma\,\alpha}$ needs to be replaced by $\widetilde{F}_{\gamma\,\alpha}$. Here we do not invoke Corollary~\ref{cIBPLS} and we formally take $F_Q=Q$. Also we take $Y_Q=X_Q^+$, the corkscrew relative to $Q$. We replace \eqref{Pdet} by 
\begin{align*}
|v(X_{\widetilde Q}^+)-v(Y_{Q})| = |u(X_{\widetilde Q}^+)-u(Y_{Q})|
\le
C\,C_0 \inf_{z\in Q} \widehat{S}^{Q}u(z)
\le
\widehat{S}^{Q_0}v(z_Q)
\le
S^{\prime, Q_0}v(z_Q)
\le
C\, C_0 \gamma\,\alpha
,
\end{align*}
where we have used \eqref{locbdd:SFE} and the fact that $z_Q\in Q\cap \widetilde{F}_{\gamma\,\alpha}$. Thus, assuming that $\gamma<(2\,C\, C_0)^{-1}\,\eps=:2\,c_0\,\eps$, one arrives at \eqref{eqn5.5} with $\widetilde{F}_{\gamma\,\alpha}$ in place of ${F}_{\gamma\,\alpha}$ and $F_Q=Q$ in the case $E_\alpha\subsetneq Q$. On the other hand, the same estimate holds in the case $Q=Q_0$ since $Y_Q=X_Q^+=X_{Q_0}^+$, hence \eqref{adqeaFWEfw} becomes trivial. 
Thus we have obtained that in either case 
\begin{equation}\label{eqn5.5:new}
E_{(1+\eps)\alpha}\cap \widetilde{F}_{\gamma\,\alpha}\cap Q
\subset
\{x\in \widetilde{F}_{\gamma\,\alpha}\cap  F_Q:  N_*^{Q}(v-v(Y_Q))(x)> \eps\,\alpha/2\}=:E_Q.
\end{equation}
Let $E_Q'$ be an arbitrary closed subset of $E_Q$ with $\sigma(E_Q')>0$. Let $x\in Q\setminus E_Q'$. Since $E_Q'$ is closed  there exists $r_x>0$ such that $B(x,r_x)\cap E_Q'=\emptyset$. Pick any $Q_x\in\dd$ with $Q_x\ni x$ and $\ell(Q_x)\ll \min\{\ell(Q), r_x\}$. Then, $x\in Q_x\cap Q$ and necessarily  $Q_x\subset Q$. Also $Q_x\subset B(x,r_x)$ since $x\in Q_x$ and $\diam(Q_x)\approx \ell(Q_x)\ll r_x$. All in one, $Q_x\subset Q\setminus E_Q'$ and there exists a maximal cube $Q_x^{\rm max}\in\dd_Q$ so that $Q_x^{\rm max}\subset Q\setminus E_Q'$. Note that $Q_x^{\rm max}\subsetneq Q$, otherwise $E_Q'=\emptyset$ which contradicts the fact that $\sigma(E_{(1+\eps)\alpha}\cap \widetilde{F}_{\gamma\,\alpha}\cap Q)>0$. Let $\widetilde{\F}$ be the family of maximal (hence pairwise disjoint) cubes $Q_x^{\rm max}$ with $x\in E_Q'$. Note that $\widetilde{\F}\subset\dd_{Q}\setminus\{Q\}$ and $Q\setminus E_Q'=\cup_{Q'\in \widetilde{\F}} Q'$.

Let $\Omega_\star=\widehat{\Omega}_{\widetilde{\F},Q}$. Let us write $\delta_\star(\cdot)=\dist(\cdot,\pom_\star)$ and $\sigma_\star= H^n|_{\pom_\star}$.  We start with Chebyshev's inequality and the fact that $E_Q'\subset E_Q$
\begin{equation}\label{eq6.31}
\sigma(E_Q')\leq \left(\frac{2}{\eps\alpha}\right)^p \int_{E_Q'} N_*^Q (v-v(Y_Q))(x)^p\,d\sigma(x),
\end{equation}
and now change the cones from those used in $N_*^Q$ (dyadic, with respect to $D$) to the traditional  ones \eqref{conetrad} with respect to $\Omega_\star$. More precisely, let $x\in E_Q'=Q\setminus \cup_{Q'\in\widetilde{\F}} Q'\subset \partial\Omega_\star\cap\partial D$ (see \cite[Proposition 6.1]{HM-I}) and $Y\in\Gamma^Q(x)$. Then $Y\in I^*(\tau)$ with $I\in \W_{Q'}$ with $x\in Q'\in\dd_Q$ and
\[
|Y-x|
\le
\diam(I)+\dist(I,Q')+\diam(Q')
\lesssim
\ell(I).
\]
Note that $Q'\in \dd_{\widetilde{\F},Q}$, otherwise $Q'\subset Q''\in\widetilde{\F}$ and hence $x\in \cup_{Q''\in\widetilde{\F}} Q''=Q\setminus E_Q'$. As a consequence,  $\interior(I^*(2\tau))\subset \interior(U_{Q', 2\tau})=\interior (\widehat{U}_{Q'})\subset\Omega_\star$ and  $\delta_\star(Y)\gtrsim \ell(I)$. All this show that
$|Y-x|\lesssim \delta_\star(Y)$ and this means for some choice of $\kappa$ (depending on the CAD character, and $\eta$ and $K$), $Y\in \Gamma_{\Omega_\star,\kappa}(x)$ (cf. \eqref{conetrad}). 
Thus, with the notation in \eqref{defN*-regular},
\[
N_*^Q (v-v(Y_Q))(x)
=
\sup_{Y\in \Gamma^Q(x)} |v(Y)-v(Y_Q)|
\le
\sup_{Y\in \Gamma_{\Omega_\star,\kappa}(x)} |v(Y)-v(Y_Q)|
=:
N_{*,\Omega_\star,\kappa} (v-v(Y_Q))(x).
\]
and \eqref{eq6.31} leads to
\begin{multline}\label{eq6.32}
\sigma(E_Q')\leq 
\left(\frac{2}{\eps\alpha}\right)^p \int_{E_Q'} N_{*,\Omega_\star,\kappa} (v-v(Y_Q))(x)^p\,d\sigma_\star(x)
\\
\le
\left(\frac{2}{\eps\alpha}\right)^p \int_{\pom_{\star}} N_{*,\Omega_\star,\kappa} (v-v(Y_Q))(x)^p\,d\sigma_\star(x)
\\
\lesssim
\frac{1}{(\eps\alpha)^p} \int_{\pom_{\star}} N_{*,\Omega_\star,\kappa_0} (v-v(Y_Q))(x)^p\,d\sigma_\star(x),
\end{multline}
where the last estimate follows from a change of aperture in the cones (see Remark \ref{remark:COA}). We remark that $Y_Q=X_Q^+$, which is a corkscrew point for $Q$ with respect to $D$. By construction, if we take $I\in \W$ so that $X_Q^+\in I$ then, $I\in \W_Q$. Hence, much as before
\[
\delta(Y_Q)\approx \ell(Q)\approx\ell(I)\lesssim \delta_\star(Y_Q)\le \delta(Y_Q). 
\]
Hence $Y_Q$ is an interior corkscrew of $\Omega_\star$ at the scale $\diam(\Omega_\star)\approx \ell(Q)$ (cf. \eqref{eq3.3aab}). Note that  $v(\cdot)-v(Y_Q)=u(\cdot)- u(Y_Q)$ and $\nabla v=\nabla u$ in $D$. This and the fact that $\Omega_\star$ is a CAD (see Lemma \ref{lemma:CAD-geom}) allow us  to invoke \eqref{eqn6.5-bis},  which together with \eqref{eq6.32}, yields
\begin{align}\label{eq6.34}
\sigma(E_Q')
&\lesssim
\frac{1}{(\eps\alpha)^p} \int_{\po_\star } N_{*,\Omega_\star,\kappa_0} (v-v(Y_Q))(x)^p\,d\sigma_\star(x) 
\\ \nonumber
&\le 
C_{\Omega_\star}\frac{1}{(\eps\alpha)^p}  \int_{\po_\star } S_{\Omega_\star,\kappa_0} v(x)^p\,d\sigma_\star(x)
\\ \nonumber
&\lesssim 
C_{\Omega_\star}\frac{1}{(\eps\alpha)^p}  \int_{\po_\star } S_{\Omega_\star,1} v(x)^p\,d\sigma_\star(x)
\\ \nonumber
&=
C_{\Omega_\star}\frac{1}{(\eps\alpha)^p}  \int_{E_Q'} S_{\Omega_\star,1} v(x)^p\,d\sigma_\star(x)
+
C_{\Omega_\star}\frac{1}{(\eps\alpha)^p}  \int_{\pom_\star\cap D} S_{\Omega_\star,1} v(x)^p\,d\sigma_\star(x)
\\ \nonumber
&=:
C_{\Omega_\star}(\mathrm{I}+\mathrm{II})
.
\end{align}
where the third estimate follows from a change of aperture in the cones (see Remark \ref{remark:COA})) and the first equality from \cite[Propositions 6.1 and 6.3]{HM-I}.

To estimate the previous terms we first need to introduce some notation. Given $x\in\pom_\star$ and for some parameter $N\ge 1$ (depending on the CAD character of $D$) to be chosen later we write
\[
\Gamma_{\Omega_\star,1}^1:=\Gamma_{\Omega_\star,1}\cap\{Y\in\Omega_\star: \delta(Y)\le \ell(Q)\},
\qquad
\Gamma_{\Omega_\star,1}^2:=\Gamma_{\Omega_\star,1}\setminus \Gamma_{\Omega_\star,1}^1.
\]
To proceed let us observe that if $Q'\in\dd_{\widetilde{\F},Q}$, then one can find $y_{Q'}\in Q'\cap E_Q'$: otherwise, $Q'\cap E_Q'=\emptyset$ and by construction there exists $Q''\in\widetilde{F}$ with $Q'\subset Q''$, contradicting the fact that $Q'\in\dd_{\widetilde{\F},Q}$. 

Given $x\in\pom_\star$, let $Y\in \Gamma_{\Omega,1}^2(x)\subset \Omega_\star=\widehat{\Omega}_{\widetilde{\F},Q}$, then $Y\in \widehat{U}_{Q'}$ for some $Q'\in \dd_{\widetilde{\F},Q}$. In particular, $Y\in \widehat{\Gamma}^{Q'}(y_{Q'})\subset \widehat{\Gamma}^{Q_0}(y_{Q'})$. Also, $\ell(Q)<\delta(Y)\approx\ell(Q')\le \ell(Q)$. This means that
\begin{multline}\label{estimate:cone-2}
\dint_{\Gamma_{\Omega_\star,1}^2(x)} |\nabla v|^2\delta^{1-n}\,dY
\le 
\sum_{\substack{Q'\in\dd_{\widetilde{\F},Q}\\ \ell(Q')\approx \ell(Q)}} \dint_{\widehat{\Gamma}^{Q_0}(y_{Q'})} |\nabla v|^2\delta^{1-n}\,dY
=
\sum_{\substack{Q'\in\dd_{\widetilde{\F},Q}\\ \ell(Q')\approx \ell(Q)}} \widehat{S}^{Q_0} v(y_{Q'})^2
\\
\le
\sum_{\substack{Q'\in\dd_{\widetilde{\F},Q}\\ \ell(Q')\approx \ell(Q)}} S^{\prime, Q_0} v(y_{Q'})^2
\le
(\gamma\,\alpha)^2\,\#\{Q'\in\dd_{Q}, \ell(Q')\approx \ell(Q)\}
\lesssim 
(\gamma\,\alpha)^2.
\end{multline}

We next turn to estimate $\mathrm{I}$. Let $x\in E_Q'\subset \partial\Omega_\star\cap\partial D$ (see \cite[Proposition 6.1]{HM-I}). Note first that if $Y\in \Gamma_{\Omega,1}(x)$, then  
$
\delta(Y)\le |Y-x|\le 2\delta_\star(Y)
$
and thus \eqref{estimate:cone-2} gives
\begin{multline}\label{egwget}
\widehat{S}_{\Omega_\star,1} v(x)^2
=
\dint_{\Gamma_{\Omega_\star,1}(x)} |\nabla v|^2\delta_{\star}^{1-n}\,dY
\lesssim
\dint_{\Gamma_{\Omega_\star,1}^1(x)} |\nabla v|^2\delta^{1-n}\,dY
+
\dint_{\Gamma_{\Omega_\star,1}^2(x)} |\nabla v|^2\delta^{1-n}\,dY
\\
\lesssim
\dint_{\Gamma_{\Omega_\star,1}^1(x)} |\nabla v|^2\delta^{1-n}\,dY
+
(\gamma\,\alpha)^2.
\end{multline}

Given $Y\in \Gamma_{\Omega,1}^1(x)\subset \Omega_\star\subset D$, one has $Y\in I_Y$ for some $I_Y\in\W$. Pick then $Q_Y\ni x$ with $\ell(I_Y)=\ell(Q_Y)$ and note that by \eqref{Whintey-4I}, and since $K$ is large enough, 
\[
\dist(Q_Y,I_Y)
\le
|x-Y|
\le 
2\,\dist(Y,\pom_\star)
\le 
2\,\delta(Y)
\le 
82\,\diam(I_Y)
=
82\sqrt{n}\ell(Q_Y)
\le
K^{1/2}\ell(Q_Y).
\]
This means that $I_Y\in \W_{Q_Y}^0\subset \W_{Q_Y}$. Besides, 
\[
\ell(Q_Y)=\ell(I_Y)\le \dist(I,\partial D)\le\delta(Y)\le\ell(Q)/N\le \ell(Q), 
\]
this together with the fact that $x\in Q_Y\cap Q$ gives that $Q_Y\in \dd_Q$. Hence, $Y\in I_Y\subset U_{Q_Y}\subset  \Gamma^{Q}(x)\subset\widehat{\Gamma}^{Q_0}(x)$and eventually
\begin{equation}\label{est:g1}
\dint_{\Gamma_{\Omega_\star,1}^1(x)} |\nabla v|^2\delta^{1-n}\,dY
\le 
\widehat{S}^{Q_0} v(x)^2
\le
S^{\prime, Q_0} v(x)^2
\le
(\gamma\,\alpha)^2,
\end{equation}
since $x\in E_Q'\subset E_Q$. This  and \eqref{egwget} imply that 
\begin{equation}\label{est-I-aaaa}
\mathrm{I}
\lesssim
\Big(\frac{\gamma\,\alpha}{\eps\alpha}\Big)^p\sigma(E'_Q)
\le
\Big(\frac{\gamma}{\eps}\Big)^p\sigma(Q).
\end{equation}

Turning to $II$, we start with the following 
\begin{claim}\label{c6.36} 
We can take choose $\eta'$ small enough and $K'$ large enough (depending on $n$, the CAD character of $D$, and  the choice of $\eta, K, \tau$) such that for any $x\in \pom_\star\cap D$ there exists $y_x\in E_Q'$ such that if $J\in \W$ verifies  $4J\cap \Gamma_{\Omega_\star,1}^1(x)\neq\emptyset$, then $4J\subset  \Gamma^{\prime,Q}(y_x)$ and, in particular, $\Gamma_{\Omega_\star,1}^1(x) \subset \Gamma^{\prime,Q}(y_x)$.
\end{claim}

\begin{proof}
Fix $x\in\pom_\star\cap D$. Then $x\in\partial\widehat{I}$ where $\widehat{I}:=I^*(2\tau)=(1+2\tau)I$ with $I\in\W_{Q'}$, $Q'\in\dd_{\widetilde{\F},Q}$. In this scenario we observed before that we can find pick $y_x=y_{Q'}\cap E_{Q'}\cap Q'$. 

Let $Y\in 4J\cap \Gamma_{\Omega_\star,1}^1(x)$ and assume first that $|Y-x|\ge \ell(I)\,\tau/(2\sqrt{n})$. Pick $Q''\in \dd$ with $y_{Q'}\in Q''$ and $\delta_\star(Y)/2<\ell(Q'')\le \delta_\star(Y)$. Note that
$\ell(Q'') \le \delta_\star(Y)\le \delta(Y) \le \ell(Q)$ since $Y\in  \Gamma_{\Omega_\star,1}^1(x)$ and hence $Q''\subset Q$. 
Then, choosing $N$ large enough, depending on $n$ and the CAD character of $D$ (recall that $\eta$, $K$ have been already fixed depending also on the CAD character of $D$),
\begin{multline*}
\dist(4J,Q'')
\le
|Y-y_{Q'}|
\le
|Y-x|+\diam(\widehat{I})+ \dist(I,Q')+\diam(Q')
\\
\le
|Y-x|+ C K^{1/2}\eta^{-1/2}\ell(I)
\le
\big(1+ C K^{1/2}\eta^{-1/2}\tau^{-1}\big)|Y-x|
\\
\le
N\,|Y-x|/2 
\le
N\,\delta_\star(Y)
\le N\,\ell(Q''),
\end{multline*}
where we have used \eqref{eq2.whitney2}. Note also that by \eqref{Whintey-4I}
\[
\ell(Q'')\le \delta_\star(Y)\le \delta(Y) 
\le
\diam(4J)+\dist(4J,\partial D)
\le 
41\diam(J)
=
41\sqrt{n}\,\ell(J)
\]
and
\[
\ell(J)
\le 
\dist(4J,\partial D)/\sqrt{n}
\le
\dist(4J,Q'')/\sqrt{n} 
\le N\,\ell(Q'').
\]
All in one we have obtained that
\[
N^{-1}\ell(J)\le \ell(Q'')\le 41\sqrt{n}\,\ell(J), \qquad
\dist(4J,Q'')\le N\ell(Q'').
\]
If we now take $J'\in \W$ with $J'\cap 4J\neq\emptyset$, then the properties of the Whitney cubes guarantee that $\ell(J')\approx\ell(J)$ and hence the previous estimates easily extend to $J'$. This means that choosing $\eta'$ smaller and $K'$ larger (depending on the CAD character of $D$), we have that $J'\in (\W'_{Q''})^0\subset \W'_{Q''}$.
Since $y_{Q'}\in Q''$ we then have that
\[
4J\subset \bigcup_{J'\in\W; J'\cap 4J\neq\emptyset} J'\subset \bigcup_{y_{Q'}\in Q'''\in\dd_Q} \Big(\bigcup_{J'\in {\W}'_{Q'''}} I^*(\tau)\Big)
=
\bigcup_{y_{Q'}\in Q'''\in\dd_Q} U_{Q'''}'=\Gamma^{\prime, Q}(y_{Q'}).
\]

Consider finally the case on which $Y\in 4J\cap \Gamma_{\Omega_\star,1}^1(x)$ satisfies  $|Y-x|<\ell(I)\,\tau/(2\sqrt{n})$  so that $Y\in (1+2\tau I)=I^*(2\tau)=:\widehat{I}$ and $\ell(I)\approx \delta(Y)\approx \ell(J)$. Note then that if $J'\cap 4J\neq\emptyset$ we have $\ell(J')\approx\ell(J)\approx\ell(I)$. Since $I\in\W_{Q'}$, $Q'\in\dd_{\widetilde{\F},Q}$ we have by \eqref{eq2.whitney2} that
\[
\eta^{1/2}\ell(Q')\lesssim \ell(I)\approx \ell(J')\lesssim K^{1/2}\ell(Q),
\]
and
\[
\dist(J',Q)
\le
\diam(J')+\diam(4J)+|Y-x|+\diam(\widehat{I})+\dist(I,Q)
\lesssim
\ell(I)+\dist(I,Q)
\lesssim 
K^{1/2} \ell(Q).
\]
Thus, by taking $\eta'$ smaller and $K'$ bigger, if needed, we obtain that $J'\in  (\W'_{Q'})^0$. Much as before 
the fact that $y_{Q'}\in Q'$ yields
\[
4J\subset \bigcup_{J'\in\W; J'\cap 4J\neq\emptyset} J'\subset \bigcup_{y_{Q'}\in Q'''\in\dd_Q} \Big(\bigcup_{J'\in {\W}'_{Q'''}} I^*(\tau)\Big)
=
\bigcup_{y_{Q'}\in Q'''\in\dd_Q} U_{Q'''}'=\Gamma^{\prime, Q}(y_{Q'}).
\]
This completes the proof.
\end{proof}

Let us now get back to the proof, specifically, to the estimate for $II$ in \eqref{eq6.34}. Let $\varpi>0$ be small enough to be chosen and set for every $x\in\pom_\star\cap D$
\[
\Gamma_{\Omega_\star,1}^3(x)=\{Y\in \Gamma_{\Omega_\star,1}^1(x): \delta_\star(Y)\ge \varpi \delta(Y)\},
\qquad
\Gamma_{\Omega_\star,1}^4(x)=\{Y\in \Gamma_{\Omega_\star,1}^2(x): \delta_\star(Y)\ge \varpi \delta(Y)\},
\]
and
\[
\Gamma_{\Omega_\star,1}^5(x)=\{Y\in \Gamma_{\Omega_\star,1}(x): \delta_\star(Y)<\varpi \delta(Y)\}
.
\]
Thus
\begin{equation}\label{S:g1-5}
S_{\Omega_\star,1} v(x)^2
=\sum_{k=3}^5 
\dint_{\Gamma_{\Omega_\star,1}^k(x)} |\nabla v|^2\delta_{\star}^{1-n}\,dY
=: 
\sum_{k=3}^5 g_k(x)^2.
\end{equation}
Note that for $x\in\pom_\star\cap D$ invoking Claim~\ref{c6.36} we obtain
\[
g_3(x)^2
\le
\varpi^{1-n} \dint_{\Gamma_{\Omega_\star,1}^1(x)} |\nabla v|^2\delta^{1-n}\,dY
\le
\varpi^{1-n} \dint_{\Gamma^{\prime,Q}(y_x) } |\nabla v|^2\delta^{1-n}\,dY
=
\varpi^{1-n} S^{\prime, Q}u(y_x)^2
\le
(\gamma\alpha^2).
\]
Analogously, by \eqref{estimate:cone-2}
\[
g_4(x)^2
\le
\varpi^{1-n} \dint_{\Gamma_{\Omega_\star,2}^1(x)} |\nabla v|^2\delta^{1-n}\,dY
\le
\varpi^{1-n} \dint_{\Gamma^{\prime,Q}(x) } |\nabla v|^2\delta^{1-n}\,dY
\lesssim
(\gamma\alpha)^2.
\]
As a result,
\begin{multline}\label{est:g:3y5}
\int_{\pom_\star\cap D} \big(g_3^p+ g_4^p\big)\,d\sigma_\star
\lesssim 
\varpi^{(1-n)\,p/2}(\gamma\alpha)^p\sigma_\star(\pom_\star)
\\
\lesssim 
\varpi^{(1-n)\,p/2}(\gamma\alpha)^p\ell(Q)^n
\approx 
\varpi^{(1-n)\,p/2}(\gamma\alpha)^p\sigma(Q),
\end{multline}
where we have used that $\pom_\star$ is ADR with $\diam(\pom_\star)\lesssim\ell(Q)$ (cf. \eqref{eq3.3aab}).

We next consider $g_5$. Set $\W_\star=\{I\in\W: I\cap \pom_\star\neq\emptyset\}$ and note that $\partial\Omega_\star\cap D\subset\bigcup_{I\in \W_\star} I$. 
For every $x\in\pom_\star\cap D$ we then have that $x\in I_x\in \W_\star$ and also that $x\in\partial \widehat{J}_x$ with $J_x\in\W_{Q_x}$, $Q_x\in\dd_{\widetilde{\F}, Q}$. If $Y\in \Gamma_{\Omega_\star,1}^5(x)$ and $\varpi<1/4$, then
\[
\delta(Y)
\le 
|Y-x|+\delta(x)
\le 
2\delta_\star(Y)+\delta(x)
<
2\varpi\delta(Y)+\delta(x)
<
\delta(Y)/2+\delta(x).
\]
This and  \eqref{Whintey-4I} yield
\[
\delta(Y)
\le 
2\delta(x)
\le 
2(\diam(4 J_x)+\dist(4J_x,\partial D))
<100 \diam(J_x)
\]
and, $\varpi$ small enough,
\[
|Y-x|
\le
2\delta_\star(Y)
\le
2\varpi\delta(Y)
<
200\varpi\diam(J_x)
<
\tau\ell(J_x)/8.
\]
Recalling that $\widehat{J}_x:=J_x^*(2\tau)$ with $\tau\le \tau_0\le 2^{-4}$ it follows that $Y\in J_x^*(7\tau/4)\subset 2\, J_x$ and also $Y\in B(x,\ell(J_x))$. Hence, easy calculations lead to 
\begin{align*}
\dint_{\Gamma_{\Omega_\star,1}^5(x)} \delta_\star^{\frac{p}{p-2}-n}\,dY
\le
\max\{2^{\frac{p}{p-2}-n},1\}
\dint_{B(x,\ell(J_x))} |x-Y|^{\frac{p}{p-2}-n}\,dY
\lesssim
\ell(J_x)^{2\frac{p-1}{p-2}}\approx \ell(I_x)^{2\frac{p-1}{p-2}}. 
\end{align*}
Using Hölder's inequality with $p/2$ we arrive at
\begin{multline*}
g_5(x)
=
\Big(  \dint_{\Gamma_{\Omega_\star,1}^5(x)} |\nabla v|^2\delta_\star^{1-n}\,dY\Big)^{\frac{1}2}
\le
\Big(  \dint_{\Gamma_{\Omega_\star,1}^5(x)} \delta_\star^{\frac{p}{p-2}-n}\,dY\Big)^{\frac{p-2}{2\,p}}\,
\Big(  \dint_{\Gamma_{\Omega_\star,1}^5(x)} |\nabla v|^p\delta_\star^{-n}\,dY\Big)^{\frac{1}p}
\\
\lesssim
\ell(I_x)^{\frac{p-1}{p}}\,\Big(  \dint_{2J_x\cap B(x,2\,\delta_\star(x))\cap\Omega_\star} |\nabla v|^p\delta_\star^{-n}\,dY\Big)^{\frac{1}p}.
\end{multline*}
Next, for every $I\in\W_\star$ we set
\[
\W_\star^I:=\{J\in\W: J=J_x \mbox{ for some }x\in \pom_\star\cap I,\ 2J_x\cap \Gamma_{\Omega_\star,1}(x)\neq\emptyset \}
\]
and obtain
\begin{align}\label{Wafraweefr0}
\int_{\pom_\star\cap D} g_5^p\,d\sigma_\star
&\le
\sum_{I\in\W_\star} \int_{\pom_\star\cap I} g_5^p\,d\sigma_\star
\\ \nonumber
&\le
\sum_{I\in\W_\star} \ell(I)^{p-1} \int_{\pom_\star\cap I}  \dint_{2 {J}_x\cap B(x,2\,\delta_\star(x))\cap\Omega_\star} |\nabla v(Y)|^p\delta_\star(Y)^{-n}\,dY d\sigma_\star(x)
\\ \nonumber
&
\le
\sum_{I\in\W_\star} \ell(I)^{p-1} \sum_{J\in \W_\star^I}  \dint_{2{J}\cap\Omega_\star} |\nabla v(Y)|^p\delta_\star(Y)^{-n}\,\sigma_\star(\pom_\star\cap B(Y,2\delta_\star(x)))\,dY
\\ \nonumber
&
\lesssim
\sum_{I\in\W_\star} \ell(I)^{p-1}  \sum_{J\in \W_\star^I} \dint_{2{J}} |\nabla v(Y)|^p\,dY
\\ \nonumber
&
\lesssim
C_0'
\sum_{I\in\W_\star} \ell(I)^{p-1} \sum_{J\in \W_\star^I} \ell(J)^{(n+1)\frac{2-p}{2}}\Big(\dint_{4{J}} |\nabla v(Y)|^2\,dY\Big)^{\frac{p}{2}} 
\\ \nonumber
&\approx
C_0'\sum_{I\in\W_\star} \ell(I)^{n}  \sum_{J\in \W_\star^I} \Big(\dint_{4{J}} |\nabla v(Y)|^2\,\delta(Y)^{1-n}dY\Big)^{\frac{p}{2}}, 
\end{align}
where we have used that $\partial\Omega_\star$ is ADR (cf. \cite[Lemma 3.61]{HM-I}), 
\eqref{revHol} (since $4J\subset D$ by  \eqref{Whintey-4I}), that $\ell(J_x)\approx \ell(I)$ since $x\in I\cap\partial \widehat{J}_x$ (hence $I\cap J\neq\emptyset$), and finally that $\delta(\cdot)\approx \ell(J)$ in $4J$ by \eqref{Whintey-4I}.

Suppose next that $I\in\W_\star$ with $\ell(I)\ll \ell(Q)$. Note that if $J=J_x$ with $x\in\pom_\star\cap I$ then $x\in\partial \widehat{J}_x\cap I$, hence $\ell(J_x)\approx \ell(I)\ll\diam(I)$ and $4 J_x\subset\{Y\in D:\delta(Y)<\ell(Q)\}$. Thus, if $2J_x\cap \Gamma_{\Omega_\star,1}(x)\neq\emptyset$ necessarily $2J_x\cap \Gamma_{\Omega_\star,1}^1(x)\neq\emptyset$. We can then invoke Claim \ref{c6.36} with $J=J_x$ to find $y_x\in E_Q'$ so that
\begin{multline}\label{Wafraweefr1}
\sum_{J\in \W_\star^I} \Big(\dint_{4{J}} |\nabla v|^2\,\delta^{1-n}dY\Big)^{\frac{p}{2}}
\le
\Big(\dint_{\Gamma^{\prime,Q}(y_x)} |\nabla v|^2\,\delta^{1-n}dY\Big)^{\frac{p}{2}}
\#\{J\in\W: \partial\widehat{J}\cap I\neq\emptyset\}
\\
\lesssim
 S^{\prime, Q}(y_x)^p\le  S^{\prime, Q_0}(y_x)^p
 \le
 (\gamma\alpha)^p.
\end{multline}

Consider next the case $I\in\W_\star$ with $\ell(I)\gtrsim \ell(Q)$. For every $J\in\W_\star^I$ we have that $J=J_x$ for some $x\in \pom_\star\cap I$ and there exists $Z\in 2J\cap\Omega_\star$. As such $J\in\W_{Q_x}$ for some $Q_x\in\dd_{\widetilde{\F}, Q}$. In particular, $\ell(Q)\lesssim \ell(I)\approx \ell(J)\approx \ell(Q_x)\le\ell(Q)$. Take then an arbitrary $Y\in 4{J}\cap\Omega_\star$. Since $Z\in 2{J}$, one has $\delta(Y)\approx\ell(J)\approx\ell(Q)$. Also, $Z\in\Omega_\star=\widehat{\Omega}_{\widetilde{\F},Q}$, then $Z\in \widehat{U}_{Q'}$ for some $Q'\in \dd_{\widetilde{\F},Q}$ and, as observed above, the latter implies that one can find $y_{Q'}\in Q'\cap E_Q'$. We claim that $4J\subset \Gamma^{\prime,Q}(y_x)$. To see this let $Y\in 4J\subset D$ and take $I_Y\in\W$ with $Y\in I_Y$. Note that by 
\eqref{Whintey-4I} and \eqref{eq2.whitney2} $\ell(I_Y)\approx\delta(Y)\approx \ell(J)\approx\ell(Q)$ and 
\[
\dist(I_Y,Q)\le \dist(Y,Q)
\le
\diam(4J)+\dist(J,Q_x)
\lesssim
\ell(Q)+\ell(Q_x)
\approx\ell(Q).
\]
Thus taking $\eta'$ smaller and $K'$ larger if needed ((depending on $n$, the CAD character of $D$, and  the choice of $\eta, K, \tau$) we can assure that $I_Y\in(\W_Q')^0\subset \W'_Q$ and since $y_{Q'}\in Q'\subset Q$ we conclude that $Z\in \Gamma^{\prime,Q}(y_x)$ as desired. All these give an estimate similar to \eqref{estimate:cone-2}
\begin{multline}\label{Wafraweefr2}
 \sum_{J\in \W_\star^I} \Big(\dint_{4{J}} |\nabla v(Y)|^2\,\delta(Y)^{1-n}dY\Big)^{\frac{p}{2}}
\le 
 \#\{J\in\W: \partial\widehat{J}\cap I\neq\emptyset\}
 \Big(\dint_{\Gamma^{\prime,Q}(y_x)} |\nabla v|^2\delta^{1-n}\,dY\Big)^{\frac{p}{2}}
 \\ 
\qquad\lesssim
 S^{\prime, Q} v(y_{Q'})
 \le  S^{\prime, Q_0} v(y_{Q'})^p
 \le
 ()\gamma\,\alpha)^p.
\end{multline}

We finally combine \eqref{Wafraweefr0}, \eqref{Wafraweefr1}, and \eqref{Wafraweefr2}, to obtain 
\begin{align}\label{qfawfr}
\int_{\pom_\star\cap D} g_5^p\,d\sigma_\star
\lesssim
C_0'(\gamma\,\alpha)^p\sum_{I\in\W_\star} \ell(I)^{n} . 
\end{align}
To complete the proof we estimate the sum in the right-hand side. For every $I\in \W_\star$ pick $Z_I\in \pom_\star\cap I$ so that $\ell(I)\approx \delta(Z_I)$ and let $\Delta_\star^I:=B(Z_I,\delta(Z_I)/2)\cap\pom_\star$, which is surface ball with respect to $\Omega_\star$. The fact that $Z_I\in  \pom_\star\subset \cap D$ implies that there exists $I'\in \W_{Q'}$ with $Q'\in\dd_{\widetilde{\F},Q}$ and $Z_I\in\partial\widehat{I}$. Then,
$\ell(I)\approx\delta(Z_I)\approx \ell(I')\approx\ell(Q')\le\ell (Q)$ by \eqref{Whintey-4I} and \eqref{eq2.whitney2}). Note that $Q\in\dd_{\widetilde{\F},Q}$, hence $U_{Q}\subset \Omega_\star$. Pick $I_Q\in\W_Q$ (which is non-empty by construction) and note that $\ell(I_Q)\approx \ell(Q)$ by \eqref{eq2.whitney2} and $I_Q\subset U_Q\subset \Omega_\star$. Hence $\ell(Q)\approx\diam(I_Q)\le \diam(\Omega_\star)\lesssim \ell(Q)$ by \eqref{eq3.3aab}. All these show that $\delta(Z_I)\lesssim\diam(\pom_\star)$. Suppose next that $\Delta_\star^I\cap \Delta_\star^J\neq\emptyset$ for some $I,J\in\W_\star$ and let $Y$ belong to that intersection. Assume for instance that $\ell(I)\le \ell(J)$ and note that
\[
\delta(Z_J)
\le
|Z_J-Y|+|Y-Z_I|+\delta(Z_I)
\le
\frac12\delta(Z_J)+\frac32\delta(Z_I).
\]
Hence, $\ell(J)\approx\delta(Z_J)\le 3\delta (Z_I)\approx\ell(I)\le \ell(J)$ and
\[
\dist(I,J)\le |Z_I-Z_J|
\le
|Z_J-Y|+|Y-Z_I|
\le
\frac12\delta(Z_J)+\frac12\delta(Z_I)
\approx
\ell(I)+\ell(J)
\approx
\ell(I)
\approx
\ell(J).
\]
As a consequence, the family $\{\Delta_\star^I\}_{I\in\W_\star}$ has bounded overlap and therefore 
\[
\sum_{I\in\W_\star} \ell(I)^{n} 
\approx
\sum_{I\in\W_\star} \sigma_\star(\Delta_\star^I)
\lesssim
\sigma_\star\Big(\bigcup_{I\in\W_\star} \Delta_\star^I\Big)
\le
\sigma_\star(\pom_\star)
\lesssim
\diam(\pom_\star)^n
\approx
\ell(Q)^n
\approx
\sigma(Q)
, 
\]
where we have used that $\pom_\star$ is ADR (cf. \cite[Lemma 3.61]{HM-I}). This and 
\eqref{qfawfr} eventually yield
\begin{align*}
\int_{\pom_\star\cap D} g_5^p\,d\sigma_\star
\lesssim
C_0'(\gamma\,\alpha)^p\sigma(Q).
\end{align*}
This, \eqref{eq6.34}, \eqref{S:g1-5}, and \eqref{est:g:3y5} give
\begin{align*}
\textrm{II}
=
\frac{1}{(\eps\alpha)^p}  \int_{\pom_\star\cap D} S_{\Omega_\star,1} v^p\,d\sigma_\star
\lesssim
\frac{1}{(\eps\alpha)^p}  \int_{\pom_\star\cap D} (g_3^p+g_4^p+g_5^p)\,d\sigma_\star
\lesssim
(1+C_0')\,\Big(\frac{\gamma}{\eps}\Big)^p\sigma(Q).
\end{align*}
We next combine this with \eqref{eq6.34} and \eqref{est-I-aaaa} to arrive at 
\[
\sigma(E_Q')
\lesssim
C_{\Omega_\star}\,(1+C_0')\,\Big(\frac{\gamma}{\eps}\Big)^p\sigma(Q).
\]
Recalling that Let $E_Q'$ be an arbitrary closed subset of $E_Q$ with $\sigma(E_Q')>0$, by inner regularity of the Hausdorff measure,  we therefore obtain that
\[
\sigma(E_{(1+\eps)\alpha}\cap \widetilde{F}_{\gamma\,\alpha}\cap Q)\le 
\sigma(E_Q)
\lesssim
C_{\Omega_\star}\,(1+C_0')\,\Big(\frac{\gamma}{\eps}\Big)^p\sigma(Q).
\]
We have then show \eqref{eqn5.1:local:new} which in turn implies \eqref{eqn5.1-bis}. With the latter estimate in hand and for any $0<q<\infty$, we proceed as in \eqref{awffevev}:
\begin{align}\label{awffevev*}
&\mathrm{I}_N:=\int_0^N q\alpha^q \sigma \{x\in Q_0:\, N_*^{Q_0} v(x)>\alpha\}\, \frac{d\alpha}{\alpha}
\\ \nonumber
&\ = (1+\eps)^q \int_0^{N/(1+\epsilon)} q\alpha^q \sigma \{x\in Q_0:\, N_*^{Q_0} v(x)>(1+\eps)\,\alpha\}\, \frac{d\alpha}{\alpha}
\\[4pt] \nonumber
&\ \leq  (1+\eps)^q\int_0^{N} q\alpha^q \sigma \{x\in Q_0:\, N_*^{Q_0}v (x)>(1+\eps)\alpha, \,S^{\prime, Q_0}v(x)\leq \gamma\alpha \}\, \frac{d\alpha}{\alpha} 
\\[4pt]\nonumber
&\ \hskip2cm+ \left(\frac{1+\eps}{\gamma}\right)^q \|S^{\prime, Q_0}v\|_{L^q(Q_0)}^q 
\\[4pt]\nonumber
&\ \leq  C_{\gamma, \eps}^*\, \,(1+\eps)^q\int_0^N q\alpha^q \sigma \{x\in Q_0:\, N_*^{Q_0}v (x)>\alpha  \}\, \frac{d\alpha}{\alpha} 
+ 
(1+\eps)^q/\gamma^q\,\|S^{\prime, Q_0}v\|_{L^q(Q_0)}^q
\\[4pt]\nonumber
&
\ =
C\,({\gamma}/{\eps})^p (1+C_0')\big(\sup_{Q\in\dd,\widetilde{\F}} C_{\Omega_{\widetilde{\F},Q}}\big)
\,(1+\eps)^q\,\mathrm{I}_N+ (1+\eps)^q/\gamma^q\,\|S^{\prime, Q_0}v)\|_{L^q(Q_0)}^q.
\end{align}
At this point we first choose $\eps=1$ and next take $0<\gamma<c_0 \,\eps/C_0$ small enough so that $C\,{\gamma}^p(1+C_0')\sup_{Q\in\dd} C_{\Omega_Q}\,2^q<1/1$. With these choices and using that $I_N\le N^q\,\sigma(Q_0)<\infty$, we can hide this term on the left-hand side of \eqref{awffevev*} to obtain 
\[
\mathrm{I}_N\le 2\,  (1+\eps)^q/\gamma^q\,\|S^{\prime, Q_0}v\|_{L^q(Q_0)}^q.
\]
Noting that $\mathrm{I}_N\nearrow \|N_*^{Q_0} v\|_{L^q(Q_0)}^q$ as $N\to\infty$ we obtain as desired \eqref{eqn6.5:new:local}.

From \eqref{eqn6.5:new:local} one can obtain \eqref{eqn6.5:new:local:balls}, and hence \eqref{eqn6.5:new} and \eqref{eqn6.5:new:unbounded} much as in the proof of Theorem \ref{theor:N<S:Lip->CAD}and we omit details. 
\end{proof}

Combining Theorems \ref{theor:N<S:Lip->CAD} and \ref{theor:N<S:CAD->CAD:all-p} we can obtain the following:

\begin{corollary}\label{corol: N<S:Lip->CAD}
	Let $D\subset \ree$ be a  CAD. Let $u\in W^{1,2}_{\rm loc}(D)\cap C(D)$ so that \eqref{locbdd} and \eqref{revHol} hold for some $p>2$. Suppose that the $N<S$ estimates are valid on $L^2$ on all bounded Lipschitz subdomains $\Omega\subset D$ (see \eqref{eqn5.9} in Theorem \ref{theor:N<S:Lip->CAD}). 
	Then \eqref{eqn6.5:new:local}--\eqref{eqn6.5:new:unbounded} hold.
\end{corollary}

\begin{proof}
Let $\Omega\subset D$ be an arbitrary bounded CAD. Since any bounded Lipschitz sub-domain of $\Omega$ is also a subdomain of $D$ we can apply Theorem \ref{theor:N<S:Lip->CAD}to obtain \eqref{eqn6.5} for $\Omega$ and for every $q>2$. That is, we have the $N<S$ estimates are valid on all bounded chord-arc subdomains  $\Omega\subset D$ for $q=p>2$. Hence, Theorem \ref{theor:N<S:CAD->CAD:all-p} applies to obtain the desired conclusions. 
\end{proof}

\section{From $N<S$ bounds on chord-arc domains to $\eps$-approximability in the complement of a UR set}\label{eps-approx}

Recall the definition of $\eps$-approximability (Definition~\ref{def1.3}). 
The second main result in \cite{HMM2}, stated there for harmonic functions but proved in full generality, can be formulated as follows. 

\begin{theorem}\label{theor:eps-approx} Let $E\subset \ree$ be an $n$-dimensional UR set, $\ree\setminus E$, and
suppose that $u\in W^{1,2}_{\rm loc}(\ree\setminus E)\cap C(\ree\setminus E)\cap L^\infty(\ree\setminus E)$ is such 
that for any cube $I$ with $2 I\subset \ree\setminus E$
\begin{equation}\label{oscbdd}
\sup_{X, Y\in I}|u(X)-u(Y)|\leq C_0 \left(\ell(I)^{1-n}\dint_{2I} |\nabla u|^2 \, dX\right)^{\frac12}
\end{equation}
and
\[
\|\nabla u\|_{\C(\ree\setminus E)}\le C_0' \|u\|_{L^\infty(\ree\setminus E)}
\]
Assume, in addition, that $N<S$ estimates are valid on $L^2$ on all bounded chord-arc subdomains  $\Omega\subset \ree\setminus E$, that is, for any bounded chord-arc subdomain  $\Omega\subset \ree\setminus E$, there holds
\begin{equation}\label{eqe7.2-bis}
\|N_{*,\Omega}(u-u(X_\Omega^+))\|_{L^2(\pom)}\leq C_\Omega \|S_{\Omega} u\|_{L^2(\pom)}. 
\end{equation} 
Here $X_{\Omega}^+$ is any interior corkscrew point of $\Omega$ at the scale of $\diam (\Omega)$, and the constant  $C_{\Omega}$ depends on the CAD character of $\Omega$, the dimension $n$, $p$, the implicit choice of $\kappa$ (the aperture of the cones in $N_{*,\Omega}$ and $S_{\Omega}$), and the implicit corkscrew constant for the point $X_{\Omega}^+$ . 
Then $u$ is $\eps$-approximable on $ \ree\setminus E$, with the implicit constants depending only on $n$, the UR character of $E$, $C_0$, and $C_0'$.
\end{theorem}

Strictly speaking, the previous result was proved in \cite[Section 5]{HMM2} for harmonic functions but it was observed in \cite[Remark 5.29]{HMM2} that the same argument can be carried out under the current 
assumptions\footnote{In \cite[Remark 5.29]{HMM2}, we inadvertently neglected to mention that our proof utilized
estimate \eqref{eqe7.2-bis}; in fact, it is utilized in an essential way.  One should bear this in mind when 
comparing the statement
of Theorem \ref{theor:eps-approx} with \cite[Remark 5.29]{HMM2}.
The former is correct.}. 
Let us note that one can weaken \eqref{oscbdd} by just assuming that for any $Q\in\dd(E)$ and for any connected component of $U_Q^i$ there holds 
\begin{equation}\label{oscbdd-weak}
\sup_{X, Y\in U_Q^i}|u(X)-u(Y)|\leq C_0 \left(\ell(Q)^{-n-1}\dint_{\widehat U_Q^i} |u|^2 \, dX\right)^{\frac12}.
\end{equation}
Also, in the course of the proof one uses \eqref{eqe7.2-bis} for the bounded chord-arc subdomains of the form 
$\Omega=\Omega_{\sbf}^\pm$ defined by \eqref{eq3.2} (with $\sbf'=\sbf$). Further details are left to the interested reader.

\section{Applications: Solutions, subsolutions, and supersolutions  of divergence form elliptic equations with bounded measurable coefficients}\label{appl}

\subsection{Estimates for solutions of second order divergence form elliptic operators with coefficients satisfying a Carleson measure condition}

Given an open set $\Omega\subset\ree$, consider a divergence form elliptic operator $L:=-\div (A(\cdot)\nabla)$,
defined in $\Omega$, where $A$ is an $(n+1)\times(n+1)$ matrix with real bounded measurable coefficients, possibly non-symmetric, satisfying the  ellipticity condition
\begin{equation}
\label{eq1.1*} \lambda^{-1}\,|\xi|^{2}\leq\,A(X)\,\xi\cdot,\xi
:= \sum_{i,j=1}^{n+1}A_{ij}(X)\xi_{j} \xi_{i}, \qquad
|A(X)\,\xi\cdot \zeta|\le \lambda|\xi|\,|\zeta|,
\end{equation}
for some $\lambda\ge 1$, and for all $\xi, \zeta\in\ree$, and for a.e $X\in \Omega$. As usual, the divergence form equation is interpreted in the weak sense, i.e., we say that $Lu=0$
in $\Omega$ if $u\in W^{1,2}_{\rm loc}(\Omega)$ and
\begin{equation}\label{eqweak}
\dint_\Omega A(X) \nabla u(X) \cdot \nabla \Psi(X)\,dX = 0\,,
\end{equation} for all  $\Psi \in C_0^\infty(\Omega)$.

Let us introduce some notation. Given an open set $\Omega\subset\ree$ and $A$, an $(n+1)\times(n+1)$ matrix defined on $\ree\setminus E$ with real bounded measurable coefficients, possibly non-symmetric, satisfying the  ellipticity condition \eqref{eq1.1*}, we say that $A\in KP(\Omega)$ (the Kenig-Pipher class) if $|\nabla A(\cdot)|\dist(\cdot,\partial\Omega)\in L^\infty(\Omega)$ and $\|\nabla A\|_{\C(\Omega)}<\infty$. 
It has been demonstrated in \cite{KPdrift}  that if $\Omega$ is a Lipschitz domain and $A\in KP(\Omega)$ weak solutions to $Lu$ satisfy square function/non-tangential maximal function estimates and Carleson measure estimates on $\Omega$.  Strictly speaking, the class of matrices is slightly smaller and the details of the proof are only provided there for $N<S$ direction (and only for $p>2$), but all ingredients are laid out for a reader to reconstruct a complete proof. One can also consult \cite{DFM2} for complete details presented in this and more general, higher co-dimensional, case. For the 
precise case we are considering here, the  following result can be found in \cite[Appendix A]{HMT}\footnote{The argument 
 in \cite[Appendix A]{HMT} follows that of  \cite{KPdrift} very closely.}:
\begin{equation}\label{CME-KP}
\parbox{.85\textwidth}{Let $\Omega$ be a Lipschitz domain and let $A\in KP(\Omega)$. Then, any weak solution $u\in W^{1,2}_\loc(\Omega)\cap L^\infty(\Omega)$ to $Lu=0$ in $\Omega$  satisfies $\|\nabla u\|_{\C(\Omega)}\lesssim \|u\|_{L^\infty(\Omega)}^2$  with implicit constant depending on $n$, the Lipschitz character of $\Omega$, ellipticity, and the the implicit constants in $A\in KP(\Omega)$. 	
}
\end{equation}

We also need the following auxiliary result (cf. \cite[Lemma~3.1]{KPdrift}):

\begin{lemma}\label{l8.3} Let $E\subset \ree$ be a closed set and let $A$ be an $(n+1)\times(n+1)$ matrix defined on $\ree\setminus E$ with real bounded measurable coefficients, possibly non-symmetric, satisfying the  ellipticity condition \eqref{eq1.1*}. If $A\in KP(\ree\setminus E)$ then $A\in KP(D)$ for any subset $D\subset \ree\setminus E$. Moreover, $\|\nabla A(\cdot)\dist(\cdot,\partial D)\|_{L^\infty(D)}\le \|\nabla A(\cdot)\dist(\cdot,E)\|_{L^\infty(\ree\setminus E)}$ and
\[
\|\nabla A\|_{\C(D)}\le C\big(\|\nabla A\|_{\C(\ree\setminus E)}+ \|\nabla A(\cdot)\dist(\cdot,E)\|_{L^\infty(\ree\setminus E)}^2\big),
\]
where $C$ depends only on dimension. 
\end{lemma}

\begin{proof} 
	
Note first that since $D\subset\ree\setminus E$ then $\dist(X,\partial D)\le \dist(X,E)$ for every $X\in D$. In particular, one has $\|\nabla A(\cdot)|\dist(\cdot,\partial D)\|_{L^\infty(D)}\le \|\nabla A(\cdot)\dist(\cdot,E)\|_{L^\infty(\ree\setminus E)}$. 

Next, we fix $B(x, r)$ with $x\in \partial D$ and $0<r<\infty$. We shall consider two cases. First, if $\dist(x,E)\leq 2r$ we pick $z\in E$ with $\dist(x,E)=|x-z|$ and observe that $B(x,r)\subset B(z,3r)$. Then, 
\[
\dint_{B(x,r)\cap D} |\nabla A(Y)|^2 \dist(Y,\partial D) \,dY
\le
\dint_{B(z,3r)\cap D} |\nabla A(Y)|^2 \dist(Y,E) \,dY
\le
(3\,r)^n \|\nabla A\|_{\C(\ree\setminus E)}.
\]
In the second case, $\dist(X,E)> 2r$, we have $\dist(Y,E)>r$ and $\dist(Y,\partial D)\le |Y-x|<r$ for every $Y\in B(x, r)\cap D$. Hence, 
\begin{multline*}
\dint_{B(x,r)\cap D} |\nabla A(Y)|^2 \dist(Y,\partial D) \,dY
\le
\|\nabla A(\cdot)\dist(\cdot,E)\|_{L^\infty(\ree\setminus E)}^2
\dint_{B(z,r)\cap D} \frac{\dist(Y,\partial D)}{\dist (Y,E)^2} \,dY
\\
\le
\|\nabla A(\cdot)\dist(\cdot,E)\|_{L^\infty(\ree\setminus E)}^2  r^{-1}\,|B(x,r)|
=
c_n\|\nabla A(\cdot)\dist(\cdot,E)\|_{L^\infty(\ree\setminus E)}^2  r^n.
\end{multline*}
All these readily give the desired estimate. 
\end{proof}

\begin{theorem}\label{t8.9} Let $E\subset \ree$ be an $n$-dimensional UR set.  Let $A$ be an $(n+1)\times(n+1)$ matrix defined on $\ree\setminus E$ with real bounded measurable coefficients, possibly non-symmetric, satisfying the  ellipticity condition \eqref{eq1.1*} and so that $A\in KP(\ree\setminus E)$. Then any weak solution $u\in W^{1,2}_\loc(\ree\setminus E)$ to $Lu=0$ in $\ree\setminus E$ satisfies the $S<N$ estimates
\begin{equation}\label{eq8.10}
\|S_{\ree\setminus E}u\|_{L^p(E)} \leq C \|N_{*,{\ree\setminus E}} u\|_{L^p(E)}, \quad 0<p<\infty,
\end{equation} 
and 
\begin{equation}\label{eq8.10:balls}
\|S_{\ree\setminus E}^r u\|_{L^p(\Delta(x,r))}\lesssim \|N^{K'r}_{*,\ree\setminus E} u\|_{L^p(\Delta(x,K' r))}, \quad 0<p<\infty,
\end{equation}
for any $x\in E$ and $0<r<2\diam(E)$, where $\Delta(x,r)=B(x,r)\cap E$, and where $K'$ depends on $n$ and the UR character of $E$; 
as well as its local dyadic analogue, for any Whitney-dyadic structure $\{\W_Q\}_{Q\in\dd(E)}$ for $\ree\setminus E$ with parameters $\eta$ and $K$,
\begin{equation}\label{eq8.11}
\|S^Qu\|_{L^p(Q)} \leq C \|\widehat N_*^Q u\|_{L^p(Q)}, \quad Q\in \dd(E), \quad  0<p<\infty.
\end{equation} 
If, in addition, bounded, $u\in L^\infty(\ree\setminus E)$ then the Carleson measure estimate 
\begin{equation}\label{eq8.12}
\|\nabla u\|_{\C(\ree\setminus E)}
\leq \,C\, \|u\|^2_{L^\infty(\ree\setminus E)}\, ,
\end{equation}
holds and $u$ is $\eps$-approximable on $\ree\setminus E$, in the sense of Definition~\ref{def1.3}. All constants depend on $n$, the UR character of $E$, the ellipticity of $A$, $\|\nabla A(\cdot)\dist(\cdot,E)\|_{L^\infty(\ree\setminus E)}$, $\|\nabla A\|_{\C(\ree\setminus E)}$, the aperture of the cone $\kappa$ implicit in \eqref{eq8.10}, and the implicit parameters $\eta, K,\tau$ implicit in \eqref{eq8.11}.
\end{theorem}

\begin{proof} 
Fix $A\in KP(\ree\setminus E)$ with ellipticity constant $\lambda$ and take any weak solution $u\in W^{1,2}_\loc(\ree\setminus E)$ to $Lu=0$ in $\ree\setminus E$. 
\begin{claim}\label{claim:elliptic-PDE:UR-sub}
For any $\Omega\subset\ree\setminus E$ with $\pom$ being UR there holds
\[
\|\nabla u\|_{\C(\Omega)}\lesssim \|u\|_{L^\infty(\Omega)}^2,
\]
with an implicit constants on $n$, the UR character of $E$, $\lambda$,  and the implicit constants in $A\in KP(\ree\setminus E)$.
\end{claim}

Assuming this momentarily, and taking $\Omega=\ree\setminus E$ we readily obtain \eqref{eq8.12}. On the other hand, given an arbitrary  $Q\in\dd(E)$ and arbitrary pairwise disjoint family $\F\subset\dd_Q$, let $G=\nabla u\in L^2_\loc(\ree\setminus E)$ and $H=u\in C(\ree\setminus E)$. Note that Proposition~\ref{prop:Sawtooths-UR} says that $\widehat{\Omega}_{\F,Q}$ is an open set with UR boundary and with UR character depending on $n$ and the UR character of $E$. Hence, Claim \ref{claim:elliptic-PDE:UR-sub} says that
\[
\|G\|_{\C(\widehat{\Omega}_{\F,Q})}=\|\nabla u\|_{\C(\widehat{\Omega}_{\F,Q})}\lesssim \|u\|_{L^\infty(\widehat{\Omega}_{\F,Q})}^2=\|H\|_{L^\infty(\widehat{\Omega}_{\F,Q})}^2
\]
with a constant which is independent of $u$, $Q$ and $\F$, and depends on $n$, the UR character of $E$, the ellipticity of $A$, and the implicit constants in $A\in KP(\ree\setminus E)$. This means that $(A_\loc)$ in Theorem~\ref{theor:good-lambda} holds for the open set $\ree\setminus E$. As such \eqref{eq:Aloc->B-dyadic}, \eqref{eq:B-dyadic->B}, 
and Remark \ref{remark:N<S:cubes->balls} imply \eqref{eq8.10}--\eqref{eq8.11}. 

\begin{proof}[Proof of Claim \ref{claim:elliptic-PDE:UR-sub}]
Take an arbitrary any open subset $\Omega\subset\ree\setminus E$ with $\pom$ being UR. We may assume that $0<\|u\|_{L^\infty(\Omega)}<\infty$, otherwise the desired estimate is trivial. 
Set $A_{\Omega}:=A$ in $\Omega$ and $A_\Omega:=Id$ (the identity matrix) in $\ree\setminus\Omega$ which is an elliptic matrix with ellipticity constant at most $\lambda$. Note that Lemma \ref{l8.3} gives
\begin{multline*}
\|\nabla A_\Omega\|_{\C(\ree\setminus\pom)}
=
\sup_{x\in \pom,\, 0<r<\infty} \frac1{r^n}\dint_{B(x,r)\setminus\pom} |\nabla A_\Omega(Y)|^2 \dist(Y,\pom) \,dY
\\
=
\sup_{x\in \pom,\, 0<r<\infty} \frac1{r^n}\dint_{B(x,r)\cap \Omega} |\nabla A(Y)|^2 \dist(Y,\pom) \,dY
\\
\le 
\|\nabla A\|_{\C(\Omega)}
\le
C_n\big(\|\nabla A\|_{\C(\ree\setminus E)}+ \|\nabla A(\cdot)\dist(\cdot,E)\|_{L^\infty(\ree\setminus E)}^2\big)
\end{multline*}
and
\[
\|\nabla A_\Omega \dist(\cdot,\pom)\|_{L^\infty(\ree\setminus\pom)}
=
\|\nabla A \dist(\cdot,\pom)\|_{L^\infty(\Omega)}
\le
\|\nabla A \dist(\cdot,E)\|_{L^\infty(\ree\setminus E)}.
\]
Write also $u_{\Omega}=u$ in $\Omega$ and $u_\Omega:=0$ in $\ree\setminus\Omega$. Note that $u_\Omega\in W^{1,2}_\loc(\ree\setminus\pom)$ satisfies, in the weak sense,  
$-\div(A_\Omega\nabla u_\Omega )=Lu=0$ in $\Omega$ and $-\div(A_\Omega\nabla u_\Omega )=0$ and $\ree\setminus\overline{\Omega}=0$. This and the fact that $\Omega$ is open imply that $-\div(A_\Omega\nabla u_\Omega )=0$ in $\ree\setminus \pom$ in the weak sense. Note also that $u_\Omega\in L^\infty(\ree\setminus\pom)$ implies  $\|u_\Omega\|_{L^\infty(\ree\setminus \pom)}= \|u\|_{L^\infty(\Omega)}<\infty$.

Fix $D\subset\Omega$ an arbitrary bounded Lipschitz subdomain and $F=\nabla u_\Omega/\|u_\Omega\|_{L^\infty(\Omega)}^2$.  By Lemma \ref{l8.3}, we have that $A_\Omega\in KP(\Omega)\subset KP(D)$ (with uniform bounds controlled by those of $A_\Omega\in KP(\Omega)$, hence ultimately on those of $A\in KP(\ree\setminus E)$). By \eqref{CME-KP} applied to $u_\Omega$ for the operator $L_\Omega$  in $D$ we obtain
\[
\|F\|_{\C(D)}=
\frac{\|\nabla u_\Omega\|_{\C(D)}}{\|u_\Omega\|_{L^\infty(\Omega)}^2}
\lesssim 
\frac{\|u_\Omega\|_{L^\infty(D)}^2}{\|u_\Omega\|_{L^\infty(\Omega)}^2}
\le 1,
\]
with implicit constant depending on $n$, the Lipschitz character of $D'$, $\lambda$ and the implicit constants of $A\in KP(\ree\setminus E)$. 
This and Corollary \ref{corol:CME:Lip->UR} (or Remark \ref{remark:CME-E-OMmga} for a more direct argument) to the UR set  $\pom$  yield
\[
\frac{\|\nabla u\|_{\C(\Omega)}}{\|u\|_{L^\infty(\Omega)}^{2}}
=
\frac{\|\nabla u_\Omega\|_{\C(\ree\setminus \pom)}}{\|u_\Omega\|_{L^\infty(\Omega)}^{2}}
=
\|F\|_{\C(\Omega)}
\lesssim
\sup_{D\subset \ree\setminus \pom} \|F\|_{\C(D)}
=
\sup_{D\subset \Omega} \|F\|_{\C(D)}
\lesssim 1, 
\]
with implicit constants depending only on $n$, the UR character of $\pom$,  $\lambda$, and the implicit constants in $A\in KP(\ree\setminus E)$.
This completes the proof of \eqref{eq8.12}.
\end{proof}

To continue with the proof of Theorem \ref{t8.9} we are left with showing that if we further assume that $u\in L^\infty(\ree\setminus E)$ then $u$ is $\eps$-approximable on $\ree\setminus E$. Firstly, all auxiliary estimates \eqref{locbdd},  \eqref{revHol}, and  \eqref{oscbdd} hold for $u$ in the open set $\ree\setminus E$, and hence in any open subset $\Omega\subset \ree\setminus E$, by the usual interior estimates for solutions of elliptic PDEs (see, e.g., \cite{KenigBook}). We point out again that $N<S$ estimates \eqref{eqn5.9} on all bounded Lipschitz subdomains of $\Omega$ hold essentially by \cite{KPdrift}. More precisely, let $D\subset\ree\setminus E$ be an arbitrary chord-arc subdomain. For every a bounded Lipschitz subdomain $\Omega\subset D$, by Lemma \ref{l8.3} it follows that $A\in KP(\Omega)$ with bounds that depend on the implicit constants in  $A\in KP(\ree\setminus E)$. In turn \eqref{CME-KP} and \cite{KKiPT} yield that the associated elliptic measure belongs to the class $A_\infty(\pom)$ with respect to surface measure. Thus, \cite{DJK} allows us to obtain $N<S$ estimates are valid on $L^q$, $0<q<\infty$, on $\Omega$. Corollary \ref{corol: N<S:Lip->CAD} readily  gives $N<S$ on $L^q$, $0<q<\infty$. This together with the fact that we have already show \eqref{eq8.12} allow us to invoke Theorem \ref{theor:eps-approx} to conclude as desired that $u$ is $\eps$-approximable with constants depending only 
on $n$, the UR character of $E$, $\lambda$,  and the implicit constants in $A\in KP(\ree\setminus E)$.
\end{proof}

\subsection{Estimates for subsolutions and supersolutions of second order divergence form elliptic operators with coefficients satisfying a Carleson measure condition} 
Our methods allow us to deal not only with solutions but also with subsolutions (thus, also with supersolutions) of the operators considered in the previous section. Before, stating the result let us recall that given an open set $\Omega\subset\ree$ and a second order divergence form elliptic operators $L:=-\div (A(\cdot)\nabla)$, defined in $\Omega$, where $A$ is an $(n+1)\times(n+1)$ matrix with real bounded measurable coefficients, possibly non-symmetric, satisfying the  ellipticity condition \eqref{eq1.1*}, we say that $u\in W^{1,2}_{\loc}(\Omega)$ is a weak $L$-subsolution (or, $Lu\le 0$) in $\Omega$ if 
\begin{equation}\label{eqweak:sub}
\dint_\Omega A(X) \nabla u(X) \cdot \nabla \Psi(X)\,dX \le  0\,,
\end{equation} 
for all  $0\le \Psi \in C_0^\infty(\Omega)$. Analogously,  $u\in W^{1,2}_{\loc}(\Omega)$ is a weak $L$-supersolution (or, $Lu\ge 0$) if $-u$ is a subsolution.

We are now ready to state our main result in this section.  We note that it applies in particular to the Laplace operator, hence the obtained estimates are valid for any subharmonic or superharmonic functions.

\begin{theorem}\label{t8.9:sub} Let $E\subset \ree$ be an $n$-dimensional UR set.  Let $A$ be an $(n+1)\times(n+1)$ matrix defined on $\ree\setminus E$ with real bounded measurable coefficients, possibly non-symmetric, satisfying the  ellipticity condition \eqref{eq1.1*} and so that $A\in KP(\ree\setminus E)$. Then any weak $L$-subsolution or $L$-supersolution $u\in W^{1,2}_\loc(\ree\setminus E)$ in $\ree\setminus E$ satisfies the $S<N$ estimates
	\begin{equation}\label{eq8.10:sub}
	\|S_{\ree\setminus E}u\|_{L^p(E)} \leq C \|N_{*,{\ree\setminus E}} u\|_{L^p(E)}, \quad 0<p<\infty,
	\end{equation} 
	and 
	\begin{equation}\label{eq8.10:balls:sub}
	\|S_{\ree\setminus E}^r u\|_{L^p(\Delta(x,r))}\lesssim \|N^{K'r}_{*,\ree\setminus E} u\|_{L^p(\Delta(x,K' r))}, \quad 0<p<\infty,
	\end{equation}
	for any $x\in E$ and $0<r<2\diam(E)$, where $\Delta(x,r)=B(x,r)\cap E$, and where $K'$ depends on $n$ and the UR character of $E$; 
	as well as its local dyadic analogue, for any Whitney-dyadic structure $\{\W_Q\}_{Q\in\dd(E)}$ for $\ree\setminus E$ with parameters $\eta$ and $K$,
	\begin{equation}\label{eq8.11:sub}
	\|S^Qu\|_{L^p(Q)} \leq C \|\widehat N_*^Q u\|_{L^p(Q)}, \quad Q\in \dd(E), \quad  0<p<\infty.
	\end{equation} 
	If, in addition, bounded, $u\in L^\infty(\ree\setminus E)$ then the following Carleson measure estimate hold
	\begin{equation}\label{eq8.12:sub}
	\|\nabla u\|_{\C(\ree\setminus E)}
	\leq \,C\, \|u\|^2_{L^\infty(\ree\setminus E)}\, .
	\end{equation}
All constants depend on $n$, the UR character of $E$, the ellipticity of $A$, $\|\nabla A(\cdot)\dist(\cdot,E)\|_{L^\infty(\ree\setminus E)}$, $\|\nabla A\|_{\C(\ree\setminus E)}$, the aperture of the cone $\kappa$ implicit in \eqref{eq8.10}, and the parameters $\eta, K,\tau$ implicit in \eqref{eq8.11}.
\end{theorem}

\begin{proof}
We start observing that we just need to consider the case where $u$ is a weak $L$-subsolution (because, if $u$ is a weak $L$-supersolution then $-u$ is a weak $L$-subsolution).   We proceed much in the proof of Theorem \ref{t8.9:sub} and a careful reading shows that we just need a version of \eqref{CME-KP} valid for weak $L$-subsolutions. That is, we need to obtain the following:
\begin{equation}\label{CME-KP:sub}
\parbox{.85\textwidth}{Let $\Omega$ be a Lipschitz domain and let $A\in KP(\Omega)$. Then, any weak $L$-subsolution $u\in W^{1,2}_\loc(\Omega)\cap L^\infty(\Omega)$  in $\Omega$  satisfies $\|\nabla u\|_{\C(\Omega)}\lesssim \|u\|_{L^\infty(\Omega)}^2$  with implicit constant depending on $n$, the Lipschitz character of $\Omega$, ellipticity, and the the implicit constants in $A\in KP(\Omega)$. 	
}
\end{equation}

With this goal in mind, fix then an arbitrary weak $L$-subsolution $u\in W^{1,2}_\loc(\Omega)\cap L^\infty(\Omega)$  in $\Omega$. We may suppose that $u$ is a.e.~non-negative. Indeed,  assume for the moment that
we have proved \eqref{CME-KP:sub} for a.e.~non-negative weak $L$-subsolutions, and let $u\in W^{1,2}_\loc(\Omega)$ 
be an arbitrary bounded
weak $L$-subsolution, so that
$\widetilde{u}:=u+\|u\|_{L^\infty(\Omega)}\in W^{1,2}_\loc(\Omega)\cap L^\infty(\Omega)$
is an a.e.~non-negative weak $L$-subsolution  in $\Omega$. 
We then observe that our assumption for a.e.~non-negative weak $L$-subsolutions yields the desired estimate for $u$:
\[
\|\nabla u\|_{\C(\Omega)}=\|\nabla\widetilde{u}\|_{\C(\Omega)}\lesssim \|\widetilde{u}\|_{L^\infty(\Omega)}^2\le 2\, \|u\|_{L^\infty(\Omega)}^2.
\]

Let us then verify \eqref{CME-KP:sub} for an a.e.~non-negative weak $L$-subsolution 
$u\in W^{1,2}_\loc(\Omega)\cap L^\infty(\Omega)$. We observe that since $A\in KP(\Omega)$, by \eqref{CME-KP} and \cite{KKiPT}, it follows that the elliptic measure $\omega_L$ belongs to $A_\infty(\sigma)$ with $\sigma=H^n|_{\partial\Omega}$. With this in hand, we carefully follow the argument in \cite[Proof of Theorem 1.1: $(b)\Longrightarrow(a)$]{CHMT} with $u$ being the fixed a.e.~non-negative weak  $L$-subsolution in $\Omega$ in place of a solution and observing that Lipschitz domains are clearly 1-sided CAD. To justify that the argument can be adapted to the present situation we just need two observations. First, that $u$ satisfies Caccioppoli's estimate (the proof is a straightforward modification of the standard argument using that $u$ is a non-negative a.e.~weak $L$-subsolution). Second, in \cite[(3.64)]{CHMT} one has to replace ``$=0$'' by ``$\le 0$'' because in the present scenario $u$ is a non-negative a.e.~weak $L$-subsolution (in place of a solution). With these two observations an interested reader could easily see that the argument goes through and eventually show that  $\|\nabla u\|_{\C(\Omega)}\lesssim \|u\|_{L^\infty(\Omega)}^2$. Hence, \eqref{CME-KP:sub} holds and this completes the proof. 
\end{proof}

\subsection{Higher order elliptic equations and systems with constant coefficients}
In \cite{DKPV} the authors obtained square function/non-tangential maximal function estimates for higher order elliptic equations and systems on bounded Lipschitz domains. These results have never been extended, even to CAD domains, and here we present a generalization of Carleson measure estimates to the complements of UR sets.

For any multiindex $\alpha=(\alpha_1, \dots, \alpha_{n+1})\in \NN_0^{n+1}$, we write $|\alpha|=\alpha_1+\dots+\alpha_{n+1}$ and  $\alpha!=\alpha_1!\cdots\alpha_{n+1}!$ where $0!=1$. Also $\partial^\alpha=\partial^{\alpha_1}\dots\partial^{\alpha_{n+1}}$ and for every $Y\in\ree$ we write
$Y^\alpha=Y_1^{\alpha_1}\cdots Y_{n+1}^{\alpha_{n+1}}$  where  $a^0=1$ for every $a\in\re$. Finally, $\nabla^k$, $k\in \N$ stands for the vector of all partial derivatives of order $k$. For $k=0$, $\nabla^0$ is just the identity operator.

Let $K, m\in\NN$. For every $1\le j,k\le K$, let $L^{jk}=\sum_{|\alpha|=2\,m} a_{\alpha \beta}^{jk}\partial^\alpha $, where  $\alpha=(\alpha_1, \dots, \alpha_{n+1})\in \NN_0^{n+1}$. The coefficients $a_{\alpha \beta}^{jk}$, $1\le \alpha,\beta\le n+1$, $1\le j, k\le K$ are real constants. Given and open set $\Omega$ and $u=(u_1, \dots, u_K)$, with $u_j \in W^{m, 2}_{\rm loc}(\Omega)$, $1\le j\le K$, we say that $Lu=0$,  if 
\[
\sum_{k=1}^K L^{jk}u^k
=
\sum_{k=1}^K \sum_{|\alpha|=|\beta|=m} a_{\alpha \beta}^{jk} \partial^\alpha \partial^\beta u^k
=
0, \qquad j=1,\dots,K, 
\]
as usual, in the weak sense, similarly to \eqref{eqweak}. Here, $W^{m,2}(\Omega)$ is the space of functions with all derivatives of orders $0,\dots,m$ in $L^2(\Omega)$ and $W^{m, 2}_{\rm loc}(\Omega)$ is the space of functions locally in $W^{m,2}(\Omega)$. We assume, in addition, that $L$ is symmetric: $L^{jk}=L^{kj}$ for $1\leq j,k\leq K$, and that the Legendre-Hadamard ellipticity condition holds: there exists $\lambda>0$ such that 
\begin{equation}\label{LH-higher}
 \sum_{j,k=1}^K \sum_{|\alpha|=|\beta|=m} a_{\alpha\beta}^{jk}\,\xi^\alpha\xi^\beta \zeta_j\zeta_k \geq \lambda\, |\xi|^{2m}|\zeta|^2, \quad \mbox{for all $\zeta=(\zeta_1, \dots, \zeta_K	)\in \re^K$, $\xi \in \ree$.}
\end{equation}

\begin{theorem}\label{t8.14} 
Let $E\subset \ree$ be an $n$-dimensional UR set. Given $K,m\in\NN$, let $L$ be a symmetric constant coefficient $2m$-order $K\times K$ system, satisfying the Legendre-Hadamard ellipticity condition, as above. Then any weak solution $u\in [W^{m, 2}_{\rm loc}(\ree\setminus E)\cap C^{m-1}(\ree\setminus E)]^K$ to $Lu=0$ in $\ree\setminus E$ satisfies the $S<N$ estimates
\begin{equation}\label{eq8.10-bis}
\|S_{\ree\setminus E}(\nabla^{m-1}u)\|_{L^p(E)} \leq C \|N_{*,{\ree\setminus E}} (|\nabla^{m-1}u|)\|_{L^p(E)}, \quad 0<p<\infty,
\end{equation} 
and 
\begin{equation}\label{eq8.10-bis:balls}
\|S_{\ree\setminus E}^r (\nabla^{m-1}u)\|_{L^p(\Delta(x,r))}\lesssim \|N^{K'r}_{*,\ree\setminus E} (|\nabla^{m-1}u|)\|_{L^p(\Delta(x,K' r))}, \quad 0<p<\infty,
\end{equation}
for any $x\in E$ and $0<r<2\diam(E)$, where $\Delta(x,r)=B(x,r)\cap E$, and where $K'$ depends on $n$ and the UR character of $E$; 
as well as its local dyadic analogue, for any Whitney-dyadic structure $\{\W_Q\}_{Q\in\dd(E)}$ for $\ree\setminus E$ with parameters $\eta$ and $K$,
\begin{equation}\label{eq8.11-bis}
\|S^Q(\nabla^{m-1}u)\|_{L^p(Q)} \leq C \|\widehat N_*^Q (|\nabla^{m-1}u|)\|_{L^p(Q)}, \quad Q\in \dd(E), \quad  0<p<\infty.
\end{equation} 
If $u$ is, in addition, such that $\nabla^{m-1} u\in L^\infty(\Omega)$, then the Carleson measure estimate 
\begin{equation}\label{eq8.12-bis}
\|\nabla^m u\|_{\C(\ree\setminus E)}
\leq \,C\, \|\nabla^{m-1}u\|^2_{L^\infty(\ree\setminus E)}\, ,
\end{equation}
holds. 
All constants depend on $n$, the UR character of $E$, the Legendre-Hadamard ellipticity constant, $\sup_{j,k,\alpha,\beta }|a_{\alpha\beta}^{jk}|$, the aperture of the cone $\kappa$ implicit in \eqref{eq8.10-bis}, and the implicit parameters $\eta, K,\tau$ implicit in \eqref{eq8.11-bis}.
\end{theorem}

\begin{remark}
It is easy to see that from the previous result, one can also obtain analogous estimates in any chord-arc domain $D\subset\ree$. To see this let us consider any weak solution $u\in [W^{m, 2}_{\rm loc}(D)]^K$ to $Lu=0$ in $D$. Let $\widetilde{u}:=u$ in $D$ and $\widetilde{u}=0\in\ree\setminus \overline{D}$. Then $\widetilde{u}\in [W^{m, 2}_{\rm loc}(\ree\setminus \partial D)]^K$ satisfies $L\widetilde{u}=0$ in $\ree\setminus \partial D$ in the weak sense. As such, and using the fact that since $D$ is a CAD then $\partial D$ is UR, we obtain \eqref{eq8.10-bis} for $\widetilde{u}$ in $\ree\setminus \partial D$ which immediately gives the corresponding estimate for $u$ in $D$. The same occurs with \eqref{eq8.12-bis}.  Further details are left to the interested reader.
\end{remark}

\begin{proof} 
The proof runs much as that of Theorem \ref{t8.9}. One replaces \eqref{CME-KP} with the fact that for any bounded Lipschitz domain $\Omega\subset\ree$, it was shown in \cite[Theorem 2, p.~1455]{DKPV} that for any weak solution $u\in [W^{m, 2}_{\rm loc}(\Omega)]^K$ to $Lu=0$ in $\Omega$ with $\nabla^{m-1} u\in L^\infty(\Omega)$ verifies 
$\|\nabla^m u\|_{\C(\Omega)}\lesssim \|\nabla^{m-1}u\|_{L^\infty(\Omega)}^2$. With this at hand the proof can be carried out \textit{mutatis mutandis}. Further details are left to the interested reader. 
\end{proof}

We can now state a higher order version of Theorems \ref{theor:N<S:Lip->CAD} and \ref{theor:N<S:CAD->CAD:all-p}:

\begin{theorem}\label{theor:N<S:HO}
	Let $D\subset \ree$ be a CAD, let $K,m\in \NN$ and let $u=(u_1,\dots,u_K)\in [W^{m,2}_{\rm loc}(D)\cap C^{m-1}(D)]^K$. 
	\begin{list}{$(\theenumi)$}{\usecounter{enumi}\leftmargin=.8cm
			\labelwidth=.8cm\itemsep=0.2cm\topsep=.01cm
			\renewcommand{\theenumi}{\roman{enumi}}}
		
		\item 	Assume that \eqref{locbdd} holds with $\nabla^{m-1} u$ in place of $u$.	Suppose that the $(m-1)$th-order $N<S$ estimates are valid on $L^2$ on all bounded Lipschitz subdomains $\Omega\subset D$, that is, \eqref{eqn5.9} holds for any  bounded Lipschitz subdomain $\Omega\subset D$ with $\nabla^{m-1}u$ in place of $u$, and where the constant may also depend on $m$ and $K$. 	
		Then \eqref{eqn5.1:local}--\eqref{eqn6.5:unbounded} hold replacing $u$ by $\nabla^{m-1} u$, and where all the constants may also depend on $m$ and $K$. 
		
		\item  Assume that \eqref{locbdd} holds with $\nabla^{m-1} u$ in place of $u$ and that \eqref{revHol} hold with $\nabla^m u$ in place of $\nabla u$ for some $p>2$. 	
		Suppose that the $(m-1)$th-order $N<S$ estimates are valid on $L^p$ on all bounded chord-arc $\Omega\subset D$, that is, \eqref{eqn6.5-bis} holds for any  bounded chord-arc subdomain $\Omega\subset D$ with $\nabla^{m-1}u$ in place of $u$, and where the constant may also depend on $m$ and $K$. 	
		Then \eqref{eqn6.5:new:local}--\eqref{eqn6.5:new:unbounded} hold with $\nabla^{m-1}u$ in place of $u$, 	
		and where all the constants may also depend on $m$ and $K$. 
		
		\item Assume that \eqref{locbdd} holds with $\nabla^{m-1} u$ in place of $u$ and that \eqref{revHol} hold with $\nabla^m u$ in place of $\nabla u$ for some $p>2$. Suppose that the $(m-1)$th-order $N<S$ estimates are valid on $L^2$ on all bounded Lipschitz subdomains $\Omega\subset D$, that is, \eqref{eqn5.9} holds for any  bounded Lipschitz subdomain $\Omega\subset D$ with $\nabla^{m-1}u$ in place of $u$, and where the constant may also depend on $m$ and $K$.  Then \eqref{eqn6.5:new:local}--\eqref{eqn6.5:new:unbounded} hold replacing $u$ by $\nabla^{m-1} u$, 	and where all the constants may also depend on $m$ and $K$. 
		
	\end{list}
	
\end{theorem}

\begin{proof}
	The proof is fairly easy. Consider the vector $v=\nabla^{m-1}u\in [W^{1,2}_\loc(D)\cap C(D)]^{K(n-1)^{m-1}}$. Note that our current assumptions in $(i)$--$(iii)$ imply that $v$ satisfies \eqref{locbdd}. Also, in items $(ii)$, $(iii)$ we will have that $v$ verifies \eqref{revHol}. Note that \eqref{eqn5.9} is satisfied by $v$ in parts $(i)$ and $(iii)$, and \eqref{eqn6.5-bis} holds for $v$ in part $(ii)$. We also no that Theorems \ref{theor:N<S:Lip->CAD} and \ref{theor:N<S:CAD->CAD:all-p}, and Corollary \ref{corol: N<S:Lip->CAD} can be easily extended to vector-valued functions $u$. With all these at hand, we readily obtain the corresponding estimates for $v$ which translated into those stated for $u$. Further details are left to the interested reader. 
\end{proof}

One can also obtain a higher-order version of Theorem \ref{theor:eps-approx} using the same ideas:

\begin{theorem}\label{theor:eps-approx:HO} Let $E\subset \ree$ be an $n$-dimensional UR set, $\ree\setminus E$, and let $m,K\in N$. Suppose that $u\in [W^{m,2}_{\rm loc}(\ree\setminus E)\cap C^{m-1}(\ree\setminus E)\cap L^\infty(\ree\setminus E)]^K$ is such 
	that for any cube $I$ with $2 I\subset \ree\setminus E$
	\begin{equation}\label{oscbdd:HO}
	\sup_{X, Y\in I}\big|\nabla^{m-1}u(X)-\nabla^{m-1}u(Y)\big|\leq C_0 \left(\ell(I)^{1-n}\dint_{2I} \big|\nabla^m u\big|^2 \, dX\right)^{\frac12}
	\end{equation}
	and
	\[
	\|\nabla^m u\|_{\C(\ree\setminus E)}\le C_0' \|\nabla^{m-1} u\|_{L^\infty(\ree\setminus E)}
	\]
	Assume, in addition, $(m-1)$th-order that $N<S$ estimates are valid on $L^2$ on all bounded chord-arc subdomains  $\Omega\subset \ree\setminus E$, that is, for any bounded chord-arc subdomain  $\Omega\subset \ree\setminus E$, there holds
	\begin{equation}\label{eqe7.2-bis:HO}
	\big\|N_{*,\Omega}\big(\nabla^{m-1}u-\nabla^{m-1}u(X_\Omega^+)\big)\big\|_{L^2(\pom)}\leq C_\Omega \big\|S_{\Omega} \big(\nabla^{m-1}u\big)\big\|_{L^2(\pom)}. 
	\end{equation} 
	Here $X_{\Omega}^+$ is any interior corkscrew point of $\Omega$ at the scale of $\diam (\Omega)$, and the constant  $C_{\Omega}$ depends on the CAD character of $\Omega$, the dimension $n$, $m$, $K$, $p$, the implicit choice of $\kappa$ (the aperture of the cones in $N_{*,\Omega}$ and $S_{\Omega}$), and the implicit corkscrew constant for the point $X_{\Omega}^+$ . 
	Then $\nabla^{m-1}u$ is $\eps$-approximable on $ \ree\setminus E$, with the implicit constants depending only on $n$, $m$, $K$,the UR character of $E$, $C_0$, and $C_0'$.
\end{theorem}

As a corollary of all these we can obtain $N<S$ estimates and $\eps$-approximability for solutions of a symmetric constant coefficient $2m$-order $K\times K$ systems. 

\begin{theorem}\label{t8.14:new} 
	Given $K,m\in\NN$, let $L$ be a symmetric constant coefficient $2m$-order $K\times K$ system, satisfying the Legendre-Hadamard ellipticity condition, as above.
	
	\begin{list}{$(\theenumi)$}{\usecounter{enumi}\leftmargin=.8cm
			\labelwidth=.8cm\itemsep=0.2cm\topsep=.01cm
			\renewcommand{\theenumi}{\roman{enumi}}}
		
		\item If $D\subset \ree$ is a CAD,  then any weak solution $u\in [W^{m, 2}_{\rm loc}(D)\cap C^{m-1}(D)]^K$ to $Lu=0$ in $D$ satisfies 
		for any $x\in\partial D$ and $0<r<2\diam(\partial D)$ and for every $\kappa>0$
		\begin{equation}\label{eqn6.5:new:local:balls:HO}
		\big\|N_{*,D,\kappa}^r \big(\nabla^{m-1}u-\nabla^{m-1}u(X_{\Delta(x,r)}^+\big)\big\|_{L^q(\Delta(x,r))}\le C \big\|S^{C'r}_{D,\kappa} \big(\nabla^{m-1}u\big)\big\|_{L^q(\Delta(x,C' r))},  \quad \mbox{for all}\quad 0<q<\infty.
		\end{equation}
		where $\Delta(x,r)=B(x,r)\cap\pom$. Here $C$ depends on $n$, $q$, $K$, $m$, the CAD character of $D$, the Legendre-Hadamard ellipticity constant, $\sup_{j,k,\alpha,\beta }|a_{\alpha\beta}^{jk}|$, and the aperture of the cone $\kappa$, and $C'$ depends on $n$ and the CAD character of $D$. 
		In particular, if $\partial D$ is bounded 
		\begin{equation}\label{eqn6.5:new:HO}
		\big\|N_{*,D,\kappa}\big(\nabla^{m-1}u-\nabla^{m-1}u(X_D^+)\big)\big\|_{L^q(\partial D)}\leq C'' \big\|S_{D,\kappa} \big(\nabla^{m-1}u\big)\big\|_{L^q(\partial D)}, \quad \mbox{for all}\quad 0<q<\infty,
		\end{equation} 
		and if $\partial D$ is unbounded and $\nabla^{m-1}u(X)\to 0$ as $|X|\to\infty$ then 
		\begin{equation}\label{eqn6.5:new:unbounded:HO}
		\|N_{*,D,\kappa}\big(\nabla^{m-1} u\big)\|_{L^q(\partial D)}\leq C'' \big\|S_{D,\kappa} \big(\nabla^{m-1}u\big)\big\|_{L^q(\partial D)}, \quad \mbox{for all}\quad 0<q<\infty.
		\end{equation}

		\item	Let $E\subset \ree$ be an $n$-dimensional UR set.  Then any weak solution $u\in [W^{m, 2}_{\rm loc}(\ree\setminus E)\cap C^{m-1}(\ree\setminus E)\cap L^\infty(\ree\setminus E)]^K$ to $Lu=0$ in $\ree\setminus E$ satisfies that $\nabla^{m-1}u$ is $\eps$-approximable in $\ree\setminus E$ with implicit constants depending on $n$, $K$, $m$, the UR character of $E$, the Legendre-Hadamard ellipticity constant, $\sup_{j,k,\alpha,\beta }|a_{\alpha\beta}^{jk}|$.
	\end{list}
	
\end{theorem}

\begin{proof} We aim to use Theorem \ref{theor:N<S:HO} part $(iii)$ and Theorem \ref{theor:eps-approx:HO}. To this end, we need to verify the interior estimates:   \eqref{locbdd} with $\nabla^{m-1} u$ in place of $u$, \eqref{revHol} with $\nabla^m u$ in place of $\nabla u$ for some $p>2$, and \eqref{oscbdd:HO}, and to 	obtain 	 
	$(m-1)$th-order $N<S$ estimates on $L^2$ on all bounded Lipschitz subdomains  $\Omega$ and for any weak solution $u\in [W^{m, 2}_{\rm loc}(\Omega)\cap C^{m-1}(\Omega)]^K$ to $Lu=0$ in $\Omega$. That is, we need to show that \eqref{eqe7.2-bis:HO} holds on all bounded Lipschitz subdomains  $\Omega$. Let us start with the latter. To see this we introduce 
	\[
	P_{m-1, X_\Omega^+}u(X)= \sum_{|\alpha|\le m-1} \frac{\partial^\alpha u(X_\Omega^+) }{\alpha!} (X-X_\Omega^+)^\alpha,\qquad X\in\Omega.
	\]	
	and observe that $\nabla^k P_{m-1, X_\Omega^+}u(X_\Omega^+)=\nabla^k u(X_\Omega^+)$ for $0\le k\le m-2$; $\nabla^{m-1} P_{m-1, X_\Omega^+}u(\cdot)\equiv\nabla^{m-1} u(X_\Omega^+)$; and $\nabla^m P_{m-1, X_\Omega^+}u\equiv 0$.  Thus if we write $v=u-P_{m-1, X_\Omega^+}u(\cdot)$ we have that 
	$v\in [W^{m, 2}_{\rm loc}(\Omega)\cap C^{m-1}(\Omega)]^K$ is a weak solution to $Lv=0$ in $\Omega$ satisfying $\nabla^k v(X_\Omega^+)=0$ for all $0\le k\le m-1$;  
	$\nabla^{m-1} v=\nabla^{m-1} u- \nabla^{m-1}u(X_\Omega^+)$; and $\nabla^{m} v=\nabla^{m} u$. As such we can invoke \cite[Theorem 3, p.~1456]{DKPV} to obtain that
	\begin{multline*}
	\big\|N_{*,\Omega}\big(\nabla^{m-1}u-\nabla^{m-1}u(X_\Omega^+)\big)\big\|_{L^2(\pom)}
	=
	\big\|N_{*,\Omega} \big(\nabla^{m-1}v\big)\big\|_{L^2(\pom)}
	\\
	\lesssim 
	\big\|S_{\Omega} \big(\nabla^{m-1}v\big)\big\|_{L^2(\pom)}
	= 
	\big\|S_{\Omega} \big(\nabla^{m-1}u\big)\big\|_{L^2(\pom)}.
	\end{multline*}
	
Turning to interior estimates, we recall from \cite[Corollary 22, p. 384]{Bar16}, that for all solutions to $Lu=0$ in $2I$ we have
\begin{equation}\label{eqinte1}
	\dint_{I} \big|\nabla^j u\big|^2 \, dX\leq C \ell(I)^{-2j}\,\dint_{2I} \big| u\big|^2 \, dX, \quad j=0,\dots,m. 
	\end{equation}
In fact, \cite{Bar16} pertains to much more general elliptic systems with  bounded measurable coefficients. It uses the weak G\r{a}rding inequality \cite[(10), p.~380]{Bar16}. To obtain the latter (with $\delta=0$) we can see that Plancherel's theorem, the fact that we are currently consider the case with real constant coefficients, and the Legendre-Hadamard condition \eqref{LH-higher} easily yield, for every smooth compactly supported function $\varphi$,
\begin{align*}
{\rm Re\,} \langle \nabla^m\varphi,A\,\nabla^m \varphi\rangle_{\ree} 
&= 
{\rm Re\,}
\dint_{\ree}  \sum_{j,k=1}^K \sum_{|\alpha|=|\beta|=m} \overline{\partial^\alpha \varphi_j(X)} \,a_{\alpha\beta}^{jk}\, \partial^\beta \varphi_k(X)\,dX
\\
&=
\sum_{j,k=1}^K \sum_{|\alpha|=|\beta|=m} a_{\alpha\beta}^{jk}\, {\rm Re\,}\dint_{\ree}  (-2\,\pi\, i\, \xi)^\alpha\, (2\,\pi\, i \xi)^\beta\, \overline{\widehat{\varphi}_j(\xi)}\, \widehat{\varphi}_k(\xi)\, d\xi
\\
&=
\dint_{\ree}  \sum_{j,k=1}^K \sum_{|\alpha|=|\beta|=m} a_{\alpha\beta}^{jk}\, (2\pi\xi)^\alpha\, (2\,\pi\, \xi)^\beta\, {\rm Re\,}\big(\,\overline{\widehat{\varphi}_j(\xi)}\, \widehat{\varphi}_k(\xi)\big)\, d\xi
\\
&\ge 
\lambda
\dint_{\ree}  \big(2\,\pi\,|\xi|\big)^{2m}\, | \widehat{\varphi}_j(X)|^2\, dX
\\
&
=
\lambda
\dint_{\ree}  |\nabla^m \varphi_j(\xi)|^2\, d\xi,
\end{align*}
and so \cite{Bar16} applies to our setting. 

Now, for constant coefficient operators any derivative of a solution is still a solution, and, in fact, we will use $v:=u-P_{m-1, X_I}u(\cdot)$ built similarly to above, only using $X_I$ being the center of $I$ in place of $X_\Omega^+$. Clearly, $\nabla ^m v=\nabla^m u$ is a solution too, and so a repeated application of \eqref{eqinte1} yields 
\begin{equation}\label{eqinte2}
	\dint_{I} \big|\nabla^{k} v\big|^2 \, dX\leq C \ell(I)^{-2(k-m)}\,\dint_{2I} \big| \nabla^m v\big|^2 \, dX, \quad k\geq m.
	\end{equation}
Taking $k>m-1$ large enough, depending on the dimension only, so that the Sobolev space $W^{k,2}(I)$ embeds into the H\"older space $C^{m-1, \alpha}(I)$, $\alpha>0$, we can show that 
\begin{multline}\label{eqinte3}
	\sup_{X, Y\in I}\big|\nabla^{m-1}u(X)-\nabla^{m-1}u(Y)\big|=\sup_{X, Y\in I}\big|\nabla^{m-1}v(X)-\nabla^{m-1}v(Y)\big|\\ \leq C \sum_{j=0}^k \left(\ell(I)^{-1-n +2(j-m+1)}\dint_{I} \big|\nabla^j v\big|^2 \, dX\right)^{\frac12}.
	\end{multline}
For $j>m$ we use \eqref{eqinte2} to descend to $j=m$. For $j<m$, we use Poincar\'e inequality to ascend to $j=m$, and all in all, the expression above is bounded by 
$$C \left(\ell(I)^{1+n}\dint_{2I} \big|\nabla^m v\big|^2 \, dX\right)^{\frac12}=C \left(\ell(I)^{1+n}\dint_{2I} \big|\nabla^m u\big|^2 \, dX\right)^{\frac12},$$
as desired. This yields \eqref{oscbdd:HO}.

In order to obtain \eqref{locbdd} with $\nabla^{m-1} u$ in place of $u$, we apply the same argument as above to $v:=\nabla^{m-1} u- \vec c$ for some constant vector $\vec c$. The function $v$ is also a solution of the initial system, and so \eqref{eqinte2} still holds. Much as above, by Morrey inequality (or generalized Sobolev embeddings), for $k$ large enough, depending on dimension only, we arrive at
\begin{equation}\label{eqinte4}
\sup_I |v|\leq C \sum_{j=0}^k \left(\ell(I)^{-1-n +2j}\dint_{I} \big|\nabla^j v\big|^2 \, dX\right)^{\frac12}\leq C \left(\ell(I)^{-1-n}\dint_{2I} \big|v\big|^2 \, dX\right)^{\frac12}, 
\end{equation}
where we have used \eqref{eqinte1} and \eqref{eqinte2} for the second inequality.

Finally, the reverse H\"older inequality \eqref{revHol} with $\nabla^m u$ in place of $\nabla u$  was also proved in \cite[Theorem 24]{Bar16}.
	
With all the previous ingredients we are ready to invoke Theorem \ref{theor:N<S:HO} part $(iii)$ and then Theorem \ref{theor:eps-approx:HO} to obtain the desired estimates. 
\end{proof}

\appendix

\section{Sawtooths have UR boundaries} \label{appa}

To start, recall from \cite[Appendix A]{HMM2} the fact that the sawtooth regions and Carleson boxes inherit the ADR property. 
In \cite[Appendix A]{HMM2}, we treated simultaneously the case that the set $E$ is ADR, but not necessarily UR, and also the case
that $E$ is UR.  The point was that the Whitney regions in the two cases (and thus also the corresponding
sawtooth regions and Carleson boxes) were somewhat different. 
In any case, the reader can easily see that, with the notation introduced in Definition \ref{def:WD-struct},  the arguments in \cite[Appendix A]{HMM2} can be carried out for any ADR set $E$ and with $\{\W_Q\}_{Q\in\dd(E)}$ any Whitney-dyadic structure for $\ree\setminus E$ with some parameters $\eta$ and $K$. In turn, both if $E$ happens to be merely an ADR set as in Section \ref{sections:ADR}, or a UR set as in Section \ref{sections:UR}, the corresponding constructions of Whitney-dyadic structure fit within the previous framework. Nonetheless, the same applies to any other Whitney-dyadic structure (constructed in a different way) but retaining the same properties.

Let us now recall some results from \cite{HMM2} that we shall use in the sequel.

\begin{proposition}\cite[Proposition A.2]{HMM2} \label{prop:Sawtooths-ADR}
Let $E\subset\ree$ be an $n$-dimensional ADR set and let $\{\W_Q\}_{Q\in\dd(E)}$ be a Whitney-dyadic structure for $\ree\setminus E$ with some parameters $\eta\ll 1$ and $K\gg 1$. Then all dyadic local sawtooths $\Omega_{\mathcal{F},Q}$ and all Carleson boxes $T_Q$ have $n$-dimensional ADR boundaries. In all cases, the implicit constants are uniform and depend only on dimension, the ADR constant of $E$, parameters $\eta$, $K$, and the constant $C$ in Definition \ref{def:WD-struct} part $(iii)$.
\end{proposition}

\begin{remark}\label{remark:ADR-sawtooth}
	Let $\Omega\subset\ree$ be an open set with ADR boundary and let $\{\W_{Q}\}_{Q\in\dd(\pom)}$ be a Whitney-dyadic structure for $\Omega$ with parameters $\eta$ and $K$. One can easily construct a Whitney-dyadic structure $\{\W_{Q}'\}_{Q\in\dd(\pom)}$ for $\ree\setminus \pom$ so that for every $I\in\W(\Omega)$ one has that $I\in \W_Q$ if and only if $I\in\W_{Q}'$, that is, the new Whitney-dyadic structure remains the same for the Whitney cubes contained in $\Omega$. To construct such a 
	Whitney-dyadic structure we define $(\W'_Q)^0$ as in \eqref{eq3.1} with the same parameters $\eta$ and $K$ but for all the Whitney cubes $I\in \W(\ree\setminus\pom)$. For every $Q\in\dd(\pom)$ we the set $\W_{Q}':=\W_Q\cup  ((\W'_Q)^0\cap\W(\ree\setminus\overline{\Omega}))$. It is straightforward to see that $\{\W_{Q}'\}_{Q\in\dd(\pom)}$ is a Whitney-dyadic structure for $\ree\setminus \pom$  with parameters $\eta$ and $K$ and agreeing with  $\{\W_{Q}\}_{Q\in\dd(\pom)}$ when restricted to the Whitney cubes contained in $\Omega$. Note also that the constants in Definition \ref{def:WD-struct} part $(iii)$ are the same for both.
	
	We then note by Proposition \eqref{prop:Sawtooths-ADR} all the associated dyadic local sawtooths $\Omega_{\mathcal{F},Q}'$ and all Carleson boxes $T_Q'$ (contained in $\ree\setminus \pom$) have $n$-dimensional ADR boundaries. In turn the agreement of  $\{\W_{Q}\}_{Q\in\dd(\pom)}$ with  $\{\W_{Q}'\}_{Q\in\dd(\pom)}$ inside $\Omega$
	implies at the very least that all the associated dyadic local sawtooths $\Omega_{\mathcal{F},Q}$ and all Carleson boxes $T_Q$ (contained now in $\Omega$) have 
	a boundary satisfying the upper ADR condition (that is the upper estimate in \ref{eq1.ADR}) with constant depending on the ADR constant of $\pom$, $\eta$, $K$ and the constant in Definition \ref{def:WD-struct} part $(iii)$.
\end{remark}

In what follows  we assume that $E$ is an ADR set and fix $\{\W_Q\}_{Q\in\dd(E)}$ a Whitney-dyadic structure for $\ree\setminus E$ with some parameters $\eta$ and $K$. As mentioned in Section \ref{sPrelim}, we always assume that if $\{\W_Q\}_{Q\in\dd(E)}$ is a Whitney-dyadic structure for $\ree\setminus E$ with some parameters $\eta$ and $K$, then $K$ is large enough (say $K\ge 40^2\,n$) so that for any $\ell(I) \lesssim \diam(E)$ we have $I\in \W_{Q_I^*}^0\subset \W_{Q_I^*}$, where 
$Q_I^*$ is some fixed nearest dyadic cube to $I$ with $\ell(I) = \ell(Q_I^*)$. To simplify the notation, it is convenient to find $m_0\in \mathbb{Z}_+,\, C_0 \in\mathbb{R}_+$ (say $2^{m_0}\approx C\,max\{K,\eta^{-1}\}^{1/2}$, $C_0=CK^{1/2}$, hence depending on $\eta$, $K$ and the constant $C$ in Definition \ref{def:WD-struct} part $(iii)$)
such that
\begin{equation}\label{eqA.1a}
2^{-m_0}\,\ell(Q)\leq \ell(I)  \leq 2^{m_0}\ell(Q),\, {\rm and}\, \dist(I,Q)\leq C_0 \ell(Q)\,,\quad \forall I\in \W_Q\,.
\end{equation} 
From now, we will use this parameters $m_0$ and $C_0$, rather than $\eta$, $K$ and the constant $C$ in Definition \ref{def:WD-struct} part $(iii)$.

Let us recall some notation from \cite[Appendix A]{HMM2}.   Given a cube $Q_0\in\dd$ and a family $\mathcal{F}$ of disjoint cubes $\F=\{Q_j\}\subset \mathbb{D}_{Q_0}$ (for the case $\F=\emptyset$ the changes are
straightforward and we leave them to the reader, also the case $\F=\{Q_0\}$ is disregarded since in that case
$\Omega_{\mathcal{F},Q_0}$ is the null set). We write $\Omega_\star=\Omega_{\mathcal{F},Q_0}$ and $\Sigma=\pom_{\star}\setminus E$.  Given $Q\in\dd$ we set
$$
\R_{Q}:=\bigcup_{Q'\in \dd_{Q}}\W_{Q'},
\qquad\mbox{and}\qquad
\Sigma_Q
=
\Sigma\bigcap \Big(\bigcup_{I\in \R_Q} I\Big).
$$

Let $C_1$ be a sufficiently large constant,
to be chosen below,
depending on $n$, the ADR constant of $E$, $m_0$ and $C_0$.
Let us introduce some new collections:
\begin{align*}
\F_{||}
&:=
\big\{
Q\in\dd\setminus\{Q_0\}:
\ell(Q)=\ell(Q_0), \ \dist(Q,Q_0)\le C_1\,\ell(Q_0)
\big\},
\\
\F_{\top}
&:=
\big\{
Q'\in \dd:
\dist(Q', Q_0)\le C_1\,\ell(Q_0),\
\ell(Q_0)<\ell(Q')\le C_1\,\ell(Q_0)
\big\},
\\
\F_{||}^*:
&=
\big\{Q\in\F_{||}: \Sigma_Q\neq\emptyset\big\}
=
\big\{Q\in\F_{||}: \exists\,I\in\R_{Q} \mbox{ such that } \Sigma\cap I\neq\emptyset\big\},
\\
\F^*:
&=
\big\{Q\in\F: \Sigma_Q\neq\emptyset\big\}
=
\big\{Q\in\F: \exists\,I\in\R_{Q} \mbox{ such that } \Sigma\cap I\neq\emptyset\big\},
\end{align*}
We also set
$$
\R_{\bot}=\bigcup_{Q\in\F^*} \R_Q,
\qquad\quad
\R_{||}=\bigcup_{Q\in\F_{||}^*} \R_Q,
\qquad\quad
\R_{\top}=\bigcup_{Q\in  \F_{\top}} \W_Q.
$$

\begin{lemma}\cite[Lemma A.3]{HMM2} \label{lemma:decomp-bdt-sawtooth}
Set $\W_\Sigma=\{I\in\W: I\cap \Sigma\neq\emptyset\}$ and define
\begin{align*}
\W_\Sigma^{\bot}
=
\bigcup_{Q\in\F^*} \W_{\Sigma,Q},
\qquad
\W_\Sigma^{||}
=
\bigcup_{Q\in\F_{||}^*} \W_{\Sigma,Q},
\qquad
\W_{\Sigma}^{\top}
=
\big\{I\in \W_\Sigma:Q_I^*\in\F_{\top}\big\}.
\end{align*}
where for every $Q\in\F^*\cup \F_{||}^*$ we set
$$\W_{\Sigma,Q}
=
\big\{I\in \W_\Sigma:Q_I^*\in\dd_{Q}\};
$$
and where we recall that $Q_I^*$ is the nearest dyadic cube to $I$ with $\ell(I) = \ell(Q_I^*)$ as defined above.
Then
\begin{equation}\label{decomp:Sigma:R}
\W_\Sigma
=
\W_\Sigma^{\bot}\cup\W_\Sigma^{||}\cup \W_\Sigma^{\top},
\end{equation}
where
\begin{equation}\label{decomp:Sigma:R:conta}
\W_\Sigma^{\bot}\subset \R_{\bot},
\qquad
\W_\Sigma^{||}\subset \R_{||},
\qquad
\W_\Sigma^{\top}\subset \R_{\top}.
\end{equation}
As a consequence,
\begin{equation}\label{decomp:Sigma}
\Sigma
=
\Sigma_{\bot}\cup\Sigma_{||}\cup\Sigma_{\top}
:=
\Big(\bigcup_{I\in\W_\Sigma^{\bot}} \Sigma\cap I\Big)
\bigcup
\Big(\bigcup_{I\in\W_\Sigma^{||}} \Sigma\cap I\Big)
\bigcup
\Big(\bigcup_{I\in\W_\Sigma^{\top}} \Sigma\cap I\Big).
\end{equation}
\end{lemma}

\begin{lemma}\cite[Lemma A.7]{HMM2} \label{lemma:I-QI}
Given $I\in \W_\Sigma$, we can find $Q_I\in\dd$,  with $Q_I\subset Q_I^*$, such
that $\ell(I)\approx \ell(Q_I)$,  $\dist(Q_I,I)\approx \ell(I)$, and in addition,
\begin{equation}\label{Sigma:bdd-overlap:I}
\sum_{I\in\W_{\Sigma,Q}} 1_{Q_I}\lesssim 1_{Q},\qquad\mbox{for any }Q\in\F^*\cup\F^*_{||},
\end{equation}
and
\begin{equation}\label{Sigma:bdd-overlap:II}
\sum_{I\in\W_\Sigma^\top} 1_{Q_I}\lesssim 1_{B_{Q_0}^*\cap E},
\end{equation}
where the implicit constants depend on $n$, the ADR constant of $E$, $m_0$ and $C_0$, and where $B_{Q_0}^*=B(x_{Q_0}, C\,\ell(Q))$ with $C$ large enough depending on the same parameters.
\end{lemma}

With the preceding results in hand, we turn to the main purpose of this appendix:
to prove that uniform rectifiability is also inherited by the sawtooth domains and Carleson boxes. 


\begin{proposition}\label{prop:Sawtooths-UR}
Let $E\subset\ree$ be an $n$-dimensional UR set and let $\{\W_Q\}_{Q\in\dd(E)}$ be a Whitney-dyadic structure for $\ree\setminus E$ with some parameters $\eta\ll 1$ and $K\gg 1$. Then all dyadic local sawtooths $\Omega_{\mathcal{F},Q}$ and all Carleson boxes $T_Q$ have $n$-dimensional UR boundaries. In all cases, the implicit constants are uniform and depend only on dimension, the UR character of $E$, and the parameters $m_0$ and $C_0$ (hence on the parameters $\eta$, $K$, and the constant $C$ in Definition \ref{def:WD-struct} part $(iii)$).
\end{proposition}

The proof of this result follows the ideas from \cite[Appendix C]{HM-I} which in turn uses some ideas from Guy David,
and uses the following singular integral characterization of UR sets, established in \cite{DS1}.
Suppose that
$E\subset\ree$ is $n$-dimensional ADR.
The singular integral operators that we shall consider are those of the form
$$T_{E,\eps}f(x)=T_{\eps} f(x):= \int_{E} \mathcal{K}_\eps (x-y)\,f(y)\,dH^n(y)\,,$$
where $\mathcal{K}_\eps(x) := \mathcal{K} (x)\,\Phi(|x|/\eps)$,  with $0\leq \Phi\leq 1$,
$\Phi(\rho)\equiv 1$ if $ \rho\geq 2,$ $\Phi(\rho) \equiv 0$ if $\rho\leq 1$, and $\Phi \in
C^\infty(\mathbb{R})$, and where
the singular kernel $\mathcal{K} $ is an odd function, smooth on $\ree\setminus\{0\}$,
and satisfying
\begin{align}\label{eqC.1}
|\mathcal{K}(x)|&\,\leq\, C_0\,|x|^{-n}
\\[4pt]\label{eqC.2}
|\nabla^m \mathcal{K}(x)|&\,\leq \,C_m\,|x|^{-n-m}\,,\qquad\forall m\ge 1.
\end{align}
Then $E$ is UR if and only if for every such kernel $\mathcal{K}$, we have that
\begin{equation}\label{eqC.3}
\sup_{\eps>0}\int_E |T_{E,\eps} f|^2\,dH^n\leq C_\mathcal{K} \int_E|f|^2\,dH^n.
\end{equation}
We refer the reader to \cite{DS1} for the proof.  For $\mathcal{K}$ as above,
set
\begin{equation}\label{eqC.4}
\T_E f(X):= \int_E \mathcal{K}(X-y)\,f(y)\,dH^n(y)\,,\qquad X \in \ree\setminus E.
\end{equation}
We define (possibly disconnected) non-tangential
approach regions $\Upsilon_\alpha(x)$ as follows.  Set $\mathcal{W}_\alpha(x):= \{I\in \mathcal{W}: \dist(I,x) <\alpha \ell(I)\}$.
Then we define
\[
\Upsilon_\alpha(x):=\bigcup_{I\in \mathcal{W}_\alpha(x)} I^*
\]
(thus, roughly speaking, $\alpha$ is the ``aperture'' of $\Upsilon_\alpha(x)$).
Here $I^*=I^*(\tau)$ as in Section \ref{sPrelim} with $0<\tau\le \tau_0/4$, which is fixed. Note that these non-tangential
approach regions are slightly different that the ones introduced in \eqref{defcone} since they do not use the Whitney regions $U_Q$.
 For $F\in \mathcal{C}(\ree\setminus E)$ we
may then also define a new non-tangential maximal function (which is different that the one \eqref{defN*} although somehow comparable much as in by Remark \ref{remark:rcones})
\[
\mathcal{N}_{*,\alpha} F(x):= \sup_{Y\in \Upsilon_\alpha(x)}|F(Y)|.
\]
We shall sometimes write simply $\mathcal{N}_*$ when there is no chance of confusion
in leaving implicit the dependence on the aperture $\alpha$.
The following lemma is a standard consequence of the usual
Cotlar inequality for maximal singular integrals, and we omit the proof.

\begin{lemma}\label{lemmacotlar}
Suppose that $E\subset \ree$ is $n$-dimensional UR, and let
$\T_E$ be defined as in \eqref{eqC.4}.  Then for each $1<p<\infty$ and $\alpha \in (0,\infty)$,
there is a constant $C_{p,\alpha,\mathcal{K}}$ depending only on $p,n,\alpha, \mathcal{K}$ and the UR character of $E$ such that
\begin{equation}\label{eqC.6}
\int_E \left(\mathcal{N}_{*,\alpha}\left(\T_E f\right)\right)^p\, dH^n \,\leq\, C_{\alpha,\mathcal{K}} \int_E |f|^p dH^n.
\end{equation}
\end{lemma}

\medskip

\begin{proof}[Proof of Proposition \ref{prop:Sawtooths-UR}]
Write $\sigma=H^n|_E$. We fix $Q_0\in\dd=\dd(E)$ and a family $\mathcal{F}$ of disjoint cubes $\F=\{Q_j\}\subset \mathbb{D}_{Q_0}$ (for the case $\F=\emptyset$ the changes are
straightforward and we leave them to the reader, also the case $\F=\{Q_0\}$ is disregarded since $\Omega_{\mathcal{F},Q_0}$ is the null set). We write $\Omega_\star=\Omega_{\mathcal{F},Q_0}$, $E_\star=\pom_{\star}$, and $\sigma_\star=H^n|_{E_\star}$.
We fix  $0\leq \Phi\leq 1$,
$\Phi(\rho)\equiv 1$ if $ \rho\geq 2,$ $\Phi(\rho) \equiv 0$ if $\rho\leq 1$, and $\Phi \in
C^\infty(\mathbb{R})$. According to the  previous considerations we fix $\epsilon_0>0$ and our goal is to show that $T_{E_\star,\epsilon_0}$ is bounded on $L^2(E_\star)$ with bounds that are independent of $\epsilon_0$. To simplify the notation we write $\mathcal{K}_0=\mathcal{K}_{\epsilon_0}$ and set, for every $X\in\ree$,
$$
\T_{E,0}f(X) = \int_E \mathcal{K}_0(X-y)\,f(y)\,d\sigma(y),
\qquad
\T_{E_\star,0}g(X) = \int_{E_\star} \mathcal{K}_0(X-y)\,g(y)\,d\sigma_\star(y).
$$
We first observe that $\mathcal{K}_0$ is not singular and therefore, for any $p$, $1<p<\infty$, and for every  $f\in L^p(E)$, respectively $g\in L^p(E_\star)$, the previous operators are well-defined (by means of an absolutely convergent integral) for every $X\in\ree$.
Also for such functions it is easy to see that the dominated convergence
theorem implies that $\T_{E,0}f, \T_{E_\star,0}g\in C(\ree)$.

\begin{remark}\label{remark:K0}
We notice that $\mathcal{K}_0$ is an odd smooth function which satisfies \eqref{eqC.1} and \eqref{eqC.2} with uniform constants (i.e. with no dependence on $\epsilon_0$) and therefore the fact that  $E$ is UR implies that \eqref{eqC.3} and \eqref{eqC.6} hold with constants that do not depend on $\epsilon_0$.
\end{remark}

We are going to see that $\T_{E,0}:L^p(E)\longrightarrow L^p(E_\star)$ for every $1<p<\infty$. To do that we take $f\in L^p(E)$ and write
\begin{align*}
\int_{E_\star} |\T_{E,0}f (x)|^p\,d\sigma_\star(x)
=
\int_{E_\star\cap E} |\T_{E,0}f (x)|^p\,d\sigma_\star(x)
+
\int_{E_\star\setminus E} |\T_{E,0}f (x)|^p\,d\sigma_\star(x)
=:\textrm{I}+\textrm{II}.
\end{align*}
The estimate for $\textrm{I}$ follows from the fact that $E$ is UR
$$
\textrm{I}
\le
\int_{E} |\T_{E,0}f (x)|^p\,d\sigma(x)
=
\int_{E} |T_{E,\epsilon_0}f (x)|^p\,d\sigma(x)
\le
C_\mathcal{K} \int_{E} |f (x)|^p\,d\sigma(x)
$$
where we have used \eqref{eqC.3} and the standard Calder\'on-Zygmund theory (taking place in the ADR set $E$)  and $C_\mathcal{K}$ does not depend on $\epsilon_0$.
For $\textrm{II}$ we use that
$\Sigma=E_\star\setminus E=\pom_{\star}\setminus E$  and invoke Lemmas \ref{lemma:decomp-bdt-sawtooth}
and \ref{lemma:I-QI}; let $Q_I$ be the cube constructed in the latter, so that
$$
\textrm{II}
=
\sum_{I\in \W_\Sigma} \int_{I\cap\Sigma} |\T_{E,0}f (x)|^p\,d\sigma_\star(x)
=
\sum_{I\in \W_\Sigma} \fint_{Q_I}\int_{I\cap\Sigma} |\T_{E,0}f (x)|^p\,d\sigma_\star(x)\,d\sigma(y).
$$
Note that if $y\in Q_I$ and $x\in\Sigma\cap I$ then $\dist(I,y)\lesssim \ell(Q_I)\approx\ell(I)$. Then taking  $\alpha>0$ large enough we obtain that $I\subset \W_{\alpha}(y)$. Write $\widetilde{\F}=\F^*\cup \F^*_{||} $,
and observe that by construction the cubes in $\widetilde{\F}$ are pairwise disjoint. Then
by the ADR property of $E_\star$, along with
Lemmas \ref{lemma:decomp-bdt-sawtooth} and \ref{lemma:I-QI},
\begin{align*}
\textrm{II}
&\le
\sum_{I\in \W_\Sigma} \sigma_\star(\Sigma\cap I)\,\fint_{Q_I} |\mathcal{N}_{*,\alpha}(\T_{E,0}f) (y)|^p\,d\sigma(y)
\\
&
\lesssim
\sum_{Q\in\widetilde{\F}}
\sum_{I\in\W_{\Sigma,Q}}
\int_{Q_I} |\mathcal{N}_{*,\alpha}(\T_{E,0}f) (y)|^p\,d\sigma(y)
+
\sum_{I\in\W_\Sigma^\top}
\int_{Q_I} |\mathcal{N}_{*,\alpha}(\T_{E,0}f) (y)|^p\,d\sigma(y)
\\
&\lesssim
\sum_{Q\in\widetilde{\F}}
\int_{Q} |\mathcal{N}_{*,\alpha}(\T_{E,0}f) (y)|^p\,d\sigma(y)
+
\int_{B_{Q_0}^*\cap E} |\mathcal{N}_{*,\alpha}(\T_{E,0}f) (y)|^p\,d\sigma(y)
\\
&\lesssim
\int_{E}|\mathcal{N}_{*,\alpha}(\T_{E,0}f) (y)|^p\,d\sigma(y)
\\
&\lesssim
\int_{E}|f(y)|^p\,d\sigma(y),
\end{align*}
where in the last estimate we have employed Lemma \ref{lemmacotlar} and Remark \ref{remark:K0}, and the implicit constants do not depend on $\epsilon_0$.

We have thus established that $\T_{E,0}:L^p(E)\longrightarrow L^p(E_\star)$ for every $1<p<\infty$. Since $\mathcal{K}$ is odd, so is $\mathcal{K}_0$, and by duality we therefore obtain that
\begin{equation}\label{Tcal-E-star-bded}
\T_{E_\star,0}:L^p(E_\star)\longrightarrow L^p(E),
\qquad 1<p<\infty.
\end{equation}
Our goal is to show that $\T_{E_\star,0}:L^2(E_\star)\longrightarrow L^2(E_\star)$ with bounds that do not depend on $\epsilon_0$. Note that $\T_{E_\star,0}f$ is a continuous function for every $f\in L^2(E_\star)$ and therefore $\T_{E_\star,0}f\big|_{E_\star}=T_{E_\star,\epsilon_0}f$ everywhere on $E_\star$.

We take $f\in L^2(E_\star)$ and write as before
\begin{multline}\label{TE:2-terms}
\int_{E_\star} |\T_{E_\star,0}f (x)|^2\,d\sigma_\star(x)
=
\int_{E_\star\cap E} |\T_{E_\star,0}f (x)|^2\,d\sigma_\star(x)
+
\sum_{I\in \W_\Sigma} \int_{I\cap\Sigma} |\T_{E_\star,0}f (x)|^2\,d\sigma_\star(x)
\\
=:
\textrm{I}+\sum_{I\in \W_\Sigma} \textrm{II}_I
=
\textrm{I}+\textrm{II}.
\end{multline}
For $\textrm{I}$ we use \eqref{Tcal-E-star-bded} with $p=2$ and conclude the desired estimate
\begin{equation}\label{TE:A}
\textrm{I}
\le
\int_{E_\star\cap E} |\T_{E_\star,0}f (x)|^2\,d\sigma_\star(x)
\le
\int_{E} |\T_{E_\star,0}f (x)|^2\,d\sigma(x)
\le
\int_{E_\star} |f (x)|^2\,d\sigma_\star(x).
\end{equation}
We next fix $I\in \W_\Sigma$ and estimate each $\textrm{II}_I$. Let $M>2$ be large parameter to be chosen below and set $\zeta_I=\ell(I)/M$, $\xi_I=M\,\ell(I)$. Write
\begin{multline}\label{decomp-K0}
\mathcal{K}_0(x)
=
\mathcal{K}_0(x)\,\Phi\Big(\frac{|x|}{\xi_I}\Big)
+
\mathcal{K}_0(x)\,\Big(\Phi\Big(\frac{|x|}{\zeta_I}\Big)- \Phi\Big(\frac{|x|}{\xi_I}\Big)\Big)
+
\mathcal{K}_0(x)\,\Big(1-\Phi\Big(\frac{|x|}{\zeta_I}\Big)\Big)
\\
=:
\mathcal{K}_{0,\xi_I}(x)+
\mathcal{K}_{0,\zeta_I,\xi_I}(x)
+
\mathcal{K}_{0}^{\zeta_I}(x).
\end{multline}
Corresponding to any of these kernels we respectively set the operators $\T_{E_\star, 0,\xi_I}$, $\T_{E_\star, 0,\zeta_I,\xi_I}$ and $\T_{E_\star, 0}^{\zeta_I}$.

We start with $\T_{E_\star, 0,\xi_I}$. Fix $x\in\Sigma\cap I$. Write $\Delta_{\star,I}=B(x, \xi_I)\cap E_\star$ and split $f=f_1+f_2:=f\,1_{\Delta_{\star,I}}+ f\,1_{E_\star\setminus \Delta_{\star,I}}$. Then we use Remark \ref{remark:K0}, the fact $\supp \Phi\subset [1,\infty)$ and that $E_\star$ is ADR to easily obtain that for every $y\in Q_I$,
with $Q_I$ as in Lemma  \ref{lemma:I-QI},
\begin{multline}\label{eqn:Cotlar-T1-loc}
|\T_{E_\star, 0,\xi_I}f_1 (x)|+|\T_{E_\star, 0,\xi_I}f_1 (y)|
\\
\le
\int_{\Delta_{\star,I}} \Big(|\mathcal{K}_0(x-z)|\,\Phi\Big(\frac{|x-z|}{\xi_I}\Big)+|\mathcal{K}_0(y-z)|\,\Phi\Big(\frac{|y-z|}{\xi_I}\Big)\Big)\,|f(z)|\,d\sigma_\star(z)
\\
\lesssim
\frac{1}{\xi_I^n}\,\int_{\Delta_{\star,I}} \,|f(y)|\,d\sigma_\star(z)
\approx
\fint_{\Delta_{\star,I}} \,|f(y)|\,d\sigma_\star(z)
\le
M_{E_\star} f(x),
\end{multline}
where $M_{E_\star}$ is the Hardy-Littlewood maximal function on $E_\star$, and the constants are independent of $\epsilon_0$ and $I$.

On the other hand, much as before we have that $\mathcal{K}_{0,\xi_I}$ is a Calder\'on-Zyg\-mund kernel with constants that are uniform in $\epsilon_0$ and $\xi_I$. Also, if $M$ is taken large enough we have that $2\,|x-y|<M\,\ell(I)\le |x-z|$ for every $z\in E_\star\setminus \Delta_{\star,I}$, $x\in\Sigma\cap I$ and $y\in Q_I$. Therefore using standard Calder\'on-Zygmund estimates and the fact that $E_\star$ is ADR we obtain that for every and $y\in Q_I$
\begin{multline}\label{eqn:Cotlar-T1-glob}
|\T_{E_\star, 0,\xi_I}f_2 (x)-\T_{E_\star, 0,\xi_I}f_2 (y)|
\le
\int_{E_\star\setminus \Delta_{\star,I}} \big|\mathcal{K}_{0,\xi_I}(x-z)-\mathcal{K}_{0,\xi_I}(y-z)\big|\,|f(z)|\,d\sigma_\star(z)
\\
\lesssim
\int_{E_\star\setminus \Delta_{\star,I}} \frac{|x-y|}{|x-z|^{n+1}}\,|f(z)|\,d\sigma_\star(z)
\lesssim_M
M_{E_\star} f(x).
\end{multline}
We next use \eqref{eqn:Cotlar-T1-loc} and \eqref{eqn:Cotlar-T1-glob} to conclude that
\begin{multline*}
\Big|\T_{E_\star, 0,\xi_I}f (x)-\fint_{Q_I} \T_{E_\star, 0,\xi_I}f (y)\,d\sigma(y)\Big|
\lesssim
|\T_{E_\star, 0,\xi_I}f_1 (x)|+\fint_{Q_I} |\T_{E_\star, 0,\xi_I}f_1 (y)|\,d\sigma(y)
\\
+
\fint_{Q_I}|\T_{E_\star, 0,\xi_I}f_2 (x)-\T_{E_\star, 0,\xi_I}f_2 (y)|\,d\sigma(y)
\lesssim
M_{E_\star} f(x),
\end{multline*}
which in turn yields
\begin{align}\label{eqn:Cotlar-T1-substr}
\int_{\Sigma\cap I}\Big|\T_{E_\star, 0,\xi_I}f (x)-\fint_{Q_I} \T_{E_\star, 0,\xi_I}f (y)\,d\sigma(y)\Big|^2\,d\sigma_\star(x)
\lesssim
\int_{\Sigma\cap I}M_{E_\star} f(x)^2\,d\sigma_\star(x).
\end{align}

We next introduce another operator
$$
T_{E_\star, 0,\xi_I}f(y)
=
\int_{z\in E_\star: |y-z|\ge \xi_I} \mathcal{K}_0(y-z)\,f(z)\,d\sigma_\star(z),
\qquad y\in E.
$$
We fix $x\in\Sigma\cap I$ and $y\in Q_I$. We first observe that, for $M$ large enough, Remark \ref{remark:K0} and the ADR property for $E_\star$ imply that
\begin{multline*}
\big|\T_{E_\star, 0,\xi_I}f (y)-T_{E_\star, 0,\xi_I}f(y)\big|
\le
\int_{E_\star} |\mathcal{K}_0(y-z)| \,\bigg|\Phi\Big(\frac{|y-z|}{\xi_I}\Big)-1_{[1,\infty)}\Big(\frac{|y-z|}{\xi_I}\Big)\bigg|\,|f(z)|\,d\sigma_\star(z)
\\
\lesssim
\frac1{\xi_I^n}\int_{z\in E_\star:|y-z|\le 2\,\xi_I} |f(z)|\,d\sigma_\star(z)
\lesssim
\frac1{\xi_I^n}\int_{z\in E_\star:|x-z|\le 3\,\xi_I} |f(z)|\,d\sigma_\star(z)
\lesssim
M_{E_\star} f(x).
\end{multline*}
On the other hand, we can introduce another decomposition
$$f=f_3+f_4:=f\,1_{B(y,\xi_I)\cap E_\star}+ f\,1_{E_\star\setminus B(y,\xi_I)}\,,$$
and then for every $\bar{y}\in Q_I$
\begin{multline}\label{eqn:Cotlar-T1-rough:decomp}
|T_{E_\star, 0,\xi_I}f(y)|
=
|\T_{E_\star,0} f_4(y)|
\le
|\T_{E_\star,0} f_4(y)-\T_{E_\star,0} f_4(\bar{y})|+|\T_{E_\star,0} f_4(\bar{y})|
\\
\le
|\T_{E_\star,0} f_4(y)-\T_{E_\star,0} f_4(\bar{y})|+|\T_{E_\star,0} f(\bar{y})|+|\T_{E_\star,0} f_3(\bar{y})|.
\end{multline}
We estimate each term in turn. We first observe that, for $M$ large enough, $2\,|y-\bar{y}|<M\,\ell(I)\le |y-z|$ for every $z\in E_\star\setminus B(y,\xi_I)$ and $\bar{y}\in Q_I$. Therefore, using standard Calder\'on-Zygmund estimates and the fact that $E_\star$ is ADR, we obtain that for every and $\bar{y}\in Q_I$
\begin{multline}\label{eqn:Cotlar-T1-rough-glob}
|\T_{E_\star,0} f_4(y)-\T_{E_\star,0} f_4(\bar{y})|
\le
\int_{E_\star\setminus B(y,\xi_I)} |\mathcal{K}_0(y-z)-\mathcal{K}_0(\bar{y}-z)|\,|f(z)|\,d\sigma_\star(z)
\\
\lesssim
\int_{E_\star\setminus B(y,\xi_I)} \frac{|y-\bar{y}|}{|y-z|^{n+1}}\,|f(z)|\,d\sigma_\star(z)
\lesssim
M_{E_\star} f(x),
\end{multline}
where we have used that, for $M$ large enough, $x\in B(y,\xi_I/2)$.
Fix $1<p<2$. We next average \eqref{eqn:Cotlar-T1-rough:decomp} on $\bar{y}\in Q_I$ and use \eqref{eqn:Cotlar-T1-rough-glob} and \eqref{Tcal-E-star-bded} to obtain
\begin{align}\label{eqn:Cotlar-T1-rough}
&|T_{E_\star, 0,\xi_I}f(y)|
\\ \nonumber
&\quad\le
\fint_{Q_I} \big(|\T_{E_\star,0} f_4(y)-\T_{E_\star,0} f_4(\bar{y})|+|\T_{E_\star,0} f(\bar{y})|+|\T_{E_\star,0} f_3(\bar{y})|\big)\,d\sigma(\bar{y})
\\ \nonumber
&\quad\lesssim
M_{E_\star} f(x)+M_E (\T_{E_\star,0} f)(y)+\sigma(Q_I)^{-\frac1p}\,\|\T_{E_\star,0} f_3\|_{L^p(E)}
\\ \nonumber
&\quad\lesssim
M_{E_\star} f(x)+M_E (\T_{E_\star,0} f)(y)+\sigma(Q_I)^{-\frac1p}\,\|f_3\|_{L^p(E_{\star})}
\\ \nonumber
&\quad\lesssim
M_{E_\star} f(x)+M_E (\T_{E_\star,0} f)(y)+\Big(\frac1{\ell(I)^n}\int_{B(y,\xi_I)\cap E_\star} |f(z)|^p\,d\sigma_\star(z)\Big)^\frac1p
\\ \nonumber
&\quad\lesssim
M_{E_\star,p} f(x)+M_E (\T_{E_\star,0} f)(y),
\end{align}
where $M_{E}$ is the Hardy-Littlewood maximal function on $E$ and we also write $M_{E_\star,p} f=M_{E_\star}(|f|^p)^\frac1p$. Note that this estimate holds for every $x\in \Sigma\cap I$ and for every $y\in Q_I$. Hence,
\begin{multline}\label{eqn:Cotlar-T1:ct}
\int_{\Sigma\cap I}\Big|\fint_{Q_I} \T_{E_\star, 0,\xi_I}f (y)\,d\sigma(y)\Big|^2\,d\sigma_\star(x)
\\
\lesssim
\int_{\Sigma\cap I}M_{E_\star,p} f(x)^2\,d\sigma_\star(x)
+
\int_{Q_I}M_E (\T_{E_\star,0} f)(y)^2\,d\sigma(y),
\end{multline}
where we have used that $\sigma_\star(\Sigma\cap I)\lesssim \ell(I)^n$. We now gather \eqref{eqn:Cotlar-T1-substr} and \eqref{eqn:Cotlar-T1:ct} to obtain that for every $I\in\W_\Sigma$
\begin{align}\label{eqn:Cotlar-T1}
&\int_{\Sigma\cap I} \big|\T_{E_\star, 0,\xi_I}f (x)\big|^2\,d\sigma_\star(x)
\\ \nonumber
&\qquad\lesssim
\int_{\Sigma\cap I} \Big|\T_{E_\star, 0,\xi_I}f (x)-\fint_{Q_I} \T_{E_\star, 0,\xi_I}f (y)\,d\sigma(y)\Big|^2\,d\sigma_\star(x)
\\ \nonumber
&\qquad\qquad
+
\int_{\Sigma\cap I}\Big|\fint_{Q_I} \T_{E_\star, 0,\xi_I}f (y)\,d\sigma(y)\Big|^2\,d\sigma_\star(x)
\\ \nonumber
&\qquad
\lesssim
\int_{\Sigma\cap I}M_{E_\star,p} f(x)^2\,d\sigma_\star(x)
+
\int_{Q_I}M_E (\T_{E_\star,0} f)(y)^2\,d\sigma(y).
\end{align}

We next consider $\T_{E_\star, 0,\zeta_I,\xi_I}$.  Note that for every $x\in \Sigma\cap I$ and $z\in E_\star$ we have
\begin{align*}
|\mathcal{K}_{0, \zeta_I,\xi_I}(z-x)|
=
|\mathcal{K}_0(z-x)|\,\Big|\Phi\Big(\frac{|z-x|}{\zeta_I}\Big)- \Phi\Big(\frac{|z-x|}{\xi_I}\Big)\Big|
\lesssim
\frac1{|z-x|^n}\,1_{\zeta_I\le |z-x|\le\, {2} \xi_I}\lesssim
\frac1{\zeta_I^n}\,1_{|z-x|\le\, {2}\xi_I},
\end{align*}
and therefore
\begin{multline}\label{eqn:Cotlar-T2}
\int_{\Sigma\cap I} |T_{E_\star, 0, \zeta_I,\xi_I}f(x)|^2\,d\sigma_\star(x)
\lesssim
\int_{\Sigma\cap I} \bigg(
\frac1{\zeta_I^n}\,\int_{B(x,2\,\xi_I)\cap E_\star}|f(z)|\,d\sigma_\star(z)
\bigg)^2\,d\sigma_\star(x)
\\
\lesssim_M\int_{\Sigma\cap I} M_{E_\star} f(x)^2\,d\sigma_\star(x).
\end{multline}

Let us finally address $\T_{E_\star, 0}^{\zeta_I}$. Observe first that
$$
\mathcal{K}_{0}^{\zeta_I}(\cdot)
=
\mathcal{K}(\cdot)\,\Phi\Big(\frac{|\cdot|}{\epsilon_0}\Big)\,\Big(1-\Phi\Big(\frac{|\cdot|}{\zeta_I}\Big)\Big).
$$
We consider three different cases.

\noindent\textbf{Case 1:} $\zeta_I\le \frac{\epsilon_0}2$. We have that $\mathcal{K}_{0}^{\zeta_I}\equiv 0$ and thus $\T_{E_\star, 0}^{\zeta_I}\equiv 0$.

\smallskip

\noindent\textbf{Case 2:}  $\frac{\epsilon_0}2< \zeta_I\le 2\,\epsilon_0 $. In this case for every $x\in \Sigma\cap I$ and $z\in E_\star$
$$
|\mathcal{K}_{0}^{\zeta_I}(x-z)|
\lesssim
\frac1{|x-z|^n}\,1_{\epsilon_0\le |z-x|\le 2\,\zeta_I}\lesssim
\frac1{\epsilon_0^n}\,1_{|z-x|\le 4\,\epsilon_0},
$$
and therefore
\begin{multline}\label{eqn:Cotlar-T3:case2}
\int_{\Sigma\cap I} |\T_{E_\star, 0}^{\zeta_I} f(x)|^2\,d\sigma_\star(x)
\lesssim
\int_{\Sigma\cap I} \bigg(
\frac1{\epsilon_0^n}\,\int_{B(x,4\,\epsilon_0)\cap E_\star}|f(z)|\,d\sigma_\star(z)
\bigg)^2\,d\sigma_\star(x)
\\
\lesssim
\int_{\Sigma\cap I} M_{E_\star} f(x)^2\,d\sigma_\star(x)
\end{multline}
where the implicit constants are independent of $\epsilon_0$ and $\zeta_I$.
\smallskip

\noindent\textbf{Case 3:}  $\zeta_I> 2\, \epsilon_0$. In this case $\T_{E_\star, 0}^{\zeta_I} f$ is a double truncated integral whose smooth Calder\'on-Zygmund kernel $\mathcal{K}_{0}^{\zeta_I}$ is odd, smooth in $\ree$ and
satisfies the estimates \eqref{eqC.1}, \eqref{eqC.2}.
with uniform bounds (i.e., independent of $\epsilon_0$ and $\zeta_I$). Fix $z_I\in \Sigma\cap I$ and notice that if $x\in\Sigma\cap I$ and $z\in B(x,2\,\zeta_I)\cap E_\star$ then, taking $M$ large enough, we have
$$
|z-z_I|
\le
|z-x|+|x-z_I|
\le
2\,\zeta_I+\diam(I)
=
\frac{\ell(I)}{2\,M}+\diam(I)
<\frac32\,\diam(I)
$$
and therefore the fact that $\supp \mathcal{K}_{0}^{\zeta_I}\subset B(0,2\,\zeta_I)$ immediately gives
$\T_{E_\star, 0}^{\zeta_I} f(x)=\T_{E_\star, 0}^{\zeta_I} (f\,1_{\widetilde{\Delta}_{\star, I}})(x)$ where $\widetilde{\Delta}_{\star, I}:=\widetilde{B}_{\star, I}\cap E_\star:= B(z_I,2\,\diam(I))\cap E_\star$. Note that \eqref{Whintey-4I} yields
$$
4\,\diam(I) \le
\dist(4\,I,E)
\le
\dist(z_I,E)
\le
\dist(\widetilde{B}_{\star, I},E)+2\,\diam(I)
$$
and therefore $\dist(\widetilde{B}_{\star, I},E)\ge 2\,\diam(I)$. This implies that $\frac32\,\widetilde{B}_{\star, I}\subset\ree\setminus E$. Also if $J\in\W$ satisfies that  $J^*\cap \widetilde{B}_{\star, I}\neq\emptyset$ we can easily check that $\ell(I)\approx\ell(J)$ and $\dist(I,J)\lesssim\ell(I)$. This implies that only a bounded number of $J$'s have the property that $J^*$ intersects $\widetilde{B}_{\star, I}$. We recall that $\Sigma=E_\star\setminus E$ is a union of portion of faces of fattened Whitney cubes $J^*$. Thus we have
$$
\widetilde{\Delta}_{\star, I}\subset \bigcup_{m=1}^{M_0} F_{m,I}\,,
$$
where $M_0$ is a uniform constant and each $F_{m,I}$ is either
a portion of a face of some $J^*$, or else $F_{m,I}=\emptyset$ (since $M_0$ is not necessarily
equal to the number of faces, but is rather an upper bound for the number of faces.) Note also that $I\subset \widetilde{B}_{\star,I}$ and therefore we also have that
$$
\Sigma\cap I\subset \bigcup_{m=1}^{M_0} F_{m,I}\,.
$$
Thus
\begin{multline*}
\int_{\Sigma\cap I} |\T_{E_\star, 0}^{\zeta_I} f(x)|^2\,d\sigma_{\star}(x)
=
\int_{\Sigma\cap I} |\T_{E_\star, 0}^{\zeta_I} (f\,1_{\widetilde{\Delta}_{\star, I}})(x)|^2\,d\sigma_{\star}(x)
\\
\lesssim
\sum_{1\le m,m'\le M_0} \int_{F_{m,I}} |\T_{E_\star, 0}^{\zeta_I} (f\,1_{F_{m',I}})(x)|^2\,d\sigma_{\star}(x).
\end{multline*}
In the case $m=m'$ we take the hyperplane $\mathcal{H}_{m,I}$ with $F_{m,I}\subset \mathcal{H}_{m,I}$ and then
\begin{multline*}
\int_{F_{m,I}} |\T_{E_\star, 0}^{\zeta_I} (f\,1_{F_{m,I}})(x)|^2\,d\sigma_{\star}(x)
\le
\int_{\mathcal{H}_{m,I}} |\T_{\mathcal{H}_{m,I}, 0}^{\zeta_I} (f\,1_{F_{m,I}})(x)|^2\,dH^n(x)
\\
\lesssim
\int_{F_{m,I}} |f(x)|^2\,dH^n(x)
=
\int_{F_{m,I}} |f(x)|^2\,d\sigma_\star(x),
\end{multline*}
where, after a rotation, we have used the $L^2$ bounds of Calder\'on-Zygmund operators with nice kernels on $\re^n$.
For $m\neq m'$ we consider two cases: either $\dist(F_{m, I}, F_{m',I})\approx\ell(I)$ or $\dist(F_{m, I}, F_{m',I})\ll\ell(I)$. In the first scenario, using that $\mathcal{K}_0^{\zeta_I}$ satisfies \eqref{eqC.1} uniformly we obtain that
\begin{multline*}
\int_{F_{m,I}} |\T_{E_\star, 0}^{\zeta_I} (f\,1_{F_{m',I}})(x)|^2\,d\sigma_{\star}(x)
\lesssim
\int_{F_{m,I}} \Big( \int_{F_{m', I}} \frac1{|x-z|^n}\, |f(z)|\,d\sigma_\star(z)\Big)^2\,d\sigma_{\star}(x)
\\
\lesssim
\int_{F_{m,I}} \Big( \frac1{\ell(I)^n}\int_{B(x,C\,\ell(I))\cap E_\star} \, |f(z)|\,d\sigma_\star(z)\Big)^2\,d\sigma_{\star}(x)
\lesssim
\int_{F_{m,I}} M_{E_\star} f(x)^2\,d\sigma_{\star}(x).
\end{multline*}
Finally if $\dist(F_{m, I}, F_{m',I})\ll\ell(I)$, we have that $F_{m,I}$ and $F_{m',I}$ are contained in respective faces which either lie in the same hyperplane, or else meet at an angle of $\pi/2$.  In the first case we may proceed as in the case $m=m'$. In the second case,  after a possible rotation of co-ordinates, we may view $F^j_{m}\cup F^j_{m'}$ as lying in a Lipschitz graph with Lipschitz constant 1, so that we may
estimate $\T_{E_\star, 0}^{\zeta_I}$ using an extension of the Coifman-McIntosh-Meyer theorem:
$$
\int_{F_{m,I}} |\T_{E_\star, 0}^{\zeta_I} (f\,1_{F_{m',I}})(x)|^2\,d\sigma_{\star}(x)
\lesssim
\int_{F_{m',I}} |f(x)|^2\,d\sigma_{\star}(x).
$$
Gathering all the possible cases we may conclude that
\begin{multline}\label{eqn:Cotlar-T3}
\int_{\Sigma\cap I} |\T_{E_\star, 0}^{\zeta_I} f(x)|^2\,d\sigma_{\star}(x)
\lesssim
\sum_{1\le m\le M_0} \int_{F_{m,I}} M_{E_\star} f(x)^2\,d\sigma_{\star}(x)
\\
\lesssim
 \sum_{I'\in \W_\Sigma: I'\cap \widetilde{\Delta}_{\star, I}\neq\mbox{\footnotesize \O}} \int_{I'\cap\Sigma}
M_{E_\star} f(x)^2\,d\sigma_{\star}(x). 
\end{multline}

We now gather \eqref{eqn:Cotlar-T1}, \eqref{eqn:Cotlar-T2} and \eqref{eqn:Cotlar-T3}  to get the following estimate for $S_I$ after using \eqref{decomp-K0}:
\begin{align}\label{SI-terms}
\textrm{II}_I
&=
\int_{\Sigma\cap I} |\T_{E_\star,0}f (x)|^2\,d\sigma_\star(x)
\\ \nonumber
&
\lesssim
\int_{\Sigma\cap I} |\T_{E_\star, 0,\xi_I}f (x)|^2\,d\sigma_\star(x)
+
\int_{\Sigma\cap I} |\T_{E_\star, 0,\zeta_I,\xi_I} f (x)|^2\,d\sigma_\star(x)
+
\int_{\Sigma\cap I} |\T_{E_\star, 0}^{\zeta_I}|^2\,d\sigma_\star(x)
\\ \nonumber
&\lesssim
\int_{\Sigma\cap I}M_{E_\star,p} f(x)^2\,d\sigma_\star(x)
+
\int_{Q_I}M_E (\T_{E_\star,0} f)(y)^2\,d\sigma(y)
\\ \nonumber
&\qquad\qquad
+
 \sum_{I'\in \W_\Sigma: I'\cap \widetilde{\Delta}_{\star, I}\neq\mbox{\footnotesize \O}} \int_{I'\cap\Sigma}
M_{E_\star} f(x)^2\,d\sigma_{\star}(x).
\end{align}
Note that since $1<p<2$ we have
\begin{equation}\label{S:1}
\sum_{I\in\W_\Sigma} \int_{\Sigma\cap I}M_{E_\star,p} f(x)^2\,d\sigma_\star(x)
\le
\int_{E_\star}M_{E_\star,p} f(x)^2\,d\sigma_\star(x)
\lesssim
\int_{E_\star}|f(x)|^2\,d\sigma_\star(x).
\end{equation}
On the other hand, recalling that $\widetilde{\F}=\F^*\cup \F^*_{||}$ is comprised of pairwise disjoint cubes, Lemmas \ref{lemma:decomp-bdt-sawtooth} and \ref{lemma:I-QI} then imply that
\begin{align}\label{S:2}
&\sum_{I\in\W_\Sigma}
\int_{Q_I}M_E (\T_{E_\star,0} f)(y)^2\,d\sigma(y)
\\ \nonumber
&\quad=
\sum_{Q\in\widetilde{\F}}\sum_{I\in\W_{\Sigma,Q}} \int_{Q_I}M_E (\T_{E_\star,0} f)(y)^2\,d\sigma(y)
+
\sum_{I\in\W_{\Sigma}^{\top}} \int_{Q_I}M_E (\T_{E_\star,0} f)(y)^2\,d\sigma(y)
\\ \nonumber
&\quad\lesssim \sum_{Q\in\widetilde{\F}}\int_{Q}M_E (\T_{E_\star,0} f)(y)^2\,d\sigma(y)
+
\int_{B_{Q_0}^*\cap E}M_E (\T_{E_\star,0} f)(y)^2\,d\sigma(y)
\\ \nonumber
&\quad\lesssim
\int_{E}M_E (\T_{E_\star,0} f)(y)^2\,d\sigma(y)
\\ \nonumber
&\quad
\lesssim
\int_{E} |\T_{E_\star,0} f(y)|^2\,d\sigma(y)
\\ \nonumber
&\quad
\lesssim
\int_{E_\star} |f(y)|^2\,d\sigma_\star(y),
\end{align}
where in the last estimate we have used \eqref{Tcal-E-star-bded} with $p=2$.

Finally, by the nature of the Whitney boxes (see \eqref{Whintey-4I}), we have that the family $\{2\,I\}_{I\in\W}$ has the bounded overlap property and therefore 
$$
\sum_{I\in\W_\Sigma}
\sum_{I'\in \W_\Sigma: I'\cap \widetilde{\Delta}_{\star, I}\neq\mbox{\footnotesize \O}} 1_{\Sigma\cap I'}
\lesssim
\sup_{I'\in\W_{\Sigma}}
\#\big\{
I\in \W_{\Sigma}: I'\cap \Delta_{\star, I}\neq\emptyset
\big\}
$$
which we claim that is uniformly bounded. Indeed, fix $I'\in\W_{\Sigma}$ and let $I_1$, $I_2\in \W_{\Sigma}$ with $I'\cap \widetilde{\Delta}_{\star, I_1}\neq\emptyset$ and $I'\cap \widetilde{\Delta}_{\star, I_2}\neq\emptyset$. Recall that $\dist(\widetilde{B}_{\star, I},E)\ge 2\,\diam(I)$ with $\widetilde{B}_{\star,I}=B(z_I,2\,\diam(I))$ and $z_I\in I\cap\Sigma$. This implies that $\ell(I_1)\approx\ell(I')\approx\ell(I_2)$ and also $\dist(I_1,I_2)\lesssim\ell(I_1)$. This easily gives our claim. Using this we conclude that
\begin{multline}\label{S:3}
\sum_{I\in\W_\Sigma}
\sum_{I'\in \W_\Sigma: I'\cap \widetilde{\Delta}_{\star, I}\neq\mbox{\footnotesize \O}} \int_{I'\cap\Sigma} M_{E_\star} f(x)^2\,d\sigma_{\star}(x)
\\
\lesssim
\int_{E_\star} M_{E_\star} f(x)^2\,d\sigma_{\star}(x)
\lesssim
\int_{E_\star} |f(x)|^2\,d\sigma_{\star}(x).
\end{multline} 

We now combine \eqref{SI-terms}, \eqref{S:1}, \eqref{S:2} and \eqref{S:3} to obtain that
$$
\textrm{II}=\sum_{I\in\W_\Sigma} \textrm{II}_I
\lesssim
\int_{E_\star} |f(x)|^2\,d\sigma_{\star}(x).
$$
This, \eqref{TE:2-terms}, and  \eqref{TE:A} give as desired that
$$
\int_{E_\star} |\T_{E_\star,0}f (x)|^2\,d\sigma_\star(x)
\lesssim
\int_{E_\star} |f(x)|^2\,d\sigma_{\star}(x),
$$
and the implicit constant does not depend on $\epsilon_0$. Hence,  $\T_{E_\star,0}:L^2(E_\star)\longrightarrow L^2(E_\star)$ with bounds that do not depend on $\epsilon_0$. Since $\T_{E_\star,0}f$ is a continuous function for every $f\in L^2(E_\star)$, we have that  $\T_{E_\star,0}f\big|_{E_\star}=T_{E_\star,\epsilon_0}f$ everywhere on $E_\star$. Thus, all these show that $T_{E_\star,0}:L^2(E_\star)\longrightarrow L^2(E_\star)$ uniformly in $\epsilon$. This in turn gives, by the aforementioned result of \cite{DS1}, that $E_\star$ is UR as desired, and the proof is complete.
\end{proof}

\end{document}